\documentclass[11pt,a4paper,openany,oneside]{book}
\usepackage{amssymb}

\usepackage[utf8]{inputenc}
\usepackage[left=2cm,right=2cm,top=2cm,bottom=2cm]{geometry}
\author{Jean Daniel Mukam}
\usepackage{amsmath,amsfonts,amssymb}
\usepackage{graphicx}
\usepackage{booktabs}
\usepackage{array} 
\usepackage[below]{placeins}
\usepackage{moreverb}  
\usepackage[pdfpagelabels,implicit=false]{hyperref}

\usepackage{fancyhdr}

\newtheorem{thm}[subsection]{Theorem}
\newtheorem{proof}[subsection]{Proof}
\newtheorem{lem}[subsection]{Lemma}
\newtheorem{cor}[subsection]{Corollary}

\newtheorem{nota}[subsection]{Notation}
\newtheorem{pro}[subsection]{Proposition}

\newtheorem{defn}[subsection]{Definition}
\newtheorem{rem}[subsection]{Remark}

\newtheorem{assumption}[subsection]{Assumptions}

\title{ Stochastic Calculus with Jumps Processes : Theory and Numerical Techniques \\(Master Thesis)}
\author{Jean Daniel Mukam (jean.d.mukam@aims-senegal.org)\\
African Institute for Mathematical Sciences (AIMS)\\
Senegal\\\\
{ Supervised by : Dr. Antoine Tambue}\\
{  AIMS-South Africa and University of Cape Town }\\
{ antonio@aims.ac.za}
}
\date{{ 13 June 2015}\\%
  {\small\it Submitted in Partial Fulfillment of 
    a Masters II at AIMS Senegal}\\
}

\begin{document}
\sloppy
\maketitle
\pagenumbering{roman}
\chapter*{Abstract}
\addcontentsline{toc}{chapter}{Abstract}

 \hspace{0.5cm} In this work we consider a stochastic differential equation (SDEs) with jump.  We prove the existence and the uniqueness of solution of this equation in the strong sense under global Lipschitz condition. Generally, exact solutions of SDEs are unknowns. The challenge is to approach them numerically. There exist several numerical techniques. In this thesis, we present the compensated stochastic theta method (CSTM) which is already developed in the literature.  We prove  that under global Lipschitz condition, the CSTM  converges strongly with standard order 0.5. We also investigated the stability behaviour of both  CSTM and  stochastic theta method (STM). Inspired by the tamed Euler scheme developed in \cite{Martin1}, we propose a new scheme for SDEs with jumps called compensated tamed Euler scheme. We prove that under non-global Lipschitz condition the compensated tamed Euler scheme converges strongly with standard order $0.5$. Inspired by \cite{Xia2}, we propose the semi-tamed Euler for SDEs with jumps under non-global Lipschitz condition and prove its strong convergence of order $0.5$. This latter result is helpful to prove the strong convergence of the tamed Euler scheme. We analyse the stability behaviours of both tamed and semi-tamed Euler scheme We present also some numerical experiments to illustrate our theoretical results.\\\\

 \textbf{Key words} : Stochastic differential equation, strong convergence, mean-square stability,  Euler scheme,  global Lipschitz condition, polynomial growth condition, one-sided Lipschitz condition. 
\vfill
\section*{Declaration}
I, the undersigned, hereby declare that the work contained in this essay is my original work,
and that any work done by others or by myself previously has been 
acknowledged and referenced accordingly.

\begin{figure}[hbtp]
\includegraphics[scale=0.1]{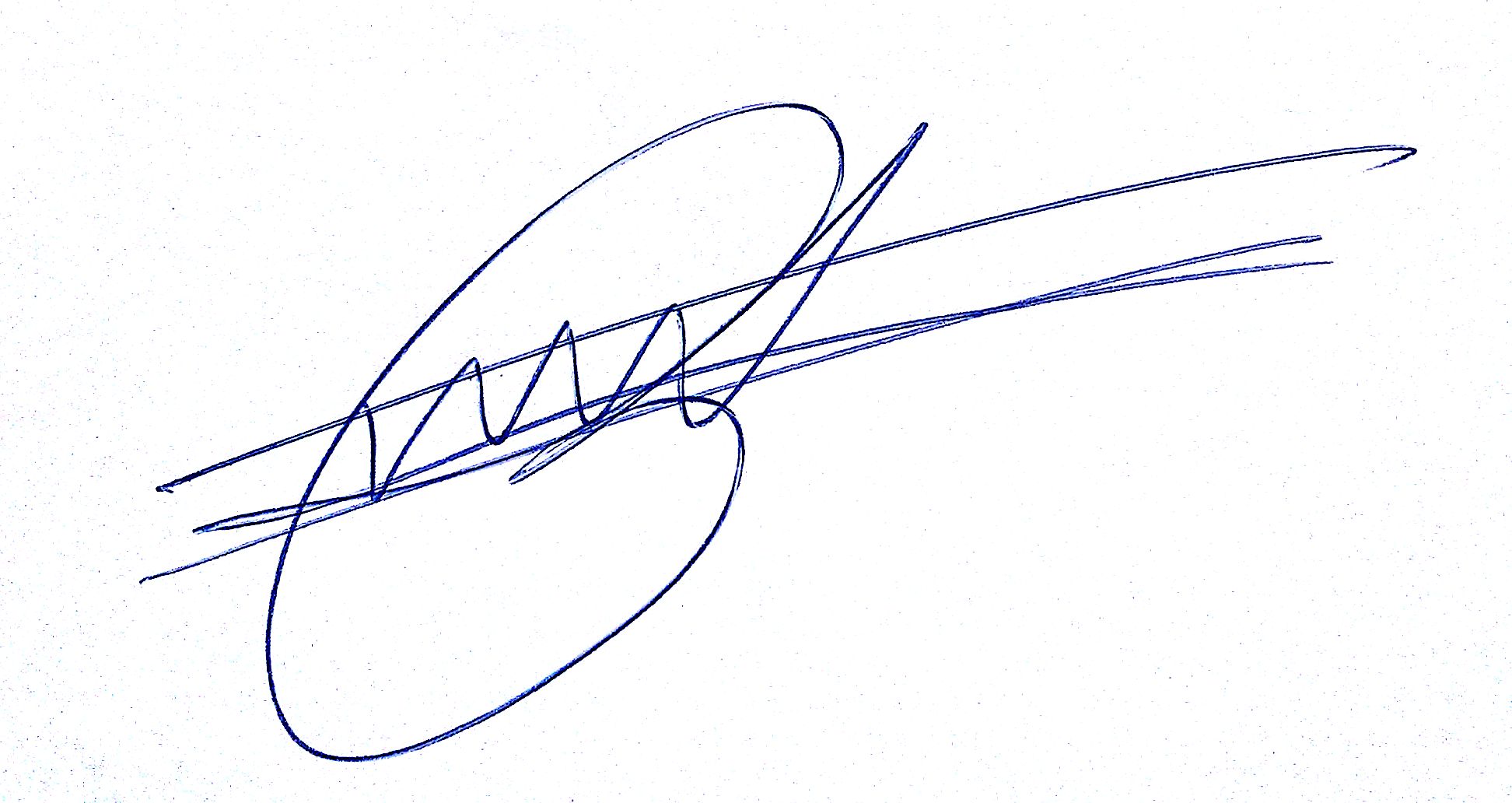}
\end{figure}
Jean Daniel Mukam, 25 May 2015
 
\tableofcontents

\pagenumbering{arabic}
\chapter*{INTRODUCTION}
 \hspace*{0.5cm}In many branches of sciences like finance,  economics,  biology,  engineering, 
 ecology one often encountered some problems  influenced by uncertainties. For example, in finance, the unpredictable nature of events such as markets crashes and  booms may have significant and sudden impact on the stock price fluctuations. Therefore, in order to have more realistic prediction of these phenomena, it is natural to model them with equations which involves the deterministic part and the random part including jump. The SDEs with jumps is the generalization of both deterministic part and random part with jumps.  SDEs with jumps have probability theory and stochastic process as prerequisites. We refer to \cite{Oksa}, \cite{Oks}, \cite{Phi} for  general notions in probability theory and stochastic process.
 
 \hspace*{0.5cm} In this thesis, under global Lipschitz condition, we prove the existence and uniqueness of solution of SDEs with jumps. We focus on the strong convergence of the compensated stochastic theta methods (CSTM) of these equations under global Lipschitz condition. In particular, we prove that CSTM have strong convergence of order $0.5$. We investigate the stability of  both CSTM and stochastic theta method (STM). For the linear case, we prove that under the assumption $\dfrac{1}{2}\leq \theta\leq 1$, CSTM holds the A-stability property. For the general nonlinear problem, we study the stability for $\theta=1$. In this case, when the drift coefficient have a negative one-sided Lipschitz coefficient, the diffusion coefficient and the jump coefficient satisfy the global Lipschitz condition, we prove that STM reproduce stability under certain step-size and the CSTM is stable for any step-size.
 
 \hspace*{0.5cm} Most phenomena are modelised by SDEs with jumps where the drift coefficient is one-sided Lipschitz  and satisfies the polynomial growth condition. For such equations, it is proved in \cite{Martin3} that Euler explicit method fails to converge strongly to the exact solution while Euler implicit method converges strongly, but requires much computational efforts. Recently, a new  explicit and efficient method was developed in \cite{Martin1} called tamed Euler scheme. In \cite{Martin1}, the authors proved that the tamed Euler converges strongly with order 0.5 to the exact solution of SDEs under non-global Lipschitz condition. In this thesis, we extend the tamed Euler scheme by introducing a compensated tamed Euler scheme for SDEs with jumps. We prove that this scheme converges strongly with standard order $0.5$.  We also extend the semi-tamed Euler developed in \cite{Xia2} and we prove that this scheme converge strongly with order $0.5$ for SDEs with jumps. As a consequence of this latter result, we prove the strong convergence of the tamed Euler scheme for SDEs with jumps. The stability analysis of both tamed Euler and  semi-tamed Euler are done in this thesis.
 
 \hspace*{0.5cm}This thesis is organized as follows. In chapter 1, we recall some basic notions in probability theory and stochastic process. Chapter 2 is devoted to the proof of the existence and uniqueness of SDEs with jumps under global Lipschitz condition. In chapter 3, we focus on the strong convergence of the CSTM and the stability analysis of both CSTM and STM. In chapter 4, under non-global Lipschitz condition we investigate the strong convergence of the compensated tamed Euler scheme. In chapter 5, under non-global Lipschitz condition, we investigate the strong convergence and stability of both semi-tamed Euler scheme and tamed Euler scheme.   Our theoretical results are illustrated by numerical examples at the end of chapter 3, chapter 4 and chapter 5.
\label{chp: chap1 Basic Notions in probability theory and stochastic process}

\chapter{Basic notions in probability theory and stochastic process}

\section{ Basic notions in probability theory }
\hspace{0.5cm} In this chapter, we present some basic concepts and results in probability theory and stochastic process useful to understand the notion of stochastic differential equations.

More details for this chapter can be found in \cite{Oksa}, \cite{Oks} and \cite{Phi}.

\subsection{\textbf{Basic notions in probability theory}}
\begin{defn} \textbf{[$\sigma$- algebra]}

Let $\Omega$ be a non-empty set. 
\begin{enumerate}
\item A $\sigma$-algebra (or $\sigma$-field) $\mathcal{F}$ on $\Omega$ is a family of subsets of $\Omega$ satisfying
\begin{enumerate}
\item [$(i)$] $\Omega \in \mathcal{F}$.
\item [$(ii)$] $\forall A\in \mathcal{F}$, $A^c\in\mathcal{F}$.
\item [$(iii)$] If $(A_i)_{i\in I}$ is a countable collection of set in $\mathcal{F}$, then $\cup_{i\in I} A_i\in \mathcal{F}$.
\end{enumerate}
\item Let $\mathcal{F}_1$ and $\mathcal{F}_2$ be two $\sigma$-algebra on $\Omega$. $\mathcal{F}_1$ is said to be a sub-$\sigma$-algebra of $\mathcal{F}_2$ if $\mathcal{F}_1\subset\mathcal{F}_2$. 
\end{enumerate}
\end{defn}

\begin{rem}
\begin{enumerate}
\item 
Given any family $\mathcal{B}$ of subset of $\Omega$, we denote by 
\begin{align*}
\sigma(\mathcal{B}):=\cap\{\mathcal{C}: \hspace{0.2cm}\mathcal{C}, \hspace{0.2cm}\sigma- \text{algebra of}\hspace{0.2cm} \Omega, \hspace{0.2cm}\mathcal{B}\subset \mathcal{C}\}
\end{align*}
the smallest $\sigma$-field of $\Omega$ containing $\mathcal{B}$, $\sigma(\mathcal{B})$ is called the $\sigma$-field generated by $\mathcal{B}$.

When $\mathcal{B}$ is a collection of all open sets of a topological space $\Omega$, $\sigma(\mathcal{B})$ is called the Borel $\sigma$-algebra on $\Omega$ and the elements of $\sigma(\mathcal{B})$ are called Borel sets.
\item If $ X: \Omega \longrightarrow\mathbb{R}^n$ is a function, then the $\sigma$-algebra generated by $X$ is the smallest $\sigma$-algebra on $\Omega$ containing all the sets of the form 
\begin{align*}
\{X^{-1}(U) : U\subset \mathbb{R}^n,  \hspace{0.5cm}\text{open}\}.
\end{align*}
\end{enumerate}
\end{rem}

\begin{defn} \textbf{[Probability measure]}.

Let $\mathcal{F}$ be a $\sigma$-field on $\Omega$. A probability measure is an application $\mathbb{P} : \mathcal{F}\longrightarrow [0,1]$ satisfying
\begin{enumerate}
\item [$(i)$] $\mathbb{P}(\Omega)=1-\mathbb{P}(\emptyset)=1$.
\item [$(ii)$] If $(A_i)_{i\in I}$ is a countable collection of elements of $\mathcal{F}$ pairwise disjoints, then 

$\mathbb{P}(\cup_{i\in I}A_i)=\sum\limits_{i\in I}\mathbb{P}(A_i)$.
\end{enumerate}
\end{defn}

\begin{defn} \textbf{[Probability space]}.

Let $\Omega$ be a non-empty set, $\mathcal{F}$ a $\sigma$-field on $\Omega$ and $\mathbb{P}$ a probability measure on $\mathcal{F}$.

 The triple $(\Omega, \mathcal{F}, \mathbb{P})$ is called a probability space.
\end{defn}

\begin{defn}\textbf{[Negligeable set]}

\begin{enumerate}
\item[$(i)$] Given a probability space $(\Omega, \mathcal{F}, \mathbb{P})$, $A\subset\Omega$ is said to be $\mathbb{P}$-null or negligeable if $\mathbb{P}(A)=0$
\item [$(ii)$] A property is said to be true almost surely (a.s) if the set on which this property is not true is negligeable.
\end{enumerate}
\end{defn}

\begin{defn}\textbf{[Measurability and random variable]}
\begin{enumerate}
\item [$(i)$]
Let $(\Omega, \mathcal{F}, \mathbb{P})$ and $(\Omega', \mathcal{F}', \mathbb{P}')$ be two probability spaces. A function $X : \Omega \longrightarrow \Omega'$ is said to be $\mathcal{F}$-measurable if and only if 
\begin{eqnarray*}
X^{-1}(U) := \{ \omega\in \Omega : X(\omega) \in U\}\subset \mathcal{F},  \hspace{0.5cm} \forall\hspace{0.2cm} U\in \mathcal{F}'
\end{eqnarray*}
\item [$(ii)$] A random variable $X$ is a function $X : \Omega \longrightarrow \Omega'$ $\mathcal{F}$-measurable.
\item [$(iii)$] If $\Omega'=\mathbb{R}$, then $X$ is called a real random variable.
\item [$(iv)$] If $\Omega'=\mathbb{R}^n$, $n>1$ then $X$ is called a vector random variable.
\end{enumerate}
\end{defn}

In the following, unless otherwise state, $(\Omega, \mathcal{F}, \mathbb{P})$ denote a probability space and $X$ a random variable,  $X :\Omega \longrightarrow \mathbb{R}^n$.

 \begin{rem}.
 
 Every random variable induces a probability measure on $\mathbb{R}^n$ denoted $\mu_X$ and define by
 
  $\mu_X(B) :=\mathbb{P}(X^{-1}(B))$, $\forall\, B$ open set of $\mathbb{R}^n$. $\mu_X$ is called the distribution function of $X$. 
 \end{rem}

\begin{defn}\textbf{[Expected value]}
 \begin{enumerate}
 \item [$(i)$]

 If $X$ is a random variable such that $\int_{\Omega}||X(\omega)||d\mathbb{P}(\omega)<\infty$ almost surely, the quantity 
 \begin{align*}
 \mathbb{E}(X) :=\int_{\Omega}X(\omega)d\mathbb{P}(\omega) =\int_{\mathbb{R}^n}d\mu_X(x)
 \end{align*}
 is called the expected value of $X$,  where $||.||$ denote the euclidean norm on $\mathbb{R}^n$.
 
\item [$(ii)$] In general, if $f : \mathbb{R}^n\longrightarrow \mathbb{R}^m$ is measurable and $\int_{\Omega}||f(X(\omega))||d\mathbb{P}(\omega)<\infty $ almost surely, then the qauntity $\mathbb{E}(f(X))$ define by
\begin{align*}
\mathbb{E}(f(X)) :=\int_{\Omega}f(X(\omega))d\mathbb{P}(\omega)=\int_{\mathbb{R}^n}f(x)d\mu_X(x)
\end{align*}
is called expected value of $f(X)$.
 \end{enumerate}
\end{defn}

\begin{defn}\textbf{[Independent random variables]}

Let  $(\Omega, \mathcal{F}, \mathbb{P})$ be a probability space. 
\begin{enumerate}
\item Two elements $A$ and $B$ of $\mathcal{F}$ are independent if 
\begin{align*}
\mathbb{P}(A\cap B)=\mathbb{P}(A)\cap\mathbb{P}(B).
\end{align*}
\item Two random variables $X_1$ and $X_2$ of  $(\Omega, \mathcal{F}, \mathbb{P})$ are independent if for every choice of different borel sets $B_1$ and $B_2$ the following holds :
\begin{align*}
\mathbb{P}(X_1\in B_1, X_2\in B_2)=\mathbb{P}(X_1\in B_1)\times\mathbb{P}(X_2\in B_2).
\end{align*}
\end{enumerate}
\end{defn}

The following proposition is from \cite{Oksa}.
\begin{pro}
Two random variables $X_1$ and $X_2$ are  independent if and only if for  any   measurable positive functions $f_1$ and $f_2$,  the following equality holds
\begin{align*}
\mathbb{E}(f_1(X_1)f_2(X_2))=\mathbb{E}(f_1(X_1)) \mathbb{E}(f_2(X_2)).
\end{align*}
\end{pro}

\begin{defn}\textbf{[Conditional probability]}

 For any event $A$ such that $P(A)>0$, the conditional probability on $A$ is the probability measure define by :
 \begin{align*}
 \mathbb{P}(B/A) :=\dfrac{\mathbb{P}(A\cap B)}{\mathbb{P}(A)},\hspace{0.3cm} \forall B\in \mathcal{F}.
\end{align*}  
\end{defn}

\subsection{Conditional expectation} 
 \begin{defn}
 Let $X$ be a random variable such that $\int_{\Omega}|X(\omega)|d\mathbb{P}(\omega)<\infty$ almost surely. Let $\mathcal{G}$ a sub $\sigma$-algebra of $\mathcal{F}$. The conditional expectation of $X$ relative  to the $\sigma$-algebra $\mathcal{G}$ is a random variable denoted by $\mathbb{E}(X/\mathcal{G})$ satisfying 
 \begin{enumerate}
 \item [$(i)$] $ \mathbb{E}(X/\mathcal{G})$ is $\mathcal{G}$-measurable.
 \item [$(ii)$] $\int_G\mathbb{E}(X/\mathcal{G})d\mathbb{P}=\int_GXd\mathbb{P}, \hspace{0.5cm}\forall \hspace{0.2cm} G\in \mathcal{G}$.
 \end{enumerate}
 \end{defn}

In the litterature, $\mathbb{E}(X/\mathcal{G})$ is called the projection of $X$ upon $\mathcal{G}$.

The proof of the following theorem can be seen in \cite{Phi}.
 \begin{pro}
 \begin{enumerate}
 \item [$(i)$] $ \mathbb{E}(\mathbb{E}(X/\mathcal{G}))=\mathbb{E}(X)$.
 \item [$(ii)$] If $X$ is $\mathcal{G}$-measurable, then $\mathbb{E}(X/\mathcal{G})=X$.
 \item [$(ii)$] $\mathbb{E}((X+Y)/\mathcal{G})=\mathbb{E}(X/\mathcal{G})+\mathbb{E}(Y/\mathcal{G})$.
 \item[$(ii)$] If $\mathcal{G}\subset \mathcal{G}'$ then $\mathbb{E}(X/\mathcal{G}')=\mathbb{E}.(\mathbb{E}(X/\mathcal{G})/\mathcal{G}')$.
 \item[$(iv)$] If $\sigma(X)$ and $\mathcal{G}$ are independent, then $\mathbb{E}(X/\mathcal{G})=\mathbb{E}(X)$.
 \item[$(v)$] If $X\leq Y$ a.s, then $\mathbb{E}(X/\mathcal{G})\leq \mathbb{E}(Y/\mathcal{G})$.
 \item[$(vi)$] If $X$ is $\mathcal{G}$ measurable, then $\mathbb{E}(XY/\mathcal{G})=X\mathbb{E}(Y/\mathcal{G})$.
 \end{enumerate}
 \end{pro}

\subsection{Convergence of random variables}.
\begin{defn}
Let $p\in[1,\infty)$, we denote by $\mathbb{L}^p(\Omega, \mathbb{R}^n)$ the  equivalence class  of  measurable functions $X :\Omega : \longrightarrow \mathbb{R}^n $, $\mathcal{F}_t$-measurable  such that 
\begin{align*}
||X||^p_{\mathbb{L}^p(\Omega,\mathbb{R}^n)} := \mathbb{E}(||X||^p) =\int_{\Omega}||X(\omega)||^pd\mathbb{P}(\omega)<+\infty.
\end{align*}
Let $(X_n)\subset \mathbb{L}^p(\Omega, \mathbb{R}^n)$ be a sequence of random variables and $X\in \mathbb{L}^p(\omega, \mathbb{R}^n)$ a random variable. Let 
\begin{align*}
N:=\{\omega : \lim_{n\longrightarrow \infty}X_n(\omega)=X(\omega)\}
\end{align*}
\end{defn}

\begin{enumerate}
\item 
$(X_n)$ converges   to $X$ almost surely if $N^c$ is negligeable.
\item  $(X_n)$ converges in probability to  $X$ if 
\begin{align*}
\forall\hspace{0.2cm}\epsilon >0\hspace{0.3cm}\lim_{n\longrightarrow\infty}\mathbb{P}(||X_n-X||>\epsilon)=0.
\end{align*}
\item  $(X_n)$ converges in $\mathbb{L}^p$ to  $X$  if
\begin{align*}
 \lim_{n\longrightarrow +\infty}\mathbb{E}(||X_n-X||^p)=0.
 \end{align*}
\end{enumerate}

\begin{defn} : \textbf{[Frobenius norm]}

The Frobenius norm of a  $m\times n$ matrix $A=(a_{ij})_{1\leq i\leq n; 1\leq j\leq m}$  is defined by
\begin{align*}
||A||:=\sqrt{\sum_{j=1}^n\sum_{i=1}^n|a_{ij}|^2}.
\end{align*}
\end{defn}
\begin{rem}
Frobenius norm and euclidean norm are the same for vectors.
\end{rem}

\begin{pro}\label{ch1Minkowski}\textbf{[Minkowski inequality : Integral form]}.

Let $1\leq p<+\infty$ and let $(X, \mathcal{A}, dx)$ and $(Y, \mathcal{B}, dy)$ be $\sigma$-finite measures spaces. Let $F$ be a measurable function on the product space $X\times Y$. Then 
\begin{eqnarray*}
\left(\int_X\left|\int_YF(x,y)dy\right|^pdx\right)^{1/p}\leq \int_Y\left(\int_X|F(x,y)|^pdx\right)^{1/p}dy.
\end{eqnarray*}
The above inequality can be writen as
\begin{eqnarray*}
\left\|\int_YF(.,y)dy\right\|_{L^p(X,\mathcal{A}, dx)}\leq \int_Y||F(.,y)||_{L^p(X, \mathcal{A}, dx)}dy.
\end{eqnarray*}
\end{pro}
\begin{pro}
\textbf{[Gronwall inequality] :  Continous form}

Let $a(t)$ and $b(t)$  be two continuous and positives functions defines on $\mathbb{R}_+$ such that 
\begin{align*}
a(t)\leq b(t)+c\int_0^ta(s)ds,  \hspace{0.5cm} \forall\, t\in \mathbb{R}_+,
\end{align*}
then 
\begin{align*}
a(t)\leq b(t)+c\int_0^tb(s)e^{c(t-s)}ds, \hspace{0.5cm}\forall\, t\in \mathbb{R}_+.
\end{align*}
\textbf{[Gronwall inequality] : Discrete form}.

Let $\theta$ and $K$ be two constants and  $(v_n)$ be a sequence satisfying :
\begin{align*}
v_{n+1}\leq (1+\theta)v_n+K,
\end{align*}
then
\begin{align*}
v_n\leq e^{n\theta}v_0+K\dfrac{e^{n\theta}-1}{e^{\theta}-1}.
\end{align*}
\end{pro}

\begin{proof} \cite{Gron}.
\end{proof}

\begin{lem} \textbf{[Borel Cantelli]}

Let $(A_n)_{n\in \mathbb{N}}$ be a family of subset of $\Omega$.
\begin{enumerate}
\item If $\sum\limits_{n\in\mathbb{N}}\mathbb{P}(A_n)<\infty$, then $\mathbb{P}\left[\limsup\limits_{n\longrightarrow\infty}A_n\right]=0$.
\item If the events $(A_n)$ are independent and $\sum\limits_{n\in\mathbb{N}}\mathbb{P}(A_n)=0$, then 
$\mathbb{P}\left[\limsup\limits_{n\longrightarrow \infty}A_n\right]=1$.
\end{enumerate}
\end{lem}
\begin{proof} \cite{Phi}.
\end{proof}

\section{Stochastic processes}
\begin{defn}
Let $(\Omega, \mathcal{F}, \mathbb{P})$ be a probability space. A family $(\mathcal{F}_t)_{t\geq 0}$ of sub $\sigma$-algebra of $\mathcal{F}$ is called filtration if $\mathcal{F}_s\subset \mathcal{F}_t$, $\forall\hspace{0.2cm} 0\leq s\leq t$.

If $(\mathcal{F}_t)$ is such that $\mathcal{F}_t=\cap_{t>s}\mathcal{F}_s$, then $(\mathcal{F}_t)_{t\geq 0}$ is said to be right continuous.
\end{defn}

\begin{defn}
A stochastic process is a family of vector  random variables $(X_t)_{t\geq 0}$. That is for all $t> 0$, the  application

$\begin{array}{cccc}
X_t & : \Omega &\longrightarrow &\mathbb{R}^n\\
& w & \longmapsto& X_t(\omega) 
\end{array}$
is measurable.

If $(X_t)_{t\geq 0}$ is a stochastic process, then for all $t\geq 0$, the application $t\longmapsto X_t$ is called sample path.
\end{defn}
 \begin{defn}
 Let $(\mathcal{F}_t)$ be a filtration on $(\Omega, \mathcal{F}, \mathbb{P})$. A stochastic process $(X_t)$ is said to be $\mathcal{F}_t$-adapted if $\forall\hspace{0.2cm} t\geq 0$ $X_t$ is $\mathcal{F}_t$-measurable.
 \end{defn}

\begin{defn}\textbf{[Martingale]}
 Let $(\mathcal{F}_t)_{t\geq 0}$ be a filtration on $(\Omega, \mathcal{F}, \mathbb{P})$. A stochastic process $(M_t)_{t\geq 0}$ is called $\mathcal{F}_t$- martingale if the following properties holds
 \begin{enumerate}
 \item [$(i)$] $(M_t)$ is $\mathcal{F}_t$-adapted.
 \item [$(ii)$] $\mathbb{E}||M_t||<\infty $, $\forall\hspace{0.1cm} t\geq 0$.
 \item[$(iii)$] $\mathbb{E}(M_t/\mathcal{F}_s)=M_s$, $\forall\hspace{0.1cm} 0\leq s\leq t$.
 \end{enumerate}
  \end{defn}

\begin{rem}
\begin{enumerate}
\item [$(i)$] If  the condition $(iii)$ of the previous definition is replaced by $\mathbb{E}(M_t/\mathcal{F}_s)\geq M_s$, 
$\forall\hspace{0.1cm} 0\leq s\leq t$, then  $(M_t)$ is called submartingale.
\item [$(ii)$] If  the condition $(iii)$ of the previous definition is replaced by $\mathbb{E}(M_t/\mathcal{F}_s)\leq M_s$, $\forall\hspace{0.1cm} 0\leq s\leq t$, then $(M_t)$ is called supermartingale.
\item [$(iii)$] A positive submartingale is a submartingale $(X_t)_{t\geq 0}$ satisfying $X_t\geq 0$ for all $t\geq 0$.
\end{enumerate}
\end{rem} 

\begin{defn}\textbf{[Predictable process]}

 Let $(\mathcal{F}_t)_{t\geq 0}$ be a filtration on $(\Omega, \mathcal{F}, \mathbb{P})$. A stochastic process $(X_t)_{t\geq 0}$ is called $\mathcal{F}_t$- predictable process if 
 for all $t> 0$, $X_t$ is  measurable with respect to the  $\sigma$-algebra generated by $\{X_s, \hspace{0.2cm}s<t\}$.
 \end{defn}

\begin{pro}

  Let $M=(M_t)$ be a  submartingale. Then for $1<p<\infty$, we have 
  \begin{enumerate}   
   \item [$(i)$] \textbf{Markov's inequality}
  \begin{eqnarray*}
  \mathbb{P}\left(\sup_{0\leq s \leq t}||M_t||\geq \alpha\right)\leq\dfrac{\mathbb{E}(||M_t||)}{\alpha},\hspace{0.6cm} \forall\, \alpha>0.
  \end{eqnarray*}
  \item [$(ii)$]  \textbf{Doop's maximal inequality}
  \begin{eqnarray*}
  \mathbb{E}\left[\left(\sup_{0\leq s \leq t}||M_t||\right)^p\right]^{1/p}\leq \dfrac{p}{p-1}\mathbb{E}\left[||M_t||^p\right]^{1/p}.
  \end{eqnarray*}
  \end{enumerate}
  \end{pro}
  \begin{proof} \cite{Phi}.
  \end{proof}

  \begin{defn}\textbf{[Wiener process or Brownian motion]}
  
  Let $(\Omega, \mathcal{F}, \mathbb{P})$ be a probability space and $(\mathcal{F}_t)_{t\geq 0}$ a filtration on this space.
  A $\mathcal{F}_t$-adapted stochastic process $(W_t)_{t\geq 0}$ is called Wiener process or  Brownian motion if : 
  \begin{enumerate}
  \item[$(i)$] $W_0=0$.
  \item[$(ii)$] $ t  \longmapsto W_t$ is almost surely continous.
  \item[$(iii)$] $(W_t)_{t\geq 0}$ has independent increments $($i.e $W_t-W_s$ is independent of $W_r,\hspace{0.2cm} r\leq s)$.
  \item [$(iv)$] $W_t-W_s\rightsquigarrow \mathcal{N}(0, t-s)$, for $0\leq s \leq t$. Usually, this property is called stationarity.
  \end{enumerate}
  \end{defn}
  
  \begin{pro}
  If $(W_t)$ is an $\mathcal{F}_t$- Brownian motion,   then the following process are $\mathcal{F}_t$- martingales
  \begin{enumerate}
  \item [$(i)$] $W_t$.
  \item[$(ii)$] $W_t^2-t$.
  \item [$(iii)$] $\exp\left(\gamma W_t-\gamma^2\dfrac{t}{2}\right),\hspace{0.3cm}\forall\; \gamma\in \mathbb{R}$. 
\end{enumerate}   
  \end{pro}

  \begin{proof} Let $0\leq s\leq t$, then 
  \begin{enumerate}
  \item [$(i)$]
  \begin{eqnarray*}
   \mathbb{E}(W_t/\mathcal{F}_s)&=&\mathbb{E}(W_t-W_s+W_s/\mathcal{F}_s)\\
   &=&W_s+\mathbb{E}(W_t-W_s/\mathcal{F}_s) \hspace{0.1cm}\text{since} \hspace{0.2cm} W_s \hspace{0.1cm} is \hspace{0.1cm} \mathcal{F}_s-\text{measurable}\\
  & =& W_s+ \mathbb{E}(W_t-W_s)\hspace{0.1cm} (\text{since the increments are independents }) \\
  &=&W_s \hspace{0.3cm}(\text{since} \hspace{0.1cm} W_t-W_s\rightsquigarrow\mathcal{N}(0,t-s)).
  \end{eqnarray*}
  \item[$(ii)$] 
  \begin{eqnarray*}
  \mathbb{E}(W_t^2-t/\mathcal{F}_s)&=&\mathbb{E}(W_t^2+W_s^2-2W_sW_t+2W_sW_t-W_s^2/\mathcal{F}_s)-t\\
  &=&\mathbb{E}((W_t-W_s)^2/\mathcal{F}_s)+W_s\mathbb{E}((2W_t-W_s)/\mathcal{F}_s)-t \\
  & & \hspace{0.1cm}(\text{since} \hspace{0.2cm} W_s \hspace{0.1cm} is \hspace{0.1cm} \mathcal{F}_s-\text{measurable  })\\
  &=&\mathbb{E}((W_t-W_s)^2)+ W_s\mathbb{E}(W_t-W_s)+W_s\mathbb{E}(W_t/\mathcal{F}_s)-t\\
  & &(\text{since the increments are independents})\\
  &=&t-s+0+W_s^2-t\hspace{0.2cm} \text{since} \hspace{0.2cm} W_t-W_s\rightsquigarrow\mathcal{N}(0, t-s)\\
  &=&W_s^2-s.
  \end{eqnarray*}
  \end{enumerate}
  \item[$iii)$] Using the same argument as above, we have :
  \begin{eqnarray*}
  \mathbb{E}(e^{\gamma W_t}/\mathcal{F}_s)&=& e^{\gamma W_s}\mathbb{E}(e^{\gamma(W_t-W_s)}/\mathcal{F}_s)\\
  &=& e^{\gamma W_s}\mathbb{E}(e^{\gamma W_{t-s}})\\
  &=&e^{\gamma W_s}\int_{-\infty}^{+\infty}\dfrac{e^{-x^2/2(t-s)}}{\sqrt{2\pi(t-s)}}dx\\
  &=&e^{\gamma W_s}e^{\gamma^2(t-s)/2}=e^{\gamma W_s+\gamma^2(t-s)/2}.
  \end{eqnarray*}
  Therefore, 
  \begin{eqnarray*}
  \mathbb{E}\left(exp\left(\gamma W_t-\gamma^2\dfrac{t}{2}\right)/\mathcal{F}_s\right)&=&\mathbb{E}(e^{\gamma W_t}/\mathcal{F}_s)e^{-\gamma^2t/2}\\
  &=&e^{\gamma W_s+\gamma^2(t-s)/2}e^{-\gamma^2t/2}\\
  &=& \exp(\gamma W_s-{\gamma^2 s}/2).
  \end{eqnarray*}
  \end{proof}

  The following proposition is from \cite{Oks}.
  \begin{pro}
 Almost all sample paths of a Brownian motion are nowhere differentiable.
  \end{pro}

\begin{defn}
 Let $(\Omega, \mathcal{F}, \mathbb{P})$ be a probability space and $(\mathcal{F}_t)$ a filtration on this space.
Let $(S_k)_{k\geq 1}$ be an $\mathcal{F}_t$-adapted stochastic process on $(\Omega, \mathcal{F}, \mathbb{P})$ with $0\leq S_1(\omega)\leq S_2(\omega)\leq ...$ for all $k\geq 1$ and $\omega \in \Omega$. The $\mathcal{F}_t$- adapted process $N=(N_t)_{t\geq 0}$ defined by :
\begin{eqnarray*}
N_t:=\sum_{k\geq 1}\mathbf{1}_{\{S_k\leq k\}}
\end{eqnarray*}
is called counting process with jump times $S_k$.
\end{defn}
 \begin{defn}
  Let $(\Omega, \mathcal{F}, \mathbb{P})$ be a probability space and $(\mathcal{F}_t)$ a filtration on this space.
 A counting process $(N)_t$,  $\mathcal{F}_t$- adapted  is called poisson process of intensity $\lambda >0$ if :
 \begin{enumerate}
 \item[$(i)$]  $N_0=0$.
 \item[$(ii)$]   $\forall\; 0\leq t_0<t_1<...<t_n$, the random variables $\{N_{t_j}-N_{t_{j-1}} \hspace{0.3cm} 1\leq j\leq n\}$ are independent.
 \item[ $(iii)$] For $0\leq s\leq t$, $N_t-N_s\approx N_{t-s}$, where $\approx$ stand for the equality in probability law. 
 \item [$(iv)$] For all $t>0$, $N_t$ follows a poisson law with parameter $\lambda t$ (and we denote $N_t\leftrightsquigarrow \mathcal{P}(\lambda t)$). That is 
 \begin{eqnarray*}
 \mathbb{P}(N_t=k)=e^{-\lambda t}\dfrac{(\lambda t)^k}{k!},\hspace{0.5cm} k\in \mathbb{N}.
 \end{eqnarray*}
 \end{enumerate}
 \end{defn} 
 \begin{defn} \textbf{[Compound poisson process]} 
 
Let $(Z_n)$ be a sequence of discrete independent identically distributed random variables with probability law $\nu_Z$. Let   $N=(N_t)$ be a poisson process with parameter $\lambda$. Let's assume that $(N_t)$ and $(Z_n)$ are independent. A compound poisson process with intensity $\lambda>0$ with a jump law $\nu_Z$ is a $\mathcal{F}_t$- adapted  stochastic process $(Y_t)$ defined by :
 \begin{eqnarray}
 Y_t: =\sum_{k=1}^{N_t}Z_k.
 \end{eqnarray}
  \end{defn}

  \begin{defn}\textbf{[Compensated poisson process]}
  
A compensated poisson process associated to a poisson process $N$ with intensity $\lambda$ is a stochastic process $\overline{N}$ defined by :
\begin{eqnarray*}
\overline{N}(t) := N(t)-\lambda t.
\end{eqnarray*} 
\end{defn}
\begin{pro}
\label{ch1quadratic}
 Let $(\Omega, \mathcal{F}, \mathbb{P})$ be a probability space and $(\mathcal{F}_t)$ a filtration on this space.
 
 If $(N)_t$ is a $\mathcal{F}_t$- adapted poisson process with intensity $\lambda$, then
\begin{enumerate}
\item $\overline{N}$ is a $\mathcal{F}_t$- adapted  martingale.
\item $\mathbb{E}(\overline{N}(t+s)-\overline{N}(t))=0$.
\item $\mathbb{E}[\overline{N}(t+s)-\overline{N}(t)]^2=\lambda s, \hspace{0.5cm} \forall\; t,s \geq 0$.
\item $\overline{N}_t^2-\lambda t$ is a martingale.
\end{enumerate}
\end{pro}

\begin{proof}
\begin{enumerate}
\item Let $\leq s\leq t$, then
\begin{eqnarray*}
\mathbb{E}(\overline{N}_t/\mathcal{F}_s)&=& \mathbb{E}(\overline{N}_t-\overline{N}_s+\overline{N}_s/ \mathcal{F}_s)\\
&=&\mathbb{E}(\overline{N}_t-\overline{N}_s/\mathcal{F}_s)+\overline{N}_s\\
&=&\mathbb{E}(N_t-N_s-\lambda t+\lambda s/\mathcal{F}_s)+N_s-\lambda s\\
&=&\mathbb{E}(N_t-N_s)-\lambda t+\lambda s +N_s-\lambda s \hspace{0.1cm}\\
& & \text{since the increments of the poisson process are independents}\\
&=& \lambda (t-s)-\lambda t+ N_s\hspace{0.2cm}\text( since\hspace{0.1cm} N_t-N_s\rightsquigarrow\mathcal{P}(\lambda(t-s)))\\
&=& N_s-\lambda s\\
&=& \overline{N}(s).
\end{eqnarray*}

\item \begin{eqnarray*}
\mathbb{E}(\overline{N}(t+s)-\overline{N}(t))&=& \mathbb{E}(N(t+s)-N(t)-\lambda s)\\
&=& \lambda (t+s-t)-\lambda s= 0.
\end{eqnarray*}
\item \begin{eqnarray*}
[\overline{N}(t+s)-\overline{N}(t)]^2&=&[N(t+s)-N(t)-\lambda s]^2\\
&=&[N(t+s)-N(t)]^2+\lambda^2s^2-2\lambda s(E(t+s)-N(t)).
\end{eqnarray*}
Since  $N(t)\rightsquigarrow \mathcal{P}(\lambda t)$, using the relation $\mathbb{E}(N_t)=var(N_t)= \lambda t$, it follows that :
\begin{eqnarray*}
\mathbb{E}[\overline{N}(t+s)-\overline{N}(s)]^2&=&\lambda(t+s-t)+\lambda^2(t+s-t)^2+\lambda^2s^2-2\lambda s(\lambda s)=\lambda s. 
\end{eqnarray*}
\item
\begin{eqnarray*}
 \mathbb{E}[\overline{N}^2_t-\lambda t/\mathcal{F}_s]&=& \mathbb{E}[\overline{N}^2_t/\mathcal{F}_s]-\lambda t\\
&=&\mathbb{E}[\overline{N}^2_t+\overline{N}^2_s-2\overline{N}_t\overline{N}_s+2\overline{N}_t\overline{N}_s-
\overline{N}_s^2/\mathcal{F}_s]-\lambda t\\
&=&\mathbb{E}[(\overline{N}_t-\overline{N}_s)^2/\mathcal{F}_s]+\mathbb{E}[\overline{N}_s(\overline{N}_t-\overline{N}_s)/\mathcal{F}_s]+\mathbb{E}[\overline{N}_t\overline{N}_s/\mathcal{F}_s].
\end{eqnarray*}
Using the fact that $\overline{N}_t$ have independent increments and using the first part of the theorem, it follows that 
\begin{eqnarray*}
 \mathbb{E}[\overline{N}^2_t-\lambda t/\mathcal{F}_s]&=&\mathbb{E}[(\overline{N}_t-\overline{N}_s)^2]+\overline{N}_s\mathbb{E}(\overline{N}_t-\overline{N}_s)+\mathbb{E}[\overline{N}_t-\overline{N}_s+\overline{N}_s/\mathcal{F}_s]-\lambda t\\
 &=&\lambda (t-s)+0+\overline{N}_s\mathbb{E}[\overline{N}_t-\overline{N}_s+\overline{N}_s/\mathcal{F}_s]-\lambda t\\
 &=&\lambda t-\lambda s+0+0+\overline{N}^2_s-\lambda t\\
 &=&\overline{N}^2_s-\lambda s
\end{eqnarray*}
 This complete the proof.
\end{enumerate}
\end{proof}

\section{Stochastic integral}
\begin{defn}
Let $\mathbb{M}^p([0,T], \mathbb{R})$ be the subspace of $\mathbb{L}^p([0,T], \mathbb{R})$ such that for any process 

$(X_t)\in \mathbb{M}^p([0,T], \mathbb{R})$ we have 
\begin{eqnarray*}
\mathbb{E}\left(\int_0^T|X(t)|^pdt\right)<\infty.
\end{eqnarray*}
\end{defn}

Consider a Brownian motion $W$ and a stochastic process $(X_t)$ both adapted to a given filtration $(\mathcal{F}_t)$. We will define the following expression called  stochastic integral
\begin{eqnarray*}
I_t(X)=\int_0^tX(s)dW(s).
\end{eqnarray*}
We will also  give some of its properties.

Let's start with the stochastic integral of simple process.

\begin{defn}\textbf{[Elementary process or simple process]}

A process $(X_t)_{t\in\mathbb{R}}\in\mathbb{L}^p([0,T], \mathbb{R})$ is called simple or elementary process if there exist a partition 
$0=t_0<t_1<...<t_n=T$ such that 
\begin{eqnarray*}
X_s(\omega)=\sum_{j=0}^{n}1_{]t_j, t_{j+1}]}\theta_j(\omega),
\end{eqnarray*}
where $\theta_j$ is a bounded $\mathcal{F}_{t_j}$-measurable random variable.
\end{defn}

\begin{defn}\textbf{[ It\^{o}'s integral]}

The It\^{o}'s Integral of the simple process  $(X_t)_{t\in\mathbb{R}}\in\mathbb{L}^2([0,T], \mathbb{R})$ is defined by 
\begin{eqnarray*}
I_t(X)=\int_0^tX(s)dW(s) :=\sum_{j=0}^{n-1}\theta_j(W_{t_{j+1}}-W_{t_j}).
\end{eqnarray*}
\end{defn}

\begin{lem}
If $f$ is an elementary function in $\mathbb{L}^2([a,b],\mathbb{R})$ and $W_t$ a Brownian motion, then :
\begin{enumerate}
\item $\mathbb{E}\left(\int_a^bf(t)dW_t\right)=0$.
\item $\mathbb{E}\left(\int_a^bf(t)dW_t\right)^2=\int_a^b\mathbb{E}(f^2(t))dt$.
\end{enumerate}
\end{lem}

\begin{proof}
\begin{enumerate}
\item
By definition we have  
\begin{eqnarray*}
\int_a^bf(t)dW_t=\sum_{j=0}^{n-1}f_j(W_{t_j+1}-W_{t_j}).
\end{eqnarray*}
By taking expectation in both sides, we obtain 
\begin{eqnarray*}
\mathbb{E}\left[\int_a^bf(t)dW_t\right]=\sum_{j=0}^{n-1}\mathbb{E}(f_j)\mathbb{E}(W_{t_{j+1}}-W_{t_j})=0,
\end{eqnarray*}
since $W_{t_{j+1}}-W_{t_j}$ is a normal distribution with mean $0$ and standard deviation $\sqrt{t_{j+1}-t_j}$.
\item 
\begin{eqnarray*}
\left(\int_a^bf(t)dW_t\right)^2&=&\left[\sum_{j=0}^{n-1}f_j(B_{t_{j+1}}-W_{t_j})\right]^2\\
&=&\sum_{j=0}^{n-1}(f_j)^2(W_{t_{j+1}}-W_{t_j})^2+\sum_{l=0}^{n-1}\sum_{k=0, k\neq l}^{n-1}f_lf_k(W_{t_{l+1}}-W_{t_l})(W_{t_{k+1}}-W_{t_k}).
\end{eqnarray*}
Taking expectation in both sides and using independence of the increments of Brownian motion, we get
\begin{eqnarray*}
\mathbb{E}\left(\int_a^bf(t)dW_t\right)^2&=&\sum_{j=0}^{n-1}\mathbb{E}(f_j)^2E\left(W_{t_{j+1}}-W_{t_j}\right)^2\\
&=&\sum_{j=0}^{n-1}\mathbb{E}(f_j)^2(t_{j+1}-t_j)\\
&=&\int_a^b\mathbb{E}(f^2(t))dt.
\end{eqnarray*}
\end{enumerate}
\end{proof}

The following proposition can be seen in \cite{Oks}.
\begin{pro}
For any process $X=(X_t)_{t\geq 0}\in\mathbb{M}^2([0,T], \mathbb{R})$, such that $\mathbb{E}|X_t|^2<\infty $ for all $t\geq 0$, there exist a sequence $(f^{(n)}_t)_{t\geq 0}$ of simple process such that $\mathbb{E}|f^{(n)}_t|^2<\infty$ and 
\begin{eqnarray*}
\lim_{n\longrightarrow \infty}\mathbb{E}\left[\int_0^t|X_s-f_s^{(n)}|^2ds\right]=0.
\end{eqnarray*}
\end{pro}

\begin{defn}
 For any process $X=(X_t)_{t\geq 0}\in\mathbb{M}^2([0,T], \mathbb{R})$, we define a stochastic integral of $X$ with respect to a Brownian motion $W$ by :
 \begin{eqnarray*}
 \int_0^tX_sdW(s)=\lim_{n\longrightarrow \infty}\int_0^tf^{(n)}_sdW(s),
 \end{eqnarray*}
 where $(f^{(n)}_t)$ is the sequence of simple process converging almost surely to $X$  according to the previous proposition. Moreover, using It\^{o} isometry for elementaries functions one can prove that the limit on this definition does not depend on the actual choice of $(f^{(n)}).$ 
 \end{defn}

 \begin{pro} \textbf{[Properties of It\^{o} integral]}.
 
 For any process $X=(X_t)_{t\geq 0}\in\mathbb{M}^2([0,T], \mathbb{R})$ such that $\mathbb{E}|X_t|^2<\infty$,  for any functions 
 
 $f,g\in \mathbb{M}^2([0,T], \mathbb{R})$ and  $0\leq S<U<T$, the following holds :
 \begin{enumerate}
 \item[$(i)$] $\int_S^TfdW(t)=\int_S^UfdW(t)+\int_U^TfdW(t)$ almost surely.
 \item [$(ii)$] $\int_S^T(cf+g)dW(t)=c\int_S^TfdW(t)+\int_S^TgdW(t)$, for any constant $c$. 
 \item [$(iii)$] $\int_S^TfdW(t)$ is $\mathcal{F}_T$-measurable.
 \item [$(iv)$] $\mathbb{E}\left(\int_0^tX_sdW(s)\right)=0$.
 \item[$(v)$] $\mathbb{E}\left(\int_0^tX_sdW(s)\right)^2=\int_0^t\mathbb{E}(X_s^2)ds$.
 \end{enumerate} 
 \end{pro}
 \begin{proof} \cite{Oks}
 \end{proof}
 \begin{pro}\cite{Oks}
 For any elementary function $f^{(n)}$ $\mathcal{F}_t$-adapted, the integral 
 \begin{eqnarray*}
 I_n(t, \omega)=\int_0^tf^{(n)}dW(r)
 \end{eqnarray*}
 is a martingale with respect to $\mathcal{F}_t$.
 \end{pro}

\begin{proof}
 For $t\leq s$, we have :
 \begin{eqnarray*}
 \mathbb{E}[I_n(s,\omega)/\mathcal{F}_t]&=&\mathbb{E}\left[\left(\int_0^sf^{(n)}dW(r)\right)/\mathcal{F}_t\right]\\
 &=&\mathbb{E}\left[\left(\int_0^tf^{(n)}dW(r)\right)/\mathcal{F}_t\right]+
 \mathbb{E}\left[\left(\int_t^sf^{(n)}dW(r)\right)/\mathcal{F}_t\right]\\
 &=&\int_0^tf^{(n)}dW(r)+\mathbb{E}\left[\sum_{t\leq t^{(n)}_j\leq t^{(n)}_{j+1}\leq s}f^{(n)}_j\Delta W_j/\mathcal{F}_t\right]\\
 &=&\int_0^tf^{(n)}dW(r)+\sum_{t\leq t^{(n)}_j\leq t^{(n)}_{j+1}\leq s}\mathbb{E}[f^{(n)}_j\Delta W_j/\mathcal{F}_t]\\
  &=&\int_0^tf^{(n)}dW(r)+\sum_{t\leq t^{(n)}_j\leq t^{(n)}_{j+1}\leq s}\mathbb{E}[\mathbb{E}[f^{(n)}_j\Delta W_j/\mathcal{F}_{t_j}]/\mathcal{F}_t]\\
  &=&\int_0^tf^{(n)}dW(r)+\sum_{t\leq t^{(n)}_j\leq t^{(n)}_{j+1}\leq s}\mathbb{E}[f^{(n)}_j\mathbb{E}[\Delta W_j/\mathcal{F}_{t_j}]/\mathcal{F}_t]\\
  &=&\int_0^tf^{(n)}dW(r),\hspace{0.2cm} \text{since} \hspace{0.2cm} E[\Delta W_j/\mathcal{F}_{t_j}]=\mathbb{E}[\Delta W_j]=0\\
  &=&I_n(t, \omega).
 \end{eqnarray*}
 \end{proof}

 \begin{pro} \textbf{[Generalisation]}
 
   Let $f(t, \omega)\in \mathbb{M}^2([0, T], \mathbb{R})$ for all $t$. Then the integral 
   \begin{eqnarray*}
   M_t(\omega)=\int_0^tf(s, \omega)dW(s)
   \end{eqnarray*}
   is a martingale with respect to $\mathcal{F}_t$ and 
   \begin{eqnarray*}
   \mathbb{P}\left[\sup_{0\leq t\leq T}|M_t|\geq \lambda\right]\leq\dfrac{1}{\lambda^2}\mathbb{E}\left[\int_0^Tf^2(s, \omega)ds\right], \hspace{0.3cm} \forall\; \lambda >0
   \end{eqnarray*}
   \end{pro}
   \begin{proof} \cite{Oks}.
   \end{proof}

   \subsection{One dimensional It\^{o} Formula}
   \begin{defn} \textbf{[1-dimensional It\^{o} process]}
   
   Let $W_t$ be a $1$-dimensional Brownian motion on $(\Omega, \mathcal{F}, \mathbb{P})$. An It\^{o} process (or Stochastic integral) is any stochastic process $X_t$ of the form 
   \begin{eqnarray}
   X_t=X_0+\int_0^tu(s,\omega)ds+\int_0^tv(s,  \omega)dW(s),
   \label{ch1Ito1}
   \end{eqnarray}
   where $u\in \mathbb{L}^1([0,T], \mathbb{R})$ and $v\in\mathbb{L}^2([0, T], \mathbb{R})$.
   \end{defn}

\begin{pro} \textbf{[ first $1$- dimensional It\^{o} formula]}

   Let $(\Omega, \mathcal{F}, \mathbb{P})$ be a complete probability space, $(W_t)_{t\in \mathbb{R}_{+}}$ a one-dimensional Brownian motion and $f :  \mathbb{R}\longrightarrow \mathbb{R}$ such that $f$ is once derivable. If $(X_t)$ is any process of the form \eqref{ch1Ito1},
   
   then $f(X_t)$ is an It\^{o} processes and 
   \begin{eqnarray*}
   f(X_t)=f(X_0)+\int_0^t f'(X_s)u_sds+\dfrac{1}{2}\int_0^tf''(X_s)v_s^2ds+\int_0^tf'(X_s)v_sdW_s.
   \end{eqnarray*}
   \end{pro}
\begin{proof} \cite{Oksa}.
\end{proof}

 \begin{pro}\textbf{[ second $1$- dimensional It\^{o} formula]}
 
   If in the previous proposition we consider  $f :[0, \infty)\times\mathbb{R}\longrightarrow\mathbb{R}$ such that $f$ is once differentiable with respect to the first variable $t$ and twice differentiable with respect to the second variable $x$,
   
   then $f(t, X_t)$ is an  It\^{o} process and 
   \begin{eqnarray*}
   f(t, X_t)=f(0, X_0)+\int_0^t\dfrac{\partial f}{\partial t}(s,X_s)ds+\int_0^t\dfrac{\partial f}{\partial x}(s,X_s)u_sds+\int_0^t\dfrac{\partial f}{\partial x}(s,X_s)v_sdW_s+\dfrac{1}{2}\int_0^t\dfrac{\partial^2 f}{\partial x^2}(s,X_s)v_s^2ds,
   \end{eqnarray*} or in its differential form :
   \begin{eqnarray*}
   df(t, X_t)=\dfrac{\partial f}{\partial t}(t, X_t)dt+\dfrac{\partial f}{\partial x}(t, X_t)dX_t+\dfrac{1}{2}\dfrac{\partial^2 f}{\partial x^2}(t, X_t)(dX_t)^2.
   \end{eqnarray*}
  $(dX_t)^2=dX_tdX_t$ is computed according to the rules 
   \begin{eqnarray*}
   dtdt=dW_tdt=dtdW_t=0,\hspace{1cm}  dW_tdW_t=dt.
   \end{eqnarray*}
         \end{pro}
  \begin{proof} \cite{Oksa}.
  \end{proof}
  
  \subsection{Multi-dimensional It\^{o} integral} 
  
  \begin{defn} \textbf{[m-dimensional Brownian motion]}\cite{Oks}.
  
  Let $W_1, \cdots W_m$ be $m$ Brownian motions. The random variable $W=(W_1, W_2, ..., W_m)$ is called $m$-dimensional Brownian motion. Let $\mathbb{L}^{n\times m}([0, T], \mathbb{R}^{n\times m})$ denotes the set of $n\times m$ matrices $v=[v_{ij}(t, \omega)]$, $1\leq i\leq n$, $1\leq j\leq m$. Where $v_{ij}(t, \omega)\in \mathbb{L}^2([0, T], \mathbb{R})$. $\int_0^tv_sdW_s$ denotes the It\^{o} integral of $v$ with respect to the m-dimensional Brownian motion $W$. It can be written into its matrix form 
  \begin{eqnarray*}
  \int_0^TvdW(s)=\int_0^T\left(\begin{array}{ccc}
  v_{11}&\cdots& v_{1m}\\
  .& &.\\
  .& &.\\
  .& &.\\
  v_{n1}&\cdots &v_{nm}
  \end{array}
  \right)\left(\begin{array}{c}
  dW_1(s)\\
  .\\
  .\\
  .\\
  dW_m(s)
    \end{array}
    \right),   
  \end{eqnarray*}
  which is a $n\times 1$ matrix (column vector) whose $i^{th}$ components are given by 
  \begin{eqnarray*}
  \sum_{j=1}^m\int_0^Tv_{ij}(s, \omega)dW_j(s).
  \end{eqnarray*}
  \end{defn}
  
  \begin{defn}\textbf{[$n$-dimensional It\^{o} process]}\cite{Oks}.
   
  Let $W$ be an $m$- Brownain motion and  $v=[v_{i,j}, 1\leq i\leq n \hspace{0.2cm} 1\leq j\leq m]$ an element of $\mathbb{L}^{n\times m}([0, t], \mathbb{R}^{n\times m})$. Let $u=(u_i)_{i=1}^n$ such that $u_i\in\mathbb{L}^2([0, T])$ pour tout $1\leq i\leq n$.
  
  The $n$-dimensional It\^{o} process is any stochastic process of the form  
  \begin{eqnarray*}
  dX(t)=udt+vdW(t),
  \end{eqnarray*}
  which is a system of $n$ It\^{o} process, where the  $i^{th}$ process is given by :
  \begin{eqnarray*}
  dX_i(t)=u_idt+\sum_{j=1}^mv_{ij}dW_j(t).
  \end{eqnarray*}
  \end{defn}

  \begin{pro} \textbf{[General It\^{o} formula]}
  
  Let $  dX(t)=udt+vdW(t)$   be an $n$-dimensional It\^{o} process. Let $g(t, x)=(g_1(t,x),..., g_p(t,x))$ be a function once differentiable with respect to $t$ and twice differentiable with respect to $x$.
  
  Then the process $Y(t)=g(t, X(t))$ is also a $p$-dimensional It\^{o} process, whose component $Y_k$ are given by :
  \begin{eqnarray*}
  Y_k(t)=\dfrac{\partial g_k}{\partial t}(t, X_t)dt+\sum_{i=1}^n\dfrac{\partial g_k}{\partial x_i}(t, X_t)dX_i+\dfrac{1}{2}+\sum_{i=1}^n\sum_{j=1}^n\dfrac{\partial^2 g_k}{\partial x_i\partial x_j}(t, X_t)dX_idX_j,
  \end{eqnarray*}
  where $dW_idW_j=\delta_{ij}dt$  and  $dW_idt=dtdW_i=0$.
  \end{pro}
  \begin{proof} \cite{Oks}.
  \end{proof}

  \section{Stochastic process with jumps and Stochastic integral with jumps}
 \begin{defn}
 \begin{enumerate}
 \item [$(i)$] A function $f : [0,T]\longrightarrow \mathbb{R}^n$ is said to be right continuous with left limit at $t\in[0,T]$ if 
 \begin{eqnarray*}   
 f(t^+) : =\lim_{s\longrightarrow t^+}f(s) \hspace{0.5cm}\text{and}\hspace{0.2cm}f(t^-) : =\lim_{s\longrightarrow t^-}f(s)\hspace{0.2cm}\text{exist}\hspace{0.2cm} \text{and} \hspace{0.5cm} f(t^+)=f(t). 
 \end{eqnarray*}
 \item [$(ii)$] A function $f : [0,T]\longrightarrow \mathbb{R}^n$ is said to be  left continuous with right limit  if  
 \begin{eqnarray*}
  f(t^+) : =\lim_{s\longrightarrow t^+}f(s) \hspace{0.5cm}\text{and}\hspace{0.2cm}f(t^-) : =\lim_{s\longrightarrow t^-}f(s)\hspace{0.2cm}\text{exist}\hspace{0.2cm} \text{and} \hspace{0.5cm} f(t^-)=f(t).
\end{eqnarray*}     
 \end{enumerate}
 In the litterature, the french short forms "c\'{a}dl\'{a}g" and "c\'{a}gl\'{a}d"    denote respectively functions which are right continous with left limit and left continous with right limit.
 
 \end{defn}
 \begin{rem}
 \begin{itemize}
 \item If $f$ is right continous with left limit at $t$, then  $\Delta f(t)=f(t)-f(t^-)$ is called the jump of $f$ at $t$.
 \item If $f$ is left continous with right limit at $t$, then  $\Delta f(t)=f(t^+)-f(t)$ is called the jump of $f$ at $t$.
 \end{itemize} 
 \end{rem}

 \begin{defn}
 A stochastic process $X=(X_t)_{t\geq 0}$ is called jump process if the sample path $s\longmapsto X_s$ is left continuous (c\'{a}gl\'{a}g) or right continuous (c\'{a}dl\'{a}g) $\forall s\geq 0$.
 \end{defn}
\begin{defn}\textbf{[L\'{e}vy process]}

A stochastic process $X=\{X_t,\hspace{0.3cm} t\geq 0\}$ is a L\'{e}vy process if the following conditions are fulfilled
\begin{enumerate}
\item[ $(i)$] The increments on disjoint time intervals are independent. That is  for $0\leq t_0<t_1<...<t_n$ $\{X_{t_j}-X_{t_{j-1}}\hspace{0.3cm} 1\leq j\leq n\}$ are independent.
\item[$(ii)$] The increments of sample paths are stationary : $X_t-X_s\approx X_{t-s}$ for $0\leq t\leq s$.
\item[$(iii)$] The sample paths are right continuous with left limit. 
\end{enumerate}
\end{defn}
\begin{rem}
The Brownian motion and the poisson process  starting at $0$ are L\'{e}vy process.
\end{rem}
\begin{defn} \cite{Oksa}
Let $\textbf{D}_{ucp}$ denote the space of c\`{a}dl\`{a}g adapted process equipped with the topology of the uniform convergence  in probability (ucp) on compact sets.
ucp : $H_n\longrightarrow H$ if $\forall\; t\geq 0$ $\sup\limits_{0\leq s\leq t}|H_n(s)-H(s)|\longrightarrow 0$ in probability ($A_n\longrightarrow A$ in probability if $\forall\; \epsilon>0, \exists\; n_{\epsilon}\in \mathbb{N}$ such that $n>n_{\epsilon}\Longrightarrow \mathbb{P}(|A_n-A|>\epsilon)<\epsilon$).

 In the sequel  $\textbf{L}_{ucp}$ denote the space of adapted c\`{a}dl\`{a}g processes (left continous with right limit) equiped with the ucp topology.
\end{defn}

\begin{defn} \cite{Oksa}
Let $H$ be an   elementary function.  i.e there exist a partition 

$0=t_0<t_1...<t_n=T$ such that
 \begin{eqnarray*}
 H=\sum_{j=0}^nH_j1_{[t_j, t_{j+1})},
 \end{eqnarray*}
where $H_j$ are $\mathcal{F}_{t_j}$-measurable.
Let $(X_t)$ be  a L\'{e}vy process. The stochastic integral $\int_0^tH(s)dX(s)$ is defined by 
\begin{eqnarray*}
J_XH(t) : =\int_0^tH(s)dX(s): =\sum_{j=0}^n H_j(X(t_{j+1})-X(t_j))  \hspace{0.3cm} t\geq 0.
\end{eqnarray*}
\end{defn}
\begin{pro}\cite{Oksa}
Let $X$ be a semimartingale, then the mapping $J_X$ can be extended to the continuous linear map
\begin{eqnarray*}
J_X : \textbf{L}_{ucp}\longrightarrow \textbf{D}_{ucp}.
\end{eqnarray*}
\end{pro}

The above proposition allows us to define a stochastic integral of the form 
\begin{eqnarray*}
\int_0^tH(s)dX(s),
\end{eqnarray*}
where $H\in \textbf{L}_{ucp}$.

\begin{defn}\cite{Oksa}
For $H\in \textbf{L}_{ucp}$ we define $\int_0^tH(s)dX(s)$ :
\begin{eqnarray*}
\int_0^tH(s)dX(s) : =\lim_{n\longrightarrow +\infty}\int_0^tH^{n}(s)dX(s),
\end{eqnarray*}
where $(H^{n})$ is a sequence of simple process converging to $H$.
\end{defn}

\begin{pro}
Let  $f\in\textbf{L}_{ucp}$ and $(\overline{N_t})$ be a compensated poisson process. The following holds
\begin{enumerate}
\item[$(i)$] $\mathbb{E}\left(\int_0^tf(s)d\overline{N}(s)\right)=0$.
\item[$(ii)$] $\mathbb{E}\left(\int_0^tf(s)d\overline{N}(s)\right)^2=\lambda\int_0^t\mathbb{E}(f(s))^2ds$.
\end{enumerate}
\end{pro}
\begin{proof} \cite{Oksa}
\end{proof}

\subsection{It\^{o} formula for jump process}  
   \begin{defn} \textbf{[It\^{o} jump-diffusion process]} 
   
   The It\^{o} jump-diffusion process is any process of the form
   \begin{eqnarray}
   X_t=X_0+\int_0^ta(X_s)ds+\int_0^tb(X_s)dW_s+\int_0^tc(X_s)dN_s.
   \label{ch1ito1}
   \end{eqnarray}
   The coefficient $a\in\mathbb{L}^1([0,T], \mathbb{R}^n)$ is called the drift coefficient, $b\in \mathbb{L}^2([0,T], \mathbb{R}^{n\times m})$ is called the diffusion coefficient and $c(X_s)\in\mathbb{L}^2([0,T], \mathbb{R}^n)$ is called the jump coefficient. $W$ is a $m$-dimentional Brownian motion and $N$ a one dimentional poisson process.
      \end{defn}
 \begin{pro} \cite[pp 6]{Oksa}\textbf{[It\^{o} formula for jump process]}
 
 If $(X_t)$ is a jump-diffusion process of the form \eqref{ch1ito1} and $f :[0,\infty)\longrightarrow \mathbb{R}^n$ any function twice derivable, then $f(X_t)$ is a jump-diffusion process and satisfies the following equation
 \begin{eqnarray*}
 f(X_t)=f(X_0)+\int_0^t\dfrac{\partial f}{\partial x}(X_s)a_sds+\dfrac{1}{2}\int_0^t\dfrac{\partial^2f}{\partial x^2}(X_s)b_s^2ds +\int_0^t\dfrac{\partial f}{\partial x}b_sdW_s+\int_0^t(f(X_{s^-}+c( X_s))-f(X_{s^-}))dN_s.
 \end{eqnarray*}
 \end{pro}

 \begin{lem}\label{ch1Itoproduct} \textbf{[It\^{o}'s lemma for product]}
 
If $X_t$ and $Y_t$ are two It\^{o}'s jump-diffusion process,  Then $X_tY_t$ is an It\^{o} jump-diffusion process and 
\begin{eqnarray*}
d(X_tY_t)=Y_tdX_t+X_tdY_t+dX_tdY_t.
\end{eqnarray*}
$dX_1dX_2$ is called the It\^{o}'s corrective term and it is computed according to the relations
\begin{eqnarray*}
dt.dt=dN_tdt=dtdN_t=dW_tdN_t=dN_tdW_t=0,\hspace{0.5cm}
dN_tdN_t=dN_t.
\end{eqnarray*}
\end{lem}
\begin{proof} \cite{Oksa}
\end{proof}

\hspace{0.5cm}After being familiar with some basic notions in probability theory and stochastic process, we are now ready to provide in the following chapter the proof of the existence and uniqueness  of solution of SDEs with jumps under global Lipschitz conditions.

\chapter{Existence and uniqueness of solution of  the  jump-diffusion It\^{o}'s stochastic differential equations}
\hspace{0.5cm} In this chapter, we give the general formulation of the compensated stochastic differential equation (CSDE) with jumps which will be helpful to prove the existence and the uniqueness solutions of the stochastic differential equation with jump.
\section{General formulation}

\hspace{0.5cm}Along this work, $||.||$ denote the Frobenius matrix norm, $(\Omega, \mathcal{F}_t, \mathbb{P})$ denote a complete probability space.  For all $x, y\in\mathbb{R}^n$, $\langle x,y\rangle=x_1y_1+\cdots, x_ny_n$ is the inner product. For all $a, b\in\mathbb{R}$, 

$ a\vee b :=\max(a,b)$.

\hspace{0.5cm} Throughout this work, we consider a jump diffusion It\^{o}'s stochastic differential (SDEs) of the form 
\begin{eqnarray}
dX(t)=f(X(t^-))dt+g(X(t^-))dW(t)+h(X(t^-))dN(t),\hspace{0.5cm} X(0^-)=X_0,
\label{ch2jdi1}
\end{eqnarray}
 where $X_0$ is the initial condition, $X(t^-)=\lim\limits_{s\longrightarrow t^-}X(s)$, $W(t)$ is an $m$-dimensional Brownian motion and $N(t)$ is a $1$-dimensional poisson process with intensity $\lambda>0$. We assume $W_t$ and $N_t$ to be both $\mathcal{F}_t$-measurable.
 $f : \mathbb{R}^n\longrightarrow \mathbb{R}^n$,  $g : \mathbb{R}^n \longrightarrow\mathbb{R}^{n\times m}$  and $h : \mathbb{R}^n\longrightarrow \mathbb{R}^n$.

\hspace{0.5cm}Our aim in this chapter is to prove the existence and the uniqueness of solution of equation \eqref{ch2jdi1} in the strong sense.

\begin{defn} \textbf{[Strong solution]}

A stochastic process $X=\{X_t\}_{t\in[0,T]}$ is called strong solution of It\^{o} jump-diffusion differential equation \eqref{ch2jdi1}  if :
\begin{enumerate}
\item $X_t$ is $\mathcal{F}_t-$measurable $\forall\; t\in[0,T]$.
\item $\mathbb{P}\left[\int_0^t|f(X_s,s)|ds<\infty\right]=\mathbb{P}\left[\int_0^t|g(X_s,s)|^2ds<\infty\right]
=\mathbb{P}\left[\int_0^t|h(X_s,s)|^2ds<\infty\right]=1$.
\item $X$ satisfies equation \eqref{ch2jdi1} almost surely.
\end{enumerate}
\end{defn}

\begin{assumption} \label{ch2assump}
 Troughout this chapter, we make the following hypothesis :
 
  There exist positive constants $L$ and $K$ such that for all $x, y\in\mathbb{R}^n$,
\begin{enumerate}
\item  $\mathbb{E}||X(0)||^2<+\infty$ and $X(0)$ is independent of the Weiner process $W(t)$ and of the  poisson process $N(t)$.
\item  $f$, $g$ and $h$ satisfy the global Lipschitz condtion :
\begin{eqnarray}
||f(x)-f(y)||^2\vee ||g(x)-g(y)||^2\vee ||h(x)-h(y)||^2\leq L||x-y||^2
\end{eqnarray}
\item  $f$, $g$ and $h$ satisfy the linear growth condition :
\begin{eqnarray}
||f(x)||^2\vee ||g(x)||^2\vee ||h(x)||^2\leq K(1+||x||^2).
\end{eqnarray}
\end{enumerate} 
\end{assumption}
\begin{rem}The globaly Lipschitz condition implies the linear growth condition. So it is enough to make the assumptions only on the global Lipschitz condition.
\end{rem}
\begin{defn}
A soluion $\{X_t\}$ of \eqref{ch2jdi1} is said to be pathwise unique if any other solution $\overline{X}$ is a stochastilly indistinguishable from it, that is $\mathbb{P}\{X(t)=\overline{X}(t)\}=1$, $\forall\; t\in [0,T]$ almost surely.
\end{defn}

In order to prove the existence and uniqueness solution of \eqref{ch2jdi1}, it is useful to write \eqref{ch2jdi1} in its compensated form.
\subsection{Compensated stochastic differential equation} (CSDE)

\hspace{0.5cm} From the relation $\overline{N}(t)= N(t)-\lambda t$, we have $d\overline{N}(t)=dN_t-\lambda dt$.
Substituting this latter relation in \eqref{ch2jdi1} leads to :
\begin{eqnarray}
dX(t)=f_{\lambda}(X(t^-))dt+g(X(t^-))dW(t)+h(X(t-))d\overline{N}(t),
\label{ch2jdi2}
\end{eqnarray}
where 
\begin{eqnarray}
f_{\lambda}(x) :=f(x)+\lambda h(x).
\label{ch2jdi4}
\end{eqnarray}
\eqref{ch2jdi2} can be rewriten into its integral form 
\begin{eqnarray}
X(t)=X_0+\int_0^tf_{\lambda}(X(s^-))ds+\int_0^tg(X(s^-))dW(s)+\int_0^th(X(s^-))d\overline{N}(s),
\label{ch2jdi3}
\end{eqnarray}
\begin{lem}
\label{ch2lemma1}
If assumptions \ref{ch2assump} are satisfied, then the function $f_{\lambda}$ satisfies the global Lipschitz and the linear growth conditions with constants $L_{\lambda}=(1+\lambda)^2L$ and $K_{\lambda}=(1+\lambda)^2K$ respectively.
\end{lem}
\begin{proof}
\begin{enumerate}
\item Using the global lipschitz condition satisfied by $f$ and $h$, it follows from \eqref{ch2jdi4} that:
\begin{eqnarray*}
||f_{\lambda}(x)-f_{\lambda}(y)||^2&=&||f(x)-f(y)+\lambda (h(x)-h(y))||^2\\
&\leq& \left(\sqrt{L}||x-y||+\lambda\sqrt{L}||x-y||\right)^2=(1+\lambda)^2L||x-y||^2.
\end{eqnarray*}
\item Along the same lines as above, we obtain the linear growth condition satisfies by $f_{\lambda}$.
\end{enumerate}
\end{proof}

\section{Well-posedness problem}
\hspace{0.5cm} Based on Lemma \ref{ch2lemma1} and using the fact that equations \eqref{ch2jdi1} and \eqref{ch2jdi2} are equivalent, the existence and uniqueness of solution of equation \eqref{ch2jdi1} is equivalent to the existence and uniqueness of solution of equation \eqref{ch2jdi2}.
\begin{thm}\label{ch2th1}
If Assumptions \ref{ch2assump} are fulfilled, then there exist a unique  strong solution of equation \eqref{ch2jdi2}.
\end{thm}

In order to prove Theorem \ref{ch2th1}, we need the following lemma.

\begin{lem}\label{ch2lemma2}
let $X^0(t)=X_0(t), \hspace{0.5cm} \forall t\in [0,T]$ and
\begin{eqnarray}
X^{n+1}(t)=X_0+\int_0^tf(X^n(s^-))ds+\int_0^tg(X^n(s^-))dW(s)+\int_0^th(X^n(s^-))d\overline{N}(s)ds,
\end{eqnarray}
then \hspace{0.5cm}
\begin{eqnarray}
\mathbb{E}||X^{n+1}(t)-X^n(t)||^2 \leq \dfrac{(Mt)^{n+1}}{(n+1)!},
\label{lemI1}
\end{eqnarray}

where $M$ is a positive constant depending on $ \lambda, K, L, X_0$.
\end{lem}
\begin{proof} : By induction
\begin{enumerate}
\item For $n=0$
\begin{eqnarray}
||X^1(t)-X^0(t)||^2&=&\left\|\int_0^tf_{\lambda}(X_0(s^-))ds+\int_0^tg(X_0(s^-))dW(s)+\int_0^th(X_0(s^-))d\overline{N}(s)\right\|^2\nonumber\\
&\leq& 3\left\|\int_0^tf_{\lambda}(X_0(s^-))ds\right\|^2+3\left\|\int_0^tg(X_0(s^-))dW(s)\right\|^2\nonumber\\
&+&3\left\|\int_0^th(X_0(s^-))d\overline{N}(s)\right\|^2.
\label{E1}
\end{eqnarray}
Using Cauchy-Schwartz inequality and the linear growth condition, it follows that:
\begin{eqnarray}
\mathbb{E}(I_1(t)):=\mathbb{E}\left(3\left\|\int_0^tf_{\lambda}(X_0(s^-))ds\right\|^2\right)&\leq & 3T\int_0^t\mathbb{E}||f_{\lambda}(X_0(s^-))||^2ds\nonumber\\
&\leq & 3TK_{\lambda}\int_0^t(1+E||X_0||^2)ds.
\label{EI1}
\end{eqnarray}
From the martingale property of $\overline{N}(t)$ and the linear growth condition, it follows that :
\begin{eqnarray}
\mathbb{E}(I_2(t)) := \mathbb{E}\left(3\left\|\int_0^th(X_0(s^-))d\tilde{N}(s)\right\|^2\right)&=&3\lambda\int_0^t\mathbb{E}||h(X_0(s^-))||^2ds\nonumber\\
&\leq & 3\lambda K\int_0^t(1+\mathbb{E}|X_0|^2)ds.
\label{EI2}
\end{eqnarray}
Using the martingale property of $W(t)$ and the linear growth condition, it follows that :
\begin{eqnarray}
\mathbb{E}(I_3(t)) := \mathbb{E}\left(3\left\|\int_0^tg(X_0(s^-))dW(s)\right\|^2\right)\leq 3K\int_0^t(1+\mathbb{E}||X_0||^2)ds.
\label{EI3}
\end{eqnarray}
Taking expectation in both sides of \eqref{E1} and using estimations \eqref{EI1}, \eqref{EI2} and \eqref{EI3} leads to : 
\begin{eqnarray}
\mathbb{E}||X^1(t)-X^0(t)||^2\leq (3TK_{\lambda}+3K+3\lambda K)\int_0^t(1+\mathbb{E}||X_0||^2)ds\leq Mt,
\label{E2}
\end{eqnarray}
where $M=(3TK_{\lambda}+3K+3\lambda K) \vee (3TL_{\lambda} +3L+3\lambda L)$.
\item Let's assume that  inequality \eqref{lemI1} holds up to a certain rank $n\geq 0$. We have to show that it remains true for $n+1$. That is,  we have to prove that $\mathbb{E}||X^{n+2}(t)-X^{n+1}(t)||^2\leq \dfrac{(Mt)^{n+2}}{(n+2)!}$.
\begin{eqnarray*}
X^{n+2}(t)&=&X_0+\int_0^tf_{\lambda}(X^{n+1}(s^-))ds+\int_0^tg(X^{n+1}(s^-))dW(s)+\int_0^th(X^{n+1}(s^-))d\overline{N}(s),\\
X^{n+1}(t)&=&X_0+\int_0^tf_{\lambda}(X^{n}(s^-))ds+\int_0^tg(X^{n}(s^-))dW(s)+\int_0^th(X^{n}(s^-))d\overline{N}(s)
\end{eqnarray*}
\begin{eqnarray*}
||X^{n+2}(t)-X^{n+1}(t)||^2&=&\left\|\int_0^t\left[f_{\lambda}(X^{n+1}(s^-)-f_{\lambda}(X^n(s^-))\right]ds\right.\\
&+&\int_0^t[g(X^{n+1}(s^-))-g(X^n(s^-))]dW(s)\\
&+&\left.\int_0^t[h(X^{n+1}(s^-))-h(X^n(s^-))]d\overline{N}(s)\right\|^2.
\end{eqnarray*}
Using the inequality $(a+b+c)^2\leq 3a^2+3b^2+3c^2$ for all $a, b,c\in\mathbb{R}$, it follows that :
\begin{eqnarray*}
||X^{n+2}(t)-X^{n+1}(t)||^2&\leq & 3\left\|\int_0^t[f_{\lambda}(X^{n+1}(s^-))-f_{\lambda}(X^n(s^-))ds\right\|^2\\ &+&3\left\|\int_0^t[g(X^{n+1}(s^-))-g(X^n(s^-))dW\right\|^2\\
&+&3\left\|\int_0^t[h(X^{n+1}(s^-))-h(X^n(s^-))]d\overline{N}\right\|^2.
\end{eqnarray*}
Using the martingale properties of $W(s)$ and $\overline{N}(s)$ and the global Lipschitz condition satisfied by $f_{\lambda}$, $g$ and $h$, it follows that :
\begin{eqnarray*}
\mathbb{E}||X^{n+2}(t)-X^{n+1}(t)||^2&\leq&(3TL_{\lambda}+3L+3\lambda L)\int_0^t\mathbb{E}||X^{n+1}(s^-)-X^n(s^-)||^2ds.
\end{eqnarray*}
Using the hypothesis of induction, it follows that :
\begin{eqnarray*}
\mathbb{E}||X^{n+2}(t)-X^{n+1}(t)||^2&\leq &M\int_0^t\dfrac{(Ms)^{n+1}}{(n+1)!}ds=\dfrac{(Mt)^{n+2}}{(n+2)!}.
\end{eqnarray*}
This complete the proof of the lemma.
\end{enumerate}

\begin{proof}\textbf{[Theorem \ref{ch2th1}]}

\textbf{\underline{Uniqueness}} : Let $X_1$ and $X_2$ be two solutions of \eqref{ch2jdi3}. Then :
\begin{eqnarray*}
X_1(t)=X_0+\int_0^tf_{\lambda}(X_1(s^-))ds+\int_0^tg(X_1(s^-))dW(s)+\int_0^th(X_1(s^-))d\overline{N}(s),
\end{eqnarray*}
\begin{eqnarray*}
X_2(t)=X_0+\int_0^tf_{\lambda}(X_2(s^-))ds+\int_0^tg(X_2(s^-))dW(s)+\int_0^th(X_2(s^-))d\overline{N}(s).
\end{eqnarray*}
Therefore,
\begin{eqnarray}
||X_1(t)-X_2(t)||^2&=&\left\|\int_0^t\left[f_{\lambda}(X_1(s^-))-f_{\lambda}(X_2(s^-))\right]ds +\int_0^t[g(X_1(s^-))-g(X_2(s^-))]dW(s)\nonumber\right.\\
&+&\left.\int_0^t[h(X_1(s^-))-h(X_2(s^-))]d\overline{N}(s)\right\|^2\nonumber\\
&\leq& 3\left\|\int_0^t[f_{\lambda}(X_1(s^-))-f_{\lambda}(X_2(s^-))]ds\right\|^2+3\left\|\int_0^t[g(X_1(s^-))-g(X_2(s^-))]dW(s)\right\|^2\nonumber\\
&+&3\left\|\int_0^t[h(X_1(s^-))-h(X_2(s^-))]d\overline{N}(s)\right\|^2.
\label{tI1}
\end{eqnarray}
Using Cauchy-Schwartz inequality and the globaly Lipschitz condition, it follows that :
\begin{eqnarray}
\mathbb{E}(I_1(t)) &:=&\mathbb{E}\left(3\left\|\int_0^t[f_{\lambda}(X_1(s^-))-f_{\lambda}(X_2(s^-))]ds\right\|^2\right)\nonumber\\
&\leq& 3t\int_0^t\mathbb{E}\left\|f_{\lambda}(X_1(s^-))-f_{\lambda}(X_2(s^-))\right\|^2\nonumber\\
&\leq& 3tL_{\lambda}\int_0^t\mathbb{E}||X_1(s^-)-X_2(s^-)||^2ds.
\label{I1}
\end{eqnarray}
Using the martingale property of $W(s)$ and the global Lipschitz condition, it follows that :
\begin{eqnarray}
\mathbb{E}(I_2(t))&:=& \mathbb{E}\left(3\left\|\int_0^t[g(X_1(s^-))-g(X_2(s^-))]dW(s)\right\|^2\right)\nonumber\\
&\leq& 3\int_0^t\mathbb{E}\left\|g(X_1(s^-))-g(X_2(s^-))\right\|^2ds\nonumber\\
&\leq& 3L\int_0^t\mathbb{E}\left\|X_1(s^-)-X_2(s^-)\right\|^2ds.
\label{I2}
\end{eqnarray}
Along the same lines as above, we obtain  : 
\begin{eqnarray}
\mathbb{E}(I_3(t))&:=& \mathbb{E}\left(3\left\|\int_0^t[h(X_1(s^-))-h(X_2(s^-))]d\overline{N}(s)\right\|^2\right)\nonumber\\
&\leq&
 3L\int_0^tE\left\|X_1(s^-)-X_2(s^-)\right\|^2ds.
 \label{I3}
\end{eqnarray}
Taking expectation in both sides of \eqref{tI1} and using estimations \eqref{I1}, \eqref{I2} and \eqref{I3}  leads to :
\begin{eqnarray}
\mathbb{E}||X_1(t)-X_2(t)||^2\leq (3tL_{\lambda}+3L+3\lambda L)\int_0^t\mathbb{E}||X_1(s^-)-X_2(s^-)||^2ds.
\label{tI2}
\end{eqnarray}
Applying  Gronwall lemma (contonuous form) to  inequality \eqref{tI2} leads to :
\begin{eqnarray*}
\mathbb{E}||X_1(t)-X_2(t)||^2=0, \hspace{0.5cm} \forall t\;\in[0,T].
\end{eqnarray*}
It follows from Markov's inequality that :
\begin{eqnarray*}
\forall\; a>0,\hspace{0.5cm} \mathbb{P}(||X_1-X_2||^2>a)=0.
\end{eqnarray*}
Therefore, 
\begin{eqnarray*}
\mathbb{P}\left(\{X_1(t)=X_2(t), \hspace{0.2cm} t\in[0,T]\}\right)=1 \hspace{0.2cm} a.s.
\end{eqnarray*}
\end{proof}
\textbf{\underline{Existence}}: 

From the sequence $X^n(t)$ defined in Lemma \ref{ch2lemma2}, it follows that :
\begin{eqnarray}
||X^{n+1}(t)-X^{n}(t)||^2&=&\left\|\int_0^t[f_{\lambda}(X^n(s^-))-f_{\lambda}(X^{n-1}(s^-))]ds+\int_0^t[g(X^n(s^-))-g(X^{n-1}(s^-))]dW(s)\right.\nonumber\\
&+&\left.\int_0^t[h(X^n(s^-))-h(X^{n-1}(s^-))]d\overline{N}(s)\right\|^2.
\label{ch2I3}
\end{eqnarray}
Taking  expectation and the supremum in the both sides of inequality \eqref{ch2I3} and using inequality $(a+b+c)^2\leq 3a^2+3b^2+3c^2$ for all $a, b, c \in\mathbb{R}$ leads to :
\begin{eqnarray}
\mathbb{E}\left(\sup_{0\leq t\leq T}||X^{n+1}(t)-X^n(t)||^2\right)&\leq& 3T\sup_{0\leq t\leq T}\left(\mathbb{E}\int_0^t||f_{\lambda}(X^n(s^-)-f_{\lambda}(s^-)||^2ds\right)\nonumber\\
&+&3\mathbb{E}\left(\sup_{0\leq t\leq T}M_1(t)\right)+3\mathbb{E}\left(\sup_{0\leq t\leq T}M_2(t)\right),
\label{ch2I4}
\end{eqnarray}
where
\begin{eqnarray*}
M_1(t)&=&\left\|\int_0^t[g(X^n(t))-g(X^{n-1}(t))]dW(s)\right\|^2\\
M_2(t)&=&\left\|\int_0^t[h(X^n(t))-h(X^{n-1}(t))]d\overline{N}(s)\right\|^2.
\end{eqnarray*}
Using  the global Lipschitz condition, it follows that
\begin{eqnarray}
\mathbb{E}\left(\sup_{0\leq t\leq T}\int_0^t||f_{\lambda }(X^n(s^-))-f_{\lambda}(X^{n-1}(s^-))||^2ds\right)\leq L_{\lambda}\int_0^T\mathbb{E}||X^n(s^-)-X^{n-1}(s^-)||^2ds.
\label{ch2M0}
\end{eqnarray}
Using respectively Doop's maximal inequality, martingale property of $W(s)$  and global Lipschitz condition satisfied by $g$, it follows that :
\begin{eqnarray}
\mathbb{E}\left(\sup_{0\leq t\leq T}M_1(t)\right)\leq 4M_1(t)&=&4\mathbb{E}\left\|\int_0^T[g(X^n(s^-))-X^{n-1}(s^-))]dW(s)\right\|^2\nonumber\\
&=&4\int_0^T\mathbb{E}||g^(X^n(s^-))-g(X^{n-1}(s^-))||^2ds\nonumber\\
&\leq& 4L\int_0^T\mathbb{E}||X^n(s^-)-X^{n-1}(s^-)||^2ds.
\label{ch2M1}
\end{eqnarray}
Along the same lines as above, we obtain 
\begin{eqnarray}
\mathbb{E}\left(\sup_{0\leq t\leq T}M_2(t)\right)\leq 4\lambda L\int_0^T\mathbb{E}||X^n(s^-)-X^{n-1}(s^-)||^2ds.
\label{ch2M2}
\end{eqnarray}
Inserting \eqref{ch2M0}, \eqref{ch2M1} and \eqref{ch2M2} in \eqref{ch2I4} leads to :
\begin{eqnarray*}
\mathbb{E}\left(\sup_{0\leq t\leq T}||X^{n+1}(t)-X^n(t)||^2\right)\leq (3L_{\lambda}+12L+12\lambda L)\int_0^T\mathbb{E}||X^n(s^-)-X^{n-1}(s^-)||^2ds.
\end{eqnarray*}
Using  Lemma \ref{ch2lemma2}, it follows that :
\begin{eqnarray*}
\mathbb{E}\left(\sup_{0\leq t\leq T}||X^{n+1}(t)-X^n(t)||^2\right)\leq C\int_0^T\dfrac{(Ms^-)^n}{n!}ds=C\dfrac{(MT^-)^{n+1}}{(n+1)!},
\end{eqnarray*}
where $C=3L_{\lambda}+12L+12\lambda L$.

It follows from Doop's and Markov's inequalities that :
\begin{eqnarray*}
\mathbb{P}\left(\sup_{0\leq t\leq T}||X^{n+1}(t)-X^n(t)||>\dfrac{1}{2^{n+1}}\right)\leq \dfrac{\mathbb{E}||X^{n+1}(t)-X^n(t)||^2}{\left(\dfrac{1}{(2^{n+1})}\right)^2}\leq C\dfrac{(2^2MT)^{n+1}}{(n+1)!}.
\end{eqnarray*}
Moreover, 
\begin{eqnarray*}
\sum_{n=0}^{\infty}\dfrac{(2^2MT)^{n+1}}{(n+1)!}=e^{2^2MT}-1<\infty.
\end{eqnarray*}

Using Borel Cantelli's Lemma, it follows that :
\begin{eqnarray*}
\mathbb{P}\left(\sup_{0\leq t\leq T}||X^{n+1}(t)-X^n(t)||>\dfrac{1}{2^{n+1}}\right)=0, \hspace{0.5cm} \text{almost surely}.
\end{eqnarray*}
Therefore,  for almost every  $\omega\hspace{0.1cm}\in\hspace{0.1cm} \Omega, \hspace{0.1cm}\exists\hspace{0.1cm} n_0(\omega)$ such that 
\begin{eqnarray*}
\sup_{0\leq t\leq T}||X^{n+1}(t)-X^n(t)||\leq\dfrac{1}{2^{n+1}},\hspace{0.2cm} \forall\: n\geq n_0(\omega).
\end{eqnarray*}
Therefore $X^n$ converge uniformly on $[0,T]$. Let  $X$ be its limit.

 Futhermore, since $X^n_t$ is continuous and $\mathcal{F}_t$-measurable for all $n\in\mathbb{N}$, it follows that  $X(t)$ is  continuous and $\mathcal{F}_t$- measurable. It remains to prove that $X$ is a solution of equation \ref{ch2jdi1}.
 
 One can see that $(X^n)$ converges in $\mathbb{L}^2([0,T]\times \Omega, \mathbb{R}^n)$. Indeed, 
 \begin{eqnarray*}
 ||X^m_t-X^n_t||^2_{\mathbb{L}^2}\leq \sum_{k=n}^{m-1}||X^{k+1}_t-X^k_t||_{\mathbb{L}^2}
 \leq \sum_{k=n}^{\infty}\dfrac{(MT)^{k+1}}{(k+1)!}\longrightarrow 0, \hspace{0.5cm}\text{as}\hspace{0.5cm} n\longrightarrow \infty.
 \end{eqnarray*}
 Therefore, $(X^n)$ is a Cauchy sequence in a Banach $\mathbb{L}^2([0,T]\times \Omega, \mathbb{R}^n)$, so its converges to $X$. 
 
 Using Fatou's lemma, it follows that :
 \begin{eqnarray*}
 \mathbb{E}\left[\int_0^T||X_t-X^n_t||^2dt\right]\leq\liminf_{m\longrightarrow +\infty}\mathbb{E}\left[\int_0^T||X^m_t-X^n_t||^2dt\right]\longrightarrow 0,\hspace{0.3cm}\text{when} \hspace{0.3cm}m\longrightarrow +\infty.
 \end{eqnarray*}
 Using the global Lipschitz condition and the Ito isometry, it follows that 
  \begin{eqnarray*}
 \mathbb{E}\left\|\int_0^t[g(X_s)-g(X^n_s)]dW_s\right\|^2\leq L\mathbb{E}\int_0^t||X_s-X^n_s||^2ds\longrightarrow 0\hspace{0.3cm}\text{when}\hspace{0.3cm}n\longrightarrow+\infty.
 \end{eqnarray*}
 So we have the following convergence in $\mathbb{L}^2([0,T]\times\mathbb{R}^n)$
 \begin{eqnarray*}
 \int_0^tg(X^n_s)dW_s\longrightarrow \int_0^tg(X_s)dW_s.
 \end{eqnarray*}
 Along the same lines as above, the following holds convergence holds in $\mathbb{L}^2([0,T]\times\mathbb{R}^n)$
 \begin{eqnarray*}
 \int_0^tg(X^n_s)d\overline{N}_s\longrightarrow \int_0^tg(X_s)d\overline{N}_s.
 \end{eqnarray*}
Using Holder inequality and the global Lipschitz condition, it follows that :
\begin{eqnarray*}
 \mathbb{E}\left\|\int_0^t[f_{\lambda}(X_s)-f_{\lambda}(X^n_s)]ds\right\|^2\leq L_{\lambda}\mathbb{E}\int_0^t||X_s-X^n_s||^2ds\longrightarrow 0\hspace{0.3cm}\text{when}\hspace{0.3cm}n\longrightarrow+\infty.
\end{eqnarray*}
So the following convergence holds in $\mathbb{L}^2([0,T]\times\mathbb{R}^n)$
\begin{eqnarray*}
 \int_0^tf_{\lambda}(X^n_s)d\overline{N}_s\longrightarrow \int_0^tf_{\lambda}(X_s)d\overline{N}_s.
 \end{eqnarray*}
 Therefore, taking  the limit in the sense of $\mathbb{L}^2([0,T]\times\mathbb{R}^n)$ in the both sides of the following  equality : 
\begin{eqnarray*}
X^{n+1}(t)=X_0+\int_0^tf(X^n(s^-))ds+\int_0^tg(X^n(s^-))dW(s)+\int_0^th(X^n(s^-))d\overline{N}(s)ds
\end{eqnarray*}
leads to :
\begin{eqnarray*}
X(t)=X_0+\int_0^tf(X(s^-))ds+\int_0^tg(X(s^-))dW(s)+\int_0^th(X(s^-))d\overline{N}(s)ds.
\end{eqnarray*}
So $X(t)$ is a strong solution of \eqref{ch2jdi1}. This complete the proof of Theorem \ref{ch2th1}.
\end{proof}

 Generally, analitycal solutions of SDEs are unknows. Knowing that the exact solution exist, one tool to approach it, is the numerical resolution. In the following chapters, we provide some numerical schemes for SDEs with jumps.

\chapter{Strong convergence and stabilty of the compensated stochastic theta methods }
\hspace{0.5cm}Our goal in this chapter is to prove the strong convergence of the compensated stochastic theta method (CSTM) and to analyse the stability behavior of both stochastic theta method(STM) and CSTM under global Lipschitz conditions. The strong convergence and stability of STM for SDEs with jumps has been investigated in \cite{Desmond1}, while  the strong convergence and stability  of the CSTM  for SDEs with jumps has been investigated in \cite{Xia1}. Most results presented in this chapter are from  \cite{Desmond1} and \cite{Xia1}. In the following section, we recall the theta method which will be used to introduce the STM and the CSTM.

\section{Theta Method}
\hspace{0.5cm}Let's consider the following deterministic differential equation 
$\left\{\begin{array}{ll}
u'(t)=f(t, u(t))\\
u(t_0)=u_0,
\end{array}
\right.$

which can be writen into its following integral  form :
\begin{eqnarray}
u(t)=u_0+\int_{t_0}^tf(s,u(s))ds.
\label{ch3Euler1}
\end{eqnarray}
\subsection{Euler explicit method}

This method use  the following approximation :
\begin{eqnarray*}
 \int_a^bf(s)ds\simeq (b-a)f(a).
 \end{eqnarray*}
So for a constant step $\Delta t$, the   Euler explicit approximation of \eqref{ch3Euler1} is given by  : 
\begin{eqnarray*}
u_{k+1}=u_k+\Delta tf(t_k,u_k),
\end{eqnarray*}
where $u_k :=u(t_k)$.
\subsection{Euler implicit method}
This method use the following approximation 
\begin{eqnarray*}
\int_a^bf(s)\simeq (b-a)g(b).
\end{eqnarray*}
Therefore, for a constant step $\Delta t$, the  Euler implicit  approximation of \eqref{ch3Euler1} is : 
\begin{eqnarray*}
u_{k+1}=u_k+\Delta tf(t_k,u_{k+1}).
\end{eqnarray*}
\subsection{ Theta Euler method}
In order to have a better approximation of the integral, we can take a convex combinaison of Euler explict and Euler implicit method. So we have the following approximation 
\begin{eqnarray*}
\int_a^bf(s)ds\simeq (b-a)[(1-\theta)f(a)+\theta f(b)],
\end{eqnarray*}
 where $\theta$ is a constant satisfying $0\leq \theta\leq 1$.

Hence, for a constant step $\Delta t$, the Euler theta  approximation of \eqref{ch3Euler1} is : 
\begin{eqnarray*}
u_{k+1}=u_k+\Delta t[(1-\theta)f(t_k, u_k)+\theta f(t_k,u_{k+1}) ].
\end{eqnarray*}
For $\theta=1$, the  Euler theta method  is called Euler backward method, which is also the Euler implicit method.

\subsection{Stochastic theta method and compensated stochastic theta method}(STM and CSTM)

\hspace{0.5cm} In order to have an approximate solution of equation \eqref{ch2jdi1}, we use the  theta Euler method for the deterministic integral and the Euler explicit method for the two random parts. So we have the following approximate solution of \eqref{ch2jdi1} called Stochastic theta method (STM) :
\begin{eqnarray}
Y_{n+1}=Y_n+(1-\theta)f(Y_n)\Delta t+\theta f(Y_{n+1})\Delta t+g(Y_n)\Delta W_n+h(Y_n)\Delta N_n,
\label{ch2approxi1}
\end{eqnarray}
where $Y_n :=X(t_n)$, $\Delta W_n :=W(t_{n+1})-W(t_n)$ and $ \Delta N_n := N(t_{n+1})-N(t_n)$

\hspace{0.5cm} Applying the same rules as for the STM to  equation  \eqref{ch2jdi3} leads to : 
\begin{eqnarray}
Y_{n+1}=Y_n+(1-\theta)f_{\lambda}(Y_n)\Delta t+\theta f_{\lambda}(Y_{n+1})\Delta t +g(Y_n)\Delta W_n+h(Y_n)\Delta\overline{N}_n,
\label{ch3comp1}
\end{eqnarray}
where 
\begin{eqnarray*}
f_{\lambda}(x)=f(x)+\lambda h(x).
\end{eqnarray*}

The numerical approximation \eqref{ch3comp1} is called compensated stochastic theta method (CSTM).

\section{Strong convergence of the CSTM on a finite time interval [0,T]}
\hspace{0.5cm} In this section, we prove the strong convergence of order $0.5$ of  the CSTM. Troughout,  $T$ is a fixed constant.
For $t\in [t_n,t_{n+1})$ we define the continuous time approximation of \eqref{ch3comp1} as follows : 
\begin{eqnarray}
\overline{Y}(t):=Y_n +(1-\theta)(t-t_n)f_{\lambda}(Y_n)+\theta(t-t_n)f_{\lambda}(Y_{n+1})+g(Y_n)\Delta W_n(t)+h(Y_n)\Delta\overline{N}_n(t),
\label{ch3contapproxi1}
\end{eqnarray}
where $\Delta W_n(t):=W(t)-W(t_n)$, $\Delta\overline{N}_n(t):=\overline{N}(t)-\overline{N}(t_n)$

The continuous approximation \eqref{ch3contapproxi1} can be writen into its following integral form : 
\begin{eqnarray}
\overline{Y}(t)&=&Y_0+\int_0^t\left[(1-\theta)f_{\lambda}(Y(s))+\theta f_{\lambda}(Y(s+\Delta t))\right]ds+\int_0^tg(Y(s))dW(s)\nonumber\\
&+&\int_0^th(Y(s))d\overline{N}(s),
\label{ch3contapproxi2}
\end{eqnarray}
where 
\begin{eqnarray*}
Y(s):=\sum_{n=0}^{\infty}\mathbf{1}_{\{t_n\leq s<t_{n+1}\}}Y_n.
\end{eqnarray*}
It follows from \eqref{ch3contapproxi1} that $\overline{Y}(t_n)=Y_n$. In others words, $\overline{Y}(t)$ and $Y_n$ coincide at the grid points.
 \hspace{0.5cm}The main result of this section is formulated in the following theorem.
\begin{thm}\label{ch3th1}
Under Assumptions \ref{ch2assump}, the continuous time approximation solution $\overline{Y}(t)$ given by \eqref{ch3contapproxi2} converges to the true solution $X(t)$ of \eqref{ch2jdi1} in the mean square sense. More precisely, 
\begin{eqnarray*}
\mathbb{E}\left(\sup_{0\leq t \leq T}||\overline{Y}(t)-X(t)||^2\right)\leq C(1+\mathbb{E}||X_0||^2)\Delta t,
\end{eqnarray*}
where $C$ is a positive constant independent of the stepsize $\Delta t$.
\end{thm}

\hspace{0.5cm} In order to prove  Theorem \ref{ch3th1}, we need the following two lemmas.
\begin{lem}\label{ch3lemma1}
Under Assumptions \ref{ch2assump}, there exist a fixed constant $\Delta t_0$ such that for any stepsize $\Delta t$ satisfying  $0<\Delta t<\Delta t_0<\dfrac{1}{K_{\lambda}+1}$,  the following bound of the numerical solution holds 
\begin{eqnarray*}
\sup_{0\leq n\Delta t\leq T}\mathbb{E}||Y_n||^2\leq C_1(1+\mathbb{E}||X_0||^2),
\end{eqnarray*}
where $C_1$ is a positive constant independent of $\Delta t$.
\end{lem}
\begin{proof}
From \eqref{ch3comp1}, it follows that :
\begin{eqnarray}
||Y_{n+1}-\theta \Delta t f_{\lambda}(Y_{n+1})||^2=||Y_n+(1-\theta)f_{\lambda}(Y_n)\Delta t+g(Y_n)\Delta W_n+h(Y_n)\Delta \overline{N}_n||^2.
\label{ch3eq1}
\end{eqnarray}
Taking expectation in both sides of \eqref{ch3eq1} leads to :
\begin{eqnarray}
\mathbb{E}||Y_{n+1}-\Delta t\theta f_{\lambda}(Y_{n+1})||^2&=&\mathbb{E}||Y_n||^2+(1-\theta)^2(\Delta t)^2\mathbb{E}||f_{\lambda}(Y_n)||^2+\mathbb{E}||g(Y_n)\Delta W_n||^2+\mathbb{E}||h(Y_n)\Delta\overline{N}_n||^2\nonumber\\
&+&2\mathbb{E}\langle Y_n,(1-\theta)f_{\lambda}(Y_n)\Delta t\rangle.
\label{ch3eq2}
\end{eqnarray}
Since $W$ is a Brwonian motion, $\Delta W_n=W_{t_{n+1}}-W_{t_n}\leftrightsquigarrow \mathcal{N}(0,t_{n+1}-t_n)$. So $\mathbb{E}(\Delta W_n)=0$.

Using the properties  $\mathbb{E}(\Delta W_n)=0$ and $\mathbb{E}(\Delta \overline{N}_n)=0$, we have
\begin{eqnarray*}
\mathbb{E}\langle Y_n,g(Y_n)\Delta W_n\rangle=\mathbb{E}\langle f_{\lambda}(Y_n)\Delta, g(Y_n)\Delta W_n\rangle=\mathbb{E} \langle f_{\lambda}(Y_n)\Delta t, h(Y_n)\Delta \overline{N}_n\rangle=0.
\end{eqnarray*}
The martingale properties of $\Delta W_n$ and $\Delta \overline{N}_n$ leads to :
\begin{eqnarray*}
\mathbb{E}||g(Y_n)\Delta W_n||^2=\mathbb{E}||g(Y_n)||^2\Delta t\hspace{0.5cm} \text{and}\hspace{0.5cm}
\mathbb{E}||h(Y_n)\Delta \overline{N}_n||^2=\lambda \Delta t\mathbb{E}||h(Y_n)||^2.
\end{eqnarray*}
Hence equality \eqref{ch3eq2} becomes :
\begin{eqnarray}
\mathbb{E}||Y_{n+1}-\Delta t\theta f_{\lambda}(Y_{n+1})||^2&=&\mathbb{E}||Y_n||^2+(1-\theta)^2(\Delta t)^2\mathbb{E}||f_{\lambda}(Y_n)||^2+\mathbb{E}||g(Y_n)||^2\Delta t\nonumber\\
&+&2(1-\theta)\Delta t\mathbb{E}\langle Y_n, f_{\lambda}(Y_n)\rangle+\lambda \Delta t\mathbb{E}||h(Y_n)||^2.
\label{ch3eq3}
\end{eqnarray}
Using Cauchy-Schwartz inequality and the linear growth condition, it follows that :
\begin{eqnarray*}
\mathbb{E}\langle Y_n, f_{\lambda}(Y_n)\rangle=\mathbb{E}||Y_nf_{\lambda}(Y_n)||&\leq&\sqrt{\mathbb{E}||Y_n||^2\mathbb{E}||f_{\lambda}(Y_n)||^2}\\
&\leq& \dfrac{1}{2}\mathbb{E}||Y_n||^2+\dfrac{1}{2}\mathbb{E}||f_{\lambda}(Y_n)||^2\\
&\leq& \dfrac{K_{\lambda}}{2}+\dfrac{1}{2}(1+K_{\lambda})\mathbb{E}||Y_n||^2.
\end{eqnarray*}
By the same arguments as above, it follows that : 
\begin{eqnarray*}
\mathbb{E}\langle Y_{n+1}, f_{\lambda}(Y_{n+1})\rangle\leq \dfrac{K_{\lambda}}{2}+\dfrac{1}{2}(1+K_{\lambda})\mathbb{E}||Y_{n+1}||^2.
\end{eqnarray*}
Since $\theta\in [0,1]$ and $\Delta t\in]0,1[$, it follows from \eqref{ch3eq3}  that :
\begin{eqnarray*}
\mathbb{E}||Y_{n+1}||^2&\leq& 2\Delta t\mathbb{E}|\langle Y_{n+1},f_{\lambda}(Y_{n+1})\rangle|+\mathbb{E}||Y_n||^2+\Delta t \mathbb{E}||f_{\lambda}(Y_n)||^2+\Delta t\mathbb{E}||g(Y_n)||^2\\
&+& 2\Delta t\mathbb{E}Y_n, f_{\lambda}(Y_n)+\lambda\Delta t\mathbb{E}||h(Y_n)||^2\\
&\leq&2\Delta t\left[\dfrac{1}{2}(1+K_{\lambda})\mathbb{E}||Y_{n+1}||^2+\dfrac{1}{2}K_{\lambda}\right]+\mathbb{E}||Y_n||^2+\Delta tK_{\lambda}(1+\mathbb{E}||Y_n||^2)+\Delta tK(1+\mathbb{E}||Y_n||^2)\\
&+&2\Delta t\left[\dfrac{1}{2}(1+K_{\lambda})\mathbb{E}||Y_n||^2+\dfrac{1}{2}K_{\lambda}\right]+\lambda \Delta tK(1+\mathbb{E}||Y_n||^2).
\end{eqnarray*}
Therefore, from the above inequality the following holds : 
\begin{eqnarray*}
(1-\Delta t-\Delta tK_{\lambda})\mathbb{E}||Y_{n+1}||^2&\leq &(1+\Delta tK_{\lambda}+\Delta tK+\Delta t+\Delta tK_{\lambda}+\lambda \Delta tK)\mathbb{E}||Y_n||^2\\
&+&\Delta tK_{\lambda}+\Delta tK_{\lambda}+\Delta tK+\Delta tK_{\lambda}+\lambda \Delta tK\\
&\leq & (1+2\Delta tK_{\lambda}+\Delta tK+\Delta t+\Delta t\lambda K)\mathbb{E}||Y_n||^2+3\Delta t K_{\lambda}+\Delta tK+\lambda\Delta t K.
\end{eqnarray*}
Then it follows from the previous inequality that :
\begin{eqnarray*}
\mathbb{E}||Y_{n+1}||^2\leq\left(1+\dfrac{3K_{\lambda}\Delta t+K\Delta t+\lambda K\Delta t+2\Delta t}{1-\Delta t-K_{\lambda}\Delta t}\right)\mathbb{E}||Y_n||^2+\dfrac{3K_{\lambda}\Delta t+K\Delta t+\lambda K\Delta t+2\Delta t}{1-\Delta t-K_{\lambda}\Delta t}.
\end{eqnarray*}
Since $\Delta t<\Delta t_0<\dfrac{1}{K_{\lambda}+1}$, we have  $1-\Delta t-K_{\lambda}\Delta t>1-\Delta t_0-K_{\lambda}\Delta t_0>0$ and then 
\begin{eqnarray*}
\mathbb{E}||Y_{n+1}||^2\leq\left(1+\dfrac{3K_{\lambda}+K+\lambda K t+2 }{1-\Delta t_0-K_{\lambda}\Delta t_0}\Delta t\right)\mathbb{E}||Y_n||^2+\dfrac{3K_{\lambda}+K+\lambda K t+2}{1-\Delta t_0-K_{\lambda}\Delta t_0}\Delta t_0.
\end{eqnarray*}
In the short form, we have  
\begin{eqnarray}
\mathbb{E}||Y_{n+1}||^2\leq (1+A)\mathbb{E}||Y_{n}||^2+B,
\label{ch3eq4}
\end{eqnarray}
where 
\begin{eqnarray*}
A=\dfrac{3K_{\lambda}+K+\lambda K+2}{1-\Delta t_0-K_{\lambda}\Delta t_0}\Delta t\hspace{0.5cm}\text{and}\hspace{0.5cm}
B=\dfrac{3K_{\lambda}+K+\lambda K+2}{1-\Delta t_0-K_{\lambda}\Delta t_0}\Delta t_0.
\end{eqnarray*}
Applying Gronwall lemma (discrete form) to \eqref{ch3eq4}  leads to : 
\begin{eqnarray}
\mathbb{E}||Y_n||^2<e^{nA}\mathbb{E}||X_0||^2+B\dfrac{e^{nA}-1}{e^A-1},
\label{ch3eq5}
\end{eqnarray}
\begin{eqnarray*}
nA=\dfrac{3K_{\lambda}+K+\lambda K+2}{1-\Delta t_0-K_{\lambda}\Delta t_0}n\Delta t \leq \dfrac{3K_{\lambda}+K+\lambda K+2}{1-\Delta t_0-K_{\lambda}\Delta t_0}T,\hspace{0.2cm} \text{since}\hspace{0.2cm} n\Delta t\leq T.
\end{eqnarray*}
Therefore, it follows from \eqref{ch3eq5} that :
\begin{eqnarray}
\mathbb{E}||Y_n||^2\leq e^C\mathbb{E}||X_0||^2+B\dfrac{e^C-1}{e^D-1},
\label{ch3eq6}
\end{eqnarray}
where
\begin{eqnarray*}
C=\dfrac{3K_{\lambda}+K+\lambda K+2}{1-\Delta t_0-K_{\lambda}\Delta t_0}T\hspace{0.2cm}\text{and}\hspace{0.5cm}
D=\dfrac{3K_{\lambda}+K+\lambda K+2}{1-\Delta t_0-K_{\lambda}\Delta t_0}\Delta t_0.
\end{eqnarray*}
 It is straightforward to see that $B$, $C$  and $D$ are independents of $\Delta t$.

\eqref{ch3eq6} can be rewritten into the following appropriate form :
\begin{eqnarray*}
\mathbb{E}||Y_n||^2\leq C_1(1+\mathbb{E}||X_0||^2)\hspace{0.3cm} \hspace{0.2cm} C_1=\max\left(e^C, B\dfrac{e^C-1}{e^D-1}\right).
\end{eqnarray*}
This complete the proof of the lemma.
\end{proof}

\begin{lem}\label{ch3lemma2}
If the conditions of  Lemma \ref{ch3lemma1} are satisfied, then there exist a positive constant $C_2$ independent of $\Delta t$ such that for $s\in [t_n, t_{n+1})$
\begin{eqnarray*}
\mathbb{E}||\overline{Y}(s)-Y(s)||^2\vee \mathbb{E}||\overline{Y}(s)-Y(s+\Delta t)||^2\leq C_2(1+\mathbb{E}||X_0||^2)\Delta t.
\end{eqnarray*}
\end{lem}
\begin{proof}
\begin{enumerate}
\item The continous interpolation of the numerical solution \eqref{ch3comp1} is given by 
\begin{eqnarray*}
\overline{Y}(s)=Y_n+(1-\theta)(s-t_n)f_{\lambda}(Y_n)+\theta(s-t_n)f_{\lambda}(Y_{n+1})+g(Y_n)\Delta W_n(s)+h(Y_n)\Delta \overline{N}(s),
\end{eqnarray*}
where
\begin{eqnarray*}
Y(s)=\sum_{n=0}^{\infty}1_{\{t_n\leq s<t_{n+1}\}}Y_n.
\end{eqnarray*}
For $s\in[t_n, t_{n+1})$, we have $Y(s)=Y_n$. Then, we have the following equality :
\begin{eqnarray}
\overline{Y}(s)-Y(s)&=&(1-\theta)(s-t_n)f_{\lambda}(Y_n)+\theta(s-t_n)f_{\lambda}(Y_{n+1})+g(Y_n)\Delta W_n(s)\nonumber\\
&+&h(Y_n)\Delta\overline{N}(s).
\label{ch3eq7}
\end{eqnarray}
By squaring both sides of \eqref{ch3eq7} and taking expectation, using the martingale properties of $\Delta W_n$ and $\Delta\overline{N}_n$ leads to : 
\begin{eqnarray}
\mathbb{E}||\overline{Y}(s)-Y(s)||^2&\leq &3(1-\theta)^2(s-t_n)^2\mathbb{E}||f_{\lambda}(Y_n)||^2+3\theta^2(s-t_n)^2\mathbb{E}|f_{\lambda}(Y_{n+1})|^2+3\mathbb{E}||g(Y_n)\Delta W_n(s)||^2\nonumber\\
&+&3\mathbb{E}||h(Y_n)\Delta\overline{N}_n(s)||^2\nonumber\\
&\leq & 3(1-\theta)^2\Delta t^2\mathbb{E}||f_{\lambda}(Y_n)||^2+3\theta ^2\Delta t^2\mathbb{E}||f_{\lambda}(Y_{n+1})||^2+3\Delta t\mathbb{E}||g(Y_n)||^2\nonumber\\
&+&3\lambda\Delta t\mathbb{E}||h(Y_n)||^2.
\label{ch3eq8}
\end{eqnarray}
By using the linear growth condition and the fact that $\theta\in[0,1]$, it follows from \eqref{ch3eq8} that  :
\begin{eqnarray}
\mathbb{E}||\overline{Y}(s)-Y(s)||^2&\leq &3\Delta tK_{\lambda}(1+\mathbb{E}||Y_n||^2)+3\Delta K_{\lambda}(1+\mathbb{E}||Y_{n+1}||^2)\nonumber\\
&+&3\Delta tK(1+\mathbb{E}||Y_n||^2)+3\Delta t\lambda K(1+\mathbb{E}||Y_n||^2).
\label{ch3eq9}
\end{eqnarray}
Now by application of  Lemma \ref{ch3lemma2} to \eqref{ch3eq9}, it follows that there exist a constant $C_1>0$ independent of $\Delta t$ such that :
\begin{eqnarray*}
\mathbb{E}||\overline{Y}(s)-Y(s)||^2\leq C_1(1+\mathbb{E}||X_0||^2)\Delta t.
\end{eqnarray*}
\item  For $s\in[t_n, t_{n+1})$,  $s+\Delta t \in[t_{n+1}, t_{n+2})$ and then $Y(s+\Delta t)=Y_{n+1}$. So it follows from \eqref{ch3contapproxi1} that : 
\begin{eqnarray*}
Y(s+\Delta t)&=&Y_{n+1}\\
&=&Y_n+(1-\theta)(t_{n+1}-t_n)f_{\lambda}(Y_n)+\theta(t_{n+1}-t_n)f_{\lambda}(Y_{n+1})+g(Y_n)\Delta W_n+h(Y_n)\Delta\overline{N}_n.
\end{eqnarray*}
So we have 
\begin{eqnarray}
\overline{Y}(s)-Y(s+\Delta t)&=&(1-\theta)(s-t_{n+1})f_{\lambda}(Y_n)+\theta(s-t_{n+1})f_{\lambda}(Y_{n+1})+g(Y_n)\left(W(s)-W(t_{n+1})\right)\nonumber\\
&+&h(Y_n)\left(\overline{N}(s)-\overline{N}(t_{n+1})\right).
\label{ch3eq10}
\end{eqnarray}
By squaring both sides of \eqref{ch3eq10},  taking expectation  and using martingale properties of $\Delta W_n$ and $\Delta\overline{N}_n$, it follows that
\begin{eqnarray*}
\mathbb{E}||\overline{N}(s)-Y(s+\Delta t)||^2\leq 3\Delta t\mathbb{E}||f_{\lambda}(Y_n)||^2+3\Delta t\mathbb{E}||f_{\lambda}(Y_{n+1})||^2+3\Delta t\mathbb{E}||g(Y_n)||^2+3\lambda \Delta t \mathbb{E}||h(Y_n)||^2.
\end{eqnarray*}
Applying respectively the linear growth condition and  Lemma \ref{ch3lemma2} to the previous inequality, it follows that there exist a positive constant $C_2$  independent of $\Delta t$ such that :
\begin{eqnarray*}
\mathbb{E}||\overline{Y}(s)-Y(s+\Delta t)||^2\leq C_2(1+\mathbb{E}||X_0||^2)\Delta t.
\end{eqnarray*}
\end{enumerate}
This complete the proof of Lemma \ref{ch3lemma2}.

\hspace{0.5cm} Now, we are ready to give the proof of Theorem \ref{ch3th1}.
\end{proof}

\begin{proof}\textbf{[Theorem \ref{ch3th1}]}

From equations \eqref{ch2jdi3} and \eqref{ch3contapproxi2}, it follows that :
\begin{eqnarray*}
||\overline{Y}(s)-X(s)||^2&=&\left\|\int_0^s[(1-\theta)(f_{\lambda}(Y(r))-f_{\lambda}(X(r^-))+\theta(f_{\lambda}(Y+\Delta t))-f_{\lambda}(X(r^-))]dr\right.\\
&+&\left.\int_0^s(g(Y(r))-g(X(r^-)))dW(r)+\int_0^s(h(Y(r))-h(X(r^-)))d\overline{N}(r)\right\|^2.
\end{eqnarray*}
Taking expectation in both sides of the above equality and using the inequality 

$(a+b+c)^2\leq 3a^2+3b^2+3c^2$ for all $a, b, c\in\mathbb{R}$ leads to :
\begin{eqnarray}
\mathbb{E}\left[\sup_{0\leq s\leq t}||\overline{Y}(s)-X(s)||^2\right]\leq 3\mathbb{E}\left(\sup_{0\leq s\leq t}M_1(t)\right)+3\mathbb{E}\left(\sup_{0\leq s\leq t}M_2(t)\right)+3\mathbb{E}\left(\sup_{0\leq s\leq t}M_3(t)\right),
\label{ch3eq11}
\end{eqnarray}
where
\begin{eqnarray*}
M_1(t)=\left\|\int_0^s[(1-\theta)(f_{\lambda}(Y(r))-f_{\lambda}(X(r^-))]dr\right\|^2,\hspace{0.5cm}
M_2(s)=\left\|\int_0^s[g(Y(r))-g(X(r^-))]dW(r)\right\|^2
\end{eqnarray*}
\begin{eqnarray*}
\text{and}\hspace{0.5cm} M_3(s)=\left\|\int_0^s[h(Y(r))-h(X(r^-))]d\overline{N}(r)\right\|^2.
\end{eqnarray*}
Using Holder inequality, it follows that :
\begin{eqnarray}
M_1(s)\leq s\int_0^s\left\|(1-\theta)(f_{\lambda}(Y(r))-f_{\lambda}(X(r^-))+\theta(f_{\lambda}(Y(r+\Delta t))-f_{\lambda}(X(r^-))\right\|^2dr.
\label{ch3eq12}
\end{eqnarray}
Using the convexity of the application $x\longmapsto ||x||^2$, it follows from \eqref{ch3eq12} that
\begin{eqnarray*}
M_1(s)\leq s\int_0^s(1-\theta)||f_{\lambda}(Y(r))-f_{\lambda}(X(r^-))||^2dr+s\int_0^s\theta||f_{\lambda}(Y(r+\Delta t))-f_{\lambda}(X(r^-))||^2dr.
\end{eqnarray*}
 Taking the supremum in both sides of the above inequality and using the global Lipschitz condition satisfied by $f_{\lambda}$, and then taking expectation it follows that
\begin{eqnarray}
\mathbb{E}\left[\sup_{0\leq s\leq t}M_1(s)\right]\leq t(1-\theta)L_{\lambda}\int_0^t\mathbb{E}||Y(r)-X(r^-)||^2dr+t\theta L_{\lambda}\int_0^t\mathbb{E}||Y(r)-X(r^-)||^2dr.
\label{ch3eq13}
\end{eqnarray}
Using Doop's inequality, it follows that :
\begin{eqnarray*}
\mathbb{E}\left[\sup_{0\leq s\leq t}M_2(s)\right]\leq 4\sup_{0\leq s\leq t}\mathbb{E}[M_2(s)]=4\sup_{0\leq s\leq t}\int_0^s\mathbb{E}||g(Y(r))-g(X(r^-))||^2dr.
\end{eqnarray*}
Using the global Lipschitz condition satisfied by $g$, it follows that :
\begin{eqnarray}
\mathbb{E}\left[\sup_{0\leq s\leq t}M_2(s)\right]\leq 4L\int_0^t\mathbb{E}||Y(r)-X(r^-)||^2dr.
\label{ch3eq14}
\end{eqnarray}
Along the same lines as above, we have 
\begin{eqnarray}
\mathbb{E}\left[\sup_{0\leq s\leq t}M_3(s)\right]=4\lambda L\int_0^t\mathbb{E}||Y(r)-X(r^-)||^2dr.
\label{ch3eq15}
\end{eqnarray}
Inserting \eqref{ch3eq13}, \eqref{ch3eq14} and \eqref{ch3eq15} in \eqref{ch3eq11} leads to :
\begin{eqnarray}
\mathbb{E}\left[\sup_{0\leq s\leq t}||\overline{Y}(s)-X(s)||^2\right]&\leq& 3T(1-\theta)L_{\lambda}\int_0^t\mathbb{E}||Y(r)-X(r^-)||^2dr\nonumber\\
&+&3T\theta L_{\lambda}\int_0^t\mathbb{E}||Y(r)-X(r^-)||^2dr\nonumber\\
&+&12L\int_0^t\mathbb{E}||Y(r)-X(r^-)||^2dr\nonumber\\
&+&12\lambda L\int_0^t\mathbb{E}||Y(r)-X(r^-)||^2dr.
\label{ch3eq16}
\end{eqnarray}
Using the fact that
\begin{eqnarray*}
||Y(r)-X(r^-)||^2=||(Y(r)-\overline{Y}(r)-(X(r^-)-\overline{Y}(r))||^2\leq 2||Y(r)-\overline{Y}(r)||^2+2||X(r^-)-\overline{Y}(r)||^2,
\end{eqnarray*}
it follows from \eqref{ch3eq16} that : 
\begin{eqnarray*}
\mathbb{E}\left[\sup_{0\leq s\leq t}||\overline{Y}(s)-X(s)||^2\right]&\leq & 6T(1-\theta)L_{\lambda}\int_0^t\left[\mathbb{E}||Y(r)-\overline{Y}(r)||^2+\mathbb{E}||\overline{Y}(r)-X(r^-)||^2\right]dr\\
&+&6T\theta L_{\lambda}\int_0^t\left[\mathbb{E}||Y(r)-\overline{Y}(r)||^2+\mathbb{E}||\overline{Y}(r)-X(r^-)||^2\right]dr\\
&+&24L(1+\lambda)\int_0^t\left[\mathbb{E}||Y(r)-\overline{Y}(r)||^2+\mathbb{E}||\overline{Y}(r)-X(r^-)||^2\right]dr.
\end{eqnarray*}
Using  lemma \ref{ch3lemma2} in the above inequality, it follows that 
\begin{eqnarray}
\mathbb{E}\left[\sup_{0\leq s\leq t}||\overline{Y}(s)-X(s)||^2\right]&\leq & [6TL_{\lambda}+24L(1+\lambda)]\int_0^t\mathbb{E}\left[\sup_{0\leq r\leq s}||\overline{Y}(r)-X(r)||^2\right]ds\nonumber\\
&+&[6T^2L_{\lambda}+24TL(1+\lambda)]C_2(1+\mathbb{E}||X_0||^2)\Delta t.
\label{ch3eq17}
\end{eqnarray}
Applying Gronwall lemma (continous form) to \eqref{ch3eq17} leads to  the existence of  a positive constant $C$ independent of $\Delta t$ such that 
\begin{eqnarray*}
\mathbb{E}\left[\sup_{0\leq s\leq t}||\overline{Y}(s)-X(s)||^2\right]\leq C(1+\mathbb{E}||X_0||^2)\Delta t.
\end{eqnarray*}
This complete the proof of Theorem \ref{ch3th1}.
\end{proof}

\hspace{0.5cm} The strong convergence of the STM has been studied in \cite{Desmond1}. Since STM and CSTM     convergence strongly  to the exact slution, it is interesting to study their stability behaviours. 
\section{Linear mean-square stability of the CSTM }
\hspace{0.5cm} In this section, we focus on the linear mean-square stability.
Let's consider the following linear test equation with real coefficients 
\begin{eqnarray}
dX(t)=aX(t^-)dt+bX(t^-)dW(t)+cX(t^-)dN(t), \hspace{0.5cm} X(0)=X_0.
\label{ch3lin1}
\end{eqnarray}
\begin{defn}
The exact solution $X$ of   SDEs  is said to be exponentially mean-square stable  if there exist constants $\alpha>0$ and  $L>0$ such that :
\begin{eqnarray*}
\mathbb{E}||X(t)||^2\leq Le^{-\alpha t}\mathbb{E}||X(0)||^2.
\end{eqnarray*}
\end{defn}

\begin{defn}
\begin{enumerate}
\item
The numerical solution  $X_n$ of  SDEs  is said to be exponentially mean-square stable  if there exist constants $\alpha>0$ and  $L>0$ such that :
\begin{eqnarray*}
\mathbb{E}||X_n||^2\leq Le^{-\alpha t}\mathbb{E}||X(0)||^2.
\end{eqnarray*}
\item The numerical solution  $X_n$ of  SDEs  is said to be  mean-square stable  if there exist constants  $0<L<1$ such that : for all $n\in[0,T]$
\begin{eqnarray*}
\mathbb{E}||X_{n+1}||^2<L\mathbb{E}||X_n||^2.
\end{eqnarray*}
\item A numerical method is said to be A-stable if it is stable for any stepsize.
\end{enumerate}
\end{defn}
It is proved in \cite{Desmond1} that the exact solution of \eqref{ch3lin1} have the following stability property :
\begin{eqnarray}
\lim_{t\longrightarrow \infty}\mathbb{E}||X(t)||^2=0\Longleftrightarrow l := 2a+b^2+\lambda c(2+c)<0,
\label{ch3lin2}
\end{eqnarray}
where $\lambda$ is the intensity of the poisson precess $(N_t)_{t\geq 0}$.
\begin{thm} Under condition \eqref{ch3lin2}, the  numerical solution of \eqref{ch3lin1} produced by compensated stochastic theta method  is mean-square stable for any stepsize $\Delta t>0$ if and only if 

$\dfrac{1}{2}\leq\theta \leq 1$. For $0\leq \theta<\dfrac{1}{2}$ this numerical solution is mean-square stable for any stepsize $\Delta t>0$ satisfying : 
\begin{eqnarray*}
\Delta t<\dfrac{-l}{(1-2\theta)(a+\lambda c)^2}.
\end{eqnarray*}
\end{thm}

\begin{proof}
Applying the compensated theta method to \eqref{ch3lin1} gives 
\begin{eqnarray*}
Y_{n+1}=Y_n+(1-\theta)\Delta t(a+\lambda c)Y_n+\theta \Delta t(a+\lambda c)Y_{n+1}+bY_n\Delta W_n+cY_n\Delta\overline{N}_n.
\end{eqnarray*}
So we have
\begin{eqnarray*}
(1-\theta \Delta ta-\theta \Delta t\lambda c)Y_{n+1}=Y_n+(1-\theta)\Delta t(a+\lambda c)Y_n+bY_n\Delta W_n+c\Delta\overline{N}_n.
\end{eqnarray*}
It follows that :
\begin{eqnarray*}
(1-\theta \Delta ta-\theta \Delta t\lambda c)^2\mathbb{E}||Y_{n+1}||^2=[1+(1-\theta)\Delta t(a+\lambda c)]^2\mathbb{E}||Y_n||^2+b^2\Delta t\mathbb{E}||Y_n||^2+c^2\lambda\Delta t\mathbb{E}||Y_n||^2.
\end{eqnarray*}
Therefore,
\begin{eqnarray*}
\mathbb{E}||Y_{n+1}||^2=\dfrac{1+[2(1-\theta)(a+\lambda c)+b^2+c^2\lambda]\Delta t+(1-\theta)^2(a+\lambda c)^2\Delta t^2}{(1-\theta \Delta t a-\theta \Delta t\lambda c)^2}\mathbb{E}||Y_n||^2.
\end{eqnarray*}
It follows that $\mathbb{E}||Y_n||^2$ is a geometric sequence which converge if and only if 
\begin{eqnarray*}
\dfrac{1+[2(1-\theta)(a+\lambda c)+b^2+c^2\lambda]\Delta t+(1-\theta)^2(a+\lambda c)^2\Delta t^2}{(1-\theta \Delta t a-\theta \Delta t\lambda c)^2}<1.
\end{eqnarray*}
That is  if and only if 
\begin{eqnarray}
(1-2\theta)(a+\lambda c)^2\Delta t<-l.
\label{ch3lin3}
\end{eqnarray}
It follows  that :
\begin{itemize}
\item If $\dfrac{1}{2}\leq \theta\leq 1$, then the condition \eqref{ch3lin3} is satisfied for any stepsize. And then the numerical solution is mean-square stable for any stepsize.
\item If $0\leq\theta<\dfrac{1}{2}$, then it follows from \eqref{ch3lin3} that if $0<\Delta t<\dfrac{-l}{(1-2\theta)(a+\lambda c)^2}$, the numerical method is stable.
\end{itemize}
\end{proof}
\begin{rem}
Changing $c$ to $-2-c$ does not  affect the  mean-square stability conditon \eqref{ch3lin2}. Hence the exact solution of \eqref{ch3lin1} have the same stability property under this transformation. It is interesting to look for  what happens to the numerical solution under this transformation.
\end{rem}
\begin{defn}
A numerical method applied to \eqref{ch3lin1} is said to be jump symmetric if whenever stable (unstable) for $\{a,b,c,\lambda , \Delta t\}$ it is also stable (unstable) for  $\{a,b,-2-c,\lambda , \Delta t\}$.
\end{defn}
\begin{cor}
The compensated stochastic theta method applied to \eqref{ch3lin1} is jump symmetric if and only if $\theta =\dfrac{1}{2}$.
\end{cor}
\begin{proof}
\begin{enumerate}
\item For  $\theta=\dfrac{1}{2}$, clearly  the stability condition \eqref{ch3lin3} of the numerical solution is equivalent to the stability condition \eqref{ch3lin2} of the exact solution. Since the stability condition \eqref{ch3lin2} is invariant under the transformation $c\longmapsto -2-c$, it follows that  the jump symmetric property holds.
\item If $\theta \neq \dfrac{1}{2}$, the right hand side of \eqref{ch3lin2} remains the same under the transformation $c\longmapsto -2-c$, but the left hand side changes. Therefore the jump symmetry property does not holds.
\end{enumerate}
\end{proof}
\begin{rem}
If the exact solution of the problem \eqref{ch3lin1} is mean-square stable, then for $\dfrac{1}{2}<\theta\leq 1$, it follows from \eqref{ch3lin3} the stability property of the CSTM is preserved under the transformation 

$c\longmapsto -2-c$.
\end{rem}

\section{Nonlinear mean-square stability}
\hspace{0.5cm} This section is devoted to the nonlinear mean-square analysis.

Troughout, this section, we make the following assumptions.
\begin{assumption}\label{ch3assump1}
We assume that there exist  constants $\mu, \sigma, \gamma$ such that for all $x,y\in\mathbb{R}^n$
\begin{eqnarray}
\langle x-y, f(x)-f(y)\rangle\leq \mu||x-y||^2\label{ch3nonlin1}\\
||g(x)-g(y)||^2\leq \sigma||x-y||^2\nonumber\\
||h(x)-h(y)||^2\leq \gamma||x-y||^2.\nonumber
\end{eqnarray}
Usually, condition \eqref{ch3nonlin1} is called "one-sided Lipschitz condition".
\end{assumption}
\subsection{Nonlinear mean-square stability of the exact solution}
\begin{thm}\cite[Theorem 4, pp 13]{Desmond2} \label{ch3th2}

Under assumptions \ref{ch3assump1}, any two solutions $X(t)$ and $Y(t)$ of the SDEs with jumps \eqref{ch2jdi1} with $\mathbb{E}||X_0||^2<\infty$ and $\mathbb{E}||Y_0||^2<\infty$ satisfy the following property
\begin{eqnarray*}
\mathbb{E}||X(t)-Y(t)||^2\leq \mathbb{E}||X_0-Y_0||^2e^{\alpha t},
\end{eqnarray*}
where $\alpha : =2\mu+\sigma +\lambda\sqrt{\gamma}(\sqrt{\gamma}+2)$.

Hence, the condition $\alpha <0$ is  sufficient  for the exponential mean-square stability property.
\end{thm}
\begin{proof}
The two solutions $X(t)$ and $Y(t)$ of \eqref{ch2jdi1} satisfy respectively
\begin{eqnarray*}
dX(t)=f(X(t^-))dt+g(X(t^-))dW(t)+h(X(t^-))dN(t)
\end{eqnarray*}
and 
\begin{eqnarray*}
dY(t)=f(Y(t^-))dt+g(Y(t^-))dW(t)+h(Y(t^-))dN(t).
\end{eqnarray*}
Applying It\^{o}'s lemma for product (Lemma \ref{ch1Itoproduct}) to the stochastic process $Z(t)=||X(t)-Y(t)||^2$ leads to 
\begin{eqnarray*}
dZ(t)&=&2\langle X(t^-)-Y(t^-),d(X(t^-))-d(Y(t^-))\rangle+||d(X(t^-))-d(Y(t^-))||^2\\
&=&\left[2\langle X(t^-)-Y(t^-), f(X(t^-))-f(Y(t^-))\rangle+2\lambda \langle X(t^-)-Y(t^-), h(X(t^-))-h(Y(t^-))\rangle\right.\\
&+&\left.||g(X(t^-))-g(Y(t^-))||^2+\lambda||h(X(t^-))-h(Y(t^-))||^2\right]dt+dM_t,
\end{eqnarray*}
where $M_t$ is a martingale and
where we used the following  rule of calculation 
\begin{eqnarray*}
dtdt=dtdW(t)=0,\hspace{0.5cm} dN_tdN(t)=dN_t,\hspace{0.5cm} dW_tdW_t=dt\hspace{0.5cm} and \hspace{0.2cm}dtdN(t)=dW_tdN_t=0.
\end{eqnarray*}
Using Assumptions \ref{ch3assump1}, ones get :
\begin{eqnarray*}
d||X(t)-Y(t)||^2&\leq& [2\mu||X(t^-)-Y(t^-)||^2+2\lambda\sqrt{\gamma}||X(t^-)-Y(t^-)||^2+\sigma||X(t^-)-Y(t^-)||^2\\
&+&\lambda\gamma||X(t^-)-Y(t^-)||^2]dt +dM(t),
\end{eqnarray*}

So we have 
\begin{eqnarray*}
d||X(t^-)-Y(t^-)||^2\leq[2\mu+\sigma+\lambda\sqrt{\gamma}(\sqrt{\gamma}+2)]||X(t^-)-Y(t^-)||^2dt+dM(t),
\end{eqnarray*}
which can be writen into its following integral form :
\begin{eqnarray}
||X(t^-)-Y(t^-)||^2\leq [2\mu+\sigma+\lambda\sqrt{\gamma}(\sqrt{\gamma}+2)]\int_0^t||X(s^-)-Y(s^-)||^2ds+\int_0^tdM(s).
\label{ch3nonlin2}
\end{eqnarray}
Taking expectation in both sides of \eqref{ch3nonlin2} and using the fact that $\mathbb{E}\left(\int_0^tdM(s)\right)=0$ leads to :
\begin{eqnarray}
\mathbb{E}||X(t^-)-Y(t^-)||^2\leq [2\mu+\sigma+\lambda\sqrt{\gamma}(\sqrt{\gamma}+2)]\int_0^t\mathbb{E}||X(s^-)-Y(s^-)||^2ds.
\label{ch3nonlin3}
\end{eqnarray}
Applying Gronwall lemma ( continuous form ) to \eqref{ch3nonlin3}  leads to : 
\begin{eqnarray*}
\mathbb{E}||X(t^-)-Y(t^-)||^2\leq \mathbb{E}||X_0-Y_0||^2e^{[2\mu+\sigma+\lambda\sqrt{\gamma}(\sqrt{\gamma}+2)]t}.
\end{eqnarray*}
This complete the proof of Theorem \ref{ch3th2}.
\end{proof}
\begin{rem}
For the linear test equation \eqref{ch3lin1}, the one-sided Lipschitz and the global Lipschitz condition become
\begin{eqnarray*}
\langle x-y, f(x)-f(y)\rangle=a|x-y|^2\\
|g(x)-g(y)|^2=b^2|x-y|^2\\
|h(x)-h(y)|^2=c^2|x-y|^2.
\end{eqnarray*}
Along the same lines as for the nonlinear case, we obtain :
\begin{eqnarray*}
\mathbb{E}|X(t^-)|^2=\mathbb{E}|X_0|^2e^{(2a+b^2+\lambda c(2+c))t}.
\end{eqnarray*}
Therefore, we have the following equivalence for the linear mean-square stability 
\begin{eqnarray}
\lim_{t\longrightarrow +\infty}\mathbb{E}|X(t)|^2=0\Longleftrightarrow l:= 2a+b^2+\lambda c(2+c)<0.
\end{eqnarray}
\end{rem}
\hspace{0.5cm} Based on Theorem \ref{ch3th2}, it is interesting to analyse whether or not the numerical solution of \eqref{ch2jdi1} reproduce the mean-square stability of the exact solution.

\subsection{Nonlinear mean-square stability of the numerical solutions}
\begin{thm}\textbf{[Stability of the stochastic theta method]}\label{ch3th3}

Under Assumptions \ref{ch3assump1} and the further hypothesis $\alpha<0$, for 
\begin{eqnarray*}
\Delta t< \dfrac{-\alpha}{\lambda^2\gamma},
\end{eqnarray*}
the Euler backward  method (STM with $\theta=1$) applied to equation \eqref{ch2jdi1} is mean-square stable in the sense that
\begin{eqnarray*}
\mathbb{E}||X_n-Y_n||^2\leq \mathbb{E}||X_0-Y_0||^2e^{\beta_1(\Delta t)n\Delta t},
\end{eqnarray*}
where 
\begin{eqnarray*}
\beta_1(\Delta t):=\dfrac{1}{\Delta t}ln\left(\dfrac{1+(\sigma+\lambda\gamma+2\lambda\sqrt{\gamma})\Delta t+\lambda^2\gamma\Delta t^2}{1-\mu\Delta t}\right).
\end{eqnarray*}
\end{thm}
\begin{proof}
Let's introduce  the following notations :
\begin{eqnarray*}
\Delta Z_n=X_n-Y_n, \hspace{0.2cm} \Delta f_n=f(X_n)-f(Y_n),\hspace{0.2cm}\Delta g_n=g(X_n)-g(Y_n), \hspace{0.2cm} \Delta h_n=h(X_n)-h(Y_n)
\end{eqnarray*}
If $\theta =1$, the numerical approximation \eqref{ch2approxi1} applied to $X$ and $Y$ gives :
\begin{eqnarray*}
Y_{n+1}=Y_n+f(Y_{n+1})\Delta t+g(Y_n)\Delta W_n +h(Y_n)\Delta N_n\\
X_{n+1}=X_n+f(X_{n+1})\Delta t+g(X_n)\Delta W_n +h(X_n)\Delta N_n.
\end{eqnarray*}
So we have :
\begin{eqnarray}
||\Delta Z_{n+1}-\Delta f_{n+1}\Delta t||^2=||\Delta Z_n+\Delta g_n\Delta W_n+\Delta h_n\Delta N_n||^2.
\label{ch3nonlin5}
\end{eqnarray}
Using the independence of $\Delta W_n$ and $\Delta N_n$ and the fact that 
\begin{eqnarray*}
\mathbb{E}|\Delta N_n|^2&=&var(\Delta N_n)+\left(\mathbb{E}(\Delta N_n)\right)^2=\lambda\Delta t+\lambda^2\Delta t^2\\
\mathbb{E}||\Delta W_n||^2&=&\Delta t, \hspace{0.5cm} \mathbb{E}||\Delta W_n||=0, \hspace{0.5cm} \mathbb{E}|\Delta N_n|=\lambda\Delta t,
\end{eqnarray*}
we obtain from \eqref{ch3nonlin5} the following estimation :
\begin{eqnarray*}
\mathbb{E}||\Delta Z_{n+1}||^2-2\Delta t\mathbb{E}\langle \Delta Z_{n+1}, \Delta f_{n+1}\rangle &\leq& \mathbb{E}||\Delta Z_n||^2+\Delta t\mathbb{E}||\Delta g_n||^2+\lambda \Delta t(1+\lambda \Delta t)\mathbb{E}||\Delta h_n||^2\\
&+&2\lambda\Delta t\mathbb{E}\langle\Delta Z_n, \Delta h_n\rangle.
\end{eqnarray*}
Using the one-sided Lipschitz condition  and the global Lipschitz condition, it follows that :
\begin{eqnarray*}
\mathbb{E}||\Delta Z_{n+1}||^2 &\leq& 2\Delta t\mu \mathbb{E}||\Delta Z_{n+1}||^2+\mathbb{E}||\Delta Z_n||^2+\sigma\Delta t\mathbb{E}||\Delta Z_n||^2+\lambda\Delta t(1+\lambda\Delta t)\gamma \mathbb{E}||\Delta Z_n||^2\\
(1-2\mu\Delta t)\mathbb{E}||\Delta Z_{n+1}||^2&\leq&[1+(\sigma + \lambda \gamma +2\lambda\sqrt{\gamma})\Delta t+\lambda^2\gamma\Delta t^2]\mathbb{E}||\Delta Z_n||^2.
\end{eqnarray*}
The latter inequality leads to : 
\begin{eqnarray*}
\mathbb{E}||\Delta Z_{n+1}||^2 \leq \left[\dfrac{1+(\sigma + \lambda \gamma + 2\lambda\sqrt{\gamma})\Delta t+\lambda^2\gamma\Delta t^2}{1-2\mu \Delta t}\right]\mathbb{E}||\Delta Z_n||^2.
\end{eqnarray*}
Therefore, we have :
\begin{eqnarray}
\mathbb{E}||\Delta Z_n||^2 \leq \left[\dfrac{1+(\sigma + \lambda \gamma + 2\lambda\sqrt{\gamma})\Delta t+\lambda^2\gamma\Delta t^2}{1-2\mu \Delta t}\right]^n\mathbb{E}||\Delta Z_0||^2.
\label{ch3nonlin6}
\end{eqnarray}
In order to have stability, we impose the following condition :
\begin{eqnarray}
\dfrac{1+(\sigma + \lambda \gamma + 2\lambda\sqrt{\gamma})\Delta t+\lambda^2\gamma\Delta t^2}{1-2\mu \Delta t}<1.
\label{ch3nonlin7}
\end{eqnarray}
The hypothesis $\alpha<0$ implies that $\mu<0$. So $1-2\mu\Delta t>0$, for all positive stepsize. It follows that \eqref{ch3nonlin7} is equivalent to 
\begin{eqnarray*}
\Delta t< \dfrac{-\alpha}{\lambda^2\gamma}.
\end{eqnarray*}
Applying the equality $a^n=e^{n\ln a}$,  for all $a>0$ and all $n\in\mathbb{N}$ to \eqref{ch3nonlin6} complete the proof of  Theorem \ref{ch3th3}.
\end{proof}
\begin{thm} \textbf{[A-stability of the compensated Euler backward method]}\label{ch3th3}

Under Assumptions \ref{ch3assump1} and the further hypothesis $\alpha<0$, for any stepsize,
the compensated backward Euler method ( CSTM with $\theta=1$) for equation \eqref{ch2jdi1} is mean square stable in the sense that :
\begin{eqnarray*}
\mathbb{E}||X_n-Y_n||^2\leq \mathbb{E}||X_0-Y_0||^2e^{\beta_2(\Delta t)n\Delta t},
\end{eqnarray*}
where 
\begin{eqnarray*}
\beta_2(\Delta t) :=\dfrac{1}{\Delta t}ln\left(\dfrac{1+(\sigma +\lambda\gamma)\Delta t}{1-2(\mu+\lambda\sqrt{\gamma})\Delta t}\right).
\end{eqnarray*}
\end{thm}
\begin{proof}
We use the same notations as for the proof of  Theorem \ref{ch3th3} except for $\Delta f^{\lambda}_n$ for which we have $\Delta f^{\lambda}_n=f_{\lambda}(X_n)-f_{\lambda}(Y_n)$.

Along the same line as for the proof of Theorem \ref{ch3th3}, we obtain :
\begin{eqnarray}
||\Delta Z_{n+1}-\Delta t\Delta f^{\lambda}_{n+1}||^2=||\Delta Z_n+\Delta g_n\Delta W_n+\Delta h_n\Delta\overline{N}_n||^2.
\label{ch3nonlin8}
\end{eqnarray}
Futhermore, the relations $\mathbb{E}|\Delta\overline{N}_n|=0$ and $ \mathbb{E}|\Delta\overline{N}_n|^2=\lambda\Delta t$ leads to :
\begin{eqnarray*}
\langle x-y, f_{\lambda}(x)-f_{\lambda}(y)\rangle&=&\langle x-y, f(x)-f(y)\rangle+\lambda\langle x-y, h(x)-h(y)\rangle\\
&\leq& (\mu+\lambda\sqrt{\gamma})||x-y||^2.
\end{eqnarray*}
Using the independence of $\Delta W_n$ and $\Delta\overline{N}_n$, it follows from \eqref{ch3nonlin8} that
\begin{eqnarray}
\mathbb{E}||\Delta Z_{n+1}||^2\leq 2\Delta t\mathbb{E}\langle \Delta Z_{n+1}, \Delta f_{n+1}\rangle+\mathbb{E}||\Delta Z_n||^2+\Delta t \mathbb{E}||\Delta g_n||^2+\lambda \Delta t\mathbb{E}||\Delta h_n||^2.
\label{ch3nonlin9}
\end{eqnarray}
Using the one-sided Lipschitz and the global Lipschitz condition, it follows \eqref{ch3nonlin9} that :
\begin{eqnarray*}
(1-2(\mu+\lambda\sqrt{\gamma})\Delta t)\mathbb{E}||\Delta Z_{n+1}||^2\leq(1+\sigma\Delta t+\lambda\gamma\Delta t)\mathbb{E}||\Delta Z_n||^2.
\end{eqnarray*}
Therefore, 
\begin{eqnarray}
\mathbb{E}||\Delta Z_n||^2\leq \left[\dfrac{1+\sigma\Delta t+\lambda\gamma\Delta t}{1-2(\mu+\lambda\sqrt{\gamma})\Delta t}\right]^n\mathbb{E}||Z_0||^2.
\label{ch3nonlin10}
\end{eqnarray}
In order to have stability, we need the following condition to be fulfilled
\begin{eqnarray}
\dfrac{1+\sigma\Delta t+\lambda\gamma\Delta t}{1-2(\mu+\lambda\sqrt{\gamma})\Delta t}<1.
\label{ch3nonlin11}
\end{eqnarray}
From the hypothesis $\alpha<0$, we have $2(\mu+\lambda\sqrt{\gamma})<0$ and then $1-2(\mu+\lambda\sqrt{\gamma})\Delta t>0$ for any stepsize. Hence condition \eqref{ch3nonlin11} is equivalent to 
$\alpha\Delta t<0$, which is satisfied for any stepsize.

Applying the relation $a^n=e^{n\ln a}$ to \eqref{ch3nonlin10} complete the proof of the theorem.
\end{proof}

\section{Numerical Experiments}
\hspace{0.5cm} The purpose of this section is to illustrate our theorical results of strong convergence and stability. We will focus in the linear case. We consider the linear  jump-diffusion It\^{o}'s stochastic integral (SDEs) 
\begin{eqnarray}\label{ch3num1}
 \left\{\begin{array}{ll}
dX(t)=aX(t^-)dt+bX(t^-)dW(t)+cX(t^-)dN(t), \hspace{0.5cm} t\geq 0,\hspace{0.5cm} c>-1,\\
X(0)=1.
\end{array}
\right.
\end{eqnarray}
\subsection{Strong convergence illustration}
In order to illustrate the strong convergence result, 
we need the exact solution of  problem \eqref{ch3num1}.
\begin{pro}
The problem \eqref{ch3num1} has the following process as a unique solution
\begin{eqnarray*}
X(t)=X_0\exp\left[\left(a-\dfrac{b^2}{2}\right)t+bW(t)\right](1+c)^{N(t)},
\end{eqnarray*}
which can be written in the following equivalent form 
\begin{eqnarray}
X(t)=X_0\exp\left[\left(a-\dfrac{b^2}{2}\right)t+bW(t)+\ln(1+c)N(t)\right].
\end{eqnarray}
\end{pro}
\begin{proof}
\begin{enumerate}
\item Obviously, the functions $f(x)=ax$, $g(x)=bx$ and $h(x)=cx$ satisfy the global Lipschitz condition and the linear growth condition. Therefore from Theorem \ref{ch2th1}, it follows that the problem \eqref{ch3num1} admit a unique solution.
\item Let's consider the following  It\^{o}'s jump-diffusion process
\begin{eqnarray*}
Z(t)=\left(a-\dfrac{b^2}{2}\right)t+bW(t)+N(t)\ln(1+c).
\end{eqnarray*}
The function $f: [0,\infty)\longrightarrow \mathbb{R}, \hspace{0.3cm} x\longmapsto x_0\exp(x)$ is infinitely differentiable.

Then  applying It\^{o} formula for jump process to the process $Z(t)$ leads to :
\begin{eqnarray}
f(Z_t)&=&f(Z_0)+\int_0^t\left(a-\dfrac{b^2}{2}\right)f'(Z_{s^-})ds+\dfrac{1}{2}\int_0^tb^2f''(Z_{s^-})ds+\int_0^tbf'(Z_{s^-})dW(s)\nonumber\\
&+&\int_0^t(f(Z_s)-f(Z_{s^-}))dN(s),
\label{ch3num2}
\end{eqnarray}
where
\begin{eqnarray}
f(Z_s)-f(Z_{s^-})&=&X_0\exp[Z_{s^-}+\ln(1+c)]-X_0\exp(Z_{s^-})\nonumber\\
&=&(1+c)X_0\exp(Z_{s^-})-X_0\exp(Z_{s^-})\nonumber\\
&=&cX_0\exp(Z_{s^-})=cf(Z_{s^-})
\label{ch3num3}
\end{eqnarray}
and 
\begin{eqnarray}
X(s^-)=f(Z_{s^-})=f'(Z_{s^-})=f''(Z_{s^-}).
\label{ch3num4}
\end{eqnarray}
Substituting \eqref{ch3num3} and \eqref{ch3num4} in \eqref{ch3num2} and rewriting the result into its differential form leads to 
\begin{eqnarray*}
dX(t)=aX(t^-)dt+bX(t^-)dW(t)+cX(t^-)dN(t).
\end{eqnarray*}
So $X(t)$ satisfies the desired equation.
\end{enumerate}
\end{proof}
\hspace{0.5cm} For the numerical simulation, we take $a=b=1$, $c=0.5$ and $\lambda=1$. We have the following graphs for the strong error. We use $5000$ sample paths. The algorithms for simulation are based on \cite{Desmond3}. We take $d t=2^{14}$ and $\Delta t=2^{p-1}$ for $p=1,2,3,4,5$. The error is computing at the end point $T=1$.
\begin{figure}[h]
\includegraphics[scale=0.7]{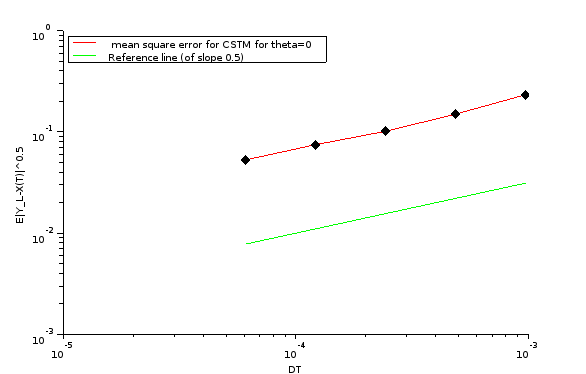}
\caption{Mean square error of the CSTM with $\theta=0$}
\end{figure}
\begin{figure}[h]
\includegraphics[scale=0.7]{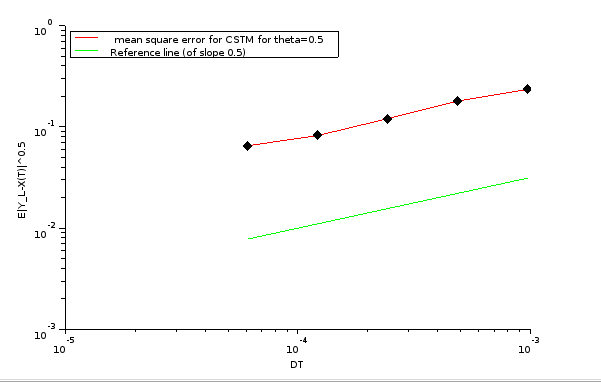}
\caption{Mean square error of the CSTM with $\theta=0.5$}
\end{figure}
\begin{figure}[h]
\includegraphics[scale=0.7]{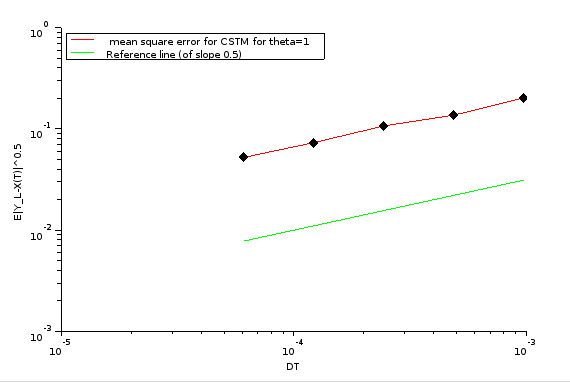}
\caption{Mean square error of the CSTM with $\theta=1$}
\end{figure}
\subsection{Mean-square stability illustration}
In order to illustrate our theoretical result of A-stability, we first consider two examples

\textbf{Example I} $a=b=2$, $c=-0.9$ and $\lambda=9$.

\textbf{Example II} $a=-7$, $b=c=1$ and $\lambda=4$.

  \hspace{0.5cm}In both examples, the  stability condition \eqref{ch3lin2} is satisfied. So  exact solutions of both examples are mean-square stable. For $\theta$ slightly less than $0.5$ (for instance $\theta=0.495$) both solutions may be unstable for a large stepsize ($\Delta t=60, 25$), but  for $\dfrac{1}{2}\leq \theta\leq 1$, numerical solutions of both examples are stable. From the top to the bottom, we present numerical examples of example I and example II respectively.
\begin{figure}[h]
\caption{A-stability for example I}
\includegraphics[scale=0.7]{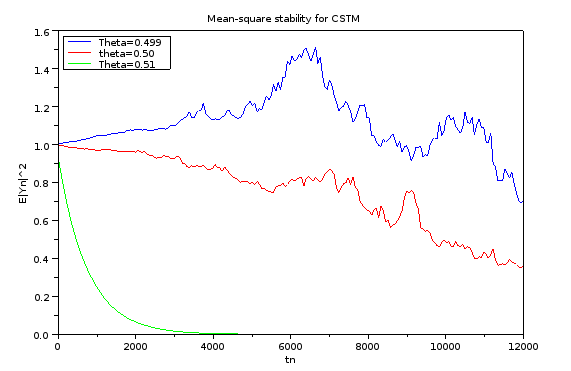}
\end{figure}
\begin{figure}[h]
\caption{A-stability for example II}
\includegraphics[scale=0.7]{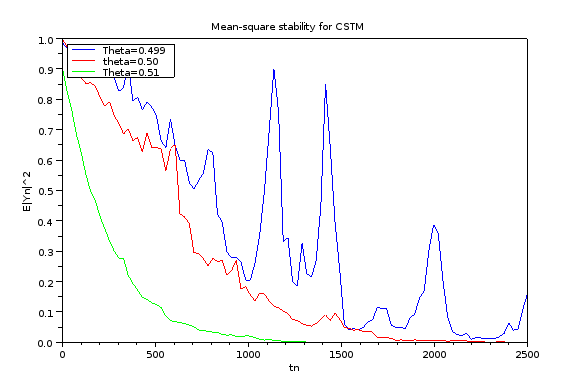}
\end{figure}

\hspace{0.5cm} The following curves provide the stability comparison between CSTM and STM. We focus on the example I. Here, $a>0$ and $c<0$. So the jump part can stabilise the problem. In this case, from the theoretical result the STM is stable for $\Delta t<0.0124829$. For $\Delta t=0.005$, both CSTM and STM stabilities behavior look the same. But for $\Delta t=0.5$ CSTM is stable while STM produce an oscillation. For $\Delta t=0.1$, numerical solution of STM grows rapidly to the scale $10^7$ and is unstable while the numerical solution of CSTM is stable. So CSTM works better than STM.


\begin{figure}[h]
\includegraphics[scale=0.7]{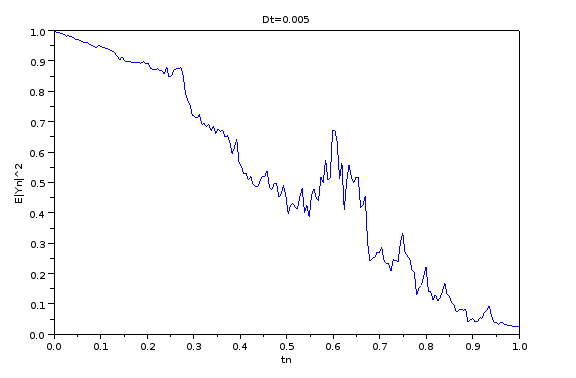}
\caption{Stability behavior of the STM}
\end{figure}

\begin{figure}[h]
\includegraphics[scale=0.7]{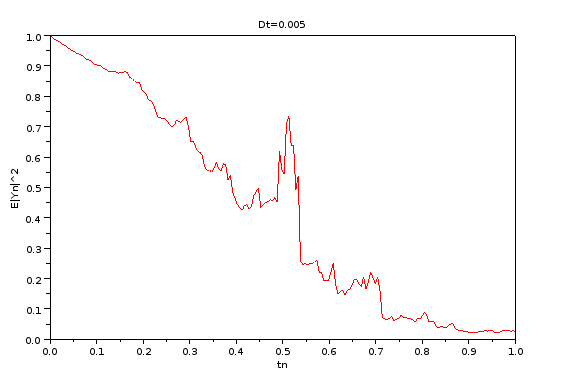}
\caption{Stability behavior of the CSTM}
\end{figure}

\begin{figure}[h]
\includegraphics[scale=0.7]{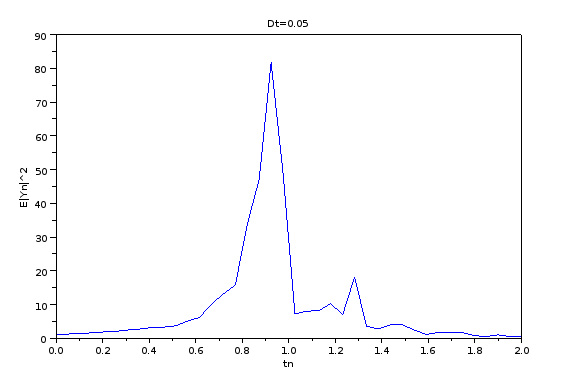}
\caption{Stability behavior of the STM}
\end{figure}

\begin{figure}[h]
\includegraphics[scale=0.7]{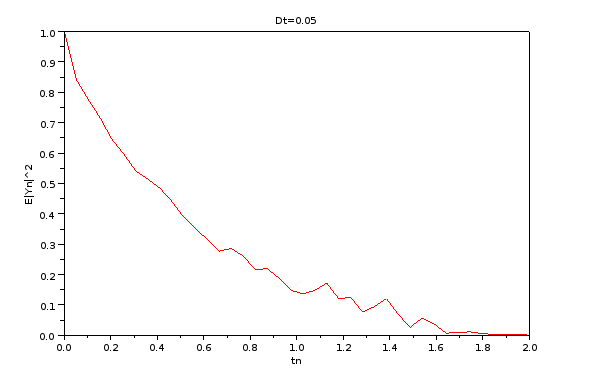}
\caption{Stability behavior of the CSTM}
\end{figure}

\begin{figure}[h]
\includegraphics[scale=0.7]{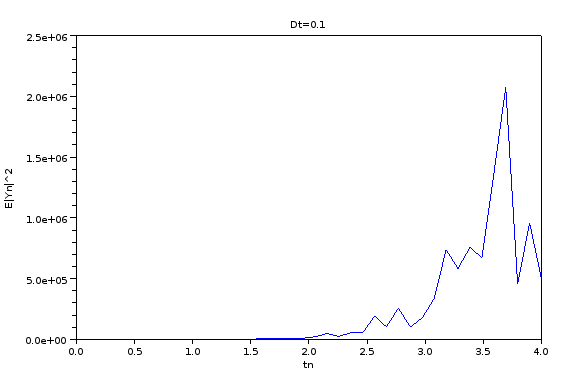} 
\caption{Stability behavior of the STM}
\end{figure}

\begin{figure}[h]
\includegraphics[scale=0.7]{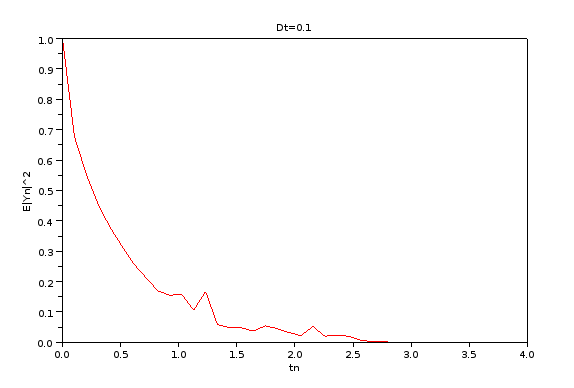}
\caption{Stability behavior of the CSTM}
\end{figure}


 \hspace{0.5cm} In this chapter, we provided the proof of the strong convergence of order $0.5$ of the CSTM under global Lipschitz condition. We also studied the stability behaviour of both STM and CSTM. We proved that the CSTM works better than the STM. Some situations in real life are modelised by SDEs with jumps, where the drift coefficient does not satisfy the global Lipschitz condtion. It is proved in \cite{Martin3} that the Euler explicit method for such equations diverge strongly. The tamed Euler scheme for SDEs without jump is the currently investigated by many authors. The  compensated tamed Euler scheme for SDEs with jumps is not yet well developped in the litterature. In the following chapter, we establish the strong convergence of the compensated tamed Euler scheme for SDEs with jumps under non-global Lipschitz condition. This scheme  is slightly different to what is already  done in the litterature.

\chapter{Strong convergence of the compensated tamed Euler scheme for stochastic differential equation with jump under non-global Lipschitz condition}
 \hspace{0.5cm} Under non-global Lipschitz condition, Euler Explicit method fails to converge strongly to the exact solution, while Euler implicit method converges but requires more computational efforts. The strong convergence of the tamed Euler scheme has been investigated in \cite{Martin1}. This scheme is explicit and requires less computational efforts than the Euler implicit method. In this chapter, we extend  the strong convergence of  the tamed Euler scheme by introducing its compensated form for stochastic differential equations with jumps.  More precisely, we prove that under non-global Lipschitz condition, the compensated tamed Euler scheme converges strongly with order $0.5$ to the exact solution of the SDEs with jumps. This scheme is different to the one proposed in \cite{Kon}. As opposed to what is done in \cite{Kon}, here we obtain the strong convergence and the rate of convergence simultaneously   under more relaxed conditions. The contents of this chapter can also be found in \cite{atjdm1}.

\section{Compensated tamed Euler scheme}
\label{ch4intro}

\hspace{0.5cm}In this chapter, we still  consider the jump-diffusion It\^{o}'s stochastic differential equations (SDEs) of the form 
\begin{eqnarray}
 dX(t)=  f(X(t^{-}))dt +g(X(t^{-}))dW(t)+h(X(t^{-}))dN(t),  \hspace{0.5cm}
 X(0)=X_0,
 \label{ch4exactsol}
\end{eqnarray}
where $W_t$ is an $m$-dimensional Brownian motion, $f :\mathbb{R}^d\longrightarrow\mathbb{R}^d$ satisfies the  one-sided Lipschitz condition and the polynomial growth condition. The functions $g : \mathbb{R}^d \longrightarrow\mathbb{R}^{d\times m}$ and $h :\mathbb{R}^d \longrightarrow\mathbb{R}^d$ satisfy the global Lipschitz condition, $N_t$ is a one dimensional poisson process with parameter $\lambda$.

We recall that the compensated poisson process
 $\overline{N}(t) := N(t)-\lambda t$ is a martingale satisfying the following properties :
 \begin{eqnarray}
 \mathbb{E}\left(\overline{N}(t+s)-\overline{N}(t)\right)=0,\,\qquad \qquad  \mathbb{E}\vert \overline{N}(t+s)-\overline{N}(t)\vert^2=\lambda s,\qquad s, t \geqslant 0.
\end{eqnarray}
We can rewrite the jump-diffusion SDEs \eqref{ch4exactsol} in the following equivalent form
\begin{eqnarray}
 dX(t)=  f_\lambda(X(t^{-}))dt +g(X(t^{-}))dW(t)+h(X(t^{-}))d\overline{N}(t),
\end{eqnarray}
where
\begin{equation}
 f_\lambda(x)=f(x)+\lambda h(x).
 \label{ch4flambda}
\end{equation}
To easy notation, we will use $X(t)$ instead of $X(t^{-})$.

If  $T$  is the final time, the  tamed Euler scheme is defined by :
\begin{equation}
 X_{n+1}^{M}=X_{n}^{M}+\dfrac{\Delta t f(X_{n}^{M})}{1+ \Delta t\Vert f(X_{n}^{N}) \Vert }+g(X_{n}^{M}) \Delta W_n +h(X_{n}^{M})\Delta N_n 
 \label{ch4tam}
\end{equation}
and the compensated tamed Euler scheme is given by :
\begin{eqnarray}
 Y_{n+1}^{M}=Y_{n}^{M}+\dfrac{\Delta t f_\lambda(Y_{n}^{M})}{1+ \Delta t \Vert f_{\lambda}(Y_{n}^{M}) \Vert }+g(Y_{n}^{M}) \Delta W_n +h(Y_{n}^{M})\Delta\overline{N}_n,
 \label{ch4tamc}
\end{eqnarray}
where $M\in\mathbb{N}$ is the  number of steps and  $\Delta t=\dfrac{T}{M}$ is the stepsize.

\hspace{0.5cm} Inspired by  \cite{Martin1}, we prove the strong convergence of the numerical approximation \eqref{ch4tamc} to the exact solution of \eqref{ch4exactsol}.

\section{Moments bounded of the numerical solution}
\begin{nota}\label{ch4nota1}
Throughout this chapter $(\Omega, \mathcal{F}, \mathbb{P})$ denote a complete probability space with a filtration $(\mathcal{F}_t)_{t\geq 0}$, $||X||_{L^p(\Omega, \mathbb{R}^d)}$ is equal to $(\mathbb{E}||X^p||)^{1/p}$, for all $p\in[1,+\infty)$ and for all $(\mathcal{F}_t)-$adapted process $X$. For all $x, y\in\mathbb{R}^d$, we denote by $\langle x, y\rangle=x.y= x_1y_1+x_2y_2+\cdots+x_dy_d$, $||x||=\left(\langle x, x\rangle\right)^{1/2}$ and $||A||=\sup_{x\in\mathbb{R}^d, ||x||\leq 1}||Ax||$ for all $A\in\mathbb{R}^{m\times d}$. We use also the following convention : $\sum_{i=u}^na_i=0$ for $u>n$.
\end{nota}
We define the continuous time interpolation of the discrete numerical approximation of \eqref{ch4tamc} by the family of processes $\left(\overline{Y}^M\right)_M $,
$
\overline{Y}^M : [0,T]\times\Omega \longrightarrow \mathbb{R}^d
$
  such that :
\begin{eqnarray}
\overline{Y}^M_t =Y^M_n+\dfrac{(t-n\Delta t)f_{\lambda}(Y^M_n)}{1+\Delta t||f_{\lambda}(Y^M_n)||}+g(Y^M_n)(W_t-W_{n\Delta t})+ h(Y^M_n)(\overline{N}_t-\overline{N}_{n\Delta t}),
\label{ch4continoussolu}
\end{eqnarray}
for all $M\in\mathbb{N}$, all $n\in\{0,\cdots, M-1\}$, and all $t\in[n\Delta t,  (n+1)\Delta t[$.
\begin{assumption}\label{ch4assumption1} Throughout this chapter, We make the following assumptions :

$(A.1)$ $f,g,h\in C^1$.

$(A.2)$ For all $p>0$, there exist a finite $M_p>0$ such that $\mathbb{E}||X_0||\leq M_p$.

$(A.3)$ $g$ and $h$ satisfy the global Lipschitz condition:
\begin{eqnarray}
||g(x)-g(y)||\vee ||h(x)-h(y)||\leq C||x-y||, \hspace{0.5cm} \forall\;x,y\in \mathbb{R}^d.
\end{eqnarray}
$(A.4)$ $f$ satisfies  the one-sided Lipschitz condition :
\begin{eqnarray*}
\langle x-y, f(x)-f(y)\rangle\leq C||x-y||^2,\hspace{0.5cm} \forall\; x,y\in \mathbb{R}^d.
\end{eqnarray*}
$(A.5)$ $f$ satisfies the superlinear growth condition :
\begin{eqnarray*}
||f(x)-f(y)||\leq C(K+ ||x||^c+||y||^c)||x-y||, \hspace{0.5cm} \forall\; x,y\in \mathbb{R}^d,
\end{eqnarray*}
where $K$, $C$ and $c$ are  strictly positive constants. 
\end{assumption}
\begin{rem}
Under conditions $(A.1)$, $(A.2)$ and $(A.3)$ of Assumptions \ref{ch4assumption1}, it is proved in \cite[Lemma 1]{Desmond2} that \eqref{ch4exactsol} has a unique solution with all moments bounbed.
\end{rem}
\begin{rem}
We note that if  Assumptions \ref{ch4assumption1} are satisfied, the function $f_{\lambda}$ defined in \eqref{ch4flambda}  satisfies the one-sided Lipschitz condition and the superlinear growth condition with constants $C_{\lambda} :=C(1+\lambda)$ and $K_{\lambda} : =K+\lambda$.

Indeed, for all $x,y\in\mathbb{R}^d$, 
\begin{eqnarray*}
\langle x-y, f_{\lambda}(x)-f_{\lambda}(y)\rangle&=&\langle x-y,f(x)\rangle+\lambda\langle x-y, h(x)-h(y)\rangle \\
&\leq& C(1+\lambda)||x-y||,\\
||f_{\lambda}(x)-f_{\lambda}(y)||&\leq& ||f(x)-f(y)||+\lambda||h(x)-h(y)||\\
&\leq& C(K+\lambda +||x||^c+||y||^c)||x-y|| \\
&=&C(K_{\lambda}+||x||^c+||y||^c)||x-y||.
\end{eqnarray*}
Since the value of the constant does not matter too much, we will use   $C$ instead of $C_{\lambda}$ and $K$ instead of $K_{\lambda}$. Throughout this work, the generic constants $C_p$ may change the value from one line to another one. We will sometimes use $Y_n^M$ instead of $Y_n^M(\omega)$ to simplify  notations.
\end{rem}
The main result of this section is given in the following theorem.

\begin{thm}\label{ch4theorem1}
   
   Let $Y_n^M : \Omega\longrightarrow \mathbb{R}^d$  be defined by \eqref{ch4tamc}  for all $M\in\mathbb{N}$ and all $ n\in\{0,\cdots, M\}$. Then the  following inequality holds :
   \begin{eqnarray*}
   \sup_{M\in\mathbb{N}}\sup_{n\in\{0,\cdots, M\}}\mathbb{E}\left[||Y_n^M||^p\right]<+\infty,
   \end{eqnarray*}
   for all $p\in[1,\infty)$.
   \end{thm}

 In order to prove  Theorem \ref{ch4theorem1} we introduce the following notations  facilitating computations.

  \begin{nota}\label{ch4notation1}
 
 \begin{eqnarray*}
 \alpha^M_k := \mathrm{1}_{\{||Y^M_k||\geq 1\}}\left\langle\dfrac{Y^M_k}{||Y^M_k||}, \dfrac{g(Y^M_k)}{||Y^M_k||}\Delta W^M_k\right\rangle,\\\\
 \beta^M_k := \mathrm{1}_{\{||Y^M_k||\geq 1\}}\left\langle\dfrac{Y^M_k}{||Y^M_k||}, \dfrac{h(Y^M_k)}{||Y^M_k||}\Delta\overline{N}^M_k\right\rangle,
 \end{eqnarray*}
 \begin{eqnarray*}
 \beta :=\left(1+K+2C+KTC+TC+||f_{\lambda}(0)||+||g(0)||+||h(0)||\right)^4,\\\\
 D^M_n := (\beta+||\varepsilon||)\exp\left(\dfrac{3\beta}{2}+\sup_{u\in\{0,\cdots,n\}}\sum_{k=u}^{n-1}\left[\dfrac{3\beta}{2}||\Delta W^M_k||^2+\dfrac{3\beta}{2}||\Delta\overline{N}^M_k||+\alpha^M_k+\beta^M_k\right]\right),\\
 \Omega^M_n :=\{\omega\in \Omega : \sup_{k\in\{0,1,\cdots, n-1\}}D^M_k(\omega)\leq M^{1/2c}, \sup_{k\in\{0,1,\cdots,n-1\}}||\Delta W^M_k(\omega)||\leq 1,\\ \sup_{k\in\{0,1,\cdots,n-1\}}||\Delta \overline{N}^M_k(\omega)||\leq 1\}.
 \end{eqnarray*}
 \end{nota}
 
 In order to prove Theorem \ref{ch4theorem1}, we need the following lemmas.
 \begin{lem}
 \label{ch4lemma1}
 For all positive real numbers $a$ and $b$, the following inequality holds 
 \begin{eqnarray*}
 1+a+b^2\leq e^{a+\sqrt{2}b}.
 \end{eqnarray*}
 \end{lem}
 
 \begin{proof}
  For $a\geq 0$ fixed, let's define the function $f(b)=e^{a+\sqrt{2}b}-1-a-b^2$. It can be easily checked that $f'(b)=\sqrt{2}e^{a+\sqrt{2}b}-2b$ and $f''(b)=2(e^{a+\sqrt{2}b}-1)$. Since $a$ and $b$ are positive, it follows that $f''(b)\geq 0$ for all $b\geq 0$. So $f'$ is a non-decreasing function. Therefore, $f'(b)\geq f'(0)=\sqrt{2}e^a>0$ for all $b\geq 0$. This implies that $f$ is a non-decreasing function. Hence $f(b)\geq f(0)=e^a-1-a$ for all $b\geq 0$. Since $1+a\leq e^a$ for all positive number $a$, it follows that $f(b)\geq 0$ for all positive number $b$. i.e $1+a+b^2\leq e^{a+\sqrt{2}b}$, $\forall\;b\geq0$. Therefore for all $a\geq 0$ fixed, $1+a+b^2\leq e^{a+\sqrt{2}b}$, $\forall\;b\geq0$.

 The proof of lemma is complete.
 \end{proof}

 Following closely  \cite[Lemma 3.1, pp 15] {Martin1}, we have  the following main lemma.
 \begin{lem}\label{ch4lemma2}
 The following inequality holds for all $M\in \mathbb{N}$ 
 and all $n\in\{0,1,\cdots, M\}$
 \begin{eqnarray}
 \mathbf{1}_{\Omega^M_n}||Y^M_n||\leq D^M_n,
 \label{ch4Denobound}
 \end{eqnarray}
 where $D_n^M$ and $\Omega^M_n$ are given in Notation \ref{ch4notation1}.
 \end{lem}
 \begin{proof}
 Using the inequality $\dfrac{\Delta t}{1+\Delta t||f_{\lambda}(x)||}\leq T$, the global Lipschitz condition of $g$ and $h$ and the polynomial growth condition of $f_{\lambda}$ we have the following estimation   on $\Omega^M_{n+1}\cap\{\omega\in \Omega : ||Y^M_n(\omega)||\leq 1\}$,  for all $n\in\{0,1,\cdots, M-1\}$
 \begin{align*}
 ||Y^M_{n+1}||&\leq ||Y^M_n||+\dfrac{\Delta t||f_{\lambda}(Y^M_n)||}{1+\Delta t||f_{\lambda}(Y^M_n)||}+||g(Y^M_n)||||\Delta W^M_n||+||h(Y^M_n)||||\Delta\overline{N} ^M_n|| \nonumber \\
 &\leq ||Y^M_n||+T||f_{\lambda}(Y^M_n)-f_{\lambda}(0)||+T||f_{\lambda}(0)||+ ||g(Y^M_n)-g(0)||+||g(0)||\nonumber\\
 &+||h(Y^M_n)-h(0)||+||h(0)||\nonumber\\
 &\leq ||Y^M_n||+TC(K+||Y^M_n||^c+||0||^c)||Y^M_n-0||+T||f_{\lambda}(0)||+C||Y^M_n||+C||Y^M_n||\nonumber\\
 &+||g(0)||+||h(0)||.\nonumber
 \end{align*}
 Since $||Y^M_n||\leq 1$, it follows that :
 \begin{align}
 ||Y^M_{n+1}||
 &\leq  1+KTC +TC+2C+T||f_{\lambda}(0)||+||g(0)||+||h(0)||\leq \beta.
 \label{ch4normY}
 \end{align}
  Futhermore, from the numerical approximation \eqref{ch4tamc}, we have 
 \begin{eqnarray} \label{ch4partnorm2}
 ||Y^M_{n+1}||^2&=&||Y^M_n||^2+\dfrac{\Delta t^2||f_{\lambda}(Y^M_n)||^2}{(1+\Delta t||f_{\lambda}(Y^M_n)||)^2}+||g(Y^M_n)\Delta W^M_n||^2+||h(Y^M_n)\Delta\overline{N}^M_n||^2\nonumber\\
 &+&\dfrac{2\Delta t\langle Y^M_n, f_{\lambda}(Y^M_n)\rangle}{1+\Delta t||f_{\lambda}(Y^M_n)||}+2\langle Y^M_n,g(Y^M_n)\Delta W^M_n\rangle+2\langle Y^M_n,h(Y^M_n)\Delta\overline{N}^M_n\rangle\nonumber\\
 &+&\dfrac{2\langle \Delta tf_{\lambda}(Y^M_n),g(Y^M_n)\Delta W^M_n\rangle}{1+\Delta t||f_{\lambda}(Y^M_n)||}+\dfrac{2\langle\Delta tf_{\lambda}(Y^M_n),h(Y^M_n)\Delta\overline{N}^M_n\rangle}{1+\Delta t||f_{\lambda}(Y^M_n)||}\nonumber\\
 &+&2\langle g(Y^M_n)\Delta W^M_n,h(Y^M_n)\Delta\overline{N}^M_n\rangle.
 \end{eqnarray}
 Using Cauchy-Schwartz inequality and the estimation $
 \dfrac{1}{1+\Delta t||f_{\lambda}(Y^M_n)||}\leq 1$, we obtain the following inequality from  \eqref{ch4partnorm2} :
 \begin{eqnarray}
||Y^M_{n+1}||^2 &\leq& ||Y^M_n||^2+\Delta t^2||f_{\lambda}(Y^M_n)||^2+||g(Y^M_n)||^2||\Delta W^M_n||^2+||h(Y^M_n)||^2|\Delta\overline{N}^M_n|^2\nonumber\\
 &+&2\Delta t\left|\langle Y^M_n, f_{\lambda}(Y^M_n)\rangle\right|+2\langle Y^M_n,g(Y^M_n)\Delta W^M_n\rangle+2\langle Y^M_n, h(Y^M_n)\Delta\overline{N}^M_n\rangle\nonumber\\ 
 &+&2\Delta t\left|\langle f_{\lambda}(Y^M_n),g(Y^M_n)\Delta W^M_n\rangle\right|+2\Delta t\left|\langle f_{\lambda}(Y^M_n), h(Y^M_n)\Delta\overline{N}^M_n\rangle\right|\nonumber\\
 &+&2\langle g(Y^M_n)\Delta W^M_n, h(Y^M_n)\Delta\overline{N}^M_n\rangle.
 \label{ch4ine14}
 \end{eqnarray}
 Using the estimation $2ab\leq a^2+b^2$, inequality \eqref{ch4ine14} becomes :
 \begin{eqnarray}
||Y^M_{n+1}||^2 &\leq & ||Y^M_n||^2+\Delta t^2||f_{\lambda}(Y^M_n)||^2+||g(Y^M_n)||^2||\Delta W^M_n||^2+||h(Y^M_n)||^2|\Delta\overline{N}^M_n|^2\nonumber\\
 &+& 2\Delta t\left|\langle Y^M_n,f_{\lambda}(Y^M_n)\rangle\right|+2\langle Y^M_n,g(Y^M_n)\Delta W^M_n\rangle+2\langle Y^M_n,h(Y^M_n)\Delta\overline{N}^M_n\rangle\nonumber\\
 &+&\Delta t^2||f_{\lambda}(Y^M_n)||^2+||g(Y^M_n)||^2||\Delta W^M_n||^2+\Delta t^2||f_{\lambda}(Y^M_n)||^2\nonumber\\
 &+&||h(Y^M_n)||^2|\Delta\overline{N}^M_n|^2 +||g(Y^M_n)||^2||\Delta W^M_n||^2+||h(Y^M_n)||^2|\Delta\overline{N}^M_n|^2.
 \label{ch4ine15}
 \end{eqnarray}
 Putting similars terms of inequality \eqref{ch4ine15} together, we obtain : 
 \begin{eqnarray}
||Y^M_{n+1}||^2 &\leq & ||Y^M_n||^2+3\Delta t^2||f_{\lambda}(Y^M_n)||^2+3||g(Y^M_n)||^2||\Delta W^M_n||^2+3||h(Y^M_n)||^2|\Delta\overline{N}^M_n|^2\nonumber\\
 &+&2\Delta t\left|\langle Y^M_n,f_{\lambda}(Y^M_n)\rangle\right|+ 2\langle Y^M_n,g(Y^M_n)\Delta W^M_n\rangle\nonumber\\
 &+&2\langle Y^M_n,h(Y^M_n)\Delta\overline{N}^M_n\rangle
 \label{ch4ine16}
 \end{eqnarray}
 on $\Omega$, for all $M\in\mathbb{N}$ and all $n\in\{0,1,\cdots, M-1\}$.
 
 In addition, for all $x\in\mathbb{R}^d$ such that $||x||\geq 1$, the global Lipschitz condition satisfied by $g$ and $h$ leads to :
 \begin{eqnarray}
 ||g(x)||^2&\leq& (||g(x)-g(0)||+||g(0)||)^2\nonumber\\
 &\leq & (C||x||+||g(0)||)^2\nonumber\\
 &\leq & (C+||g(0)||)^2||x||^2\nonumber\\
 &\leq &\beta||x||^2.
 \label{ch4normdeg2}
 \end{eqnarray}
 Along the same lines as above, for all $x\in\mathbb{R}^d$ such that $||x||\geq 1$, we have :
 \begin{eqnarray}
 ||h(x)||^2\leq \beta||x||^2.
 \label{ch4normdeh2}
 \end{eqnarray}
  Also, for all $x\in\mathbb{R}^d$ such that $||x||\geq 1$, the one-sided Lipschitz condition satisfied by $f_{\lambda}$ leads to :
 \begin{eqnarray}
 \langle x,f_{\lambda}(x)\rangle&=&\langle x,f_{\lambda}(x)-f_{\lambda}(0)+f_{\lambda}(0)\rangle=\langle x,f_{\lambda}(x)-f_{\lambda}(0)\rangle+\langle x,f_{\lambda}(0)\rangle\nonumber\\
 &\leq & C||x||^2+||x||||f_{\lambda}(0)||\nonumber\\
 &\leq &(C+||f_{\lambda}(0)||)||x||^2\nonumber\\
 &\leq &\sqrt{\beta}||x||^2.
 \label{ch4crochetf}
 \end{eqnarray}
  Futhermore, for all $x\in\mathbb{R}^d$ such that $1\leq ||x||\leq M^{1/2c}$ and for all $M\in\mathbb{N}$,  using the polynomial growth condition of $f_{\lambda}$, the following inequality holds  
 \begin{eqnarray}
 ||f_{\lambda}(x)||^2&\leq&\left(||f_{\lambda}(x)-f_{\lambda}(0)||+||f_{\lambda}(0)||\right)^2\nonumber\\
 &\leq & \left(C(K+||x||^c)||x||+||f_{\lambda}(0)||\right)^2\nonumber\\
 &\leq & \left(C(K+1)||x||^{c+1}+||f_{\lambda}(0)||\right)^2\nonumber\\
 &\leq & (KC+C+||f_{\lambda}(0)||)^2||x||^{2(c+1)}\nonumber\\
 &\leq & M\sqrt{\beta}||x||^2.
 \label{ch4normf2}
 \end{eqnarray}
 Now combining  inequalities \eqref{ch4ine16}, \eqref{ch4normdeg2}, \eqref{ch4normdeh2}, \eqref{ch4crochetf} and \eqref{ch4normf2},  we obtain :
 \begin{eqnarray}
 ||Y^M_{n+1}||^2 &\leq& ||Y^M_n||^2+\dfrac{3T^2\sqrt{\beta}}{M}||Y^M_{n}||^2+3\beta||Y^M_n||^2||\Delta W^M_n||^2+3\beta||Y^M_n||^2|\Delta \overline{N}^M_n|^2\nonumber\\
 &+&\dfrac{2T\sqrt{\beta}}{M}||Y^M_n||^2+2\langle Y^M_n, g(Y^M_n)\Delta W^M_n\rangle+2\langle Y^M_n, h(Y^M_n)\Delta\overline{N}^M_n\rangle\nonumber\\
 &\leq &||Y^M_n||^2+\dfrac{(3T^2+2T)\sqrt{\beta}}{M}||Y^M_n||^2+3\beta||Y^M_n||^2||\Delta W^M_n||^2+3||Y^M_n||^2|\Delta\overline{N}^M_n|^2\nonumber\\
 &+&2\langle Y^M_n, g(Y^M_n)\Delta W^M_n\rangle
 +2\langle Y^M_n, h(Y^M_n)\Delta\overline{N}^M_n\rangle.
 \end{eqnarray}
 Using the inequality $3T^2+2T\leq 3\sqrt{\beta}$, it follows that :
 \begin{eqnarray}
 ||Y^M_{n+1}||^2&\leq& ||Y^M_n||^2+\dfrac{3\beta}{M}||Y^M_n||^2+3\beta||Y^M_n||^2||\Delta W^M_n||^2+3\beta||Y^M_n||^2|\Delta\overline{N}^M_n|^2\nonumber\\
 &+&2\langle Y^M_n,g(Y^M_n)\Delta W^M_n\rangle+2\langle Y^M_n,h(Y^M_n)\Delta\overline{N}^M_n\rangle\nonumber\\
 &=&||Y^M_n||^2 \left(1+\dfrac{3\beta}{M}+3\beta||\Delta W^M_n||^2+3\beta||\Delta\overline{N}^M_n||^2+2\left<\dfrac{Y^M_n}{||Y^M_n||}, \dfrac{g(Y^M_n)}{||Y^M_n||}\Delta W^M_n\right>   \right.\nonumber\\
 &+&\left. 2\left\langle\dfrac{Y^M_n}{||Y^M_n||}, \dfrac{h(Y^M_n)}{||Y^M_n||}\Delta\overline{N}^M_n\right\rangle\right)\nonumber\\
 &=&||Y^M_n||^2\left(1+\dfrac{3\beta}{M}+3\beta||\Delta W^M_n||^2+3\beta|\Delta\overline{N}^M_n|^2+2\alpha^M_n+2\beta^M_n\right).
 \label{ch4expY1}
 \end{eqnarray}
 Using Lemma \ref{ch4lemma1} for  
  $a=\dfrac{3\beta}{M}+3\beta||\Delta W^M_n||^2+2\alpha^M_n+2\beta^M_n$ and $b=\sqrt{3\beta}|\Delta\overline{N}^M_n| $ it follows from \eqref{ch4expY1} that  :
 \begin{eqnarray}
 ||Y^M_{n+1}||^2\leq ||Y^M_n||^2\exp\left(\dfrac{3\beta}{M}+3\beta||\Delta W^M_n||^2+3\beta|\Delta\overline{N}^M_n|+2\alpha^M_n+2\beta^M_n\right)
 \label{ch4expY2}
 \end{eqnarray}
 on $\{w\in\Omega : 1\leq ||Y^M_n(\omega)||\leq M^{1/2c}\}$, for all $M\in\mathbb{N}$ and all $n\in\{0,1,\cdots,M-1\}$.
 
 In order to complete our proof, we need the following map
 \begin{eqnarray*}
 \tau^M_l : \Omega \longrightarrow\{-1,0,1,\cdots,l\},\hspace{0.5cm}l\in\{0,1,\cdots, M\},
 \end{eqnarray*}
 such that :
 \begin{eqnarray*}
 \tau^M_l(\omega) :=\max\left(\{-1\}\cup\{n\in\{0,1,\cdots,l-1\} : ||Y^M_n(\omega)||\leq 1\}\right),
 \end{eqnarray*}
 for all $\omega\in\Omega$, $M\in\mathbb{N}$ and all $l\in\{0,1,\cdots,M\}$.
 
 For $M\in\mathbb{N}$ fixed we prove by induction on $n\in\{0,1,\cdots,M\}$  that 
 \begin{eqnarray}
 \mathbf{1}_{\Omega^M_n}||Y^M_n||\leq D^M_n.
 \end{eqnarray}
 \label{ch4borneD1}
 \begin{itemize}
 \item For $n=0$, $D_0^M =(\beta+||X_0||)\exp(\beta)$ and $||Y^M_0||=||X_0||$.
  Since $\beta\geq 1 $ we have $\exp(\beta)\geq 1$. So the following inequality holds 
 \begin{eqnarray*}
 \mathbf{1}_{\Omega^M_0}||Y^M_0||\leq D^M_0.
 \end{eqnarray*}
 \item Let $l\in\{0,1,\cdots,M-1\}$ be arbitrary and let's assume that
 
 $\mathbf{1}_{\Omega^M_n}||Y^M_n||\leq D^M_n$ for all $n\in\{0,1,\cdots,l\}$. We want to prove that inequality \eqref{ch4Denobound} holds for $n=l+1$.
 
 Let $\omega\in\Omega^M_{l+1}$ we have  to prove that $||Y^M_{l+1}(\omega)||\leq D^M_{l+1}(\omega)$.
 
 Since $(\Omega^M_n)$ is a decreasing sequence and $\omega\in\Omega^M_{l+1}$, we have $\omega\in\Omega^M_k$ and it follows from the hypothesis of induction that :
  $||Y^M_k(\omega)||\leq D^M_k(\omega)$, for all $k\in\{0,\cdots, l\}$.
  
  Also, since $\omega\in\Omega^M_{k+1}$, by definition of $\Omega^M_{k+1}$ it follows that $D^M_k(\omega)\leq M^{1/2c}$, for all $k\in\{0, \cdots, l\}$. So for all $k\in\{0,1,\cdots,l\}$, 
  \begin{eqnarray*}
  ||Y^M_k(\omega)||\leq D^M_k(\omega)\leq M^{1/2c}.
  \end{eqnarray*}
  For all $k\in\{\tau^M_{l+1}(\omega)+1,\tau^M_{l+1}(\omega)+2,\cdots,l\}$ we have 
  \begin{eqnarray}
  1\leq ||Y^M_k(\omega)||\leq M^{1/2c}.
  \label{ch4con}
  \end{eqnarray}
  Since \eqref{ch4con} holds, it follows from  \eqref{ch4expY2}, that
  \begin{eqnarray*}
  ||Y^M_{k+1}(\omega)||&\leq& ||Y^M_k(\omega)||\exp\left(\dfrac{3\beta}{2M}+\dfrac{3\beta}{2}||\Delta W^M_k(\omega)||^2\right.\\
  &+&\left.\dfrac{3\beta}{2}|\Delta\overline{N}^N_k(\omega)|+\alpha^M_k(\omega)+\beta^M_k(\omega)\right)
  \end{eqnarray*}
  for all $k\in\{\tau^M_{l+1}(\omega)+1,\tau^M_{l+1}+2,\cdots,l\}$.
  
  For $k=l$ from the previous inequality, we have :
  \begin{eqnarray}
  ||Y^M_{l+1}(\omega)||&\leq &||Y^M_l(\omega)||\exp\left(\dfrac{3\beta}{2M}+\dfrac{3\beta}{2}||\Delta W^M_l(\omega)||^2+\dfrac{3\beta}{2}|\Delta\overline{N}^N_l(\omega)|\right.\nonumber\\
  &+&\left.\alpha^M_l(\omega)+\beta^M_l(\omega)\right).
  \label{ch4iter}
  \end{eqnarray}
  Iterating \eqref{ch4iter} $l-\tau^M_{l+1}(\omega)$ times leads to
  \begin{eqnarray*}
 ||Y^M_{l+1}(\omega)|| &\leq &||Y^M_{\tau^M_{l+1}(\omega)+1}(\omega)||\exp\left(\sum_{m=\tau^M_{l+1}(\omega)+1}^l\left[\dfrac{3\beta}{2M}+\dfrac{3\beta}{2}||\Delta W^M_m(\omega)||^2\right.\right.\\
  &+&\left.\left.\dfrac{3\beta}{2}|\Delta\overline{N}^M_m(\omega)|+\alpha^M_m(\omega)+\beta^M_m(\omega)\right]\right).
  \end{eqnarray*}
 By definition of $\tau^M_l(\omega)$,  we have  $||Y^M_{\tau^M_{l+1}(\omega)}(\omega)||\leq 1$. 
 
 Then it  follows from \eqref{ch4normY} that  $||Y^M_{\tau^M_{l+1}(\omega)+1}(\omega)||\leq \beta$. So the above estimation of $||Y^M_{l+1}(\omega)||$ becomes :
  \begin{eqnarray*}
  ||Y^M_{l+1}(\omega)||&\leq& \beta\exp\left(\sum_{m=\tau^M_{l+1}(\omega)+1}^l\left[\dfrac{3\beta}{2M}\right.\right.\\
  &+&\dfrac{3\beta}{2}||\Delta W^M_m(\omega)||^2+\dfrac{3\beta}{2}|\Delta\overline{N}^M_m(\omega)|+\left.\left.\alpha^M_m(\omega)+\beta^M_m(\omega)\right]\right)\\
  &\leq& (\beta+||X_0||)\exp\left(\dfrac{3\beta}{2}+\sup_{u\in\{0,1,\cdots,l+1\}}\sum_{m=u}^l\left[\dfrac{3\beta}{2}||\Delta W^M_m(\omega)||^2\right.\right.\\
  &+&\left.\left.\dfrac{3\beta}{2}|\Delta\overline{N}^M_m(\omega)|+\alpha^M_m(\omega)+\beta^M_m(\omega)\right]\right)=D^M_{l+1}(\omega).
  \end{eqnarray*}
  Therefore $||Y^M_{l+1}(\omega)||\leq D^M_{l+1}(\omega)$. This complete the proof of Lemma \ref{ch4lemma2}.
 \end{itemize}
 \end{proof}

 The following  is from \cite[Lemma 3.2 pp 15]{Martin1}.
 \begin{lem}\label{ch4lemma3}
 Let $n\in\mathbb{N}$ and $Z : \Omega \longrightarrow \mathbb{R}^m$ be an $m-$dimensional standard normal random variable. Then for all $a\in\left[0,\dfrac{1}{4}\right]$ the following inequality holds 
 \begin{eqnarray*}
 \mathbb{E}\left[\exp(a||Z||^2)\right]=(1-2a)^{-m/2}\leq e^{2am}.
 \end{eqnarray*}
 \end{lem}
 \begin{proof}
 Using the relation   $||Z||^2=|Z_1|^2+|Z_2|^2+\cdots+|Z_n|^2$ and the fact that $(Z_i)$ are independent and identically distributed,  we have :
 \begin{eqnarray}
 \mathbb{E}\left[\exp(a||Z||^2)\right]=\mathbb{E}\left[\exp\left(\sum_{i=1}^ma|Z_i|^2\right)\right]=\mathbb{E}\left[\prod_{i=1}^m\exp\left(a|Z_i|^2\right)\right]=\left[\mathbb{E}\left(\exp(a|Z_1|^2)\right)\right]^m.
 \label{ch4normal1}
 \end{eqnarray}
 From the definition of the expected value of the standard normal random variable, we have :
 \begin{eqnarray*}
 \mathbb{E}[\exp(a|Z_1|^2)]=\int_{-\infty}^{+\infty}e^{ax^2}\dfrac{1}{\sqrt{2\pi}}e^{-x^2/2}dx=\dfrac{1}{\sqrt{1-2a}}.
 \end{eqnarray*}
 Using the inequality $\dfrac{1}{1-x}\leq e^{2x} \hspace{0.3cm }\forall x\in\left[0,\dfrac{1}{2}\right]$, it follows that 
 \begin{eqnarray}
 \mathbb{E}[\exp(a|Z_1|^2)]=\dfrac{1}{\sqrt{1-2a}}\leq e^{2a},\hspace{0.5cm}\forall a\in\left[0,\dfrac{1}{4}\right].
 \label{ch4normal2}
 \end{eqnarray}
 Combining \eqref{ch4normal1} and \eqref{ch4normal2} leads to :
 \begin{eqnarray*}
 \mathbb{E}\left[\exp(a||Z||^2)\right]=(1-2a)^{-m/2}\leq e^{2am},\hspace{0.5cm}\forall a\in\left[0,\dfrac{1}{4}\right].
 \end{eqnarray*}
 \end{proof}

 The following lemma and its proof are based on \cite[Lemma 3.3, pp 15] {Martin1} with only different value of the coefficient $\beta$.
 \begin{lem}\label{ch4lemma4}
 The following inequality holds :
 \begin{eqnarray*}
 \sup_{M\in\mathbb{N}, M\geq 4\beta pT}\mathbb{E}\left[\exp\left(\beta p\sum_{k=0}^{M-1}||\Delta W^M_k||^2\right)\right]<\infty.
 \end{eqnarray*}
 \end{lem}
 
 \begin{proof}
 Let $Z=\mathcal{N}(0,1)$ be an $m$-dimensional standard normal random variable.  Since for  
 
  $k=0,\cdots, M-1$, $\Delta W^M_k$ are independent, stationary and follows the normal distribution with mean $0$ and variance $\dfrac{T}{M}$, $||\Delta W^M_n||^2=\dfrac{T}{M}||Z||^2$, it follows that :
 \begin{eqnarray*}
 \mathbb{E}\left(\exp\left[\beta p\sum_{k=0}^{M-1}||\Delta W^M_k||^2\right]\right)&=&\prod_{k=0}^{M-1}\mathbb{E}\left[\exp(\beta p||\Delta W^M_k||^2)\right]\\
 &=&\left(\mathbb{E}\left[\exp\left(\beta p\dfrac{T}{M}||Z||^2\right)\right]\right)^M  \\
 &\leq &\left[\exp\left(2\beta pm\dfrac{T}{M}\right)\right]^M\text{(using Lemma \ref{ch4lemma3})}\\
 &\leq &\exp(2\beta pTm)<\infty,
 \end{eqnarray*}
 for all $p\in[1,+\infty)$   and all $M\in\mathbb{N}\cap[4\beta pT, \infty)$.
 \end{proof}

 \begin{lem}\label{ch4lemma5}
 Let $Y$ be a  standard normal random of dimension $m$ variable   and $c\in\mathbb{R}^m$, then
 \begin{eqnarray*}
 \mathbb{E}[\exp(cY)]=\exp\left(\dfrac{c^2}{2}\right).
 \end{eqnarray*} 
 \end{lem}
 
 \begin{proof}
 $\mathbb{E}[\exp(cY)]$ is the moment generating function of $Y$ at $c$. Since the mean $\mu=0$ and the standard deviation $\sigma=1$, it follows directly that
 \begin{eqnarray*}
 \mathbb{E}[\exp(cY)]=\exp\left(\mu+\dfrac{1}{2}\sigma^2c^2\right)=\exp\left(\dfrac{c^2}{2}\right).
 \end{eqnarray*} 
 \end{proof}

 The following lemma is from   \cite[Lemma 5.7, pp 15]{Martin2}.
 \begin{lem}\label{ch4lemma6}
 The following inequality holds
\begin{eqnarray*}
 \mathbb{E}\left[\left|pz\mathbf{1}_{\{||x||\geq 1\}}\left<\dfrac{x}{||x||},\dfrac{g(x)}{||x||}\Delta W^M_k\right>\right|^2\right]\leq \exp\left[\dfrac{p^2T(C+||g(0)||)^2}{M}\right],
 \end{eqnarray*}
 for all $x\in\mathbb{R}^d, k\in\{0,1,\cdots,M-1\}, p\in[1,\infty)$ and all $z\in\{-1,1\}$.
 
 \end{lem}

 \begin{proof}
 Let the notation $a^{\top}$ stand for the transposed of a vector $a$ and $Y$ the $m$ column vector define by $Y=\sqrt{\dfrac{T}{M}}(1,\cdots, 1)$. Then we have :
 \begin{eqnarray*}
 \mathbb{E}\left[\exp\left(pz\left<\dfrac{x}{||x||},\dfrac{g(x)}{||x||}\Delta W^M_k\right>\right)\right]&=&\mathbb{E}\left[\exp\left(pz\dfrac{g(x)^{\top}x}{||x||^2}\Delta W^M_k\right)\right]\\
 &=&\mathbb{E}\left[\exp\left(pz\dfrac{g(x)^{\top}x}{||x||^2}\sqrt{\dfrac{T}{M}}\mathcal{N}(0,1)\right)\right]\\
 &=&\mathbb{E}\left[\exp\left(pzY\dfrac{g(x)^{\top}x}{||x||^2}\right)\right]
 \end{eqnarray*}
 Using Lemma \ref{ch4lemma6}, it follows that 
 \begin{eqnarray*}
 \mathbb{E}\left[\exp\left(pz\left<\dfrac{x}{||x||},\dfrac{g(x)}{||x||}\Delta W^M_k\right>\right)\right]&=&\exp\left[\dfrac{1}{2}\left|pz\dfrac{g(x)^{\top}x}{||x||^2}Y\right|^2\right]\\
 &\leq &\exp\left[p^2\dfrac{||g(x)||^2}{||x||^2}||Y||^2\right]\\
 &\leq &\exp\left[\dfrac{p^2T}{M}\dfrac{||g(x)||^2}{||x||^2}\right].
 \end{eqnarray*}
 Using the global Lipschitz condition and the fact that $||x||\geq 1$, we have :
 \begin{eqnarray*}
 \dfrac{||g(x)||^2}{||x||^2}\leq\dfrac{(||g(x)-g(0)||+||g(0)||)^2}{||x||^2}\leq\dfrac{(C+||g(0)||)^2||x||^2}{||x||^2}\leq (C+||g(0)||)^2.
 \end{eqnarray*}
 Therefore, for all $x\in\mathbb{R}^d$ such that $||x||\geq 1$, we have 
 \begin{eqnarray*}
  \mathbb{E}\left[\exp\left(pz\left<\dfrac{x}{||x||},\dfrac{g(x)}{||x||}\Delta W^M_k\right>\right)\right]\leq \exp\left[\dfrac{p^2T(C+||g(0)||)^2}{M}\right],
 \end{eqnarray*}
 for all $M\in\mathbb{N}$, $k\in\{0,\cdots,M-1\}$, all $p\in[1,\infty)$ and $z\in\{-1,1\}$.
 \end{proof}

 Following closely  \cite[Lemma 3.4, pp 15] {Martin1} we have the following lemma.
 \begin{lem}\label{ch4lemma7}
 Let $\alpha^M_n :\Omega\longrightarrow\mathbb{R}$ for $M\in\mathbb{N}$ and $n\in\{0,1,\cdots,M\}$ defined in Notation \ref{ch4notation1}, then the following inequality holds :
 \begin{eqnarray*}
 \sup_{z\in\{-1,1\}}\sup_{M\in\mathbb{N}}\left\|\sup_{n\in\{0,1,\cdots,M\}}\exp\left(z\sum_{k=0}^{n-1}\alpha^M_k\right)\right\|_{L^p(\Omega, \mathbb{R})}<\infty,
 \end{eqnarray*}
 for all $p\in[2,+\infty)$.
 \end{lem}
 
 \begin{proof}
 The time discrete stochastic process $z\sum_{k=0}^{n-1}\alpha^M_k$, $n\in\{0,1,\cdots, M\}$ is an $(\mathcal{F}_{nT/M})_{n\in\{0,\cdots,M\}}-$ martingale for every $z\in\{-1,1\}$ and $M\in \mathbb{N}$. So $\exp\left(z\sum_{k=0}^{n-1}\alpha^M_k\right)$ is a positive $(\mathcal{F}_{nT/M})_{n\in\{0,\cdots,M\}}-$ submartingale for every $z\in\{-1,1\}$ and $M\in\mathbb{N}$ since $\exp$ is a convex function.
 
 Applying Doop's maximal inequality leads to : 
 \begin{eqnarray}
 \left\|\sup_{n\in\{0,\cdots,M\}}\exp\left(z\sum_{k=0}^{n-1}\alpha^M_k\right)\right\|_{L^p{(\Omega, \mathbb{R})}}&=&\left(\mathbb{E}\left|\sup_{n\in\{0,\cdots,M\}}\exp\left(pz\sum_{k=0}^{n-1}\alpha^M_k\right)\right|\right)^{1/p}\nonumber\\
 &\leq &\left(\dfrac{p}{p-1}\right)\left(\mathbb{E}\left |\exp\left(pz\sum_{k=0}^{M-1}\alpha^M_k\right)\right|\right)^{1/p}\nonumber\\
 &= &\dfrac{p}{p-1}\left\|\exp\left(z\sum_{k=0}^{M-1}\alpha^M_k\right)\right\|_{L^p{(\Omega,\mathbb{R})}}.
 \label{ch4alpha1}
 \end{eqnarray}
 Using  Lemma \ref{ch4lemma6}, it follows from the previous inequality that :
 \begin{eqnarray}
 \mathbb{E}\left[\exp(pz\alpha^M_k)/\mathcal{F}_{kT/M}\right]\leq\exp\left(\dfrac{p^2T(C+||g(0)||)^2}{M}\right).
 \label{ch4alpha2}
 \end{eqnarray}
 Using  inequality \eqref{ch4alpha2},  it follows that :
 \begin{eqnarray*}
 \mathbb{E}\left[\exp\left(pz\sum_{k=0}^{M-1}\alpha^M_k\right)\right]&=&\mathbb{E}\left[\exp\left(pz\sum_{k=0}^{M-2}
 \alpha^M_k\right)\mathbb{E}[\exp(p\alpha^M_{M-1}/\mathcal{F}_{(M-1)T/M}\right]\\
 &\leq &\mathbb{E}\left[\exp\left(pz\sum_{k=0}^{M-2}\alpha^M_k\right)\right]\exp\left(\dfrac{p^2T(C+||g(0)||)^2}{M}\right).
 \end{eqnarray*}
 Iterating the previous inequality $M$ times gives :
 \begin{eqnarray}
\mathbb{E}\left[\exp\left(pz\sum_{k=0}^{M-1}\alpha^M_k\right)\right] \leq \exp(p^2T(C+||g(0)||)^2).
\label{ch4alpha3}
 \end{eqnarray}
 Now combining  inequalities  \eqref{ch4alpha1} and \eqref{ch4alpha3} leads to 
 \begin{eqnarray*}
 \sup_{z\in\{-1,1\}}\sup_{M\in\mathbb{N}}\left\|\sup_{n\in\{0,\cdots,M\}}\exp\left(z\sum_{k=0}^{n-1}\alpha^M_k\right)\right\|_{L^p(\Omega,\mathbb{R})} \leq 2\exp(p^2T(C+||g(0)||)^2) <\infty,
 \end{eqnarray*}
 for all $p\in[2,\infty)$.
 \end{proof}

 \begin{lem} \label{ch4lemma8}
 For all $c\in\mathbb{R}$, we have :
 \begin{eqnarray*}
 \mathbb{E}[\exp(c\Delta \overline{N}^M_n)]=\exp\left[\dfrac{(e^c+c-1)\lambda T}{M}\right],
 \end{eqnarray*}
 for all $M\in\mathbb{N}$ and all $n\in\{0,\cdots,M\}$.
 \end{lem}
 \begin{proof}
It is known that if $Y$ is a random variable following the poisson law with parameter $\lambda$, then its moment generating function is given by :
 \begin{eqnarray*}
 \mathbb{E}[\exp(cY)]=\exp(\lambda(e^c-1)).
 \end{eqnarray*}
 \label{gene}
 Since $\Delta N_n$ follows  a poisson law with parameter $\lambda\Delta t$,  it follows that 
 \begin{eqnarray*}
 \mathbb{E}[\exp(c\Delta\overline{N}^M_n)]&=&\mathbb{E}[\exp(c\Delta N^M_n+c\lambda\Delta t)]\\
 &=&\mathbb{E}\left[\exp\left(\dfrac{\lambda T}{M}\right)\exp(c\Delta N^M_n)\right]\\
 &=&\exp\left(\dfrac{c\lambda T}{M} \right)\exp\left[ \dfrac{\lambda T}{M}(e^c-1)\right]\\
 &=&\exp\left[\dfrac{(e^c+1-1)\lambda T}{M}\right].
 \end{eqnarray*}
 \end{proof}

 \begin{lem}\label{ch4lemma9}
 The following inequality holds
 \begin{eqnarray*}
 \mathbb{E}\left[\exp\left(pz\mathbf{1}_{\{||x||\geq 1\}}\left<\dfrac{x}{||x||},\dfrac{h(x)}{||x||}\Delta\overline{N}^M_n\right>\right)\right]\leq\exp\left[\dfrac{\lambda \left(e^{p(C+||h(0)||)}+p(C+||h(0)||\right)}{M}\right],
 \end{eqnarray*}
 for all $M\in\mathbb{N}$, $z\in\{-1,1\}$, all $p\in[1, +\infty)$ and all $n\in\{0,\cdots,M\}$.
 \end{lem}
 
 \begin{proof} For $x\in\mathbb{R}^d$ such that $||x||\neq 0$, we have :
 \begin{eqnarray*}
  \mathbb{E}\left[\exp\left(pz\left<\dfrac{x}{||x||},\dfrac{h(x)}{||x||}\Delta\overline{N}^M_n\right>\right)\right]  &\leq& \mathbb{E}\left[\exp\left(pz\dfrac{||x||||h(x)||}{||x||^2}\Delta\overline{N}^M_n\right)\right]\\
  &=& \mathbb{E}\left[\exp\left(pz\dfrac{||h(x)||}{||x||}\Delta \overline{N}^M_n\right)\right].
  \end{eqnarray*}
 For all $x\in\mathbb{R}^d$ such that $||x||\geq 1$, using the global Lipschitz condition satisfied by $h$,  we have :
  \begin{eqnarray}
  \dfrac{||h(x)||}{||x||}\leq \dfrac{||h(x)-h(0)||+||h(0)||}{||x||}\leq C+||h(0)||. \label{ch4normeh}
  \end{eqnarray}
  So from inequality \eqref{ch4normeh} and using Lemma \ref{ch4lemma8} it follows that :
  \begin{eqnarray*}
   \mathbb{E}\left[\exp\left(pz\mathbf{1}_{\{||x||\geq 1\}}\left<\dfrac{x}{||x||},\dfrac{h(x)}{||x||}\Delta\overline{N}^M_n\right>\right)\right] &\leq &\mathbb{E}[\exp(pz(C+||h(0)||)\Delta\overline{N}^M_n)]\\
   &\leq &\exp\left[\dfrac{\left(e^{p(C+||h(0)||)}+p(C+||h(0)||-1\right)\lambda T}{M}\right]\\
   &\leq &\exp\left[\dfrac{\left(e^{p(C+||h(0)||)}+p(C+||h(0)||\right)\lambda T}{M}\right].
  \end{eqnarray*}
 \end{proof}

 \begin{lem}\label{ch4lemma10}
 Let $\beta^M_n :\Omega \longrightarrow \mathbb{R}$ define as in Notation \ref{ch4notation1} for all $M\in\mathbb{N}$ and all $n\in\{0,\cdots,M\}$, then we have the following inequality
 \begin{eqnarray*}
 \sup_{z\in\{-1,1\}}\sup_{M\in\mathbb{N}}\left\|\sup_{n\in\{0,\cdots, M\}}\exp\left(z\sum_{K=0}^{n-1}\beta^M_k\right)\right\|_{L^p(\Omega, \mathbb{R})}<+\infty.
 \end{eqnarray*} 
 \end{lem}
 
 \begin{proof}
 For the same reason as for $\alpha^M_k$, $\beta^M_k$ is an $(\mathcal{F}_{nT/M})$- martingale. So $\exp\left(pz\sum_{k=0}^{n-1}\beta^M_k\right)$ is a positive $(\mathcal{F}_{nT/M})$- submartingale for all $M\in\mathbb{N}$  and all $n\in\{0,\cdots, M\}$. Using Doop's maximal inequality we have :
 \begin{eqnarray}
 \left\|\sup_{n\in\{0,\cdots, M\}}\exp\left(z\sum_{k=0}^{n-1}\beta^M_k\right)\right\|_{L^p(\Omega, \mathbb{R})} &\leq &\left(\dfrac{p}{p-1}\right)\left\|\exp\left(z\sum_{k=0}^{M-1}\beta^M_k\right)\right\|_{L^p(\Omega,\mathbb{R})},
 \label{ch4beta1}
 \end{eqnarray}
 \begin{eqnarray*}
 \left\|\exp\left(z\sum_{k=0}^{M-1}\beta^M_k\right)\right\|^p_{L^p(\Omega,\mathbb{R})}&=&\mathbb{E}\left[\exp\left(pz\sum_{k=0}^{M-1}\beta^M_k\right)\right]=\mathbb{E}
 \left[\exp\left(pz\left(\sum_{k=0}^{M-2}\beta^M_k\right)+pz\beta_{M-1}^M\right)\right]\\
 &=&\mathbb{E}\left[\exp\left(pz\sum_{k=0}^{M-2}\beta^M_k\right)
 \mathbb{E}\left[\exp\left(pz\beta^M_{M-1}\right)/\mathcal{F}_{(M-1)T/M}\right]
 \right].
 \end{eqnarray*}
 Using Lemma \ref{ch4lemma9} it follows that
 \begin{eqnarray*}
 \left\|\exp\left(z\sum_{k=0}^{M-1}\beta^M_k\right)\right\|^p_{L^p(\Omega,\mathbb{R})} \leq \mathbb{E}\left[\exp\left(pz\sum_{k=0}^{M-2}\beta^M_k\right)\right]\exp\left[\dfrac{\left(e^{p(C+||h(0)||)}+p(C+||h(0)||)\right)\lambda T}{M}\right].
  \end{eqnarray*}
 Iterating this last inequality $M$ times leads to :
 \begin{eqnarray}
 \left(\mathbb{E}\left[\exp\left(pz\sum_{k=0}^{M-1}\beta^M_k\right)\right]\right)^p\leq \exp\left[\lambda T\left(e^{p(C+||h(0)||)}+Tp(C+||h(0)||\right)\right],
 \label{ch4beta2}
 \end{eqnarray}
 for all $M\in\mathbb{N}$, all $p\in (1,\infty)$ and all $z\in\{-1,1\}$.
 
 Combining   inequalities \eqref{ch4beta1} and \eqref{ch4beta2} complete the proof of Lemma \ref{ch4lemma10}
 \end{proof}

 \begin{lem}\label{ch4lemma11}
The following inequality holds 
\begin{eqnarray*}
\sup_{M\in\mathbb{N}}\mathbb{E}\left[\exp\left(p\beta\sum_{k=0}^{M-1}||\Delta\overline{N}^M_k||\right)\right]<+\infty,
\end{eqnarray*}
for all $p\in[1, +\infty)$.
\end{lem}

  \begin{proof}
  Using independence, stationarity of $\Delta\overline{N}_k^M$ and Lemma \ref{ch4lemma8}, it follows that :
  \begin{eqnarray*}
  \sup_{M\in\mathbb{N}}\mathbb{E}\left[\exp\left(p\beta\sum_{k=0}^{M-1}||\Delta\overline{N}^M_k||\right)\right]&=&\prod_{k=0}^{M-1}\mathbb{E}[\exp(p\beta||\Delta\overline{N}^M_k||)]\\
  &=&\left(\mathbb{E}[\exp(p\beta||\Delta\overline{N}^M_k||)]\right)^M\\
  &=&\left(\exp\left[\dfrac{(e^{p\beta}+p\beta-1)\lambda T}{M}\right]\right)^M\\
  &=&\exp[e^{p\beta}+p\beta-1]<+\infty,
  \end{eqnarray*}
  for all $p\in[1, +\infty)$.
  \end{proof}

  Inspired by   \cite[Lemma 3.5, pp 15]{Martin1}, we have the following estimation.
  \begin{lem}\label{ch4lemma12}
  [Uniformly bounded moments of the dominating stochastic processes].
  
  Let $M\in\mathbb{N}$ and $D_n^M : \Omega \longrightarrow [0,\infty)$ for $n\in\{0,1,\cdots, M\}$ be define as above, then we have :
  \begin{eqnarray*}
  \sup_{M\in\mathbb{N}, M\geq8\lambda pT}\left\|\sup_{n\in\{0,1,\cdots, M\}}D_n^M\right\|_{L^p(\Omega, \mathbb{R})}<\infty,
  \end{eqnarray*}
  for all $p\in[1,\infty)$. 
  \end{lem}
  
  \begin{proof}
  Let's recall  that :
  \begin{eqnarray*}
  D_n^M =(\beta+||\varepsilon||)\exp\left(\dfrac{3\beta}{2}+\sup_{u\in\{0,\cdots,n\}}\sum_{k=u}^{n-1}\dfrac{3\beta}{2}||\Delta W^M_k||^2+\dfrac{3\beta}{2}|\Delta\overline{N}^M_k|+\alpha^M_k+\beta^M_k\right).
  \end{eqnarray*}
  Using Holder inequality, it follows that :
  \begin{eqnarray*}
  \sup_{M\in\mathbb{N}, M\geq 8\lambda pT}\left\|\sup_{n\in\{0,\cdots, M\}}D_n^M\right\|_{L^p(\Omega, \mathbb{R})}&\leq &e^{3\beta/2}\left(\beta+||\varepsilon||_{L^{4p}(\Omega, \mathbb{R})}\right)\\
  &\times& \sup_{M\in\mathbb{N}, M\geq 8\lambda pT}\left\|\exp\left(\dfrac{3\beta}{2}\sum_{k=0}^{M-1}||\Delta W^M_k||^2\right)\right\|_{L^{2p}(\Omega, \mathbb{R})}\\
  &\times &\sup_{M\in\mathbb{N}}\left\|\exp\left(\dfrac{3\beta}{2}\sum_{k=0}^{M-1}|\Delta\overline{N}_k^M|\right)\right\|_{L^{8p}(\Omega, \mathbb{R})}\\
  &\times &\left(\sup_{M\in\mathbb{N}}\left\|\sup_{n\in\{0,\cdots, M\}}\exp\left(\sup_{u\in\{0,\cdots,n\}}\sum_{k=u}^{n-1}\alpha_k^M\right)\right\|_{L^{16p}(\Omega, \mathbb{R})}\right)\\
  &\times &\left(\sup_{M\in\mathbb{N}}\left\|\sup_{n\in\{0,\cdots, M\}}\exp\left(\sup_{u\in\{0,\cdots,n\}}\sum_{k=u}^{n-1}\beta_k^M\right)\right\|_{L^{16p}(\Omega, \mathbb{R})}\right)\\
  &=& A_1\times A_2\times A_3\times A_4\times A_5.
  \end{eqnarray*}
  By assumption $A_1$ is bounded. Lemma \ref{ch4lemma4} and \ref{ch4lemma11} show that $A_2$ and $A_3$ are bounded.  Using again Holder inequality and Lemma \ref{ch4lemma7} it follows that :
  \begin{eqnarray*}
  A_4=\left\|\sup_{n\in\{0,\cdots, M\}}\exp\left(\sup_{u\in\{0,\cdots,n\}}\sum_{k=u}^{n-1}\alpha_k^M\right)\right\|_{L^{16p}(\Omega, \mathbb{R})}
  \end{eqnarray*}
  \begin{eqnarray*}
   \leq\left\|\sup_{n\in\{0,\cdots, M\}}\exp\left(\sum_{k=0}^{n-1}\alpha^M_k\right)\right\|_{L^{32p}(\Omega,\mathbb{R})}
  \times \left\|\sup_{u\in\{0,\cdots, M\}}\exp\left(-\sum_{k=0}^{u-1}\alpha^M_k\right)\right\|_{L^{32p}(\Omega,\mathbb{R})} <+\infty,
  \end{eqnarray*}
  for all $M\in\mathbb{N}$ and all $p\in[1,\infty)$.
  
  Along the same lines as above,  we prove that $A_5$ is bounded.
  
  Since each of the terms $A_1, A_2, A_3, A_4$ and $A_5$ is bounded, this complete the proof of  Lemma \ref{ch4lemma12}.
  \end{proof}

  The following lemma is an extension of \cite[Lemma 3.6, pp 16]{Martin1}. Here, we include  the jump part.
  \begin{lem}\label{ch4lemma13}
  Let $M\in\mathbb{N}$ and $\Omega_M^M\in\mathcal{F}$. The following holds :
  \begin{eqnarray*}
  \sup_{M\in\mathbb{N}}\left(M^p\mathbb{P}[(\Omega_M^M)^c]\right)<+\infty,
  \end{eqnarray*}
  for all $p\in[1,\infty)$.
  \end{lem}
   \begin{proof}
   Using the subadditivity of the probability measure and the Markov's inequality, it follows that
   \begin{eqnarray*}
   \mathbb{P}[(\Omega_M^M)^c] &\leq & \mathbb{P}\left[\sup_{n\in\{0,\cdots, M-1\}}D_n^M>M^{1/2c}\right]+M\mathbb{P}\left[\|W_{T/M}\|>1\right]+M\mathbb{P}\left[|\overline{N}_{T/M}|>1\right]\nonumber\\
    &\leq & \mathbb{P}\left[\sup_{n\in\{0,\cdots, M-1\}}|D_n^M|>M^{q/2c}\right]+M\mathbb{P}\left[\|W_{T}\|>\sqrt{M}\right]+M\mathbb{P}\left[|\overline{N}_{T}|>M\right]\nonumber\\
     &\leq & \mathbb{P}\left[\sup_{n\in\{0,\cdots, M-1\}}|D_n^M|>M^{q/2c}\right]+M\mathbb{P}\left[\|W_{T}\|^q>M^{q/2}\right]+M\mathbb{P}\left[|\overline{N}_{T}|^q>M^q\right]\nonumber\\
   &\leq &\mathbb{E}\left[\sup_{n\in\{0,\cdots, M-1\}}|D_n^M|^q\right]M^{-q/2c}+\mathbb{E}[\|W_T\|^q]M^{1-q/2}+ \mathbb{E}[|\overline{N}_{T}|^q] M^{1-q} \nonumber,\\
   \end{eqnarray*}
   for all $q>1$.
   
   Multiplying both sides of the above inequality  by $M^p$ leads to
   \begin{eqnarray*}
   M^p\mathbb{P}[(\Omega_M^M)^c]\leq \mathbb{E}\left[\sup_{n\in\{0,\cdots, M-1\}}|D_n^M|^q\right]M^{p-q/2c}+\mathbb{E}[\|W_T\|^q]M^{p+1-q/2}+\mathbb{E}[|\overline{N}_{T}|^q] M^{p+1-q}
   \end{eqnarray*}
   for all $q>1$.
   
    For $q>\max\{2pc, 2p+2\}$, we have  $M^{p+1-q/2}<1$, $M^{p-q/2c}<1$ and $ M^{p+1-q}<1$. It follows for this choice of $q$ that 
   \begin{eqnarray*}
   M^p\mathbb{P}[(\Omega_M^M)^c]\leq \mathbb{E}\left[\sup_{n\in\{0,\cdots, M-1\}}|D_n^M|^p\right]+\mathbb{E}[\|W_T\|^q]+\mathbb{E}[|\overline{N}_{T}|^q].
   \end{eqnarray*}
   Using Lemma \ref{ch4lemma12} and the fact that  $W_T$ and  $\overline{N}_{T}$ are independents of $M$, it follows that
   \begin{eqnarray*}
   \sup_{M\in\mathbb{N}}\left(M^p\mathbb{P}[(\Omega_M^M)^c]\right)<+\infty.
   \end{eqnarray*}
   \end{proof}

 The following lemma can be found in  \cite[Theorem 48 pp 193]{Protter} or in \cite[Theorem 1.1, pp 1]{Gundy}.
\begin{lem}\label{ch4lemma18b}[Burkholder-Davis-Gundy inequality]

 Let $M$ be a martingale with c\`{a}dl\`{a}g paths and let $p\geq 1$ be fixed. Let
   $M_t^*=\sup\limits_{s\leq t}||M_s||$. Then there exist constants $c_p$ and $C_p$ such that for any $M$
 \begin{eqnarray*}
c_p \left[\mathbb{E}\left([M,M]_t\right)^{p/2}\right]^{1/p}\leq \left[\mathbb{E}(M_t^*)^p\right]^{1/p}\leq C_p\left[\mathbb{E}\left([ M, M]_t\right)^{p/2}\right]^{1/p},
 \end{eqnarray*}
 for all $0\leq t\leq \infty$, where $[M,M]$ stand for the quadratic variation of the process $M$. The constants $c_p$ and $C_p$ are universal : They does not depend on the choice of $M$. 
\end{lem}

  The following lemma can be found  in \cite[Lemma 3.7, pp 16]{Martin1}.
  \begin{lem} \label{ch4lemma15}
  
  Let $k\in\mathbb{N}$ and let $Z : [0,T]\times \Omega \longrightarrow \mathbb{R}^{k\times m}$ be a predictable stochastic process satisfying     
  $\mathbb{P}\left[\int_0^T||Z_s||^2ds<+\infty\right]=1$. Then we have the following inequality 
  \begin{eqnarray*}
  \left\|\sup_{s\in[0,t]}\left\|\int_0^sZ_udW_u\right\|\right\|_{L^p(\Omega, \mathbb{R})}\leq C_p\left(\int_0^t\sum_{i=1}^m||Z_s\vec{e}_i||^2_{L^p(\Omega, \mathbb{R}^k)}ds\right)^{1/2}
  \end{eqnarray*}
  for all $t\in[0,T]$ and all $p\in[1,\infty)$. Where $(\vec{e}_1,\cdots,\vec{e}_m)$ is the canonical basis of $\mathbb{R}^m$.
  \end{lem}
  \begin{proof}
 Since $W$ is a continuous martingale satisfying $d[ W,W]_s= ds$,  it follows from the property of the quadratic variation (see \cite[8.21, pp 219]{Fima} ) that  
 \begin{eqnarray}
 \left[ \int_0Z_sdW_s, \int_0Z_sdW_s\right]_t=\int_0^t||Z_s||^2d[W,W]_s=\int_0^t||Z_s||^2 ds.
 \label{ch4Bur2}
 \end{eqnarray}
 Applying Lemma \ref{ch4lemma18b} for $M_t=\sup\limits_{0\leq s\leq T}\int_0^tZ_sdW_s$ and using \eqref{ch4Bur2} leads to :

 \begin{eqnarray}
 \left[\mathbb{E}\left[\sup_{0\leq t\leq T}\left\|\int_0^tZ_udW_u\right\|^p\right]\right]^{1/p}\leq C_p\left[\mathbb{E}\left(\int_0^T||Z_s||^2ds\right)^{p/2}\right]^{1/p},
 \end{eqnarray}
where $C_p$ is a positive constant depending on $p$  : 

 Using the definition of $||X||_{L^p(\Omega, \mathbb{R})}$ for any random variable $X$, it follows that 
 \begin{eqnarray*}
 \left\|\sup_{s\in[0,T]}\left\|\int_0^sZ_udW_u\right\|\right\|_{L^p(\Omega, \mathbb{R})}&\leq& C_p\left\|\int_0^T||Z_s||^2ds\right\|^{1/2}_{L^{p/2}(\Omega, \mathbb{R})}\\
 &\leq&C_p\left\|\int_0^T\sum_{i=1}^m||Z_s.\vec{e}_i||^2ds\right\|^{1/2}_{L^{p/2}(\Omega, \mathbb{R})}.
 \end{eqnarray*}
 Using Minkowski inequality in its integral form (see Proposition \ref{ch1Minkowski}) yields : 
  \begin{eqnarray*}
\left\|\sup_{s\in[0,T]}\left\|\int_0^sZ_udW_u\right\|\right\|_{L^p(\Omega, \mathbb{R})}\leq C_p\left(\int_0^T\sum_{i=1}^m\left\|||Z_s.\vec{e}_i||^2\right\|_{L^{p/2}(\Omega, \mathbb{R}^k)}ds\right)^{1/2}.
 \end{eqnarray*}
 Using Holder inequality, it follows that :
 \begin{eqnarray}
\left\|\sup_{s\in[0,T]}\left\|\int_0^sZ_udW_u\right\|\right\|_{L^p(\Omega, \mathbb{R})}\leq C_p\left(\int_0^T\sum_{i=1}^m||Z_s.\vec{e}_i||^2_{L^p(\Omega, \mathbb{R}^k)}ds\right)^{1/2}.
 \end{eqnarray}
 This complete the proof of the lemma.
 \end{proof}

  The following lemma and its proof can be found in \cite[Lemma 3.8, pp 16]{Martin1}.
  \begin{lem}\label{ch4lemma16} 
  
  Let $k\in\mathbb{N}$ and let $ Z^M_l : \Omega \longrightarrow \mathbb{R}^{k\times m}$, $l\in\{0,1,\cdots, M-1\}$, $M\in\mathbb{N}$ be a familly  of mappings such that $Z^M_l$ is $\mathcal{F}_{lT/M}/\mathcal{B}(\mathbb{R}^{k\times m})$-measurable for all $l\in\{0,1,\cdots,M-1\}$ and $M\in\mathbb{N}$. Then the following inequality holds :
  \begin{eqnarray*}
  \left\|\sup_{j\in\{0,1,\cdots,n\}}\left\|\sum_{l=0}^{j-1}Z_l^M\Delta W^M_l\right\|\right\|_{L^p(\Omega, \mathbb{R})} \leq C_p\left(\sum_{l=0}^{n-1}\sum_{i=1}^m||Z^M_l.\vec{e}_i||^2_{L^p(\Omega, \mathbb{R}^k)}\dfrac{T}{M}\right)^{1/2}.
  \end{eqnarray*}
  \end{lem}
  \begin{proof}
  Let $\overline{Z}^M : [0,T]\times\Omega \longrightarrow \mathbb{R}^{k\times m}$ such that $\overline{Z}_s  :=Z^M_l$ for all $s\in\left[\dfrac{lT}{M},\dfrac{(l+1)T}{M}\right)$,
  
   $l\in\{0,1,\cdots,M-1\}$ and all $M\in\mathbb{N}$.
  
  Using Lemma \ref{ch4lemma15}, one obtain :
  \begin{eqnarray*}
  \left\|\sup_{j\in\{0,1,\cdots,n\}}\left\|\sum_{l=0}^{j-1}Z_l^M\Delta W^M_l\right\|\right\|_{L^p(\Omega, \mathbb{R})} &=&\left\|\sup_{j\in\{0,1,\cdots, n\}}\left\|\int_0^{jT/M}\overline{Z}^M_udW_u\right\|\right\|_{L^p(\Omega, \mathbb{R})}\\
  &\leq & \left\|\sup_{s\in\left[0,\dfrac{nT}{M}\right]}\left\|\int_0^s\overline{Z}^M_udW_u\right\|\right\|_{L^p(\Omega, \mathbb{R})}\\
  &\leq &C_p\left(\int_0^{nT/M}\sum_{i=1}^m||Z_s.\vec{e}_i||^2_{L^p(\Omega,\mathbb{R}^k)}ds\right)^{1/2}\\
  &=&C_p\left(\sum_{l=0}^{n-1}\sum_{i=1}^m||Z^M_l.\vec{e}_i||^2_{L^p(\Omega, \mathbb{R}^k)}\dfrac{T}{M}\right)^{1/2}.
  \end{eqnarray*}
  \end{proof}

 \begin{lem}\label{ch4lemma18}
 Let $k\in\mathbb{N}$ and $Z :[0, T]\times\Omega \longrightarrow \mathbb{R}^k$ be a predictable stochastic process satisfying $\mathbb{E}\left(\int_0^T||Z_s||^2ds\right)<+\infty$. Then the following inequality holds :
 \begin{eqnarray*}
 \left\|\sup_{s\in[0,T]}\left\|\int_0^sZ_ud\overline{N}_u\right\|\right\|_{L^p(\Omega, \mathbb{R})}\leq C_p\left(\int_0^T||Z_s||^2_{L^p(\Omega, \mathbb{R}^k)}ds\right)^{1/2},
 \end{eqnarray*}
 for all $t\in[0, T]$ and all $p\in[1,+\infty)$.
 \end{lem}
 
 \begin{proof}
 Since $\overline{N}$ is a martingale with c\`{a}dl\`{a}g paths satisfying $d[ \overline{N},\overline{N}]_s=\lambda s$ (see Proposition \ref{ch1quadratic}), it follows from  the property of the quadratic variation (see \cite[8.21, pp 219]{Fima}) that 
 \begin{eqnarray}
 \left[\int_0Z_sd\overline{N}_s, \int_0Z_sd\overline{N}_s\right]_t=\int_0^t||Z_s||^2\lambda ds.
 \label{ch4Bur1}
 \end{eqnarray}
 Applying Lemma \ref{ch4lemma18b} for $M_t=\sup\limits_{0\leq s\leq T}\int_0^tZ_sd\overline{N}_s$ and using \eqref{ch4Bur1} leads to :

 \begin{eqnarray}
 \left[\mathbb{E}\left[\sup_{0\leq t\leq T}\left\|\int_0^tZ_ud\overline{N}_u\right\|^p\right]\right]^{1/p}\leq C_p\left[\mathbb{E}\left(\int_0^T||Z_s||^2ds\right)^{p/2}\right]^{1/p},
 \end{eqnarray}
where $C_p$ is a positive constant depending on $p$ and $\lambda$.

 Using the definition of $||X||_{L^p(\Omega, \mathbb{R})}$ for any random variable $X$, it follows that 
 \begin{eqnarray}
 \left\|\sup_{s\in[0,T]}\left\|\int_0^sZ_ud\overline{N}_u\right\|\right\|_{L^p(\Omega, \mathbb{R})}\leq C_p\left\|\int_0^T||Z_s||^2ds\right\|^{1/2}_{L^{p/2}(\Omega, \mathbb{R})}.
 \end{eqnarray}
 Using Minkowski inequality in its integral form (see Proposition \ref{ch1Minkowski}) yields : 
 \begin{eqnarray*}
\left\|\sup_{s\in[0,T]}\left\|\int_0^sZ_ud\overline{N}_u\right\|\right\|_{L^p(\Omega, \mathbb{R})}\leq C_p\left(\int_0^T\left\|||Z_s||^2\right\|_{L^{p/2}(\Omega, \mathbb{R}^k)}ds\right)^{1/2}.
 \end{eqnarray*}
Using Holder inequality leads to :
 \begin{eqnarray}
\left\|\sup_{s\in[0,T]}\left\|\int_0^sZ_ud\overline{N}_u\right\|\right\|_{L^p(\Omega, \mathbb{R})}\leq C_p\left(\int_0^T||Z_s||^2_{L^p(\Omega, \mathbb{R}^k)}ds\right)^{1/2}.
 \end{eqnarray}
 This complete the proof of the lemma.
 \end{proof}

 \begin{lem}\label{ch4lemma19}
 Let $k\in\mathbb{N}$,  $M\in\mathbb{N}$ and $Z^M_l : \Omega \longrightarrow \mathbb{R}^k, l\in\{0,1,\cdots, M-1\}$ be a family of mappings such that $Z^M_l$ is $\mathcal{F}_{lT/M}/\mathcal{B}(\mathbb{R}^k)$-measurable for all $l\in\{0,1,\cdots, M-1\}$, then  $\forall\; n\in\{0,1\cdots, M\}$ the following inequality holds : 
 \begin{eqnarray*}
 \left\|\sup_{j\in\{0,1,\cdots, n\}}\left\|\sum_{l=0}^{j-1}Z^M_l\Delta\overline{N}^M_l\right\|\right\|_{L^p(\Omega, \mathbb{R})}\leq C_p\left(\sum_{j=0}^{n-1}||Z^M_j||^2_{L^p(\Omega, \mathbb{R}^k)}\dfrac{T}{M}\right)^{1/2},
 \end{eqnarray*}
 where $C_p$ is a positive constant independent of $M$.
 \end{lem} 
 \begin{proof}
 Let's define $\overline{Z}^M : [0, T]\times\Omega\longrightarrow \mathbb{R}^k$ such that $\overline{Z}^M_s := Z^M_l$ for all $s\in\left[\dfrac{lT}{M}, \dfrac{(l+1)T}{M}\right)$, $l\in\{0, 1,\cdots, M-1\}$.
 
 Using the definition of stochastic integral and Lemma \ref{ch4lemma18}, it follows that :
 \begin{eqnarray*}
 \left\|\sup_{j\in\{0,1,\cdots, n\}}\left\|\sum_{l=0}^{j-1}Z^M_l\Delta\overline{N}^M_l\right\|\right\|_{L^p(\Omega, \mathbb{R})}
  &= &\left\|\sup_{j\in\{0,1,\cdots, n\}}\left\|\int_0^{jT/M}\overline{Z}^M_ud\overline{N}^M_u\right\|\right\|_{L^p(\Omega, \mathbb{R})}\\
 &\leq &\left\|\sup_{s\in[0, nT/M]}\left\|\int_0^s\overline{Z}^M_ud\overline{N}_u\right\|\right\|_{L^p(\Omega, \mathbb{R}^k)}\\
 &\leq & C_p\left(\int_0^{nT/M}||\overline{Z}^M_u||^2_{L^p(\Omega, \mathbb{R}^k)}ds\right)^{1/2}\\
 &=&C_p\left(\sum_{j=0}^{n-1}||Z_j^M||^2_{L^p(\Omega,\mathbb{R}^k)}\dfrac{T}{M}\right)^{1/2}.
 \end{eqnarray*}
 This complete the proof of lemma.
 \end{proof}

   Now we are ready to prove Theorem \ref{ch4theorem1}.
   \begin{proof}\textbf{[ Theorem \ref{ch4theorem1}]}
   
    Let's first represent the numerical approximation $Y^M_n$ in the following appropriate form  :
   \begin{eqnarray*}
   Y^M_n&=&Y_{n-1}^M+\dfrac{\Delta tf_{\lambda}(Y^M_{n-1})}{1+\Delta t||f_{\lambda}(Y^M_{n-1})||}+g(Y_{n-1})\Delta W^M_{n-1}+h(Y^M_{n-1})\Delta\overline{N}^M_{n-1}\\
   &=&X_0+\sum_{k=0}^{n-1}\dfrac{\Delta t f_{\lambda}(Y_k^M)}{1+\Delta t||f_{\lambda}(Y^M_k)||}+\sum_{k=0}^{n-1}g(Y^M_k)\Delta W^M_k+\sum_{k=0}^{n-1}h(Y^M_k)\Delta\overline{N}^M_k\\
   &=& X_0+ \sum_{k=0}^{n-1}g(0)\Delta W^M_k+\sum_{k=0}^{n-1}h(0)\Delta\overline{N}^M_k+\sum_{k=0}^{n-1}\dfrac{\Delta tf_{\lambda}(Y^M_{n-1})}{1+\Delta t||f_{\lambda}(Y^M_{n-1})||}\\
   &+&\sum_{k=0}^{n-1}(g(Y^M_k)-g(0))\Delta W^M_k+\sum_{k=0}^{n-1}(h(Y^M_k)-h(0))\Delta\overline{N}^M_k,
   \end{eqnarray*}
for all $M\in\mathbb{N}   $ and all $n\in\{0,\cdots,M\}$.

Using the inequality 
\begin{eqnarray*}
\left\|\dfrac{\Delta tf_{\lambda}(Y^M_k)}{1+\Delta t||f_{\lambda}(Y^M_k)||}\right\|_{L^P(\Omega, \mathbb{R}^d)} <1
\end{eqnarray*} 
it follows that :
\begin{eqnarray*}
||Y^M_n||_{L^p(\Omega, \mathbb{R}^d)}  &\leq &||X_0||_{L^p(\Omega,\mathbb{R}^d)}+\left\|\sum_{k=0}^{n-1}g(0)\Delta W^M_k\right\|_{L^p(\Omega, \mathbb{R}^d)}+\left\|\sum_{k=0}^{n-1}h(0)\Delta\overline{N}^M_k\right\|_{L^p(\Omega,\mathbb{R}^d)}+M\\
&+&\left\|\sum_{k=0}^{n-1}(g(Y^M_k)-g(0))\Delta W^M_k\right\|_{L^p(\Omega,\mathbb{R}^d)}+\left\|\sum_{k=0}^{n-1}(h(Y^M_k)-h(0))\Delta\overline{N}^M_k\right\|_{L^p(\Omega,\mathbb{R}^d)}.
\end{eqnarray*}
Using Lemma \ref{ch4lemma16} and Lemma \ref{ch4lemma19}, it follows that :
\begin{eqnarray}
||Y^M_n||_{L^p(\Omega,\mathbb{R}^d)} &\leq &||X_0||_{L^p(\Omega,\mathbb{R})}+C_p\left(\sum_{k=0}^{n-1}\sum_{i=1}^{m}||g_i(0)||^2\dfrac{T}{M}\right)^{1/2}+C_p\left(\sum_{k=0}^{n-1}||h(0)||^2\dfrac{T}{M}\right)^{1/2}\nonumber\\
&+&M+ C_p\left(\sum_{k=0}^{n-1}\sum_{i=1}^m||(g_i(Y_k^M)-g_i(0))\Delta W^M_k||^2_{L^p(\Omega, \mathbb{R}^d)}\dfrac{T}{M}\right)^{1/2}\nonumber\\ &+&C_p\left(\sum_{k=0}^{n-1}\lambda||(h(Y_k^M)-h(0))\Delta W^M_k||^2_{L^p(\Omega, \mathbb{R}^d)}\dfrac{T}{M}\right)^{1/2}\nonumber\\
&\leq&||X_0||_{L^p(\Omega, \mathbb{R}^d)}+C_p\left(\dfrac{nT}{M}\sum_{i=1}^m||g_i(0)||^2\right)^{1/2}+C_p\left(\dfrac{nT}{M}||h(0)||^2\right)^{1/2}\nonumber\\
 &+&M+C_p\left(\sum_{k=0}^{n-1}\sum_{i=1}^m||g_i(Y^M_k)-g_i(0)||^2_{L^p(\Omega,\mathbb{R}^d)}\dfrac{T}{M}\right)^{1/2}\nonumber\\
&+&C_p\left(\sum_{k=0}^{n-1}||h(Y^M_k)-h(0)||^2_{L^p(\Omega,\mathbb{R}^d)}\dfrac{T}{M}\right)^{1/2}.
\label{ch4MB1}
\end{eqnarray}
From $||g_i(0)||^2\leq ||g(0)||^2$  and the global Lipschitz condition satisfied by   $g$ and $h$, we obtain 

$||g_i(Y^M_k)-g_i(0)||_{L^p(\Omega, \mathbb{R}^d)}\leq  C||Y^M_k||_{L^p(\Omega,\mathbb{R}^d)}$ and $||h(Y^M_k)-h(0)||_{L^p(\Omega, \mathbb{R}^d)}\leq  C||Y^M_k||_{L^p(\Omega,\mathbb{R}^d)}$. So using \eqref{ch4MB1}, we obtain 
\begin{eqnarray*}
||Y^M_n||_{L^p(\Omega, \mathbb{R}^d)} &\leq & ||X_0||_{L^p(\Omega,\mathbb{R}^d)}+C_p\sqrt{Tm}||g(0)||+C_p\sqrt{T}||h(0)||+M\\
&+&C_p\left(\dfrac{Tm}{M}\sum_{k=0}^{n-1}||Y^M_k||^2_{L^p(\Omega,\mathbb{R}^d)}\right)^{1/2}
+C_p\left(\dfrac{T}{M}\sum_{k=0}^{n-1}||Y^M_k||^2_{L^p(\Omega,\mathbb{R}^d)}\right)^{1/2}.
\end{eqnarray*}
Using the inequality $(a+b+c)^2\leq 3a^2+3b^2+3c^2$, it follows that :
\begin{eqnarray*}
||Y^M_n||^2_{L^p(\Omega,\mathbb{R}^d)}&\leq& 3\left(||X_0||_{L^p(\Omega,\mathbb{R}^d)}+C_p\sqrt{Tm}||g(0)||+C_p\sqrt{T}||h(0)||+M\right)^2\nonumber\\
&+&\dfrac{3T(C_p\sqrt{m}+C_p)^2}{M}\sum_{k=0}^{n-1}||Y^M_k||^2_{L^p(\Omega, \mathbb{R}^d)},
\label{ch4MB2}
\end{eqnarray*}
for all $p\in[1,\infty)$.

Using the fact that 
$
\dfrac{3T(p\sqrt{m}+C_p)^2}{M}<3(p\sqrt{m}+C_p)^2
$  we obtain the following estimation
\begin{eqnarray}
||Y^M_n||^2_{L^p(\Omega,\mathbb{R}^d)}&\leq& 3\left(||X_0||_{L^p(\Omega,\mathbb{R}^d)}+C_p\sqrt{Tm}||g(0)||+C_p\sqrt{T}||h(0)||+M\right)^2\nonumber\\
&+&3T(C_p\sqrt{m}+C_p)^2\sum_{k=0}^{n-1}||Y^M_k||^2_{L^p(\Omega, \mathbb{R}^d)},
\label{ch4MB}
\end{eqnarray}

Applying  Gronwall lemma to \eqref{ch4MB} leads to
\begin{eqnarray}
||Y^M_n||^2_{L^p(\Omega, \mathbb{R}^d)}\leq 3e^{3(C_p\sqrt{m}+C_p)^2}\left(||X_0||_{L^p(\Omega,\mathbb{R}^d)}+C_p\sqrt{Tm}||g(0)||+C_p\sqrt{T}||h(0)||+M\right)^2.
\label{ch4MB3}
\end{eqnarray}
Taking the square root and the supremum in the both sides of \eqref{ch4MB3} leads to :

$
\sup\limits_{n\in\{0,\cdots, M\}}||Y^M_n||_{L^p(\Omega, \mathbb{R}^d)}
$
\begin{eqnarray}
\leq \sqrt{3}e^{3(C_p\sqrt{m}+C_p)^2}\left(||X_0||_{L^p(\Omega,\mathbb{R}^d)}+C_p\sqrt{Tm}||g(0)||+C_p\sqrt{T}||h(0)||+M\right)
\label{ch4MB4}
\end{eqnarray}
Unfortunately, \eqref{ch4MB4} is not enough to conclude the proof of the lemma due to the term  $M$ in the right hand side. Using the fact that $(\Omega_n^M)_n$ is a decreasing sequence and by exploiting Holder inequality, we obtain :
\begin{eqnarray}
\sup_{M\in\mathbb{N}}\sup_{n\in\{0,\cdots,M\}}\left\|\mathbf{1}_{(\Omega_n^M)^c}Y^M_n\right\|_{L^p(\Omega,\mathbb{R}^d)}& \leq &\sup_{M\in\mathbb{N}}\sup_{n\in\{0,\cdots,M\}}\left\|\mathbf{1}_{(\Omega^M_M)^c}\right\|_{L^{2p}(\Omega,\mathbb{R}^d)}\left\|Y^M_n\right\|_{L^{2p}(\Omega,\mathbb{R}^d)}\nonumber\\
&\leq &\left(\sup_{M\in\mathbb{N}}\sup_{n\in\{0,\cdots,M\}}\left(M\left\|\mathbf{1}_{(\Omega^M_M)^c}\right\|_{L^{2p}(\Omega,\mathbb{R}^d)}\right)\right)\nonumber\\
&\times &\left(\sup_{M\in\mathbb{N}}\sup_{n\in\{0,\cdots,M\}}\left(M^{-1}||Y^M_n||_{L^{2p}(\Omega,\mathbb{R}^d)}\right)\right).
\label{ch4MB5}
\end{eqnarray}
Using  inequality \eqref{ch4MB4} we have

$
\left(\sup\limits_{M\in\mathbb{N}}\sup\limits_{n\in\{0,\cdots,M\}}\left(M^{-1}||Y^M_n||_{L^{2p}(\Omega,\mathbb{R}^d)}\right)\right)$
\begin{eqnarray}
\leq \sqrt{3}e^{3(C_p\sqrt{m}+C_p)^2}\left(\dfrac{||X_0||_{L^{2p}(\Omega, \mathbb{R}^d)}}{M}+\dfrac{C_p\sqrt{Tm}||g(0)||+C_p\sqrt{T}||h(0)||}{M}+1\right)\nonumber\\
\leq \sqrt{3}e^{3(C_p\sqrt{m}+C_p)^2}\left(||X_0||_{L^{2p}(\Omega,\mathbb{R}^d)}+C_p\sqrt{Tm}||g(0)||+C_p\sqrt{T}||h(0)||+1\right)<+\infty,
\label{ch4MB6}
\end{eqnarray}
for all $p\geq 1$.
From the relation 
\begin{eqnarray*}
\left\|\mathbf{1}_{(\Omega^M_M)^c}\right\|_{L^{2p}(\Omega, \mathbb{R}^d)}=
\mathbb{E}\left[\mathbf{1}_{(\Omega^M_M)^c}\right]^{1/2p}=
\mathbb{P}\left[(\Omega^M_M)^c\right]^{1/2p},
\end{eqnarray*}
it follows using Lemma  \ref{ch4lemma13} that :
\begin{eqnarray}
\sup_{M\in\mathbb{N}}\sup_{n\in\{0,\cdots, M\}}\left(M\left\|\mathbf{1}_{(\Omega^M_M)^c}\right\|_{L^{2p}(\Omega,\mathbb{R}}\right)=\sup_{M\in\mathbb{N}}\sup_{n\in\{0,\cdots, M\}}\left(M^{2p}\mathbb{P}\left[(\Omega^M_M)^c\right]\right)^{1/2p}<+\infty,
\label{ch4MB7}
\end{eqnarray}
for all $p\geq 1$. 

So  plugging  \eqref{ch4MB6} and \eqref{ch4MB7} in \eqref{ch4MB5}  leads to : 
\begin{eqnarray}
\sup_{M\in\mathbb{N}}\sup_{n\in\{0,\cdots, M\}}\left\|\mathbf{1}_{(\Omega^M_n)^c}Y^M_n\right\|_{L^p(\Omega,\mathbb{R}^d)}<+\infty.
\label{ch4MB8}
\end{eqnarray}
Futhermore, we have 
\begin{eqnarray}
\sup_{M\in\mathbb{N}}\sup_{n\in\{0,\cdots, M\}}\left\|Y^M_n\right\|_{L^p(\Omega,\mathbb{R}^d)} &\leq &\sup_{M\in\mathbb{N}}\sup_{n\in\{0,\cdots, M\}}\left\|\mathbf{1}_{(\Omega^M_n)}Y^M_n\right\|_{L^p(\Omega,\mathbb{R}^d)}\nonumber\\
&+&\sup_{M\in\mathbb{N}}\sup_{n\in\{0,\cdots, M\}}\left\|\mathbf{1}_{(\Omega^M_n)^c}Y^M_n\right\|_{L^p(\Omega,\mathbb{R}^d)}.
\label{ch4MB9}
\end{eqnarray}
From  \eqref{ch4MB8}, the second term  of inequality \eqref{ch4MB9} is bounded, while using Lemma \ref{ch4lemma2} and Lemma \ref{ch4lemma12} we  have :
\begin{eqnarray}
\sup_{M\in\mathbb{N}}\sup_{n\in\{0,\cdots, M\}}\left\|\mathbf{1}_{(\Omega^M_n)}Y^M_n\right\|_{L^p(\Omega,\mathbb{R}^d)}\leq \sup_{M\in\mathbb{N}}\sup_{n\in\{0,\cdots, M\}}\left\|D_n^M\right\|_{L^p(\Omega,\mathbb{R}^d)}<+\infty.
\label{ch4MB10}
\end{eqnarray}
Finally plugging \eqref{ch4MB8} and \eqref{ch4MB10} in \eqref{ch4MB9} leads to : 
\begin{eqnarray*}
\sup_{M\in\mathbb{N}}\sup_{n\in\{0,\cdots, M\}}\left\|Y^M_n\right\|_{L^p(\Omega,\mathbb{R}^d)}<+\infty.
\end{eqnarray*}
   \end{proof}

 \section{Strong convergence of the compensated tamed Euler scheme}
 \hspace{0.5cm} The main result of this chapter is given in the following theorem.
 \begin{thm}\label{ch4theorem2}
 Under Assumptions \ref{ch4assumption1}, for all $p\in[1,+\infty)$  there exist a positive  constant $C_p$ such that :
 \begin{eqnarray}
 \left(\mathbb{E}\left[\sup_{t\in[0,T]}\left\|X_t-\overline{Y}^M_t\right\|^p\right]\right)^{1/p}\leq C_p\Delta t^{1/2},
 \label{ch4inetheo}
 \end{eqnarray}
 for all $M\in \mathbb{N}$.
 
 Where $X : [0,T]\times \Omega\longrightarrow \mathbb{R}^d$ is the exact solution of equation \eqref{ch4exactsol} and $\overline{Y}^M_t$ is the time continous approximation defined by \eqref{ch4continoussolu}.
 \end{thm} 
 In order to prove Theorem \ref{ch4theorem2}, we need the following two lemmas.

 Following closely  \cite[Lemma 3.10, pp 16] {Martin1}, we have the following lemma.
\begin{lem} \label{ch4lemma21}
Let $Y_n^M$ be defined by \eqref{ch4tamc} for all $M\in\mathbb{N}$ and all $n\in\{0,1,\cdots, M\}$, then we have 
\begin{eqnarray*}
\sup_{M\in\mathbb{N}}\sup_{n\in\{0,1,\cdots, M\}}\left(\mathbb{E}\left[||f_{\lambda}(Y_n^M)||^p\right]\vee \mathbb{E}\left[\left\|g(Y_n^M)\right\|^p\right]\vee \mathbb{E}\left[\left\|h(Y_n^M)\right\|^p\right]\right)<+\infty,
\end{eqnarray*}
for all $p\in[1,\infty)$.
\end{lem}   

\begin{proof}
From the polynomial growth condition of $f_{\lambda}$,  for all $x\in\mathbb{R}^d$ we have 
\begin{eqnarray*}
||f_{\lambda}(x)||\leq C(K+||x||^c)||x||+||f_{\lambda}(0)||=CK||x||+C||x||^{c+1}+||f_{\lambda}(0)||.
\end{eqnarray*}
\begin{itemize}
\item 
If $||x||\leq 1$, then $CK||x||\leq CK$, hence 
\begin{eqnarray}
||f_{\lambda}(x)||&\leq& CK+C||x||^{c+1}+||f_{\lambda}(0)||\nonumber\\
&\leq & KC+KC||x||^{c+1}+C+C||x||^{c+1}+||f_{\lambda}(0)||+||f_{\lambda}(0)||||x||^{c+1}\nonumber\\
&=&(KC+C+||f_{\lambda}(0)||)(1+||x||^{c+1}).
\label{ch4eq1}
\end{eqnarray}
\item
If $||x||\geq 1$, then $C||x||\leq C||x||^{c+1}$, hence
\begin{eqnarray}
||f_{\lambda}(x)||&\leq & KC||x||^{c+1}+C||x||^{c+1}+||f_{\lambda}(0)||\nonumber\\
&\leq & KC+KC||x||^{c+1}+C+C||x||^{c+1}+||f_{\lambda}(0)||+||f_{\lambda}(0)||||x||^{c+1}\nonumber\\
&=&(KC+C+||f_{\lambda}(0)||)(1+||x||^{c+1}).
\label{ch4eq2}
\end{eqnarray}
\end{itemize}
So it follows from \eqref{ch4eq1} and \eqref{ch4eq2} that 
\begin{eqnarray}
||f_{\lambda}(x)|| \leq (KC+C+||f_{\lambda}(0)||)(1+||x||^{c+1}), \hspace{0.5cm} \text{for all} \hspace{0.3cm} x\in\mathbb{R}^d.
\label{ch4eq3}
\end{eqnarray}
Using  inequality \eqref{ch4eq3} and Theorem \ref{ch4theorem1}, it follows that :
\begin{eqnarray*}
\sup_{M\in\mathbb{N}}\sup_{n\in\{0,\cdots,M\}}\left\|f_{\lambda}(Y_n^M)\right\|_{L^p(\Omega, \mathbb{R}^d)}&\leq& (KC+C+||f_{\lambda}(0)||)\\
&\times&\left(1+\sup_{M\in\mathbb{N}}\sup_{n\in\{0,\cdots, M\}}\left\|Y^M_n\right\|^{c+1}_{L^{p(c+1)}(\Omega, \mathbb{R}^d)}\right)\\
&<&+\infty,
\end{eqnarray*}
for all $p\in[1,\infty)$.

In other hand, using the global Lipschitz condition satisfied by   $g$ and $h$, it follows that :
\begin{eqnarray}
||g(x)|| \leq C||x||+ ||g(0)|| \hspace{0.2cm} \text{and}\hspace{0.2cm} ||h(x)||\leq C||x||+||h(0)||.
\label{ch4stron2}
\end{eqnarray}

Using once again  Theorem \ref{ch4theorem1}, it follows from \eqref{ch4stron2} that :
\begin{eqnarray*}
 \sup_{M\in\mathbb{N}, n\in\{0, \cdots, M\}}\left\|g(Y^M_n)\right\|_{L^p(\Omega, \mathbb{R}^d)}\leq||g(0)||+C\sup_{M\in\mathbb{N}}\sup_{n\in\{0,\cdots, M\}}\left\|Y^M_n\right\|_{L^p(\Omega, \mathbb{R}^d)}<+\infty,
\end{eqnarray*}
for all $p\in[1,\infty)$.

 Using the same argument as for $g$ the following holds 
\begin{eqnarray*}
\sup_{M\in\mathbb{N}}\sup_{n\in\{0,\cdots, M\}}\left\|h(Y^M_n)\right\|_{L^p(\Omega,\mathbb{R}^d)}<+\infty, 
\end{eqnarray*}
for all $p\in[1, +\infty)$.

 This complete the proof of  Lemma \ref{ch4lemma21}.
\end{proof}
 
 For $s\in[0,T]$  let $\lfloor s\rfloor$ be the greatest grid point less than $s$. We have the following lemma.
 \begin{lem}\label{ch4lemma22}
 For any stepsize $\Delta t$, the following inequalities holds
 \begin{align*}
 \sup_{t\in[0,T]}\left\|\overline{Y}^M_t-\overline{Y}^M_{\lfloor t\rfloor}\right\|_{L^p(\Omega, \mathbb{R}^d)}\leq C_p\Delta t^{1/2},
 \end{align*}
 \begin{eqnarray*}
 \sup_{M\in\mathbb{N}}\sup_{t\in[0,T]}\left\|\overline{Y}^M_t\right\|_{L^p(\Omega, \mathbb{R}^d)}<\infty,
 \end{eqnarray*}
 \begin{eqnarray*}
 \sup_{t\in[0,T]}\left\|f_{\lambda}(\overline{Y}^M_t)-f_{\lambda}(\overline{Y}^M_{\lfloor t\rfloor})\right\|_{L^p(\Omega, \mathbb{R}^d)}\leq C_p\Delta t^{1/2}.
 \end{eqnarray*}
 \end{lem}
 
 \begin{proof}
 \begin{itemize}
 \item
  Using Lemma \ref{ch4lemma18}, Lemma \ref{ch4lemma15} and the time continous approximation \eqref{ch4continoussolu}, it follows that :
 
 $
 \sup\limits_{t\in[0,T]}\left\|\overline{Y}^M_t-\overline{Y}^M_{\lfloor t\rfloor}\right\|_{L^p(\Omega, \mathbb{R}^d)}
 $
 \begin{eqnarray}
  &\leq &\dfrac{T}{M}\left(\sup_{t\in[0,T]}\left\|\dfrac{f_{\lambda}(\overline{Y}^M_{\lfloor t\rfloor})}{1+\Delta t||f_{\lambda}(\overline{Y}^M_{\lfloor t\rfloor})||}\right\|_{L^2(\Omega, \mathbb{R}^d)}\right)+\sup_{t\in[0, T]}\left\|\int^t_{\lfloor t\rfloor} g(\overline{Y}^M_{\lfloor t\rfloor})dW_s\right\|_{L^p(\Omega, \mathbb{R}^d)}\nonumber\\
 &+& \sup_{t\in[0,T]}\left\|\int^t_{\lfloor t\rfloor}h(\overline{Y}^M_{\lfloor t\rfloor})d\overline{N}_s\right\|_{L^p(\Omega, \mathbb{R}^d)}\nonumber\\
 &\leq &\dfrac{T}{\sqrt{M}}\left(\sup_{n\in\{0,\cdots, M\}}||f_{\lambda}(Y^M_n)||_{L^p(\Omega, \mathbb{R}^d)}\right)+\sup_{t\in[0,T]}\left(\dfrac{T}{M}\sum_{i=1}^m\int^t_{\lfloor t\rfloor}||g_i(\overline{Y}^M_s)||^2_{L^p(\Omega, \mathbb{R}^k)}ds\right)^{1/2}\nonumber\\
 &+& \sup_{t\in[0,T]}\left(\dfrac{TC_p}{M}\int^t_{\lfloor t\rfloor}||h(\overline{Y}^M_s)||^2_{L^p(\Omega, \mathbb{R}^d)}ds\right)^{1/2}\nonumber\\ 
 &\leq &\dfrac{T}{\sqrt{M}}\left(\sup_{n\in\{0,\cdots, M\}}||f_{\lambda}(Y^M_n)||_{L^p(\Omega, \mathbb{R}^d)}\right)+\dfrac{\sqrt{Tm}}{\sqrt{M}}\left(\sup_{i\in\{1,\cdots, m\}}\sup_{n\in\{0,\cdots, M\}}||g_i(Y^M_n)||_{L^p(\Omega, \mathbb{R}^k)}\right)\nonumber\\
 &+&\dfrac{C_p\sqrt{T}}{\sqrt{M}}\left(\sup_{n\in\{0, \cdots,M\}}||h(Y^M_n)||_{L^p(\Omega, \mathbb{R}^d)}\right),
 \label{ch4Thcontinous}
 \end{eqnarray}
 for all $M\in\mathbb{N}$.
 
  Using  inequality \eqref{ch4Thcontinous}  and Lemma \ref{ch4lemma21}, it follows  that : 
\begin{eqnarray}
\left[\sup_{t\in[0,T]}\left\|\overline{Y}^M_t-\overline{Y}^M_{\lfloor t\rfloor}\right\|_{L^p(\Omega, \mathbb{R}^d)}\right]<C_p\Delta t^{1/2},
\label{Ch4bon1}
\end{eqnarray}
for all $p\in[1,\infty)$ and all stepsize $\Delta t$. 
 \item  Using  the inequalities \eqref{Ch4bon1}, $||a||\leq ||a-b||+||b||$ for all $a,b\in\mathbb{R}^d$ and Theorem \ref{ch4theorem1}  it follows that 
\begin{eqnarray*}
\sup_{t\in[0,T]}||\overline{Y}^M_t||_{L^p(\Omega, \mathbb{R}^d)}&\leq&\left[\sup_{t\in[0,T]}\left\|\overline{Y}^M_t-\overline{Y}^M_{\lfloor t\rfloor}\right\|_{L^p(\Omega, \mathbb{R}^d)}\right]+\sup_{t\in[0,T]}\left\|\overline{Y}^M_{\lfloor t\rfloor}\right\|_{L^p(\Omega, \mathbb{R}^d)}\\
&\leq&\dfrac{C_p}{M^{1/2}}+\sup_{t\in[0,T]}\left\|\overline{Y}^M_{\lfloor t\rfloor}\right\|_{L^p(\Omega, \mathbb{R}^d)}\\
&<&C_pT^{1/2}+\sup_{t\in[0,T]}\left\|\overline{Y}^M_{\lfloor t\rfloor}\right\|_{L^p(\Omega, \mathbb{R}^d)}\\
&<&\infty,
\end{eqnarray*}
for all $p\in[1,+\infty)$ and all $M\in\mathbb{N}$.
\item Further, using the polynomial growth condition :
\begin{eqnarray*}
||f_{\lambda}(x)-f_{\lambda}(y)||\leq C(K+||x||^c+||y||^c)||x-y||,
\end{eqnarray*}
for all $x, y\in\mathbb{R}^d$, it follows using Holder  inequality that :
\begin{eqnarray}
\sup_{t\in[0,T]}||f_{\lambda}(\overline{Y}^M_t)-f_{\lambda}(\overline{Y}^M_{\lfloor t\rfloor})||_{L^p(\Omega, \mathbb{R}^d)} &\leq &C\left(K+2\sup_{t\in[0,T]}||\overline{Y}^M_t||^c_{L^{2pc}(\Omega, \mathbb{R}^d)}\right)\nonumber\\
&\times & \left(\sup_{t\in[0,T]}||\overline{Y}^M_t-\overline{Y}^M_{\lfloor t\rfloor}||_{L^{2p}(\Omega, \mathbb{R}^d)}\right)
\label{ch4Thfcontinou}
\end{eqnarray}
Using \eqref{ch4Thfcontinou} and the first part of Lemma \ref{ch4lemma22},   the following inequality holds  
\begin{eqnarray}
\left[\sup_{t\in[0,T]}||f_{\lambda}(\overline{Y}^M_t)-f_{\lambda}(\overline{Y}^M_{\lfloor t\rfloor})||_{L^p(\Omega, \mathbb{R}^d)}\right]<C_p\Delta t^{1/2},
\label{ch4Thffinal}
\end{eqnarray} 
for all $p\in[1,\infty)$ and for all stepsize $\Delta t$.
\end{itemize}
\end{proof}

 Now we are ready to give the proof of Theorem \ref{ch4theorem2}.
 \begin{proof} \textbf{[ Theorem \ref{ch4theorem2}]}

 Let's recall that for $s\in[0,T]$, $\lfloor s\rfloor$ denote the greatest grid point less than $s$. The time continuous solution \eqref{ch4continoussolu} can be writen into  its integral form as bellow :
 \begin{eqnarray}
 \overline{Y}^M_s=X_0+\int_0^s\dfrac{f_{\lambda}(\overline{Y}^M_{\lfloor u\rfloor})}{1+\Delta t||f_{\lambda}(\overline{Y}^M_{\lfloor u\rfloor})||}du+ \int_0^s g(\overline{Y}^M_{\lfloor u\rfloor})dW_u+\int_0^s h(\overline{Y}^M_{\lfloor u\rfloor})d\overline{N}_u,
 \label{ch4continoussol2}
 \end{eqnarray}
 for all $s\in[0, T]$ almost surely and all $M\in\mathbb{N}$.
 
 Let's estimate first the quantity  $||X_s-\overline{Y}^M_s||^2$
 \begin{eqnarray*}
 X_s-\overline{Y}_s&=&\int_0^s\left(f_{\lambda}(X_u)-\dfrac{f_{\lambda}(\overline{Y}^M_{\lfloor u\rfloor})}{1+\Delta t||f_{\lambda}(\overline{Y}^M_{\lfloor u\rfloor})||}\right)du+\int_0^s\left(g(X_u)-g(\overline{Y}^M_{\lfloor u\rfloor})\right)dW_u\\
 &+&\int_0^s\left(h(X_u)-h(\overline{Y}^M_{\lfloor u\rfloor})\right)d\overline{N}_u.
 \end{eqnarray*}
 Using the relation $d\overline{N}_u=dN_u-\lambda du$, it follows that
 \begin{eqnarray*}
 X_s-\overline{Y}_s&=&\int_0^s\left[\left(f_{\lambda}(X_u)-\dfrac{f_{\lambda}(\overline{Y}^M_{\lfloor u\rfloor})}{1+\Delta t||f_{\lambda}(\overline{Y}^M_{\lfloor u\rfloor})||}\right)-\lambda \left(h(X_u)-h(\overline{Y}^M_{\lfloor u\rfloor})\right)\right]du\\
 &+&\int_0^s\left(g(X_u)-g(\overline{Y}^M_{\lfloor u\rfloor})\right)dW_u+\int_0^s\left(h(X_u)-h(\overline{Y}^M_{\lfloor u\rfloor})\right)dN_u.
 \end{eqnarray*}
 The function $ k :\mathbb{R}^m\longrightarrow \mathbb{R}$, $x \longmapsto ||x||^2$ is twice differentiable. Applying It\^{o}'s formula for jumps process to the process $X_s-\overline{Y}^M_s$ leads to :
 \begin{eqnarray*}
 \left\|X_s-\overline{Y}^M_s\right\|^2&=&2\int_0^s\left<X_u-\overline{Y}^M_u, f_{\lambda}(X_u)-\dfrac{f_{\lambda}(\overline{Y}^M_{\lfloor u\rfloor})}{1+\Delta t||f_{\lambda}(\overline{Y}^M_{\lfloor u\rfloor})||}\right>du\\
 &-&2\lambda\int_0^s\left<X_u-\overline{Y}^M_u, h(X_u)-h(\overline{Y}^M_{\lfloor u\rfloor})\right>du +\sum_{i=1}^m\int_0^s||g_i(X_u)-g_i(\overline{Y}^M_{\lfloor u\rfloor})||^2du\\
 &+&2\sum_{i=1}^m\int_0^s\left<X_u-\overline{Y}^M_u, g_i(X_u)-g_i(\overline{Y}^M_{\lfloor u\rfloor})\right>dW^i_u\\
 &+& \int_0^s\left[||X_u-\overline{Y}^M_u+h(X_u)-h(\overline{Y}^M_{\lfloor u\rfloor})||^2-||X_u-\overline{Y}^M_u||^2\right]dN_u.
 \end{eqnarray*}
 Using again the relation $d\overline{N}_u=dN_u-\lambda du$ leads to :
 \begin{eqnarray}
 \left\|X_s-\overline{Y}^M_s\right\|^2&=&2\int_0^s\left<X_u-\overline{Y}^M_u, f_{\lambda}(X_u)-\dfrac{f_{\lambda}(\overline{Y}^M_{\lfloor u\rfloor})}{1+\Delta t||f_{\lambda}(\overline{Y}^M_{\lfloor u\rfloor})||}\right>du\nonumber\\
 &-&2\lambda\int_0^s\left<X_u-\overline{Y}^M_u, h(X_u)-h(\overline{Y}^M_{\lfloor u\rfloor})\right>du+\sum_{i=1}^m\int_0^s||g_i(X_u)-g_i(\overline{Y}^M_{\lfloor u\rfloor})||^2du\nonumber\\
 &+&2\sum_{i=1}^m\int_0^s\left<X_u-\overline{Y}^M_u, g_i(X_u)-g_i(\overline{Y}^M_{\lfloor u\rfloor})\right>dW^i_u\nonumber\\
 &+& \int_0^s\left[||X_u-\overline{Y}^M_u+h(X_u)-h(\overline{Y}^M_{\lfloor u\rfloor})||^2-||X_u-\overline{Y}^M_u||^2\right]d\overline{N}_u\nonumber\\
 &+&\lambda\int_0^s\left[||X_u-\overline{Y}^M_u+h(X_u)-h(\overline{Y}^M_{\lfloor u \rfloor})||^2-||X_u-\overline{Y}^M_u||^2\right]du\nonumber\\
 &=&A_1+A_2+A_3+A_4+A_5+A_6.
 \label{ch4Th1}
 \end{eqnarray}
 In the next step, we give some useful estimations of $A_1, A_2, A_3$ and $A_6$. 
 \begin{eqnarray*}
 A_1& : =&2\int_0^s\left<X_u-\overline{Y}^M_u,f_{\lambda}-\dfrac{f_{\lambda}(\overline{Y}_{\lfloor u\rfloor})}{1+\Delta t||f_{\lambda}(\overline{Y}^M_{\lfloor u\rfloor})||}\right>du\\
 &=&2\int_0^s\left\langle X_s-\overline{Y}^M_u,f_{\lambda}(X_u)-f_{\lambda}(\overline{Y}^M_u)\right\rangle du\\
 &+&2\int_0^s\left<X_s-\overline{Y}^M_u,f_{\lambda}(\overline{Y}^M_u)-\dfrac{f_{\lambda}(\overline{Y}^M_{\lfloor u\rfloor})}{1+\Delta t||f_{\lambda}(\overline{Y}^M_{\lfloor u\rfloor})||}\right>du.\\
 &=& A_{11}+A_{12}
 \end{eqnarray*}
 Using the one-sided Lipschitz condition satisfied by $f_{\lambda}$ leads to :
 \begin{eqnarray}
 A_{11} &: =&2\int_0^s\left\langle X_s-\overline{Y}^M_u,f_{\lambda}(X_u)-f_{\lambda}(\overline{Y}^M_u)\right\rangle du\nonumber\\
 &\leq& 2C\int_0^u||X_u-\overline{Y}^M_u||^2du.
 \label{ch4ThA11}
 \end{eqnarray}
 Moreover, using the inequality $\langle a, b\rangle\leq |a||b|\leq \dfrac{a^2}{2}+\dfrac{b^2}{2}$ leads to :
 \begin{eqnarray}
 A_{12}&=& 2\int_0^s\left<X_u-\overline{Y}^M_u,f_{\lambda}(\overline{Y}^M_u)-\dfrac{f_{\lambda}(\overline{Y}^M_{\lfloor u\rfloor})}{1+\Delta t||f_{\lambda}(\overline{Y}^M_{\lfloor u\rfloor})||}\right>du\nonumber\\
 &=&2\int_0^s\left<X_u-\overline{Y}^M_u, f_{\lambda}(\overline{Y}^M_u)-f_{\lambda}(\overline{Y}^M_{\lfloor u\rfloor})\right>ds\nonumber\\
 &+&2\Delta t\int_0^s\left<X_u-\overline{Y}^M_u, \dfrac{f_{\lambda}(\overline{Y}^M_{\lfloor u\rfloor})||f_{\lambda}(\overline{Y}^M_{\lfloor u\rfloor})||}{1+\Delta t||f_{\lambda}(\overline{Y}^M_{\lfloor u\rfloor})||}\right>du\nonumber\\
 &\leq &\int_0^s||X_u-\overline{Y}^M_u||^2du+\int_0^s||f_{\lambda}(\overline{Y}^M_u)-f_{\lambda}(\overline{Y}^M_{\lfloor u\rfloor})||^2du\nonumber\\
 &+&\int_0^s||X_u-\overline{Y}^M_u||^2du+\dfrac{T^2}{M^2}\int_0^s||f_{\lambda}(\overline{Y}^M_{\lfloor u\rfloor})||^4du\nonumber\\
 &\leq &2\int_0^s||X_u-\overline{Y}^M_u||^2du+\int_0^s||f_{\lambda}(\overline{Y}^M_u)-f_{\lambda}(\overline{Y}_{\lfloor u\rfloor})||^2du\nonumber\\
 &+&\dfrac{T^2}{M^2}\int_0^s||f_{\lambda}(\overline{Y}^M_{\lfloor u\rfloor})||^4du.
 \label{ch4ThA12}
 \end{eqnarray}
 Combining \eqref{ch4ThA11} and \eqref{ch4ThA12} give the following estimation of $A_1$ :
 \begin{eqnarray}
 A_1 &\leq & (2C+2)\int_0^s||X_u-\overline{Y}^M_u||^2du+\int_0^s||f_{\lambda}(\overline{Y}^M_u)-f_{\lambda}(\overline{Y}_{\lfloor u\rfloor})||^2du\nonumber\\
 &+&\dfrac{T^2}{M^2}\int_0^s||f_{\lambda}(\overline{Y}^M_{\lfloor u\rfloor})||^4du.
 \label{ch4ThA1}
 \end{eqnarray}
 Using again the inequality $2\langle a, b\rangle\leq 2|a||b|\leq a^2+b^2$ and the global Lipschitz condition satisfied by $h$ leads to :
 \begin{eqnarray}
 A_2 & : =&-2\lambda\int_0^s\left<X_u-\overline{Y}^M_u, h(X_u)-h(\overline{Y}^M_{\lfloor u\rfloor})\right>du\nonumber\\
 &=&-2\lambda\int_0^s\left<X_u-\overline{Y}^M_u, h(X_u)-h(\overline{Y}^M_u)\right>du-2\lambda\int_0^s\left<X_u-\overline{Y}^M_u, h(\overline{Y}^M_u)-h(\overline{Y}^M_{\lfloor u\rfloor})\right>du\nonumber\\
 &\leq &(2\lambda+\lambda C^2)\int_0^s||X_u-\overline{Y}^M_u||^2du+\lambda C^2\int_0^s||\overline{Y}^M_u-\overline{Y}^M_{\lfloor u\rfloor}||^2du.
 \label{ch4ThA2}
 \end{eqnarray}
 Using the inequalities $||g_i(x)-g_i(y)||\leq ||g(x)-g(y)||$ and $(a+b)^2\leq 2a^2+2b^2$ and the global Lipschitz condition  we have 
 \begin{eqnarray}
 A_3 &: =&\sum_{i=1}^m\int_0^s||g_i(X_u)-g_i(\overline{Y}^M_{\lfloor u\rfloor})||^2du\nonumber\\
 &\leq &m\int_0^s||g(X_u)-g(\overline{Y}^M_{\lfloor u\rfloor})||^2du\nonumber\\
 &\leq &m\int_0^s||g(X_u)-g(\overline{Y}^M_u)+g(\overline{Y}^M_u)-g(\overline{Y}^M_{\lfloor u\rfloor})||^2du\nonumber\\
 &\leq &2m\int_0^s||g(X_u)-g(\overline{Y}^M_u)||^2du+2m\int_0^s||g(\overline{Y}^M_u)-g(\overline{Y}^M_{\lfloor u\rfloor})||^2du\nonumber\\
 &\leq& 2mC^2\int_0^s||X_u-\overline{Y}^M_u||^2du+2mC^2\int_0^s||\overline{Y}^M_u-\overline{Y}^M_{\lfloor u\rfloor}||^2du.
 \label{ch4ThA3}
 \end{eqnarray}
 Using the same idea as above we obtain the following estimation of $A_6$ :
 \begin{eqnarray}
 A_6 & : =&\lambda\int_0^s\left[X_u-\overline{Y}^M_u+h(\overline{Y}^M_u)-h(\overline{Y}^M_{\lfloor u\rfloor})||^2-||X_u-\overline{Y}^M_u||^2\right]du\nonumber\\
 &\leq &3\lambda\int_0^s||X_u-\overline{Y}^M_u||^2du+2\lambda\int_0^s||h(X_u)-h(\overline{Y}^M_{\lfloor u\rfloor})||^2du\nonumber\\
 &\leq &3\lambda\int_0^s||X_u-\overline{Y}^M_u||^2du+4\lambda\int_0^s||h(X_u)-h(\overline{Y}^M_u)||^2du\nonumber\\
 &+ &4\lambda\int_0^s||h(\overline{Y}^M_u)-h(\overline{Y}^M_{\lfloor u\rfloor})||^2du\nonumber\\
 &\leq &(3\lambda+4\lambda C^2)\int_0^s||X_u-\overline{Y}^M_u||^2du+4\lambda C^2\int_0^s||\overline{Y}^M_u-\overline{Y}^M_{\lfloor u\rfloor}||^2du.
 \label{ch4ThA6}
 \end{eqnarray}
 Inserting \eqref{ch4ThA1}, \eqref{ch4ThA2}, \eqref{ch4ThA3} and \eqref{ch4ThA6} in \eqref{ch4Th1} we obtain :
  \begin{eqnarray*}
  \left\|X_s-\overline{Y}^M_s\right\|^2&\leq &(2C+2+2mC^2+5\lambda+5\lambda C^2)\int_0^s||X_u-\overline{Y}^M_u||^2du\nonumber\\
  &+&(2mC^2+5\lambda C^2)\int_0^s||\overline{Y}^M_u-\overline{Y}^M_{\lfloor u\rfloor}||^2du\nonumber\\
  &+&\int_0^s||f_{\lambda}(\overline{Y}^M_u)-f_{\lambda}(\overline{Y}^M_{\lfloor u\rfloor})||^2du+\dfrac{T^2}{M^2}\int_0^s||f_{\lambda}(\overline{Y}^M_{\lfloor u\rfloor})||^4du\nonumber\\
  &+&2\sum_{i=1}^m\int_0^s\left<X_u-\overline{Y}^M_u, g_i(X_u)-g_i(\overline{Y}^M_{\lfloor u\rfloor})\right>dW^{i}_u\nonumber\\
  &+&\int_0^s\left[||X_u-\overline{Y}^M_u+h(X_u)-h(\overline{Y}^M_{\lfloor u\rfloor})||^2-||X_u-\overline{Y}^M_u||^2\right]d\overline{N}_u.
  \end{eqnarray*}
Taking the supremum in both sides of the previous inequality leads to
\begin{eqnarray}
\sup_{s\in[0,t]}\left\|X_s-\overline{Y}^M_s\right\|^2&\leq &(2C+2+2mC^2+5\lambda+5\lambda C^2)\int_0^t||X_u-\overline{Y}^M_u||^2du\nonumber\\
  &+&(2mC^2+5\lambda C^2)\int_0^t||\overline{Y}^M_u-\overline{Y}^M_{\lfloor u\rfloor}||^2du\nonumber\\
  &+&\int_0^t||f_{\lambda}(\overline{Y}^M_u)-f_{\lambda}(\overline{Y}^M_{\lfloor u\rfloor})||^2du+\dfrac{T^2}{M^2}\int_0^t||f_{\lambda}(\overline{Y}^M_{\lfloor u\rfloor})||^4du\nonumber\\
  &+&2\sup_{s\in[0,t]}\left|\sum_{i=1}^m\int_0^s\left<X_u-\overline{Y}^M_u, g_i(X_u)-g_i(\overline{Y}^M_{\lfloor u\rfloor})\right>dW^{i}_u\right|\nonumber\\
  &+&\sup_{s\in[0,t]}\left|\int_0^s\left[||X_u-\overline{Y}^M_u+h(X_u)-h(\overline{Y}^M_{\lfloor u\rfloor})||^2\right]d\overline{N}_u\right|\nonumber\\
  &+&\sup_{s\in[0,t]}\left|\int_0^s||X_u-\overline{Y}^M_u||^2d\overline{N}_u\right|
  \label{ch4Th2}
\end{eqnarray}
 Using Lemma \ref{ch4lemma15} we  have the following estimation for all $p\geq 2$
 \begin{eqnarray*}
 B_1& :=&\left\|2\sup_{s\in[0,t]}\left|\sum_{i=1}^m\int_0^s\left<X_u-\overline{Y}^M_u, g_i(X_u)-g_i(\overline{Y}^M_{\lfloor u\rfloor})\right>dW^i_u\right|\right\|_{L^{p/2}(\Omega, \mathbb{R})}\\
 &\leq &C_p\left(\int_0^t\sum_{i=1}^m\left\|\left<X_u-\overline{Y}^M_u, g_i(X_u)-g_i(\overline{Y}^M_{\lfloor u\rfloor})\right>\right\|^2_{L^{p/2}(\Omega, \mathbb{R})}ds\right)^{1/2}.
 \end{eqnarray*}
Moreover, using Holder inequality, the inequalities $ab\leq \dfrac{a^2}{2}+\dfrac{b^2}{2}$ and $(a+b)^2\leq 2a^2+b^2$ we  have the following estimations  for all $p\geq 2$
 \begin{eqnarray}
 B_1 &\leq&  C_p\left(\int_0^t\sum_{i=1}^m\left\|\left<X_u-\overline{Y}^M_u, g_i(X_u)-g_i(\overline{Y}^M_{\lfloor u\rfloor})\right>\right\|^2_{L^{p/2}(\Omega, \mathbb{R})}ds\right)^{1/2}\nonumber\\ 
 &\leq &C_p\left(\int_0^t\sum_{i=1}^m||X_u-\overline{Y}_u^M||^2_{L^p(\Omega, \mathbb{R})}||g_i(X_u)-g_u(\overline{Y}^M_{\lfloor u\rfloor})||^2_{L^p(\Omega, \mathbb{R}^d)}ds\right)^{1/2}\nonumber\\
 &\leq &\dfrac{C_p}{\sqrt{2}}\left(\sup_{s\in[0,t]}||X_s-\overline{Y}^M_s||_{L^p(\Omega, \mathbb{R}^d)}\right)\left(2pC^2m\int_0^t||X_s-\overline{Y}^M_{\lfloor s\rfloor}||^2_{L^p(\Omega, \mathbb{R}^d)}ds\right)^{1/2}\nonumber\\
 &\leq &\dfrac{1}{4}\sup_{s\in[0,t]}||X_s-\overline{Y}^M_s||^2_{L^p(\Omega, \mathbb{R}^d)}+p^2C_p^2m\int_0^t||X_s-\overline{Y}^M_{\lfloor s\rfloor}||^2_{L^p(\Omega,\mathbb{R}^d)}ds\nonumber\\
 &\leq &\dfrac{1}{4}\sup_{s\in[0,t]}||X_s-\overline{Y}^M_s||^2_{L^p(\Omega, \mathbb{R}^d)}+2p^2C_p^2m\int_0^t||X_s-\overline{Y}^M_s||^2_{L^p(\Omega,\mathbb{R}^d)}ds\nonumber\\
 &+&2p^2C_p^2m\int_0^t||\overline{Y}^M_s-\overline{Y}^M_{\lfloor u\rfloor}||^2_{L^p(\Omega,\mathbb{R}^d)}ds, 
 \label{ch4ThB1}
 \end{eqnarray}
 Using Lemma \ref{ch4lemma18} and the inequality $(a+b)^4\leq 16a^4+16b^4$, it follows that 
 \begin{eqnarray*}
 B_2 &: =&\left\|\sup_{s\in[0,t]}\left|\int_0^s||X_u-\overline{Y}^M_u+h(X_u)-h(\overline{Y}^M_{\lfloor u\rfloor})||^2d\overline{N}_u\right|\right\|_{L^{p/2}(\Omega, \mathbb{R}^d)}\nonumber\\
 &\leq &C_p\left(\int_0^t||X_u-\overline{Y}^M_u+h(X_u)-h(\overline{Y}^M_{\lfloor u\rfloor})||^4_{L^{p/2}(\Omega, \mathbb{R}^d)}du\right)^{1/2}\nonumber\\
 &\leq &C_p\left(\int_0^t16||X_u-\overline{Y}^M_u||^4_{L^{p/2}(\Omega, \mathbb{R}^d)}+16||h(X_u)-h(\overline{Y}^M_{\lfloor u\rfloor})||^4_{L^{p/2}(\Omega, \mathbb{R}^d)}du\right)^{1/2},
 \end{eqnarray*}
 for all $p\geq 2$.
 
 Using the inequality $\sqrt{a+b}\leq\sqrt{a}+\sqrt{b}$, it follows that :
 \begin{eqnarray}
 B_2&\leq &2C_p\left(\int_0^t||X_u-\overline{Y}^M_u||^4_{L^{p/2}(\Omega, \mathbb{R}^d)}du\right)^{1/2} +2C_p\left(\int_0^t||h(X_u)-h(\overline{Y}^M_{\lfloor u\rfloor})||^4_{L^{p/2}(\Omega, \mathbb{R}^d)}du\right)^{1/2}\nonumber\\
 & =& B_{21}+B_{22}.
 \label{ch4ThB}
 \end{eqnarray}
 Using Holder inequality, it follows that 
 \begin{eqnarray*}
 B_{21} &: = &2 C_p\left(\int_0^t||X_u-\overline{Y}^M_u||^4_{L^{p/2}(\Omega, \mathbb{R}^d)}du\right)^{1/2}\nonumber\\
 &\leq &2C_p\left(\int_0^t||X_u-\overline{Y}^M_u||^2_{L^p(\Omega, \mathbb{R}^d)}||X_u-\overline{Y}^M_u||^2_{L^p(\Omega, \mathbb{R}^d)}du\right)^{1/2}\nonumber\\
 &\leq &\dfrac{1}{4}\sup_{u\in[0,t]}||X_u-\overline{Y}^M_u||_{L^p(\Omega,\mathbb{R}^d)}8C_p\left(\int_0^t||X_u-\overline{Y}^M_u||^2_{L^p(\Omega,\mathbb{R}^d)}du\right)^{1/2}.
 \end{eqnarray*}
 Using the inequality $2ab\leq a^2+b^2$ leads to :
 \begin{eqnarray}
 B_{21}&\leq &\dfrac{1}{16}\sup_{u\in[0,t]}||X_u-\overline{Y}^M_u||^2_{L^p(\Omega, \mathbb{R}^d)}
 +16C^2_p\int_0^t||X_u-\overline{Y}^M_u||^2_{L^p(\Omega, \mathbb{R}^d)}du.
 \label{ch4ThB21}
 \end{eqnarray}
 Using the inequalities $(a+b)^4\leq 4a^4+4b^4$ and $\sqrt{a+b}\leq\sqrt{a}+\sqrt{b}$,  we obtain 
 \begin{eqnarray*}
 B_{22}&: = &2C_p\left(\int_0^t||h(X_u)-h(\overline{Y}^M_{\lfloor u\rfloor})||^4_{L^{p/2}(\Omega, \mathbb{R}^d)}du\right)^{1/2}\nonumber\\
 &\leq &2C_p\left(\int_0^t\left[4||h(X_u)-h(\overline{Y}^M_u)||^4_{L^{p/2}(\Omega,\mathbb{R}^d)}+4||h(\overline{Y}^M_u)-h(\overline{Y}^M_{\lfloor u\rfloor})||^4_{L^{p/2}(\Omega,\mathbb{R}^d)}\right]du\right)^{1/2}\\
 &\leq &4C_p\left(\int_0^t||h(X_u)-h(\overline{Y}^M_u)||^4_{L^{p/2}(\Omega, \mathbb{R}^d)}du\right)^{1/2}
 +4C_p\left(\int_0^t||h(\overline{Y}^M_u)-h(\overline{Y}^M_{\lfloor u\rfloor})||^4_{L^{p/2}(\Omega, \mathbb{R}^d)}du\right)^{1/2}.
 \end{eqnarray*}
 Using the global Lipschitz condition, leads to :
 \begin{eqnarray*}
 B_{22}&\leq &4C_p\left(\int_0^tC||X_u-\overline{Y}^M_u||^4_{L^{p/2}(\Omega, \mathbb{R}^d)}du\right)^{1/2}+4C_p\left(\int_0^tC||\overline{Y}^M_u-\overline{Y}^M_{\lfloor u\rfloor}||^4_{L^{p/2}(\Omega, \mathbb{R}^d)}du\right)^{1/2}.
 \end{eqnarray*}
 Using the same estimations as for $B_{21}$, it follows that :
 \begin{eqnarray}
 B_{22} &\leq &\dfrac{1}{16}\sup_{u\in[0,t]}||X_u-\overline{Y}^M_u||^2_{L^p(\Omega, \mathbb{R}^d)} +64C_p\int_0^t||X_u-\overline{Y}^M_u||^2_{L^p(\Omega, \mathbb{R}^d)}du\nonumber\\
 &+&\dfrac{1}{4}\sup_{u\in[0,t]}||\overline{Y}^M_u-\overline{Y}^M_{\lfloor u\rfloor}||^2_{L^p(\Omega, \mathbb{R}^d)}+64C_p\int_0^t||\overline{Y}^M_u-\overline{Y}^M_{\lfloor u\rfloor}||^2_{L^p(\Omega, \mathbb{R}^d)}du.\nonumber
 \end{eqnarray}
 Taking the supremum under the integrand in the last term of the above inequality and using the fact that we don't care about the value of the constant  leads to :
 \begin{eqnarray}
 B_{22}&\leq &\dfrac{1}{16}\sup_{u\in[0,t]}||X_u-\overline{Y}^M_u||^2_{L^p(\Omega, \mathbb{R}^d)} +64C_p\int_0^t||X_u-\overline{Y}^M_u||^2_{L^p(\Omega, \mathbb{R}^d)}du\nonumber\\
 &+&C_p\sup_{s\in[0,t]}||\overline{Y}^M_s-\overline{Y}^M_{\lfloor s\rfloor}||^2_{L^p(\Omega, \mathbb{R}^d)}.
 \label{ch4ThB22}
 \end{eqnarray}
 Inserting \eqref{ch4ThB21} and \eqref{ch4ThB22} into \eqref{ch4ThB} gives :
 \begin{eqnarray}
 B_2&\leq &\dfrac{1}{8}\sup_{u\in[0,t]}||X_u-\overline{Y}^M_u||^2_{L^p(\Omega, \mathbb{R}^d)}+C_p\int_0^t||X_u-\overline{Y}^M_u||^2_{L^p(\Omega, \mathbb{R}^d)}du\nonumber\\
 &+&C_p\sup_{s\in[0,t]}||\overline{Y}^M_s-\overline{Y}^M_{\lfloor s\rfloor}||^2_{L^p(\Omega, \mathbb{R}^d)}.
 \label{ch4ThB2}
 \end{eqnarray}
 Using again Lemma \ref{ch4lemma18} leads to :
 \begin{eqnarray*}
B_3 &: =& \left\|\sup_{u\in[0,t]}\left(\int_0^s||X_u-\overline{Y} ^M_u||^2d\overline{N}_u\right)^{1/2}\right\|_{L^{p/2}(\Omega, \mathbb{R}^d)}\\
&\leq &C_p\left(\int_0^t||X_u-\overline{Y}^M_u||^4_{L^{p/2}(\Omega, \mathbb{R}^d)}du\right)^{1/2}.
 \end{eqnarray*}
 Using the same argument as for $B_{21}$, we obtain :
 \begin{eqnarray}
 B_3 &\leq &\dfrac{1}{8}\sup_{u\in[0,t]}||X_u-\overline{Y}^M_u||^2_{L^p(\Omega, \mathbb{R}^d)}+C_p\int_0^t||X_u-\overline{Y}^M_u||^2_{L^p(\Omega, \mathbb{R}^d)}du.
 \label{ch4ThB3}
 \end{eqnarray}
 Taking the $L^{p}$ norm in both  side of \eqref{ch4Th2}, inserting inequalities \eqref{ch4ThB1}, \eqref{ch4ThB2}, \eqref{ch4ThB3} and using Holder inequality in its integral form (see Proposition \ref{ch1Minkowski}) leads to :
\begin{eqnarray*}
\left\|\sup_{s\in[0,t]}||X_s-\overline{Y}^M_s||\right\|^2_{L^p(\Omega, \mathbb{R}^d)}&=&\left\|\sup_{s\in[0,t]}||X_s-\overline{Y}^M_s||^2\right\|_{L^{p/2}(\Omega, \mathbb{R}^d)}
\end{eqnarray*}
  \begin{eqnarray*}
  &\leq & C_p\int_0^t||X_s-\overline{Y}^M_s||^2_{L^p(\Omega, \mathbb{R}^d)}ds+C_p\int_0^t||\overline{Y}^M_s-\overline{Y}^M_{\lfloor s\rfloor}||^2_{L^p(\Omega,\mathbb{R}^d)}ds\\
 &+&\int_0^t||f_{\lambda}(X_s)-f_{\lambda}(\overline{Y}^M_{\lfloor s\rfloor})||^2_{L^{p}(\Omega, \mathbb{R}^d)}ds+C_p\sup_{u\in[0,t]}||\overline{Y}^M_u-\overline{Y}^M_{\lfloor u\rfloor}||^2_{L^p(\Omega, \mathbb{R}^d)}\\
 &+&\dfrac{T^2}{M^2}\int_0^t||f_{\lambda}(\overline{Y}^M_{\lfloor s\rfloor})||^4_{L^{2p}(\Omega, \mathbb{R}^d)}ds +2C_p\int_0^t||\overline{Y}^M_s-\overline{Y}^M_{\lfloor s\rfloor}||^2_{L^p(\Omega, \mathbb{R}^d)}ds\\
 &+&\dfrac{1}{2}\left\|\sup_{s\in[0,t]}||X_s-\overline{Y}^M_s||\right\|^2_{L^{p}(\Omega, \mathbb{R}^d)}.
 \end{eqnarray*} 
 for all $t\in[0,T]$ and all $p\in[2,+\infty)$.
 
The previous inequality can be rewrite in the following appropriate form : 

$
\dfrac{1}{2}\left\|\sup\limits_{s\in[0,t]}\left\|X_s-\overline{Y}^M_s\right\|\right\|^2_{L^{p}(\Omega, \mathbb{R}^d)}
$
  \begin{eqnarray*}
  &\leq & C_p\int_0^t||X_s-\overline{Y}^M_s||^2_{L^p(\Omega, \mathbb{R}^d)}ds+C_p\int_0^t||\overline{Y}^M_s-\overline{Y}^M_{\lfloor s\rfloor}||^2_{L^p(\Omega,\mathbb{R}^d)}ds\\
 &+&\int_0^t||f_{\lambda}(X_s)-f_{\lambda}(\overline{Y}^M_{\lfloor s\rfloor})||^2_{L^{p}(\Omega, \mathbb{R}^d)}ds+C_p\sup_{u\in[0,t]}||\overline{Y}^M_u-\overline{Y}^M_{\lfloor u\rfloor}||^2_{L^p(\Omega, \mathbb{R}^d)}\\
 &+&\dfrac{T^2}{M^2}\int_0^t||f_{\lambda}(\overline{Y}^M_{\lfloor s\rfloor})||^4_{L^{2p}(\Omega, \mathbb{R}^d)}ds +2C^2m\int_0^t||\overline{Y}^M_s-\overline{Y}^M_{\lfloor s\rfloor}||^2_{L^p(\Omega, \mathbb{R}^d)}ds.
 \end{eqnarray*}

 Applying Gronwall lemma to the previous inequality leads to :

$
 \dfrac{1}{2}\left\|\sup\limits_{s\in[0,t]}||X_s-\overline{Y}^M_s||\right\|^2_{L^p(\Omega, \mathbb{R}^d)} 
 $
 \begin{eqnarray*}
 &\leq &C_pe^{C_p}\left(\int_0^T||f_{\lambda}(\overline{Y}^M_s)-f_{\lambda}(\overline{Y}^M_{\lfloor s\rfloor})||^2_{L^p(\Omega, \mathbb{R}^d)}ds+C_p\sup_{t\in[0,t]}||\overline{Y}^M_u-\overline{Y}^M_{\lfloor u\rfloor}||^2_{L^p(\Omega, \mathbb{R}^d)}\right.\\
 & &\left. +\dfrac{T^2}{M^2}\int_0^T||f_{\lambda}(\overline{Y}^M_{\lfloor s\rfloor})||^4_{L^{2p}(\Omega, \mathbb{R}^d)}ds +C_p\int_0^T||\overline{Y}^M_s-\overline{Y}^M_{\lfloor s\rfloor}||^2_{L^p(\Omega, \mathbb{R}^d)}ds\right).
 \end{eqnarray*}
 From the inequality $\sqrt{a+b+c}\leq \sqrt{a}+\sqrt{b}+\sqrt{c}$, it follows that
 
 $
 \left\|\sup\limits_{t\in[0, T]}||X_t-\overline{Y}^M_t||\right\|_{L^p(\Omega, \mathbb{R}^d)}
 $
 \begin{eqnarray}
 &\leq &C_pe^{C_p}\left(\sup_{t\in[0,T]}||f_{\lambda}(\overline{Y}^M_t)-f_{\lambda}(\overline{Y}^M_{\lfloor t\rfloor})||_{L^p(\Omega, \mathbb{R}^d)}+C_p\sup_{t\in[0,t]}||\overline{Y}^M_t-\overline{Y}^M_{\lfloor t\rfloor}||_{L^p(\Omega, \mathbb{R}^d)} \right.\nonumber\\
 &+& \left. \dfrac{T}{M}\left[\sup_{n\in\{0,\cdots, M\}}||f_{\lambda}(Y^M_n)||^2_{L^p(\Omega, \mathbb{R}^d)}\right]+C_p\sup_{t\in[0,T]}||\overline{Y}^M_t-\overline{Y}^M_{\lfloor t\rfloor}||_{L^p(\Omega, \mathbb{R}^d)}\right),
 \label{ch4Gronwall}
 \end{eqnarray}
 for all $p\in[2,\infty)$.
 
 Using Lemma \ref{ch4lemma21} and Lemma \ref{ch4lemma22} it follows from \ref{ch4Gronwall} that 
 \begin{eqnarray}
 \left(\mathbb{E}\left[\sup_{t\in[0,T]}\left\|X_t-\overline{Y}^M_t\right\|^p\right]\right)^{1/p}\leq C_p(\Delta t)^{1/2},
 \label{ch4final}
 \end{eqnarray}
 for all $p\in[2,\infty)$. Using Holder inequality, one can prove that \eqref{ch4final} holds for all  $p\in[1,2]$. The proof of the theorem is complete.
 \end{proof}

 \section{Numerical Experiments}
In order to illustrate our theoretical result, we consider the following stochastic differential equation

$\left\{\begin{array}{ll}
dX_t=-X_t^4dt+X_tdW_t+X_tdN_t\\
X_0=1,
\end{array}
\right.$

$\lambda=1$. It is straighforward to verify that Assumptions \ref{ch4assumption1} are satisfied. We use Monte carlo method to evaluate the error. The exact solution is consider as the numerical one with small stepsize $dt=2^{-14}$. We have the following curve for $5000$ paths.

\begin{figure}[h]
\includegraphics[scale=0.7]{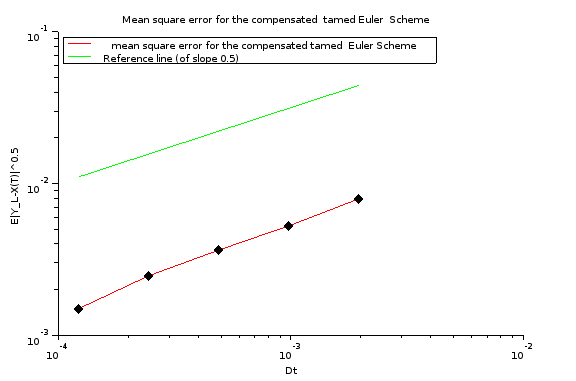}
\caption{Strong error of the compensated tamed Euler scheme}
\end{figure}

\hspace{0.5cm}In this chapter, we proposed a compensated tamed Euler scheme to solve numerically SDEs with jumps under non-global Lipschitz condition. We proved the strong convergence of order $0.5$ of the compensated Euler scheme. This scheme is explicit and then requires less computational efforts than the implicit scheme. In some situations, the drift part can be equipped with the Lipschitz continuous part and the non-Lipschitz continuous part. In the following chapter, we combine the tamed Euler scheme and the Euler scheme and obtain another scheme called semi-tamed Euler scheme in order to solve numerically  the kind of equation mentioned above.

\chapter{Strong convergence and stability of the  semi-
tamed Euler and the tamed Euler scheme  for stochastic differential
 equations with jumps, under non-global Lipschitz continous coefficients}
 
\hspace{0.5cm}Explicit numerical method called compensated tamed Euler scheme is developped in  the previous chapter. More precisely, it is   proved that such numerical approximation have strong convergence of order $0.5$ for stochastic differential equations with jumps under non-global Lipschitz condition. In this chapter, following the idea of \cite{Xia2}, we propose a  semi-tamed Euler  scheme to solve stochastic differential equations with jumps, where the drift coefficient is equipped with the
Lipschitz continuous part and  the non-Lipschitz continuous part. We prove that for SDEs with jumps, the  semi-tamed Euler scheme converges strongly with order $0.5$. We use this result to deduce a strong convegrence of order $0.5$ of the tamed Euler scheme for SDEs with jumps, where the drift coefficient satisfies the  non-global Lipschitz condition. We also investigate the stability analysis of both semi-tamed Euler scheme and tamed Euler scheme. The contents of this chapter can also be found in \cite{atjdm1} and \cite{atjdm2}.

\section{Semi-tamed Euler scheme}
\label{ch5intro}

\hspace{0.5cm}In this chapter, we consider again a jump-diffusion It\^{o}'s stochastic differential equations (SDEs) of the form :
\begin{eqnarray}
 dX(t)=  f((X(t^{-}))dt +g((X(t^{-}))dW(t)+h(X(t^{-}))dN(t),  \hspace{0.5cm}
 X(0)=X_0,
 \label{ch5exactsol1}
\end{eqnarray}
where $W_t$ is a $m$-dimensional Brownian motion, $f :\mathbb{R}^d\longrightarrow\mathbb{R}^d$ such that $f(x)=u(x)+v(x)$ satisfies the  global one-sided Lipschitz condition. $u, v : \mathbb{R}^d\longrightarrow\mathbb{R}^d$, $u$ is the global Lipschitz continous part while $v$ is the non-global Lipschitz continous part. The functions  $ g : \mathbb{R}^d \longrightarrow\mathbb{R}^{d\times m}$ and $h :\mathbb{R}^d \longrightarrow\mathbb{R}^d$ satisfy the global Lipschitz condition, $N_t$ is a one dimensional poisson process with parameter $\lambda$. Using the relation $f=u+v$, equation \eqref{ch5exactsol1} can be rewritten into its equivalent form :
\begin{eqnarray}
dX(t)=u(X(t^-))dt+v(X(t^-))dt+g(X(t^-))dW(t)+h(X(t^-))dN(t), \hspace{0.2cm} X(0)=X_0.
\end{eqnarray}

We can rewrite the jump-diffusion SDEs \eqref{ch5exactsol1} in the following equivalent form
\begin{eqnarray}
 dX(t)=  f_\lambda(X(t^{-}))dt +g(X(t^{-}))dW(t)+h(X(t^{-}))d\overline{N}(t),
\end{eqnarray}
where
\begin{equation}
 f_\lambda(x)=f(x)+\lambda h(x)= u(x)+\lambda h(x)+v(x).
 \label{ch5flambda}
\end{equation}
If  $T$  is the final time   we consider  the  tamed Euler scheme
\begin{equation}
 X_{n+1}^{M}=X_{n}^{M}+\dfrac{\Delta t f(X_{n}^{M})}{1+ \Delta t\Vert f(X_{n}^{N}) \Vert }+g(X_{n}^{M}) \Delta W_n +h(X_{n}^{M})\Delta N_n 
 \label{ch5tam}
\end{equation}
and the semi-tamed Euler scheme
\begin{eqnarray}
 Y_{n+1}^{M}=Y_{n}^{M}+u(Y^M_n)\Delta t+\dfrac{\Delta t v(Y_{n}^{M})}{1+ \Delta t \Vert v(Y_{n}^{M}) \Vert }+g(Y_{n}^{M}) \Delta W_n +h(Y_{n}^{M})\Delta N_n,
 \label{ch5semi}
\end{eqnarray}
where $\Delta t=\dfrac{T}{M}$ is the time step-size,  $M\in\mathbb{N}$ is the  number of steps.

Inspired by  \cite{Xia2} and \cite{Martin1} we prove the strong convergence of the numerical approximation \eqref{ch5semitam} and deduce the strong convergence of \eqref{ch5tam} to the exact solution of \eqref{ch5exactsol1}.

\section{Moments bounded of the numerical solution }

\hspace{0.5cm}Throughout this chapter, we use Notations \ref{ch4nota1}.

\begin{rem}
\hspace{0.5cm}Note that the numerical approximation \eqref{ch5semi} can be writen into its following equivalent form 
\begin{eqnarray}
 Y_{n+1}^{M}=Y_{n}^{M}+u(Y^M_n)\Delta t+\lambda h(Y^M_n)\Delta t+\dfrac{\Delta t v(Y_{n}^{M})}{1+ \Delta t \Vert v(Y_{n}^{M}) \Vert }+g(Y_{n}^{M}) \Delta W_n +h(Y_{n}^{M})\Delta \overline{N}_n.
 \label{ch5semitam}
\end{eqnarray}
We define the continous time interpolation of the discrete numerical approximation of $\eqref{ch5semitam}$ by the familly of processes $\left(\overline{Y}^M\right)_M $,
$
\overline{Y}^M : [0,T]\times\Omega \longrightarrow \mathbb{R}^d
$
  such that :
\begin{eqnarray}
\overline{Y}^M_t &=&Y^M_n+u(Y^M_n)(t-n\Delta t)+\lambda h(Y^M_n)(t-n\Delta t)+\dfrac{(t-n\Delta t)v(Y^M_n)}{1+\Delta t||v(Y^M_n)||}\nonumber\\
&+&g(Y^M_n)(W_t-W_{n\Delta t})+ h(Y^M_n)(\overline{N}_t-\overline{N}_{n\Delta t}),
\label{ch5continoussolu}
\end{eqnarray}
for all $M\in\mathbb{N}$, all $n\in\{0,\cdots, M-1\}$, and all $t\in[n\Delta t,  (n+1)\Delta t)$.
\end{rem}

\begin{assumption}\label{ch5assumption1} We assume that : 

$(A.1)$ $f,g,h\in C^1$.

$(A.2)$ For all $p>0$, there exist a finite $M_p>0$ such that $\mathbb{E}||X_0||^p\leq M_p$.

$(A.3)$ $g$, $h$  and $u$ satisfy  global Lipschitz condition:
\begin{eqnarray}
||g(x)-g(y)||\vee ||h(x)-h(y)||\vee ||u(x)-u(y)|| \leq C||x-y||, \hspace{0.5cm} \forall\; x,y\in \mathbb{R}^d.
\end{eqnarray}
$(A.4)$ $v$ satisfies the  one-sided Lipschitz condition :
\begin{eqnarray*}
\langle x-y, f(x)-f(y)\rangle\leq C||x-y||^2,\hspace{0.5cm} \forall\;x,y\in \mathbb{R}^d,
\end{eqnarray*}
$(A.5)$ $v$ satisfies the superlinear growth condition :
\begin{eqnarray*}
||v(x)-v(y)||\leq C(K+ ||x||^c+||y||^c)||x-y||, \hspace{0.5cm} \forall\;x,y\in \mathbb{R}^d,
\end{eqnarray*}
where $K$, $C$ and $c$ are constants strictetly positives. 
\end{assumption}

\begin{rem}
Under conditions $(A.1)$, $(A.2)$ and $(A.3)$ of Assumptions \ref{ch5assumption1} it is proved in  \cite[Lemma 1]{Desmond2} that \eqref{ch5exactsol1} has a unique solution with all moments bounded.
\end{rem}

\begin{rem} Let's define $u_{\lambda}(x)=u(x)+\lambda h(x)$. From Assumptions \ref{ch5assumption1},  it is straightforward to prove that $u_{\lambda}$ satisfies the global Lipschitz condition with constant $C_{\lambda}=(1+\lambda)C$ and $v$ satisfies the one-sided Lipschitz condition.
We denote by $C_p$ a generic constant. Throughout this work, this constant may change the value from one line to another one. We will sometimes use $Y_n^M$ instead of $Y_n^M(\omega)$ to simplify  notations.
\end{rem}

\hspace{0.5cm}The main result of this section is formulated in the following theorem.  
   which is based on  \cite[Lemma 3.9 pp 16]{Martin1}. Here, we include the jump part.
   \begin{thm}\label{ch5theorem1}
   
   Let $Y_n^M : \Omega\longrightarrow \mathbb{R}^d$ be defined by \eqref{ch5semitam}  for all $M\in\mathbb{N}$ and all $ n\in\{0,\cdots, M\}$. The following inequality holds :
   \begin{eqnarray*}
   \sup_{M\in\mathbb{N}}\sup_{n\in\{0,\cdots, M\}}\mathbb{E}\left[||Y_n^M||^p\right]<+\infty,
   \end{eqnarray*}
   for all $p\in[1,\infty)$.
   \end{thm}
 
 In order to prove  Theorem \ref{ch5theorem1} we introduce the following notations  facilitating computations.

 \begin{nota}\label{ch5notation1}
 
 \begin{eqnarray*}
 \alpha^M_k := \mathrm{1}_{\{||Y^M_k||\geq 1\}}\left\langle\dfrac{Y^M_k+u_{\lambda}(Y^M_n)\Delta t}{||Y^M_k||}, \dfrac{g(Y^M_k)}{||Y^M_k||}\Delta W^M_k\right\rangle,\\\\
 \beta^M_k := \mathrm{1}_{\{||Y^M_k||\geq 1\}}\left\langle\dfrac{Y^M_k+u_{\lambda}(Y^M_n)\Delta t}{||Y^M_k||}, \dfrac{h(Y^M_k)}{||Y^M_k||}\Delta\overline{N}^M_k\right\rangle,
 \end{eqnarray*}
 \begin{eqnarray*}
 \beta=\left(1+K+3C_{\lambda}+KTC_{\lambda}+TC_{\lambda}+||u_{\lambda}(0)||+||g(0)||+||h(0)||\right)^4,\\\\
 D^M_n := (\beta+||\varepsilon||)\exp\left(4\beta+\sup_{u\in\{0,\cdots,n\}}\sum_{k=u}^{n-1}\left[2\beta||\Delta W^M_k||^2+2\beta||\Delta\overline{N}^M_k||+\alpha^M_k+\beta^M_k\right]\right),\\
 \Omega^M_n :=\{\omega\in \Omega : \sup_{k\in\{0,1,\cdots, n-1\}}D^M_k(\omega)\leq M^{1/2c}, \sup_{k\in\{0,1,\cdots,n-1\}}||\Delta W^M_k(\omega)||\leq 1,\\ \sup_{k\in\{0,1,\cdots,n-1\}}||\Delta \overline{N}^M_k(\omega)||\leq 1\},
 \end{eqnarray*}
 for all $M\in\mathbb{N}$ and $k\in\{0,\cdots M\}$.
 \end{nota}

 Following closely  \cite[Lemma 2.1]{Xia2} we have  the following main lemma.
 \begin{lem}\label{ch5lemma2}
 The following inequality holds for all $M\in \mathbb{N}$ 
 and all $n\in\{0,1,\cdots, M\}$
 \begin{eqnarray}
 \mathbf{1}_{\Omega^M_n}||Y^M_n||\leq D^M_n.
 \label{ch5Denobound}
 \end{eqnarray}
 \end{lem}
 \begin{proof}
 Using the inequality $\dfrac{\Delta t}{1+\Delta t||u_{\lambda}(x)||}\leq T$, the global Lipschitz condition of $g$ and $h$ and the polynomial growth condition of $v$ we have the following estimation   on $\Omega^M_{n+1}\cap\{\omega\in \Omega : ||Y^M_n(\omega)||\leq 1\}$,  for all $n\in\{0,1,\cdots, M-1\}$
 \begin{align}
 ||Y^M_{n+1}||&\leq ||Y^M_n||+||u_{\lambda}(Y^M_n)||\Delta t+\dfrac{\Delta t||v(Y^M_n)||}{1+\Delta t||v(Y^M_n)||}+||g(Y^M_n)||||\Delta W^M_n||+||h(Y^M_n)||||\Delta\overline{N} ^M_n|| \nonumber \\
 &\leq ||Y^M_n||+T||u_{\lambda}(Y^M_n)-u_{\lambda}(0)||+T||u_{\lambda}(0)||+T||v(Y^M_n)-v(0)||+T||v(0)||\nonumber\\
 &+ ||g(Y^M_n)-g(0)||+||g(0)||+||h(Y^M_n)-h(0)||+||h(0)||\nonumber\\
 &\leq 1+C||Y^M_n||+T||u_{\lambda}(0)||+TC(K+||Y^M_n||^c+||0||^c)||Y^M_n-0||+T||v(0)||\nonumber\\
 &+C||Y^M_n||+C||Y^M_n||+||g(0)||+||h(0)||\nonumber\\
 &\leq  1+KTC +TC+3C++T||u_{\lambda}(0)||+T||v(0)||+||g(0)||+||h(0)||\leq \beta.
 \label{ch5normY}
 \end{align}
  Futhermore, from the numerical approximation \eqref{ch5semitam}, we have 
 \begin{eqnarray} \label{ch5partnorm2}
 ||Y^M_{n+1}||^2&=&||Y^M_n||^2+||u_{\lambda}(Y^M_n)||^2\Delta t^2+\dfrac{\Delta t^2||v(Y^M_n)||^2}{(1+\Delta t||f_{\lambda}(Y^M_n)||)^2}+||g(Y^M_n)\Delta W^M_n||^2\nonumber\\
 &+& ||h(Y^M_n)\Delta\overline{N}^M_n||^2+ 2\Delta t\langle Y^M_n, u_{\lambda}(Y^M_n)\rangle+\dfrac{2\Delta t\langle Y^M_n, v(Y^M_n)\rangle}{1+\Delta t||f_{\lambda}(Y^M_n)||}\nonumber\\
 &+&2\langle Y^M_n,g(Y^M_n)\Delta W^M_n\rangle+2\langle Y^M_n,h(Y^M_n)\Delta\overline{N}^M_n\rangle+\dfrac{2\Delta t^2\langle u_{\lambda}(Y^M_n), v(Y^M_n)\rangle}{1+\Delta t||v(Y^M_n)||}\nonumber\\
 &+&2\Delta t\langle u_{\lambda}(Y^M_n), g(Y^M_n)\Delta W_n\rangle+2\Delta t\langle u_{\lambda}(Y^M_n), h(Y^M_n)\Delta\overline{N}^M_n\rangle\nonumber\\
 &+&\dfrac{2\langle \Delta tv(Y^M_n),g(Y^M_n)\Delta W^M_n\rangle}{1+\Delta t||v(Y^M_n)||}+\dfrac{2\langle \Delta tv(Y^M_n),h(Y^M_n)\Delta\overline{N}^M_n\rangle}{1+\Delta t||v(Y^M_n)||}\nonumber\\
 &+&2\langle g(Y^M_n)\Delta W^M_n,h(Y^M_n)\Delta\overline{N}^M_n\rangle.
 \end{eqnarray}
 Using  the estimations $a\leq |a|$ and  $
 \dfrac{1}{1+\Delta t||v(Y^M_n)||}\leq 1$, we obtain the following inequality from equation \eqref{ch5partnorm2} :
 \begin{eqnarray}
||Y^M_{n+1}||^2 &\leq& ||Y^M_n||^2+||u_{\lambda}(Y^M_n)||\Delta t^2+\Delta t^2||v(Y^M_n)||^2+||g(Y^M_n)||^2||\Delta W^M_n||^2\nonumber\\
 &+&||h(Y^M_n)||^2||\Delta\overline{N}^M_n||^2+2\Delta t\langle Y^M_n, u_{\lambda}(Y^M_n)\rangle+2\Delta t|\langle Y^M_n, v(Y^M_n)\rangle|\nonumber\\
 &+&2\langle Y^M_n, g(Y^M_n)\Delta W^M_n\rangle+2\langle Y^M_n, h(Y^M_n)\Delta\overline{N}^M_n\rangle+2\Delta t^2|\langle u_{\lambda}(Y^M_n), v(Y^M_n)\rangle|\nonumber\\
 &+&2\langle \Delta tu_{\lambda}(Y^M_n), g(Y^M_n)\Delta W_n^M\rangle+2\langle \Delta tu_{\lambda}(Y^M_n), h(Y^M_n)\Delta\overline{N}^M_n\rangle\nonumber\\
 &+&2\Delta t|\langle v(Y^M_n), g(Y^M_n)\Delta W^M_n\rangle|+2\Delta t|\langle v(Y^M_n), h(Y^M_n)\Delta\overline{N}^M_n\rangle|\nonumber\\
 &+&2\langle g(Y^M_n)\Delta W^M_n, h(Y^M_n)\Delta\overline{N}^M_n\rangle.
 \label{ch5ine14}
 \end{eqnarray}
 Using the estimation $2ab\leq a^2+b^2$, the inequality \eqref{ch5ine14} becomes :
 \begin{eqnarray}
||Y^M_{n+1}||^2 &\leq & ||Y^M_n||^2+||u_{\lambda}(Y^M_n)||^2\Delta t^2+||v(Y^M_n)||^2\Delta t^2+||g(Y^M_n)||^2||\Delta W^M_n||^2\nonumber\\
&+&||h(Y^M_n)||^2||\Delta\overline{N}^M_n||^2+2\Delta t\langle Y^M_n, u_{\lambda}(Y^M_n)\rangle+2\Delta t|\langle Y^M_n, v(Y^M_n)\rangle|\nonumber\\
&+&2\langle Y^M_n+\Delta tu_{\lambda}(Y^M_n), g(Y^M_n)\Delta W^M_n\rangle+2\rangle Y^M_n+\Delta tu_{\lambda}(Y^M_n), h(Y^M_n)\Delta\overline{N}^M_n\rangle\nonumber\\
&+&||u_{\lambda}(Y^M_n)||^2\Delta t^2+||u_{\lambda}(Y^M_n)||^2\Delta t^2+||v(Y^M_n)||^2\Delta t^2+||h(Y^M_n)||^2||\Delta\overline{N}^M_n||^2\nonumber\\
&+&||v(Y^M_n)||^2\Delta t^2+||g(Y^M_n)||^2||\Delta W^M_n||^2+||v(Y^M_n)||^2\Delta t^2+||h(Y^M_n)||^2||\Delta\overline{N}^M_n||^2\nonumber\\
&+&||g(Y^M_n)||^2||\Delta W^M_n||^2+||h(Y^M_n)||^2||\Delta\overline{N}^M_n||^2
 \label{ch5ine15}
 \end{eqnarray}
 Putting similars terms of  inequality  \eqref{ch5ine15} together, we obtain : 
 \begin{eqnarray}
||Y^M_{n+1}||^2 &\leq & ||Y^M_n||^2+3||u_{\lambda}(Y^M_n)||^2\Delta t^2+4||v(Y^M_n)||^2\Delta t^2+3||g(Y^M_n)||^2||\Delta W^M_n||^2\nonumber\\
&+&4||h(Y^M_n)||^2||\Delta\overline{N}^M_n||^2+2\Delta t\langle Y^M_n, u_{\lambda}(Y^M_n)\rangle+2\Delta t|\langle Y^M_n, v(Y^M_n)\rangle|\nonumber\\
&+&2\langle Y^M_n+\Delta tu_{\lambda}(Y^M_n), g(Y^M_n)\Delta W^M_n\rangle\nonumber\\
&+&2\langle Y^M_n+\Delta tu_{\lambda}(Y^M_n), h(Y^M_n)\Delta\overline{N}^M_n\rangle,
 \label{ch5ine16}
 \end{eqnarray}
 on $\Omega$, for all $M\in\mathbb{N}$ and all $n\in\{0,1,\cdots, M-1\}$.
 
 In addition, for all $x\in\mathbb{R}^d$ such that $||x||\geq 1$, the global Lipschitz condition satisfied by $g$, $h$ and $u_{\lambda}$ leads to :
 \begin{eqnarray}
 ||g(x)||^2&\leq& (||g(x)-g(0)||+||g(0)||)^2\nonumber\\
 &\leq & (C||x||+||g(0)||)^2\nonumber\\
 &\leq & (C+||g(0)||)^2||x||^2\nonumber\\
 &\leq &\beta||x||^2
 \label{ch5normdeg2}
 \end{eqnarray}
 Along the same lines as above, for all $x\in\mathbb{R}^d$ such that $||x||\geq 1$, we have :
 \begin{eqnarray}
 ||h(x)||^2\leq \beta||x||^2 \hspace{0.3cm} \text{and} \hspace{0.3cm} ||u_{\lambda}(x)||\leq \beta ||x||.
 \label{ch5normdeh2}
 \end{eqnarray}
  Also, for all $x\in\mathbb{R}^d$ such that $||x||\geq 1$, the one-sided Lipschitz condition satisfied by $v$ leads to :
 \begin{eqnarray}
 \langle x,v(x)\rangle&=&\langle x,v(x)-v(0)+v(0)\rangle=\langle x-0,v(x)-v(0)\rangle +\langle x,v(0)\rangle\nonumber\\
 &\leq & C||x||^2+||x||||v(0)||\nonumber\\
 &\leq &(C+||v(0)||)||x||^2\nonumber\\
 &\leq &\sqrt{\beta}||x||^2.
 \label{ch5crochetv}
 \end{eqnarray}
 Along the same lines as above, $\forall x\in\mathbb{R}^d$ such that $||x||\geq 1$,  we have 
 \begin{eqnarray}
 \langle x, u(x)\rangle\leq \sqrt{\beta}||x||^2.
 \label{ch5crochetu}
 \end{eqnarray}
  Futhermore, using the polynomial growth condition satisfied by $v$, the following inequality holds  for all $x\in\mathbb{R}^d$ with $1\leq ||x||\leq M^{1/2c}$ and for all $M\in\mathbb{N}$
 \begin{eqnarray}
 ||v(x)||^2&\leq&(||v(x)-v(0)||+||v(0)||)^2\nonumber\\
 &\leq & (C(K+||x||^c)||x||+||v(0)||)^2\nonumber\\
 &\leq & (C(K+1)||x||^{c+1}+||v(0)||)^2\nonumber\\
 &\leq & (KC+C+||v(0)||)^2||x||^{2(c+1)}\nonumber\\
 &\leq & M\sqrt{\beta}||x||^2.
 \label{ch5normv2}
 \end{eqnarray}
 Using the global-Lipschitz condition of $u_{\lambda}$ leads to 
 \begin{eqnarray}
||u_{\lambda}(x)||^2\leq \sqrt{\beta}||x||^2.
\label{ch5normu2}
 \end{eqnarray}
 Now combining  inequalities \eqref{ch5ine16}, \eqref{ch5normdeg2}, \eqref{ch5normdeh2}, \eqref{ch5crochetv},\eqref{ch5crochetu}, \eqref{ch5normv2} and  \eqref{ch5normu2}   we obtain 
 \begin{eqnarray*}
 ||Y^M_{n+1}||^2 &\leq& ||Y^M_n||^2+\dfrac{3T^2\sqrt{\beta}}{M}||Y^M_n||^2+\dfrac{4T^2\sqrt{\beta}}{M}||Y^M_n||^2+3\beta||Y^M_n||^2||\Delta W^M_n||^2\nonumber\\
 &+&+4\beta||Y^M_n||^2||\Delta\overline{N}^M_n||^2+\dfrac{2T\sqrt{\beta}}{M}||Y^M_n||^2+\dfrac{2T\sqrt{\beta}}{M}||Y^M_n||^2\nonumber\\
 &+&2\langle Y^M_n+\Delta tu_{\lambda}(Y^M_n), g(Y^M_n)\Delta W^M_n\rangle+2\langle Y^M_n+\Delta tu_{\lambda}(Y^M_n), h(Y^M_n)\Delta\overline{N}^M_n\rangle\\
 &\leq &||Y^M_n||^2+\left(\dfrac{8T^2\sqrt{\beta}}{M}+\dfrac{4T\sqrt{\beta}}{M}\right)||Y^M_n||^2+4\beta||Y^M_n||^2||\Delta W^M_n||^2+4\beta||Y^M_n||^2||\Delta\overline{N}^M_n||^2\\
 &+&2\langle Y^M_n+\Delta tu_{\lambda}(Y^M_n), g(Y^M_n)\Delta W^M_n\rangle+2\langle Y^M_n+\Delta tu_{\lambda}(Y^M_n), h(Y^M_n)\Delta\overline{N}^M_n\rangle.
 \end{eqnarray*}
 Using the inequality $8T^2+4T\leq 8\sqrt{\beta}$, it follows that :
 \begin{eqnarray}
 ||Y^M_{n+1}||^2&\leq& ||Y^M_n||^2+\dfrac{8\beta}{M}||Y^M_n||^2+4\beta||Y^M_n||^2||\Delta W^M_n||^2+4\beta||Y^M_n||^2||\Delta\overline{N}^M_n||^2\nonumber\\
 &+&2\langle Y^M_n+\Delta tu_{\lambda}(Y^M_n),g(Y^M_n)\Delta W^M_n\rangle+2\langle Y^M_n+\Delta tu_{\lambda}(Y^M_n),h(Y^M_n)\Delta\overline{N}^M_n\rangle\nonumber\\
 &=&||Y^M_n||^2 \left(1+\dfrac{8\beta}{M}+4\beta||\Delta W^M_n||^2+4\beta||\Delta\overline{N}^M_n||^2   \right.\nonumber\\
 &+&\left. \left\langle\dfrac{Y^M_n+\Delta tu_{\lambda}(Y^M_n)}{||Y^M_n||}, \dfrac{g(Y^M_n)}{||Y^M_n||}\Delta W^M_n\right\rangle+2\left\langle\dfrac{Y^M_n+\Delta tu_{\lambda}(Y^M_n)}{||Y^M_n||}, \dfrac{h(Y^M_n)}{||Y^M_n||}\Delta\overline{N}^M_n\right\rangle\right)\nonumber\\
 &=&||Y^M_n||^2\left(1+\dfrac{8\beta}{M}+4\beta||\Delta W^M_n||^2+4\beta||\Delta\overline{N}^M_n||^2+2\alpha^M_n+2\beta^M_n\right).
 \label{ch5expY1}
 \end{eqnarray}
 Using Lemma \ref{ch4lemma1} for $a=\dfrac{8\beta}{M}+4\beta||\Delta W^M_n||^2+2\alpha^M_n+2\beta^M_n$ and $b=2\sqrt{\beta}||\Delta\overline{N}^M_n|| $ it follows from \eqref{ch5expY1} that  :
 \begin{eqnarray}
 ||Y^M_{n+1}||^2\leq ||Y^M_n||^2\exp\left(\dfrac{8\beta}{M}+4\beta||\Delta W^M_n||^2+4\beta||\Delta\overline{N}^M_n||+2\alpha^M_n+2\beta^M_n\right),
 \label{ch5expY2}
 \end{eqnarray}
 on $\{w\in\Omega : 1\leq ||Y^M_n(\omega)||\leq M^{1/2c}\}$, for all $M\in\mathbb{N}$ and  $n\in\{0,\cdots,M-1\}.$ Now combining \eqref{ch5normY} and \eqref{ch5expY2} and using mathematical induction as used in the proof of Lemma \ref{ch4lemma2} complete the proof of Lemma  \ref{ch5lemma2}.
 \end{proof}

 The following lemma and its proof are similars  to \cite[Lemma 3.3 pp 15]{Martin1} with only different value of the coefficient $\beta$.
 \begin{lem}\label{ch5lemma4}
 The following inequality holds :
 \begin{eqnarray*}
 \sup_{M\in\mathbb{N}, M\geq 4\beta pT}\mathbb{E}\left[\exp\left(\beta p\sum_{k=0}^{M-1}||\Delta W^M_k||^2\right)\right]<\infty.
 \end{eqnarray*}
 \end{lem}

 The following lemma is based on   \cite[Lemma 5.7]{Martin2}.
 \begin{lem}\label{ch5lemma6}
 The following inequality holds
 \begin{eqnarray*}
 \mathbb{E}\left[\exp\left(pz\mathbf{1}_{\{||x||\geq 1\}}\left<\dfrac{x+u(x)T/M}{||x||},\dfrac{g(x)}{||x||}\Delta W^M_k\right>\right)\right]\leq \exp\left[\dfrac{p^2T(1+TC+T||u(0)||)^2(C+||g(0)||)^2}{M}\right].
 \end{eqnarray*} 
 \end{lem}
 
 \begin{proof}
Let $a^{\top}$ stand for the transposed of a vector $a$, let $Y$ be the $m$ column vector defined by : $Y=\sqrt{\dfrac{T}{M}}(1,\cdots,1)$ and let $\mathcal{N}(0,1)$ be a $1$-dimensional standard normal random variable. Then we have 
 \begin{eqnarray*}
 \mathbb{E}\left[\exp\left(pz\left<\dfrac{x+u(x)T/M}{||x||},\dfrac{g(x)}{||x||}\Delta W^M_k\right>\right)\right]= \mathbb{E}\left[\exp\left(pz\left<\dfrac{x+u(x)T/M}{||x||},\dfrac{g(x)}{||x||}Y\mathcal{N}(0,1)\right>\right)\right].
 \end{eqnarray*}
 Using Lemma \ref{ch4lemma5} it follows that :
 \begin{eqnarray*}
\mathbb{E}\left[\exp\left(pz\left<\dfrac{x+u(x)T/M}{||x||},\dfrac{g(x)}{||x||}\Delta W^M_k\right>\right)\right]&\leq&\exp\left[p^2z^2\left\|\dfrac{g(x)^{\top}(x+u(x)T/M)Y||}{||x||^2}\right\|^2\right]\\
&=&\exp\left[\dfrac{p^2T}{M}\dfrac{||g(x)^{\top}(x+u(x)T/M)||^2}{||x||^4}\right]\nonumber\\
 &\leq & \exp\left[\dfrac{p^2T}{M}\dfrac{||g(x)||^2||x+u(x)T/M||^2}{||x||^4}\right].
 \end{eqnarray*}
From the global Lipschitz condition, for all $x\in\mathbb{R}^d$ such that $||x||\geq1$ we have 
 \begin{eqnarray*}
 ||g(x)||^2\leq (||g(x)-u(0)||+||g(0)||)^2\leq (C||x||^2+||g(0)||)^2\leq (C+||g(0)||)^2||x||^2
 \end{eqnarray*}
 \begin{eqnarray*}
 ||x+u(x)T/M||\leq ||x||+T/M||u(x)||&\leq&||x||+T/M||u(x)-u(0)||+T/M||u(0)||\\
 &\leq &||x||+TC||x||+T||u(0)||\\
 &\leq &(1+TC+T||u(0)||)||x||.
 \end{eqnarray*}
 Therefore, it follows that : 
 \begin{eqnarray*}
 \mathbb{E}\left[\exp\left(pz_{\{||x||\geq 1\}}\left<\dfrac{x+u(x)T/M}{||x||},\dfrac{g(x)}{||x||}\Delta W^M_k\right>\right)\right]\leq \exp\left[\dfrac{p^2T(1+TC+T||u(0)||)^2(C+||g(0)||)^2}{M}\right].
 \end{eqnarray*}

 for all $M\in\mathbb{N}$, $k\in\{0,\cdots,M-1\}$, $p\in[1,\infty)$ and $z\in\{-1,1\}$.
 \end{proof}

 Following closely \cite[Lemma 2.3 ]{Xia2} we have the following lemma.
 \begin{lem}\label{ch5lemma7}
 Let $\alpha^M_n :\Omega\longrightarrow\mathbb{R}$ for $M\in\mathbb{N}$ and $n\in\{0,1,\cdots,M\}$ define as in notation \eqref{notation 1}. Then the following inequality holds :
 \begin{eqnarray*}
 \sup_{z\in\{-1,1\}}\sup_{M\in\mathbb{N}}\left\|\sup_{n\in\{0,1,\cdots,M\}}\exp\left(z\sum_{k=0}^{n-1}\alpha^M_k\right)\right\|_{L^p(\Omega, \mathbb{R})}<\infty,
 \end{eqnarray*}
 for all $p\in[2,+\infty)$ 
 \end{lem}
 \begin{proof}
 The time discrete stochastic process $z\sum_{k=0}^{n-1}\alpha^M_k$, $n\in\{0,1,\cdots, M\}$ is an $(\mathcal{F}_{nT/M})_{n\in\{0,\cdots,M\}}-$ martingale for every $z\in\{-1,1\}$ and $M\in \mathbb{N}$. So $\exp\left(z\sum_{k=0}^{n-1}\alpha^M_k\right)$ is a positive $(\mathcal{F}_{nT/M})_{n\in\{0,\cdots,M\}}-$ submartingale for every $z\in\{-1,1\}$ and $M\in\mathbb{N}$ since $\exp$ is a convex function.
 
 Applying Doop's maximal inequality leads to : 
 \begin{eqnarray}
 \left\|\sup_{n\in\{0,\cdots,M\}}\exp\left(z\sum_{k=0}^{n-1}\alpha^M_k\right)\right\|_{L^p{(\Omega, \mathbb{R})}}&=&\left(\mathbb{E}\left|\sup_{n\in\{0,\cdots,M\}}\exp\left(pz\sum_{k=0}^{n-1}\alpha^M_k\right)\right|\right)^{1/p}\nonumber\\
 &\leq &\left(\dfrac{p}{p-1}\right)\left(\mathbb{E}\left |\exp\left(pz\sum_{k=0}^{M-1}\alpha^M_k\right)\right|\right)^{1/p}\nonumber\\
 &= &\dfrac{p}{p-1}\left\|\exp\left(z\sum_{k=0}^{M-1}\alpha^M_k\right)\right\|_{L^p{(\Omega,\mathbb{R})}}.
 \label{ch5alpha1}
 \end{eqnarray}
 
 Using  Lemma \ref{ch5lemma6}, it follows that :
 \begin{eqnarray}
 \mathbb{E}\left[\exp(pz\alpha^M_k)/\mathcal{F}_{kT/M}\right]\leq\exp\left(\dfrac{p^2T(C+||g(0)||)^2(1+TC+T||u(0)||)^2}{M}\right).
 \label{ch5alpha4}
 \end{eqnarray}
 Using  inequality \eqref{ch5alpha4},  it follows that :
 \begin{eqnarray*}
 \mathbb{E}\left[\exp\left(pz\sum_{k=0}^{M-1}\alpha^M_k\right)\right]&=&\mathbb{E}\left[\exp\left(pz\sum_{k=0}^{M-2}
 \alpha^M_k\right)\mathbb{E}[\exp(p\alpha^M_{M-1}/\mathcal{F}_{(M-1)T/M}\right]\\
 &\leq &\mathbb{E}\left[\exp\left(pz\sum_{k=0}^{M-2}\alpha^M_k\right)\right]\exp\left(\dfrac{p^2T(C+||g(0)||)^2(1+TC+T||u(0)||)^2}{M}\right).
 \end{eqnarray*}
 Iterating the previous inequality $M$ times gives :
 \begin{eqnarray}
\mathbb{E}\left[\exp\left(pz\sum_{k=0}^{M-1}\alpha^M_k\right)\right] \leq \exp(p^2T(C+||g(0)||)^2(1+TC+T||u(0)||)^2).
\label{ch5alpha5}
 \end{eqnarray}
 Now combining  inequalities  \eqref{ch5alpha1} and \eqref{ch5alpha5} leads to 
 \begin{eqnarray*}
 \sup_{z\in\{-1,1\}}\sup_{M\in\mathbb{N}}\left\|\sup_{n\in\{0,\cdots,M\}}\exp\left(z\sum_{k=0}^{n-1}\alpha^M_k\right)\right\|_{L^p(\Omega,\mathbb{R})} &\leq& 2\exp(p^2T(C+||g(0)||)^2(1+TC+T||u(0)||)^2) \\
 &< &+\infty,
 \end{eqnarray*}
 for all $p\in[2,\infty)$.
 \end{proof}

  \begin{lem}\label{ch5lemma9}
 The following inequality holds
 
 $
 \mathbb{E}\left[\exp\left(pz\mathbf{1}_{\{||x||\geq 1\}}\left<\dfrac{x+u(x)T/M}{||x||},\dfrac{h(x)}{||x||}\Delta\overline{N}^M_n\right>\right)\right]
 $
 \begin{eqnarray*}
 \leq\left[\exp\left(\dfrac{\left[e^{p(C+||h(0)||)(1+TC+T||u(0)||)}+p(C+||h(0)||)(1+TC+T||u(0)||)\right]\lambda T}{M}\right)\right],
 \end{eqnarray*}
 for all $M\in\mathbb{N}$, all $p\in[1, +\infty)$ and all $n\in\{0,\cdots,M\}$, $z\in\{-1,1\}$.
 \end{lem}
 
 \begin{proof} For $x\in\mathbb{R}^d$ such that $||x||\neq 0$, we have :
 \begin{eqnarray*}
  \mathbb{E}\left[\exp\left(pz\left<\dfrac{x+u(x)T/M}{||x||},\dfrac{h(x)}{||x||}\Delta\overline{N}^M_n\right>\right)\right]  &\leq& \mathbb{E}\left[\exp\left(pz\dfrac{||x+u(x)T/M||||h(x)||}{||x||^2}\Delta\overline{N}^M_n\right)\right].
  \end{eqnarray*}
  Using the global Lipschitz condition on $h$, for all $x\in\mathbb{R}^d$ such that $||x||\geq 1$,  we have :
  \begin{eqnarray}
  ||h(x)||\leq ||h(x)-h(0)||+||h(0)||\leq (C+||h(0)||)||x||. \label{ch5normeh}\\
  ||x+u(x)T/M||\leq (1+TC+T||u(0)||)||x||\label{ch5normeh1}
  \end{eqnarray}
  So using inequalities \eqref{ch5normeh} and \eqref{ch5normeh1},  it follows that :
  \begin{eqnarray*}
  \mathbb{E}\left[\exp\left(pz\mathbf{1}_{\{||x||\geq 1\}}\left<\dfrac{x+u(x)T/M}{||x||},\dfrac{h(x)}{||x||}\Delta\overline{N}^M_n\right>\right)\right]
  \leq\mathbb{E}\left( \exp[pz(C+||h(0)||)(1+TC+T||u(0)||)\Delta\overline{N}^M_n \right).
  \end{eqnarray*}
  Using Lemma \ref{ch4lemma8}, it follows that : 
  
  $
  \left[\exp\left(pz\mathbf{1}_{\{||x||\geq 1\}}\left<\dfrac{x+u(x)T/M}{||x||},\dfrac{h(x)}{||x||}\Delta\overline{N}^M_n\right>\right)\right]
  $
  \begin{eqnarray*}
  &\leq &\mathbb{E}[\exp(pz(C+||h(0)||)(1+TC+T||u(0)||)\Delta\overline{N}^M_n)]\\
   &\leq &\left[\exp\left(\dfrac{\left[e^{p(C+||h(0)||)(1+TC+T||u(0)||)}+p[(C+||h(0)||)(1+TC+T||u(0)||)-1\right]\lambda T}{M}\right)\right]\\
   &\leq &\left[\exp\left(\dfrac{\left[e^{p(C+||h(0)||)(1+TC+T||u(0)||)}+p[(C+||h(0)||)(1+TC+T||u(0)||)\right]\lambda T}{M}\right)\right].
  \end{eqnarray*}
 \end{proof}

 The following lemma is similar  to \cite[Lemma 2.3]{Xia2}.
\begin{lem}\label{ch5lemma10}
 Let $\beta^M_n :\Omega \longrightarrow \mathbb{R}$ defined in Notation \ref{ch5notation1} for all $M\in\mathbb{N}$ and all $n\in\{0,\cdots,M\}$ then we have the following inequality
 \begin{eqnarray*}
 \sup_{z\in\{-1,1\}}\sup_{M\in\mathbb{N}}\left\|\sup_{n\in\{0,\cdots, M\}}\exp\left(z\sum_{K=0}^{n-1}\beta^M_k\right)\right\|_{L^p(\Omega, \mathbb{R})}<+\infty.
 \end{eqnarray*} 
 \end{lem}
 
 \begin{proof}
    Following  the proof of  \cite[Lemma 3.4 ]{Martin1}, the result is straightforward using lemmas \ref{ch5lemma9} and \ref{ch5lemma6}.
 \end{proof}

\begin{lem}\label{ch5lemma11}
The following inequality holds 
\begin{eqnarray*}
\sup_{M\in\mathbb{N}}\mathbb{E}\left[\exp\left(p\beta\sum_{k=0}^{M-1}||\Delta\overline{N}^M_k||\right)\right]<+\infty,
\end{eqnarray*}
for all $p\in[1, +\infty)$.
\end{lem}

  \begin{proof}
 The proof is similar to the proof of  Lemma \ref{ch4lemma11} with only different value of $\beta$.
  \end{proof}

  The following lemma is based on Lemma \ref{ch5lemma10}.
  \begin{lem}\label{ch5lemma12}
  [Uniformly bounded moments of the dominating stochastic processes].
  
  Let $M\in\mathbb{N}$ and $D_n^M : \Omega \longrightarrow [0,\infty)$ for $n\in\{0,1,\cdots, M\}$ be define as in notation \eqref{ch5notation1}. Then we have :
  \begin{eqnarray*}
  \sup_{M\in\mathbb{N}, M\geq8\lambda pT}\left\|\sup_{n\in\{0,1,\cdots, M\}}D_n^M\right\|_{L^p(\Omega, \mathbb{R})}<\infty,
  \end{eqnarray*}
  for all $p\in[1,\infty)$. 
  \end{lem}

  The following lemma is an extension  of \cite[Lemma 3.6, pp 16 ]{Martin1}. Here, we include  the jump part.
  \begin{lem}\label{ch5lemma13}
  Let $M\in\mathbb{N}$ and $\Omega_M^M\in\mathcal{F}$. The following inequality  holds :
  \begin{eqnarray*}
  \sup_{M\in\mathbb{N}}\left(M^p\mathbb{P}[(\Omega_M^M)^c]\right)<+\infty,
  \end{eqnarray*}
  for all $p\in[1,\infty)$.
  \end{lem}

   \begin{proof} \textbf{[Theorem \ref{ch5theorem1}]}.
   
   Let's first represent the numerical approximation $Y^M_n$ in the following appropriate form  :
   \begin{eqnarray*}
   Y^M_n&=&Y_{n-1}^M +u_{\lambda}(Y^M_{n-1})T/M+\dfrac{\Delta tv(Y^M_{n-1})}{1+\Delta t||v(Y^M_{n-1})||}+g(Y_{n-1})\Delta W^M_{n-1}+h(Y^M_{n-1})\Delta\overline{N}^M_{n-1}\\
   &=&X_0+\sum_{k=0}^{n-1}u(Y^M_{k})T/M+\sum_{k=0}^{n-1}\dfrac{\Delta t v(Y_k^M)}{1+\Delta t||v(Y^M_k)||}+\sum_{k=0}^{n-1}g(Y^M_k)\Delta W^M_k+\sum_{k=0}^{n-1}h(Y^M_k)\Delta\overline{N}^M_k\\
   &=& X_0+u(0)nT/M+ \sum_{k=0}^{n-1}g(0)\Delta W^M_k+\sum_{k=0}^{n-1}h(0)\Delta\overline{N}^M_k+\sum_{=0}^{n-1}T/M(u(^M_k)-u(0))\\
   &+&\sum_{k=0}^{n-1}\dfrac{\Delta tv(Y^M_{k})}{1+\Delta t||v(Y^M_{k})||}+\sum_{k=0}^{n-1}(g(Y^M_k)-g(0))\Delta W^M_k+\sum_{k=0}^{n-1}(h(Y^M_k)-h(0))\Delta\overline{N}^M_k,
   \end{eqnarray*}
for all $M\in\mathbb{N}   $ and all $n\in\{0,\cdots,M\}$.

Using the inequality 
\begin{eqnarray*}
\left\|\dfrac{\Delta tv(Y^M_k)}{1+\Delta t||v(Y^M_k)||}\right\|_{L^P(\Omega, \mathbb{R}^d)} <1
\end{eqnarray*} 
it follows that :
\begin{eqnarray*}
||Y^M_n||_{L^p(\Omega, \mathbb{R}^d)}  &\leq &||X_0||_{L^p(\Omega,\mathbb{R}^d)}+ ||u(0)||nT/M+\left\|\sum_{k=0}^{n-1}g(0)\Delta W^M_k\right\|_{L^p(\Omega, \mathbb{R}^d)}+M\\
&+&\left\|\sum_{k=0}^{n-1}h(0)\Delta\overline{N}^M_k\right\|_{L^p(\Omega,\mathbb{R}^d)}+\left\|\sum_{k=0}^{n-1}(g(Y^M_k)-g(0))\Delta W^M_k\right\|_{L^p(\Omega,\mathbb{R}^d)}\\
&+&\left\|\sum_{k=0}^{n-1}(h(Y^M_k)-h(0))\Delta\overline{N}^M_k\right\|_{L^p(\Omega,\mathbb{R}^d)}.
\end{eqnarray*}
Using Lemma \ref{ch4lemma16} and Lemma \ref{ch4lemma19}, it follows that
\begin{eqnarray}
||Y^M_n||_{L^p(\Omega,\mathbb{R}^d)} &\leq &||X_0||_{L^p(\Omega,\mathbb{R})}+||u(0)||nT/M+C_p\left(\sum_{k=0}^{n-1}\sum_{i=1}^{m}||g_i(0)||^2\dfrac{T}{M}\right)^{1/2}+C_p\left(\sum_{k=0}^{n-1}||h(0)||^2\dfrac{T}{M}\right)^{1/2}\nonumber\\
&+&M+ C_p\left(\sum_{k=0}^{n-1}\sum_{i=1}^m||(g_i(Y_k^M)-g_i(0))\Delta W^M_k||^2_{L^p(\Omega, \mathbb{R}^d)}\dfrac{T}{M}\right)^{1/2}\nonumber\\ &+&C_p\left(\sum_{k=0}^{n-1}\lambda||(h(Y_k^M)-h(0))\Delta W^M_k||^2_{L^p(\Omega, \mathbb{R}^d)}\dfrac{T}{M}\right)^{1/2}\nonumber\\
&\leq&||X_0||_{L^p(\Omega, \mathbb{R}^d)}+T||u(0)||+C_p\left(\dfrac{nT}{M}\sum_{i=1}^m||g_i(0)||^2\right)^{1/2}+C_p\left(\dfrac{nT}{M}||h(0)||^2\right)^{1/2}\nonumber\\
 &+&M+C_p\left(\sum_{k=0}^{n-1}\sum_{i=1}^m||g_i(Y^M_k)-g_i(0)||^2_{L^p(\Omega,\mathbb{R}^d)}\dfrac{T}{M}\right)^{1/2}\nonumber\\
&+&C_p\left(\sum_{k=0}^{n-1}||h(Y^M_k)-h(0)||^2_{L^p(\Omega,\mathbb{R}^d)}\dfrac{T}{M}\right)^{1/2}.
\label{ch5MB1}
\end{eqnarray}
From $||g_i(0)||^2\leq ||g(0)||^2$  and the global Lipschitz condition satisfied by   $g$ and $h$, we obtain $||g_i(Y^M_k)-g_i(0)||\leq  C||Y^M_k||_{L^p(\Omega,\mathbb{R}^d)}$ and $||h(Y^M_k)-h(0)||\leq  C||Y^M_k||_{L^p(\Omega,\mathbb{R}^d)}$. So using \eqref{ch5MB1}, we have :
\begin{eqnarray*}
||Y^N_n||_{L^p(\Omega, \mathbb{R}^d)} &\leq & ||X_0||_{L^p(\Omega,\mathbb{R}^d)}+T||u(0)||+C_p\sqrt{Tm}||g(0)||+C_p\sqrt{T}||h(0)||+M\\
&+&C_p\left(\dfrac{Tm}{M}\sum_{k=0}^{n-1}||Y^M_k||^2_{L^p(\Omega,\mathbb{R}^d)}\right)^{1/2}
+C_p\left(\dfrac{T}{M}\sum_{k=0}^{n-1}||Y^M_k||^2_{L^p(\Omega,\mathbb{R}^d)}\right)^{1/2}.
\end{eqnarray*}
Using the inequality $(a+b+c)^2\leq 3a^2+3b^2+3c^2$, it follows that :
\begin{eqnarray*}
||Y^M_n||^2_{L^p(\Omega,\mathbb{R}^d)}&\leq& 3\left(||X_0||_{L^p(\Omega,\mathbb{R}^d)}+T||u(0)||+C_p\sqrt{Tm}||g(0)||+C_p\sqrt{T}||h(0)||+M\right)^2\nonumber\\
&+&\dfrac{3T(C_p\sqrt{m}+C_p)^2}{M}\sum_{k=0}^{n-1}||Y^M_k||^2_{L^p(\Omega, \mathbb{R}^d)},
\end{eqnarray*}
Using the inequality  the inequality
$
\dfrac{3T(C_p\sqrt{m}+C_p)^2}{M}<3(C_p\sqrt{m}+C_p)^2
$,  it follows  that :

\begin{eqnarray}
||Y^M_n||^2_{L^p(\Omega,\mathbb{R}^d)}&\leq& 3\left(||X_0||_{L^p(\Omega,\mathbb{R}^d)}+T||u(0)||+C_p\sqrt{Tm}||g(0)||+C_p\sqrt{T}||h(0)||+M\right)^2\nonumber\\
&+&3T(C_p\sqrt{m}+C_p)^2\sum_{k=0}^{n-1}||Y^M_k||^2_{L^p(\Omega, \mathbb{R}^d)},
\label{ch5MB2}
\end{eqnarray}
for all $p\in[2,\infty)$.

Applying  Gronwall inequality to \eqref{ch5MB2} leads to :
\begin{eqnarray*}
||Y^M_n||^2_{L^p(\Omega,\mathbb{R}^d)}\leq 3e^{3(C_p\sqrt{m}+C_p)^2}\left(||X_0||_{L^p(\Omega,\mathbb{R}^d)}+T||u(0)||+C_p\sqrt{Tm}||g(0)||+C_p\sqrt{T}||h(0)||+M\right)^2
\label{ch5MB3}
\end{eqnarray*}
Taking the square root and the supremum in both sides of the previous inequality leads to :

$\sup\limits_{n\in\{0,\cdots, M\}}||Y^M_n||_{L^p(\Omega, \mathbb{R}^d)}$
\begin{eqnarray}
 \leq\sqrt{3}e^{3(C_p\sqrt{m}+C_p)^2}\left(||X_0||_{L^p(\Omega,\mathbb{R}^d)}+T||u(0)||+C_p\sqrt{Tm}||g(0)||+C_p\sqrt{T}||h(0)||+M\right).
\label{ch5MB4}
\end{eqnarray}
Unfortunately, \eqref{ch5MB4} is not enough to conclude the proof of the theorem due to the term  $M$ in the right hand side. Using the fact that $(\Omega_n^M)_n$ is a decreasing sequence and exploiting Holder inequality, we obtain :
\begin{eqnarray}
\sup_{M\in\mathbb{N}}\sup_{n\in\{0,\cdots,M\}}\left\|\mathbf{1}_{(\Omega_n^M)^c}Y^M_n\right\|_{L^p(\Omega,\mathbb{R}^d)}& \leq &\sup_{M\in\mathbb{N}}\sup_{n\in\{0,\cdots,M\}}\left\|\mathbf{1}_{(\Omega^M_M)^c}\right\|_{L^{2p}(\Omega,\mathbb{R}^d)}\left\|Y^M_n\right\|_{L^{2p}(\Omega,\mathbb{R}^d)}\nonumber\\
&\leq &\left(\sup_{M\in\mathbb{N}}\sup_{n\in\{0,\cdots,M\}}\left(M\left\|\mathbf{1}_{(\Omega^M_M)^c}\right\|_{L^{2p}(\Omega,\mathbb{R}^d)}\right)\right)\nonumber\\
&\times &\left(\sup_{M\in\mathbb{N}}\sup_{n\in\{0,\cdots,M\}}\left(M^{-1}||Y^M_n||_{L^{2p}(\Omega,\mathbb{R}^d)}\right)\right).
\label{ch5MB5}
\end{eqnarray}
From  inequality \eqref{ch5MB4} we have

$
\left(\sup\limits_{M\in\mathbb{N}}\sup\limits_{n\in\{0,\cdots,M\}}\left(M^{-1}||Y^M_n||_{L^p(\Omega,\mathbb{R}^d)}\right)\right)$
\begin{eqnarray}
\leq \sqrt{2}e^{(C_p\sqrt{m}+C_p)^2}\left(\dfrac{||X_0||_{L^p(\Omega, \mathbb{R}^d)}}{M}+\dfrac{C_p\sqrt{Tm}||g(0)||+C_p\sqrt{T}||h(0)||}{M}+1\right)\nonumber\\
\leq \sqrt{2}e^{(p\sqrt{m}+C_p)^2}\left(||X_0||_{L^p(\Omega,\mathbb{R}^d)}+C_p\sqrt{Tm}||g(0)||+C_p\sqrt{T}||h(0)||+1\right)<+\infty,
\label{ch5MB6}
\end{eqnarray}
From the relation 
\begin{eqnarray*}
\left\|\mathbf{1}_{(\Omega^M_M)^c}\right\|_{L^{2p}(\Omega, \mathbb{R}^d)}=
\mathbb{E}\left[\mathbf{1}_{(\Omega^M_M)^c}\right]^{1/2p}=
\mathbb{P}\left[(\Omega^M_M)^c\right]^{1/2p},
\end{eqnarray*}
it follows using Lemma  \ref{ch5lemma13} that :
\begin{eqnarray}
\sup_{M\in\mathbb{N}}\sup_{n\in\{0,\cdots, M\}}\left(M\left\|\mathbf{1}_{(\Omega^M_M)^c}\right\|_{L^{2p}(\Omega,\mathbb{R}}\right)=\sup_{M\in\mathbb{N}}\sup_{n\in\{0,\cdots, M\}}\left(M^{2p}\mathbb{P}\left[(\Omega^M_M)^c\right]\right)^{1/2p}<+\infty.
\label{ch5MB7}
\end{eqnarray}
So  plugging  \eqref{ch5MB6} and \eqref{ch5MB7} in \eqref{ch5MB5}  leads to : 
\begin{eqnarray}
\sup_{M\in\mathbb{N}}\sup_{n\in\{0,\cdots, M\}}\left\|\mathbf{1}_{(\Omega^M_n)^c}Y^M_n\right\|_{L^p(\Omega,\mathbb{R}^d)}<+\infty.
\label{ch5MB8}
\end{eqnarray}
Futhermore, we have 
\begin{eqnarray}
\sup_{M\in\mathbb{N}}\sup_{n\in\{0,\cdots, M\}}\left\|Y^M_n\right\|_{L^p(\Omega,\mathbb{R}^d)} &\leq &\sup_{M\in\mathbb{N}}\sup_{n\in\{0,\cdots, M\}}\left\|\mathbf{1}_{(\Omega^M_n)}Y^M_n\right\|_{L^p(\Omega,\mathbb{R}^d)}\nonumber\\
&+&\sup_{M\in\mathbb{N}}\sup_{n\in\{0,\cdots, M\}}\left\|\mathbf{1}_{(\Omega^M_n)^c}Y^M_n\right\|_{L^p(\Omega,\mathbb{R}^d)}.
\label{ch5MB9}
\end{eqnarray}
From  \eqref{ch5MB8} the second term  of inequality \eqref{ch5MB9} is bounded, while using Lemma \ref{ch5lemma2} and Lemma \ref{ch5lemma12} we  have :
\begin{eqnarray}
\sup_{M\in\mathbb{N}}\sup_{n\in\{0,\cdots, M\}}\left\|\mathbf{1}_{(\Omega^M_n)}Y^M_n\right\|_{L^p(\Omega,\mathbb{R}^d)}\leq \sup_{M\in\mathbb{N}}\sup_{n\in\{0,\cdots, M\}}\left\|D_n^M\right\|_{L^p(\Omega,\mathbb{R}^d)}<+\infty.
\label{ch5MB10}
\end{eqnarray}
Finally plugging \eqref{ch5MB8} and \eqref{ch5MB10} in \eqref{ch5MB9} leads to : 
\begin{eqnarray*}
\sup_{M\in\mathbb{N}}\sup_{n\in\{0,\cdots, M\}}\left\|Y^M_n\right\|_{L^p(\Omega,\mathbb{R}^d)}<+\infty.
\end{eqnarray*}
This complete the proof of Theorem \ref{ch5theorem1}.
\end{proof}

\section{Strong convergence of the semi-tamed Euler scheme}

\begin{thm}\label{ch5theorem2}
 Under Assumptions \ref{ch5assumption1}, for all $p\in[1,+\infty)$  there exist a positive  constant $C_p$ such that :
 \begin{eqnarray}
 \left(\mathbb{E}\left[\sup_{t\in[0,T]}\left\|X_t-\overline{Y}^M_t\right\|^p\right]\right)^{1/p}\leq C_p\Delta t^{1/2},
 \label{ch5inetheo}
 \end{eqnarray}
 for all $M\in \mathbb{N}$.
 
 Where $X : [0,T]\times \Omega\longrightarrow \mathbb{R}^d$ is the exact solution of equation \eqref{ch5exactsol1} and $\overline{Y}^M_t$ is the time continous solution defined in  \eqref{ch5continoussolu}. 
 \end{thm}

 In order to prove Theorem \ref{ch5theorem2}, we need the following two lemmas.
\begin{lem} \label{ch5lemma14}[Based on  \cite[Lemma 3.10, pp 16]{Martin1}].

Let $Y_n^M$ be defined by \eqref{ch5semitam} for all $M\in\mathbb{N}$ and all $n\in\{0,1,\cdots, M\}$. Then the following inequalities holds :
\begin{eqnarray*}
\sup_{M\in\mathbb{N}}\sup_{n\in\{0,1,\cdots, M\}}\left(\mathbb{E}\left[||u_{\lambda}(Y_n^M)||^p\right]\vee  \mathbb{E}[||v(Y^M_n)||^p]\vee \mathbb{E}\left[||g(Y_n^M)||^p\right]\vee \mathbb{E}\left[||h(Y_n^M)||^p\right]\right)<+\infty,
\end{eqnarray*}
for all $p\in[1,\infty)$.
\end{lem}   

\begin{proof}
The proof is similar to the proof  \cite[Lemma 3.10, pp 16]{Martin1}.
\end{proof}

For $s\in[0,T]$  let $\lfloor s\rfloor$ be the greatest grid point less than $s$. We have the following lemma.
 \begin{lem}\label{ch5lemma16}
 The following inequalities holds for any stepsize $\Delta t$. 
 \begin{align*}
 \sup_{t\in[0,T]}\left\|\overline{Y}^M_t-\overline{Y}^M_{\lfloor t\rfloor}\right\|_{L^p(\Omega, \mathbb{R}^d)}\leq C_p\Delta t^{1/2},
 \end{align*}
 \begin{eqnarray*}
 \sup_{M\in\mathbb{N}}\sup_{t\in[0,T]}\left\|\overline{Y}^M_t\right\|_{L^p(\Omega, \mathbb{R}^d)}<\infty,
 \end{eqnarray*}
 \begin{eqnarray*}
 \sup_{t\in[0,T]}\left\|v(\overline{Y}^M_t)-v(\overline{Y}^M_{\lfloor t\rfloor})\right\|_{L^p(\Omega, \mathbb{R}^d)}\leq C_p\Delta t^{1/2}.
 \end{eqnarray*}
 \end{lem}

\begin{proof}
\begin{itemize}
\item
 Using the time continous approximation \eqref{ch5continoussolu}, Lemma \ref{ch4lemma19} and Lemma \ref{ch4lemma16}, it follows that :
 
 $
 \sup\limits_{t\in[0,T]}||\overline{Y}^M_t-\overline{Y}^M_{\lfloor t\rfloor}||_{L^p(\Omega, \mathbb{R}^d)}
 $
 \begin{eqnarray}
  &\leq &\dfrac{T}{M}\sup_{t\in[0,T]}\left\|u_{\lambda}(\overline{Y}^M_{\lfloor t\rfloor})\right\|_{L^p(\Omega, \mathbb{R}^d)}+\dfrac{T}{M}\left(\sup_{t\in[0,T]}\left\|\dfrac{v(\overline{Y}^M_{\lfloor t\rfloor})}{1+\Delta t||v(\overline{Y}^M_{\lfloor t\rfloor})||}\right\|_{L^p(\Omega, \mathbb{R}^d)}\right)\nonumber\\
 &+& \sup_{t\in[0, T]}\left\|\int^t_{\lfloor t\rfloor} g(\overline{Y}^M_{\lfloor t\rfloor})dW_s\right\|_{L^p(\Omega, \mathbb{R}^d)}+\sup_{t\in[0,T]}\left\|\int^t_{\lfloor t\rfloor}h(\overline{Y}^M_{\lfloor t\rfloor})d\overline{N}_s\right\|_{L^p(\Omega, \mathbb{R}^d)}\nonumber\\
 &\leq &\dfrac{T}{\sqrt{M}}\left(\sup_{n\in\{0,\cdots M\}}\|u_{\lambda}(Y^M_n)\|_{L^p(\Omega,\mathbb{R}^d)}\right)+\dfrac{T}{\sqrt{M}}\left(\sup_{n\in\{0,\cdots, M\}}||v(Y^M_n)||_{L^p(\Omega, \mathbb{R}^d)}\right)\nonumber\\
 &+& C_p\sup_{t\in[0,T]}\left(\dfrac{T}{M}\sum_{i=1}^m\int^t_{\lfloor t\rfloor}||g_i(\overline{Y}^M_s)||^2_{L^p(\Omega, \mathbb{R}^d)}ds\right)^{1/2}+C_p\sup_{t\in[0,T]}\left(\dfrac{TC_p}{M}\int^t_{\lfloor t\rfloor}||h(\overline{Y}^m_s)||^2_{L^p(\Omega, \mathbb{R}^d)}ds\right)^{1/2}\nonumber\\ 
 &\leq &\dfrac{T}{\sqrt{M}}\left(\sup_{n\in\{0,\cdots, M\}}||u_{\lambda}(Y^M_n)||_{L^p(\Omega, \mathbb{R}^d)}\right)+\dfrac{T}{\sqrt{M}}\left(\sup_{n\in\{0,\cdots, M\}}||v(Y^M_n)||_{L^p(\Omega, \mathbb{R}^d)}\right)\nonumber\\
 &+& \dfrac{C_p\sqrt{Tm}}{\sqrt{M}}\left(\sup_{i\in\{1,\cdots, m\}}\sup_{n\in\{0,\cdots, M\}}||g_i(Y^M_n)||_{L^p(\Omega, \mathbb{R}^d)}\right)\nonumber\\
 &+&\dfrac{C_p\sqrt{T}}{\sqrt{M}}\left(\sup_{n\in\{0, \cdots,M\}}||h(Y^M_n)||_{L^p(\Omega, \mathbb{R}^d)}\right)
 \label{Thcontinous}
 \end{eqnarray}
 for all $p\in[1,+\infty)$ and all $M\in\mathbb{N}$. Using  inequality \eqref{Thcontinous} and  Lemma \ref{ch5lemma14}  it follows  that : 
\begin{eqnarray}
\left[\sup_{t\in[0,T]}||\overline{Y}^M_t-\overline{Y}^M_{\lfloor t\rfloor}||_{L^p(\Omega, \mathbb{R}^d)}\right]<C_p\Delta t^{1/2},
\label{Ch5bon1}
\end{eqnarray}
for all $p\in[1, +\infty)$ and for all stepsize $\Delta t$. 
\item  Using   inequality \eqref{Ch5bon1}, inequality $||a||\leq ||a-b||+||b||$ for all $a,b\in\mathbb{R}^d$ and Theorem \ref{ch5theorem1}  it follows that 
\begin{eqnarray*}
\sup_{t\in[0,T]}||\overline{Y}^M_t||_{L^p(\Omega, \mathbb{R}^d)}&\leq&\left[\sup_{t\in[0,T]}\left\|\overline{Y}^M_t-\overline{Y}^M_{\lfloor t\rfloor}\right\|_{L^p(\Omega, \mathbb{R}^d)}\right]+\sup_{t\in[0,T]}\left\|\overline{Y}^M_{\lfloor t\rfloor}\right\|_{L^p(\Omega, \mathbb{R}^d)}\\
&\leq&C_p\Delta t^{1/2}+\sup_{t\in[0,T]}\left\|\overline{Y}^M_{\lfloor t\rfloor}\right\|_{L^p(\Omega, \mathbb{R}^d)}\\
&<&C_pT^{1/2}+\sup_{t\in[0,T]}\left\|\overline{Y}^M_{\lfloor t\rfloor}\right\|_{L^p(\Omega, \mathbb{R}^d)}<\infty,
\end{eqnarray*}
for all $p\in[1,+\infty)$ and all $M\in\mathbb{N}$.
\item Further, using the polynomial growth condition :
\begin{eqnarray*}
||v(x)-v(y)||\leq C(K+||x||^c+||y||^c)||x-y||,
\end{eqnarray*}
for all $x, y\in\mathbb{R}^d$, it follows using Holder  inequality that :
\begin{eqnarray}
\sup_{t\in[0,T]}||v(\overline{Y}^M_t)-v(\overline{Y}^M_{\lfloor t\rfloor})||_{L^p(\Omega, \mathbb{R}^d)} &\leq &C\left(K+2\sup_{t\in[0,T]}||\overline{Y}^M_t||^c_{L^{2pc}(\Omega, \mathbb{R}^d)}\right)\nonumber\\
&\times & \left(\sup_{t\in[0,T]}||\overline{Y}^M_t-\overline{Y}^M_{\lfloor t\rfloor}||_{L^{2p}(\Omega, \mathbb{R}^d)}\right)
\label{ch4Thfcontinous}
\end{eqnarray}
Using \eqref{ch4Thfcontinous} and the first part of Lemma \ref{ch5lemma16},   the following inequality holds  for all $p\in[1,+\infty)$ 
\begin{eqnarray}
\left[\sup_{t\in[0,T]}||v(\overline{Y}^M_t)-v(\overline{Y}^M_{\lfloor t\rfloor})||_{L^p(\Omega, \mathbb{R}^d)}\right]<C_p\Delta t^{1/2},
\label{ch4Thffinal}
\end{eqnarray} 
for all $p\in[1,\infty)$ and all $M\in\mathbb{N}$.
 \end{itemize}
 \end{proof}

  Now we are ready to  prove  Theorem \ref{ch5theorem2}.

 \begin{proof} \textbf{[ Theorem \ref{ch5theorem2}]}
 
 Let's recall that for $z\in[0,T]$,  $\lfloor z\rfloor$ is the greatest grid point less than $z$. The time continuous solution \eqref{ch5continoussolu} can be written into  its integral form as bellow :
 \begin{eqnarray}
 \overline{Y}^M_s=\varepsilon+\int_0^su(\overline{Y}^M_{\lfloor z\rfloor})dz+\int_0^s\dfrac{v(\overline{Y}^M_{\lfloor z\rfloor})}{1+\Delta t||v(\overline{Y}^M_{\lfloor z\rfloor})||}dz+ \int g(\overline{Y}^M_{\lfloor z\rfloor})dW_z+\int h(\overline{Y}^M_{\lfloor z\rfloor})d\overline{N}_z,
 \label{ch5continoussol2}
 \end{eqnarray}
 for all $z\in[0, T]$ almost sure (a.s) and all $M\in\mathbb{N}$.
 
 Let's estimate first the quantity  $||X_s-\overline{Y}^M_s||^2$
 \begin{eqnarray*}
 X_s-\overline{Y}_s&=&\int_0^s\left(u_{\lambda}(X_z)-u_{\lambda}(\overline{Y}^M_{\lfloor z \rfloor})\right)dz+\int_0^s\left(v(X_z)-\dfrac{v(\overline{Y}^M_{\lfloor z\rfloor})}{1+\Delta t||v(\overline{Y}^M_{\lfloor z\rfloor})||}\right)dz\\
 &+&\int_0^s\left(g(X_z)-g(\overline{Y}^M_{\lfloor z\rfloor})\right)dW_z+\int_0^s\left(h(X_z)-h(\overline{Y}^M_{\lfloor z\rfloor})\right)d\overline{N}_z.
 \end{eqnarray*}
 Using the relation $d\overline{N}_z=dN_z-\lambda dz$, it follows that
 \begin{eqnarray*}
 X_s-\overline{Y}_s&=&\int_0^s\left[\left(v(X_z)-\dfrac{v(\overline{Y}^M_{\lfloor z\rfloor})}{1+\Delta t||v(\overline{Y}^M_{\lfloor z\rfloor})||}\right)+\left(u(X_z)-u(\overline{Y}^M_{\lfloor z\rfloor})\right)+\lambda\left(h(X_z)-h(\overline{Y}^M_{\lfloor z\rfloor})\right)\right]dz\\
 &+&\int_0^s\left(g(X_z)-g(\overline{Y}^M_{\lfloor z\rfloor})\right)dW_z+\int_0^s\left(h(X_z)-h(\overline{Y}^M_{\lfloor z\rfloor})\right)dN_z.
 \end{eqnarray*}
 The function $ k :\mathbb{R}^n\longrightarrow \mathbb{R}$, $x \longmapsto ||x||^2$ is twice differentiable. Applying It\^{o}'s formula for jump process to the process $X_s-\overline{Y}^M_s$ with the function $k$ leads to :
 \begin{eqnarray*}
 ||X_s-\overline{Y}^M_s||^2&=&2\int_0^s\left<X_z-\overline{Y}^M_z, v(X_z)-\dfrac{v(\overline{Y}^M_{\lfloor z\rfloor})}{1+\Delta t||v(\overline{Y}^M_{\lfloor z\rfloor})||}\right>dz+2\lambda\int_0^s\left\langle X_z-\overline{Y}^M_z, h(X_z)-h(\overline{Y}^M_{\lfloor z\rfloor})\right\rangle dz\\
 &+&2\int_0^s\left\langle X_z-\overline{Y}^M_z, u(X_z)-u(\overline{Y}^M_{\lfloor z\rfloor})\right\rangle dz +\sum_{i=1}^m\int_0^s||g_i(X_z)-g_i(\overline{Y}^M_{\lfloor z\rfloor})||^2dz\\
 &+&2\sum_{i=1}^m\int_0^s\left<X_z-\overline{Y}^M_z, g_i(X_z)-g_i(\overline{Y}^M_{\lfloor z\rfloor})\right>dW^i_z\\
 &+& \int_0^s\left[||X_z-\overline{Y}^M_z+h(X_z)-h(\overline{Y}^M_{\lfloor z\rfloor})||^2-||X_z-\overline{Y}^M_z||^2\right]dN_z.
 \end{eqnarray*}
 Using again the relation $d\overline{N}_z=dN_z-\lambda dz$ leads to 
 \begin{eqnarray}
 ||X_s-\overline{Y}^M_s||^2&=&2\int_0^s\left<X_z-\overline{Y}^M_z, v(X_z)-\dfrac{v(\overline{Y}^M_{\lfloor z\rfloor})}{1+\Delta t||v(\overline{Y}^M_{\lfloor z\rfloor})||}\right>dz+2\lambda\int_0^s\left\langle X_z-\overline{Y}^M_z, h(X_z)-h(\overline{Y}^M_{\lfloor z\rfloor})\right\rangle dz\nonumber\\
 &+&2\int_0^s\left\langle X_z-\overline{Y}^M_z, u(X_z)-u(\overline{Y}^M_{\lfloor z\rfloor})\right\rangle dz+\sum_{i=1}^m\int_0^s||g_i(X_z)-g_i(\overline{Y}^M_{\lfloor z\rfloor})||^2dz \nonumber\\
 &+&2\sum_{i=1}^m\int_0^s\left<X_z-\overline{Y}^M_z, g_i(X_z)-g_i(\overline{Y}^M_{\lfloor z\rfloor})\right>dW^i_z\nonumber\\
 &+& \int_0^s\left[||X_z-\overline{Y}^M_u+h(X_z)-h(\overline{Y}^M_{\lfloor z\rfloor})||^2-||X_z-\overline{Y}^M_z||^2\right]d\overline{N}_z\nonumber\\
 &+&\lambda\int_0^s\left[||X_z-\overline{Y}^M_z+h(X_z)-h(\overline{Y}^M_{\lfloor z \rfloor})||^2-||X_z-\overline{Y}^M_z||^2\right]dz\nonumber\\
 &=&A_1+A'_1+A_2+A_3+A_4+A_5+A_6.
 \label{ch5Th1}
 \end{eqnarray}
 In the next step, we give some useful estimations of $A_1, A'_1,  A_2, A_3$ and $A_6$. 
 \begin{eqnarray*}
 A_1&=&2\int_0^s\left\langle X_z-\overline{Y}^M_z,v(X_z)-\dfrac{v(\overline{Y}_{\lfloor z\rfloor})}{1+\Delta t||v(\overline{Y}^M_{\lfloor z\rfloor})||}\right\rangle dz\\
 &=&2\int_0^s<X_s-\overline{Y}^M_z,v(X_z)-v(\overline{Y}^M_z)>dz\\
 &+&2\int_0^s\left\langle X_s-\overline{Y}^M_z,v(\overline{Y}^M_z)-\dfrac{v(\overline{Y}^M_{\lfloor z\rfloor})}{1+\Delta t||v(\overline{Y}^M_{\lfloor z\rfloor})||}\right\rangle dz\\
 &=& A_{11}+A_{12}.
 \end{eqnarray*}
 Using the one-sided Lipschitz condition satiasfied by $v$ leads to
 \begin{eqnarray}
 A_{11} &=&2\int_0^s\langle X_s-\overline{Y}^M_z,v(X_z)-v(\overline{Y}^M_z)\rangle dz\nonumber\\
 &\leq& 2C\int_0^s||X_z-\overline{Y}^M_z||^2dz.
 \label{ch5ThA11}
 \end{eqnarray}
 Moreover, using the inequality $\langle a, b\rangle\leq |a||b|\leq \dfrac{a^2}{2}+\dfrac{b^2}{2}$ leads to :
 \begin{eqnarray}
 A_{12}&=& 2\int_0^s\left\langle X_z-\overline{Y}^M_z,v(\overline{Y}^M_z)-\dfrac{v(\overline{Y}^M_{\lfloor z\rfloor})}{1+\Delta t||v(\overline{Y}^M_{\lfloor z\rfloor})||}\right\rangle dz\nonumber\\
 &=&2\int_0^s\left\langle X_z-\overline{Y}^M_z, v(\overline{Y}^M_z)-v(\overline{Y}^M_{\lfloor z\rfloor})\right\rangle dz\nonumber\\
 &+&2\Delta t\int_0^s\left\langle X_z-\overline{Y}^M_z, \dfrac{v(\overline{Y}^M_{\lfloor z\rfloor})||v(\overline{Y}^M_{\lfloor z\rfloor})||}{1+\Delta t||v(\overline{Y}^M_{\lfloor z\rfloor})||}\right\rangle dz\nonumber\\
 &\leq &\int_0^s||X_z-\overline{Y}^M_z||^2dz+\int_0^s||v(\overline{Y}^M_z)-v(\overline{Y}^M_{\lfloor z\rfloor})||^2dz\nonumber\\
 &+&\int_0^s||X_z-\overline{Y}^M_z||^2dz+\dfrac{T^2}{M^2}\int_0^s||v(\overline{Y}^M_{\lfloor z\rfloor})||^4dz\nonumber\\
 &\leq &2\int_0^s||X_z-\overline{Y}^M_z||^2dz+\int_0^s||v(\overline{Y}^M_z)-v(\overline{Y}_{\lfloor z\rfloor})||^2dz\nonumber\\
 &+&\dfrac{T^2}{M^2}\int_0^s||v(\overline{Y}^M_{\lfloor z\rfloor})||^4dz
 \label{ch5ThA12}
 \end{eqnarray}
 Combining \eqref{ch5ThA11} and \eqref{ch5ThA12} give the following estimation for $A_1$ :
 \begin{eqnarray}
 A_1 &\leq & (2C+2)\int_0^s||X_z-\overline{Y}^M_z||^2dz+\int_0^s||v(\overline{Y}^M_z)-v(\overline{Y}_{\lfloor z\rfloor})||^2dz\nonumber\\
 &+&\dfrac{T^2}{M^2}\int_0^s||v(\overline{Y}^M_{\lfloor z\rfloor})||^4dz.
 \label{ch5ThA1}
 \end{eqnarray}
 Using again the inequality $\langle a, b\rangle\leq |a||b|\leq \dfrac{a^2}{2}+\dfrac{b^2}{2}$ and the global-Lipschitz condition satisfied by $u$ leads to :
 \begin{eqnarray}
 A_2 &=&2\int_0^s\left\langle X_z-\overline{Y}^M_z, u(X_z)-u(\overline{Y}^M_{\lfloor z\rfloor})\right\rangle dz\nonumber\\
 &=&2\int_0^s\left\langle X_z-\overline{Y}^M_z, u(X_z)-u(\overline{Y}^M_z)\right\rangle dz+\int_0^s\left\langle X_z-\overline{Y}^M_z, u(\overline{Y}^M_z)-u(\overline{Y}^M_{\lfloor z\rfloor})\right\rangle dz\nonumber\\
 &\leq &2C\int_0^s||X_z-\overline{Y}^M_z||^2dz+2C\int_0^s||\overline{Y}^M_z-\overline{Y}^M_{\lfloor z\rfloor}||^2dz.
 \label{ch5ThA2}
 \end{eqnarray} 
 Using the same arguments as for $A_2$ leads to the following estimation of $A'_1$
 \begin{eqnarray}
 A'_1&=&2\lambda\int_0^s\left\langle X_z-\overline{Y}^M_z, h(X_z)-h(\overline{Y}^M_{\lfloor z\rfloor})\right\rangle dz \nonumber\\
 &\leq& 2\lambda C\int_0^s||X_z-\overline{Y}^M_z||^2dz+2\lambda C\int_0^s||\overline{Y}^M_z-\overline{Y}^M_{\lfloor z\rfloor}||^2dz.
 \label{ch5ThA1n}
 \end{eqnarray}
 Using the inequalities $||g_i(x)-g_i(y)||\leq ||g(x)-g(y)||$ and $(a+b)^2\leq 2a^2+2b^2$ and the global Lipschitz condition satisfyed by $g$,  we have 
 \begin{eqnarray}
 A_3 &=&\sum_{i=1}^m\int_0^s||g_i(X_z)-g_i(\overline{Y}^M_{\lfloor z\rfloor})||^2dz\nonumber\\
 &\leq &m\int_0^s||g(X_z)-g(\overline{Y}^M_{\lfloor z\rfloor})||^2dz\nonumber\\
 &= &m\int_0^s||g(X_z)-g(\overline{Y}^M_z)+g(\overline{Y}^M_z)-g(\overline{Y}^M_{\lfloor z\rfloor})||^2dz\nonumber\\
 &\leq &2m\int_0^s||g(X_z)-g(\overline{Y}^M_z)||^2dz+2m\int_0^s||g(\overline{Y}^M_z)-g(\overline{Y}^M_{\lfloor z\rfloor})||^2dz\nonumber\\
 &\leq& 2mC^2\int_0^s||X_z-\overline{Y}^M_z||^2dz+2mC^2\int_0^s||\overline{Y}^M_z-\overline{Y}^M_{\lfloor z\rfloor}||^2dz
 \label{ch5ThA3}
 \end{eqnarray}
 Using the same  reasons as  above we obtain the following estimation for $A_6$ :
 \begin{eqnarray}
 A_6 &=&\lambda\int_0^s\left[X_z-\overline{Y}^M_z+h(X_z)-h(\overline{Y}^M_{\lfloor z\rfloor})||^2-||X_z-\overline{Y}^M_z||^2\right]dz\nonumber\\
 &\leq &3\lambda\int_0^s||X_z-\overline{Y}^M_z||^2dz+2\lambda\int_0^s||h(X_z)-h(\overline{Y}^M_{\lfloor z\rfloor})||^2dz\nonumber\\
 &\leq &3\lambda\int_0^s||X_z-\overline{Y}^M_z||^2dz+4\lambda\int_0^s||h(X_z)-h(\overline{Y}^M_z)||^2dz\nonumber\\
 &+ &4\lambda\int_0^s||h(\overline{Y}^M_z)-h(\overline{Y}^M_{\lfloor z\rfloor})||^2dz\nonumber\\
 &\leq &(3\lambda+4\lambda C^2)\int_0^s||X_z-\overline{Y}^M_z||^2dz+4\lambda C^2\int_0^s||\overline{Y}^M_z-\overline{Y}^M_{\lfloor z\rfloor}||^2dz.
 \label{ch5ThA6}
 \end{eqnarray}
 Inserting \eqref{ch5ThA1}, \eqref{ch5ThA2}, \eqref{ch5ThA1n}, \eqref{ch5ThA3} and \eqref{ch5ThA6} in \eqref{ch5Th1} we obtain :
  \begin{eqnarray*}
  ||X_s-\overline{Y}^M_s||^2&\leq &(4C+2+2mC^2+3\lambda+4\lambda C^2+2\lambda C)\int_0^s||X_z-\overline{Y}^M_z||^2dz\nonumber\\
  &+&(2C+2mC^2+4\lambda C^2+2\lambda C)\int_0^s||\overline{Y}^M_z-\overline{Y}^M_{\lfloor z\rfloor}||^2dz\nonumber\\
  &+&\int_0^s||v(\overline{Y}^M_z)-v(\overline{Y}^M_{\lfloor z\rfloor})||^2dz+\dfrac{T^2}{M^2}\int_0^s||v(\overline{Y}^M_{\lfloor z\rfloor})||^4dz\nonumber\\
  &+&2\sum_{i=1}^m\int_0^s\left<X_z-\overline{Y}^M_z, g_i(X_z)-g_i(\overline{Y}^M_{\lfloor z\rfloor})\right>dW^{i}_z\nonumber\\
  &+&\int_0^s\left[||X_z-\overline{Y}^M_z+h(X_z)-h(\overline{Y}^M_{\lfloor z\rfloor})||^2-||X_z-\overline{Y}^M_z||^2\right]d\overline{N}_z.
  \end{eqnarray*}

Taking the supremum in both sides of the previous inequality leads to
\begin{eqnarray}
\sup_{s\in[0,t]}||X_s-\overline{Y}^M_s||^2&\leq &(4C+2+2mC^2+3\lambda+4\lambda C^2+2\lambda C)\int_0^t||X_z-\overline{Y}^M_z||^2dz\nonumber\\
  &+&(2C+2mC^2+4\lambda C^2+2\lambda C)\int_0^t||\overline{Y}^M_z-\overline{Y}^M_{\lfloor z\rfloor}||^2dz\nonumber\\
  &+&\int_0^t||v(\overline{Y}^M_z)-v(\overline{Y}^M_{\lfloor z\rfloor})||^2dz+\dfrac{T^2}{M^2}\int_0^t||v(\overline{Y}^M_{\lfloor z\rfloor})||^4dz\nonumber\\
  &+&2\sup_{s\in[0,t]}\left|\sum_{i=1}^m\int_0^s\left<X_z-\overline{Y}^M_z, g_i(X_z)-g_i(\overline{Y}^M_{\lfloor z\rfloor})\right>dW^{i}_z\right|\nonumber\\
  &+&\sup_{s\in[0,t]}\left|\int_0^s\left[||X_z-\overline{Y}^M_z+h(X_z)-h(\overline{Y}^M_{\lfloor z\rfloor})||^2\right]d\overline{N}_z\right|\nonumber\\
  &+&\sup_{s\in[0,t]}\left|\int_0^s||X_z-\overline{Y}^M_z||^2d\overline{N}_z\right|.
  \label{ch5Th2}
\end{eqnarray}
Using  Lemma \ref{ch4lemma15} we  have the following estimations 
 \begin{eqnarray*}
 B_1& :=&\left\|2\sup_{s\in[0,t]}\left|\sum_{i=1}^m\int_0^s\left<X_z-\overline{Y}^M_z, g_i(X_z)-g_i(\overline{Y}^M_{\lfloor z\rfloor})\right>dW^i_z\right|\right\|_{L^{p/2}(\Omega, \mathbb{R})}\\
 &\leq &C_p\left(\int_0^t\sum_{i=1}^m\left\|\left<X_z-\overline{Y}^M_z, g_i(X_z)-g_i(\overline{Y}^M_{\lfloor z\rfloor})\right>\right\|^2_{L^{p/2}(\Omega, \mathbb{R}^d)}dz\right)^{1/2},
 \end{eqnarray*}
for all $p\in[2,\infty)$.

Moreover, using  Holder inequality, the inequality $ab\leq \dfrac{a^2}{2}+\dfrac{b^2}{2}$ and the global Lipschitz condition satisfied by $g$, we  have the following estimation for $B_1$.
 \begin{eqnarray}
 B_1 &\leq&  C_p\left(\int_0^t\sum_{i=1}^m\left\|\left<X_z-\overline{Y}^M_z, g_i(X_z)-g_i(\overline{Y}^M_{\lfloor z\rfloor})\right>\right\|^2_{L^{p/2}(\Omega, \mathbb{R}^d)}dz\right)^{1/2}\nonumber\\ 
 &\leq &C_p\left(\int_0^t\sum_{i=1}^m||X_z-\overline{Y}_z^M||^2_{L^p(\Omega, \mathbb{R}^d)}||g_i(X_z)-g_i(\overline{Y}^M_{\lfloor z\rfloor})||^2_{L^p(\Omega, \mathbb{R}^d)}dz\right)^{1/2}\nonumber\\
 &\leq &C_p\left(\int_0^t\dfrac{1}{2}||X_z-\overline{Y}_z^M||^2_{L^p(\Omega, \mathbb{R}^d)}2m||g(X_z)-g(\overline{Y}^M_{\lfloor z\rfloor})||^2_{L^p(\Omega, \mathbb{R}^d)}dz\right)^{1/2}\nonumber\\
 &\leq &\dfrac{C_p}{\sqrt{2}}\left(\sup_{s\in[0,t]}||X_s-\overline{Y}^M_s||_{L^p(\Omega, \mathbb{R}^d)}\right)\left(2C^2m\int_0^t||X_s-\overline{Y}^M_{\lfloor s\rfloor}||^2_{L^p(\Omega, \mathbb{R}^d)}dz\right)^{1/2}\nonumber\\
 &\leq &\dfrac{1}{4}\sup_{s\in[0,t]}||X_s-\overline{Y}^M_s||^2_{L^p(\Omega, \mathbb{R}^d)}+C_p^2m\int_0^t||X_s-\overline{Y}^M_{\lfloor s\rfloor}||^2_{L^p(\Omega,\mathbb{R}^d)}dz\nonumber\\
 &\leq &\dfrac{1}{4}\sup_{s\in[0,t]}||X_s-\overline{Y}^M_s||^2_{L^p(\Omega, \mathbb{R}^d)}+2C_p^2m\int_0^t||X_s-\overline{Y}^M_s||^2_{L^p(\Omega,\mathbb{R}^d)}dz\nonumber\\
 &+&2C_p^2m\int_0^t||\overline{Y}^M_s-\overline{Y}^M_{\lfloor s\rfloor}||^2_{L^p(\Omega,\mathbb{R}^d)}dz.
 \label{ch5ThB1}
 \end{eqnarray}
 Using Lemma \ref{ch4lemma18} and the inequality $(a+b)^4\leq 4a^4+4b^4$, it follows that :
 \begin{eqnarray*}
 B_2 &=&\left\|\sup_{s\in[0,t]}\left|\int_0^s||X_z-\overline{Y}^M_z+h(X_z)-h(\overline{Y}^M_{\lfloor z\rfloor})||^2d\overline{N}_z\right|\right\|_{L^{p/2}(\Omega, \mathbb{R}^d)}\nonumber\\
 &\leq &C_p\left(\int_0^t||X_z-\overline{Y}^M_z+h(X_z)-h(\overline{Y}^M_{\lfloor z\rfloor})||^4_{L^{p/2}(\Omega, \mathbb{R}^d)}dz\right)^{1/2}\nonumber\\
 &\leq &C_p\left(\int_0^t\left[4||X_z-\overline{Y}^M_z||^4_{L^{p/2}(\Omega, \mathbb{R}^d)}+4||h(X_z)-h(\overline{Y}^M_{\lfloor z\rfloor})||^4_{L^{p/2}(\Omega, \mathbb{R}^d)}\right]dz\right)^{1/2},
 \end{eqnarray*}
 for all $p\in[2,\infty)$.
 
 Using the inequality $\sqrt{a+b}\leq\sqrt{a}+\sqrt{b}$, it follows that :
 \begin{eqnarray}
 B_2&\leq &2C_p\left(\int_0^t||X_z-\overline{Y}^M_z||^4_{L^{p/2}(\Omega, \mathbb{R}^d)}dz\right)^{1/2} +2C_p\left(\int_0^t||h(X_z)-h(\overline{Y}^M_{\lfloor z\rfloor})||^4_{L^{p/2}(\Omega, \mathbb{R}^d)}dz\right)^{1/2}\nonumber\\
 & =& B_{21}+B_{22}.
 \label{ch5ThB}
 \end{eqnarray}
 Using Holder inequality, it follows that :
 \begin{eqnarray*}
 B_{21} &: = &2 C_p\left(\int_0^t||X_z-\overline{Y}^M_z||^4_{L^{p/2}(\Omega, \mathbb{R}^d)}dz\right)^{1/2}\nonumber\\
 &\leq &2C_p\left(\int_0^t||X_z-\overline{Y}^M_z||^2_{L^p(\Omega, \mathbb{R}^d)}||X_z-\overline{Y}^M_z||^2_{L^p(\Omega, \mathbb{R}^d)}dz\right)^{1/2}\nonumber\\
 &= &2C_p\left(\int_0^t\dfrac{1}{16}||X_z-\overline{Y}^M_z||^2_{L^p(\Omega, \mathbb{R}^d)}16||X_z-\overline{Y}^M_z||^2_{L^p(\Omega, \mathbb{R}^d)}dz\right)^{1/2}\nonumber\\ 
 &\leq &\dfrac{1}{4}\sup_{s\in[0,t]}||X_z-\overline{Y}^M_z||_{L^p(\Omega,\mathbb{R}^d)}8C_p\left(\int_0^t||X_z-\overline{Y}^M_z||^2_{L^p(\Omega,\mathbb{R}^d)}dz\right)^{1/2}.
 \end{eqnarray*}
 Using the inequality $2ab\leq a^2+b^2$ leads to :
 \begin{eqnarray}
 B_{21}&\leq &\dfrac{1}{16}\sup_{s\in[0,t]}||X_s-\overline{Y}^M_s||^2_{L^p(\Omega, \mathbb{R}^d)}
 +16C_p\int_0^t||X_s-\overline{Y}^M_z||^2_{L^p(\Omega, \mathbb{R}^d)}dz.
 \label{ch5ThB21}
 \end{eqnarray}
 Using the inequalities $(a+b)^4\leq 4a^4+4b^4$ and $\sqrt{a+b}\leq\sqrt{a}+\sqrt{b}$,  we have the following bound for $B_{22}$
 \begin{eqnarray*}
 B_{22}&: = &2C_p\left(\int_0^t||h(X_z)-h(\overline{Y}^M_{\lfloor z\rfloor})||^4_{L^{p/2}(\Omega, \mathbb{R}^d)}dz\right)^{1/2}\nonumber\\
 &\leq &2C_p\left(\int_0^t4||h(X_z)-h(\overline{Y}^M_z)||^4_{L^{p/2}(\Omega,\mathbb{R}^d)}+4||h(\overline{Y}^M_z)-h(\overline{Y}^M_{\lfloor z\rfloor})||^4_{L^{p/2}(\Omega,\mathbb{R}^d)}dz\right)^{1/2}\\
 &\leq &4C_p\left(\int_0^t||h(X_z)-h(\overline{Y}^M_z)||^4_{L^{p/2}(\Omega, \mathbb{R}^d)}dz\right)^{1/2}\\
 &+&4C_p\left(\int_0^t||h(\overline{Y}^M_z)-h(\overline{Y}^M_{\lfloor z\rfloor})||^4_{L^{p/2}(\Omega, \mathbb{R}^d)}dz\right)^{1/2}.
 \end{eqnarray*}
 Using the global Lipschitz condition satisfied by $h$ leads to :
 \begin{eqnarray*}
 B_{22}&\leq &4C_p\left(\int_0^tC^4||X_z-\overline{Y}^M_z||^4_{L^{p/2}(\Omega, \mathbb{R}^d)}dz\right)^{1/2}+4C_p\left(\int_0^tC^4||\overline{Y}^M_z-\overline{Y}^M_{\lfloor z\rfloor}||^4_{L^{p/2}(\Omega, \mathbb{R}^d)}dz\right)^{1/2}.
 \end{eqnarray*}
 Using the same idea as for $B_{21}$, it follows that :
 \begin{eqnarray}
 B_{22} &\leq &\dfrac{1}{16}\sup_{s\in[0,t]}||X_s-\overline{Y}^M_s||^2_{L^p(\Omega, \mathbb{R}^d)} +64C_p\int_0^t||X_z-\overline{Y}^M_z||^2_{L^p(\Omega, \mathbb{R}^d)}dz\nonumber\\
 &+&\dfrac{1}{16}\sup_{s\in[0,t]}||\overline{Y}^M_s-\overline{Y}^M_{\lfloor s\rfloor}||^2_{L^p(\Omega, \mathbb{R}^d)}+64C_p\int_0^t||\overline{Y}^M_z-\overline{Y}^M_{\lfloor z\rfloor}||^2_{L^p(\Omega, \mathbb{R}^d)}dz.\nonumber
 \end{eqnarray}
 Taking the supremum under the integrand of the last term in the above inequality and using the fact that we don't care about the value of the constant  leads to :
 \begin{eqnarray}
 B_{22}&\leq &\dfrac{1}{16}\sup_{s\in[0,t]}||X_s-\overline{Y}^M_s||^2_{L^p(\Omega, \mathbb{R}^d)} +64C_p\int_0^t||X_s-\overline{Y}^M_s||^2_{L^p(\Omega, \mathbb{R}^d)}ds\nonumber\\
 &+&C_p\sup_{s\in[0,t]}||\overline{Y}^M_s-\overline{Y}^M_{\lfloor s\rfloor}||^2_{L^p(\Omega, \mathbb{R}^d)}.
 \label{ch5ThB22}
 \end{eqnarray}
 Inserting \eqref{ch5ThB21} and \eqref{ch5ThB22} into \eqref{ch5ThB} gives :
 \begin{eqnarray}
 B_2&\leq &\dfrac{1}{8}\sup_{s\in[0,t]}||X_s-\overline{Y}^M_s||^2_{L^p(\Omega, \mathbb{R}^d)}+C_p\int_0^t||X_s-\overline{Y}^M_s||^2_{L^p(\Omega, \mathbb{R}^d)}ds\nonumber\\
 &+&C_p\sup_{s\in[0,t]}||\overline{Y}^M_s-\overline{Y}^M_{\lfloor s\rfloor}||^2_{L^p(\Omega, \mathbb{R}^d)}.
 \label{ch5ThB2}
 \end{eqnarray}
 Using again Lemma \ref{ch4lemma18} leads to :
 \begin{eqnarray*}
B_3 &: =& \left\|\sup_{s\in[0,t]}\int_0^s||X_z-\overline{Y} ^M_z||^2d\overline{N}_z\right\|_{L^{p/2}(\Omega, \mathbb{R}^d)}\\
&\leq &C_p\left(\int_0^t||X_z-\overline{Y}^M_z||^4_{L^{p/2}(\Omega, \mathbb{R}^d)}dz\right)^{1/2}.
 \end{eqnarray*}
 Using the same argument as for $B_{21}$, we obtain :
 \begin{eqnarray}
 B_3 &\leq &\dfrac{1}{8}\sup_{s\in[0,t]}||X_s-\overline{Y}^M_s||^2_{L^p(\Omega, \mathbb{R}^d)}+C_p\int_0^t||X_s-\overline{Y}^M_s||^2_{L^p(\Omega, \mathbb{R}^d)}ds
 \label{ch5ThB3}
 \end{eqnarray}
 Taking the $L^{p}$ norm in both  side of \eqref{ch5Th2} and inserting inequalities \eqref{ch5ThB1}, \eqref{ch5ThB2} and \eqref{ch5ThB3} leads to :
\begin{eqnarray*}
\left\|\sup_{s\in[0,t]}||X_s-\overline{Y}^M_s||\right\|^2_{L^{p}(\Omega, \mathbb{R}^d)}&=&\left\|\sup_{s\in[0,t]}||X_s-\overline{Y}^M_s||^2\right\|_{L^{p/2}(\Omega, \mathbb{R}^d)}
\end{eqnarray*}
  \begin{eqnarray*}
  &\leq & C_p\int_0^t||X_s-\overline{Y}^M_s||^2_{L^p(\Omega, \mathbb{R}^d)}ds+C_p\int_0^t||\overline{Y}^M_s-\overline{Y}^M_{\lfloor s\rfloor}||^2_{L^p(\Omega,\mathbb{R}^d)}ds\\
 &+&\int_0^t||v(X_s)-v(\overline{Y}^M_{\lfloor s\rfloor})||^2_{L^{p}(\Omega, \mathbb{R}^d)}ds+C_p\sup_{s\in[0,t]}||\overline{Y}^M_s-\overline{Y}^M_{\lfloor s\rfloor}||^2_{L^p(\Omega, \mathbb{R}^d)}\\
 &+&\dfrac{T^2}{M^2}\int_0^t||v(\overline{Y}^M_{\lfloor s\rfloor})||^4_{L^{2p}(\Omega, \mathbb{R}^d)}ds +2C_p\int_0^t||\overline{Y}^M_s-\overline{Y}^M_{\lfloor s\rfloor}||^2_{L^p(\Omega, \mathbb{R}^d)}ds\\
 &+&\dfrac{1}{2}\left\|\sup_{s\in[0,t]}||X_s-\overline{Y}^M_s||\right\|^2_{L^{p/2}(\Omega, \mathbb{R}^d)},
 \end{eqnarray*} 
 for all $t\in[0,T]$ and all $p\in[2,+\infty)$. Where $C_p$ is the generic constant.
  
 The previous inequality can  be writen in the following appropriate form 
 
 $
\dfrac{1}{2}\left\|\sup\limits_{s\in[0,t]}||X_s-\overline{Y}^M_s||\right\|^2_{L^{p}(\Omega, \mathbb{R}^d)}
$
  \begin{eqnarray*}
  &\leq & C_p\int_0^t||X_s-\overline{Y}^M_s||^2_{L^p(\Omega, \mathbb{R}^d)}ds+C_p\int_0^t||\overline{Y}^M_s-\overline{Y}^M_{\lfloor s\rfloor}||^2_{L^p(\Omega,\mathbb{R}^d)}ds\\
 &+&\int_0^t||v(X_s)-v(\overline{Y}^M_{\lfloor s\rfloor})||^2_{L^{p}(\Omega, \mathbb{R}^d)}ds+C_p\sup_{s\in[0,t]}||\overline{Y}^M_s-\overline{Y}^M_{\lfloor s\rfloor}||^2_{L^p(\Omega, \mathbb{R}^d)}\\
 &+&\dfrac{T^2}{M^2}\int_0^t||v(\overline{Y}^M_{\lfloor s\rfloor})||^4_{L^{2p}(\Omega, \mathbb{R}^d)}ds +2C^2m\int_0^t||\overline{Y}^M_s-\overline{Y}^M_{\lfloor s\rfloor}||^2_{L^p(\Omega, \mathbb{R}^d)}ds.
 \end{eqnarray*}

 Applying Gronwall lemma to the previous inequality leads to :
 
 $
 \dfrac{1}{2}\left\|\sup\limits_{s\in[0,t]}||X_s-\overline{Y}^M_s||\right\|^2_{L^p(\Omega, \mathbb{R}^d)} 
$

 \begin{eqnarray*}
 &\leq &C_pe^{C_p}\left(\int_0^T||v(\overline{Y}^M_s)-v(\overline{Y}^M_{\lfloor s\rfloor})||^2_{L^p(\Omega, \mathbb{R}^d)}ds+C_p\sup_{u\in[0,t]}||\overline{Y}^M_u-\overline{Y}^M_{\lfloor u\rfloor}||^2_{L^p(\Omega, \mathbb{R}^d)}\right.\\
 & &\left. +\dfrac{T^2}{M^2}\int_0^T||v(\overline{Y}^M_{\lfloor s\rfloor})||^4_{L^{2p}(\Omega, \mathbb{R}^d)}ds +C_p\int_0^T||\overline{Y}^M_s-\overline{Y}^M_{\lfloor s\rfloor}||^2_{L^p(\Omega, \mathbb{R}^d)}ds\right).
 \end{eqnarray*}
 Using the inequality $\sqrt{a+b+c}\leq \sqrt{a}+\sqrt{b}+\sqrt{c}$, it follows that

$
 \left\|\sup\limits_{t\in[0, T]}||X_t-\overline{Y}^M_t||\right\|_{L^p(\Omega, \mathbb{R}^d)}
 $
 \begin{eqnarray}
 &\leq &C_pe^{C_p}\left(\sup_{s\in[0,T]}||v(\overline{Y}^M_s)-v(\overline{Y}^M_{\lfloor s\rfloor})||_{L^p(\Omega, \mathbb{R}^d)}+C_p\sup_{s\in[0,T]}||\overline{Y}^M_s-\overline{Y}^M_{\lfloor s\rfloor}||_{L^p(\Omega, \mathbb{R}^d)} \right.\nonumber\\
 &+& \left. \dfrac{T^2}{M}\left[\sup_{n\in\{0,\cdots, M\}}||v(Y^M_n)||^2_{L^{2p}(\Omega, \mathbb{R}^d)}\right]+C_p\sup_{s\in[0,T]}||\overline{Y}^M_s-\overline{Y}^M_{\lfloor s\rfloor}||_{L^p(\Omega, \mathbb{R}^d)}\right),
 \label{ch5Gronwall}
 \end{eqnarray}
 for all $p\geq 2$. Using Lemma \ref{ch5lemma14} and Lemma \ref{ch5lemma16}, it follows that :
 \begin{eqnarray*}
 \mathbb{E}\left[\sup_{t\in[0,T]}\left\|X_t-\overline{Y}^M_t\right\|^p\right]^{1/p}\leq C_p(\Delta t)^{1/2},
 \end{eqnarray*}
for all $M\in\mathbb{N}$ and all $p\in[2,\infty)$. The application of Holder inequality shows that the latter inequality is satisfied for all $p\in[1,\infty)$, and this complete the proof of Theorem \ref{ch5theorem2}.
 \end{proof}

 \section{Strong convergence of the tamed Euler Scheme}
 \begin{thm}\label{ch5theorem2b}
 \hspace{0.5cm}Under Assumptions \ref{ch5assumption1}, for all $p\in[1, +\infty)$ there exist a constant $C_p>0$ such that
 \begin{eqnarray}
 \left(\mathbb{E}\left[\sup_{t\in[0,T]}\left\|X_t-\overline{X}^M_t\right\|^p\right]\right)^{1/p}\leq C_p\Delta t^{1/2},
 \end{eqnarray}
 for all $M\in\mathbb{N}$. Where $X : [0, T]\times\Omega \longrightarrow \mathbb{R}^d$ is the exact solution of \eqref{ch5exactsol1} and $\overline{X}^M_t$ the continous interpolation  of the numerical solution \eqref{ch5tam} defined by :
 \begin{eqnarray*}
 \overline{X}^M_t &=&X^M_n+\dfrac{(t-n\Delta t)f(X^M_n)}{1+||f(X^M_n)||}+g(X^M_n)(W_t-W_{n\Delta t})+h(X^M_n)(\overline{N}_t-\overline{N}_{n\Delta t}),
 \label{ch5continoussol1}
 \end{eqnarray*}
 for all $M\in\mathbb{N}$, all $n\in\{0,\cdots, M\}$ and all $t\in[n\Delta t, (n+1)\Delta t)$.
 \end{thm}
 
   \begin{proof}
    Using the relation $\Delta\overline{N}^M_n=\Delta N^M_n-\lambda\Delta t$,  the continous interpolation of \eqref{ch5tam} can  be express in the following form
    \begin{eqnarray*}
    \overline{X}^M_t =X^M_n+\lambda(t-n\Delta t)h(X^M_n)+\dfrac{(t-n\Delta t)f(X^M_n)}{1+||f(X^M_n)||}+g(X^M_n)(W_t-W_{n\Delta t})+h(X^M_n)(N_t-N_{n\Delta t}),
    \end{eqnarray*}
    for all $t\in[n\Delta t, (n+1)\Delta t[$.
    
   From the numerical solution \eqref{ch5tam} and using the relation $\Delta N^M_n=\Delta\overline{N}^M_n+\lambda\Delta t$, it follows that:
   \begin{eqnarray}
   X^M_{n+1}&=&X^M_n+\dfrac{\Delta tf(X^M_n)}{1+\Delta t||f(X^M_n)||}+g(X^M_n)\Delta W^M_n+h(X^M_n)\Delta N^M_n\nonumber\\
   &=&X^M_n+\lambda h(X^M_n)T/M+\dfrac{\Delta tf(X^M_n)}{1+\Delta t||f(X^M_n)||}+g(X^M_n)\Delta W^M_n\nonumber\\
   &+&h(X^M_n)\Delta \overline{N}^M_n.
   \label{ch5contam}
   \end{eqnarray}
  The functions $\lambda h$ and $f$ satisfy the same conditions as $u_{\lambda}$ and $v$ respectively. So from \eqref{ch5tam} it follows that the numerical solution \eqref{ch5tam} satisfied the same hypothesis as the numerical solution \eqref{ch5semitam}. Hence, it follows from Theorem \ref{ch5theorem2} that there exist a constant $C_p>0$ such that  
  \begin{eqnarray}
 \left(\mathbb{E}\left[\sup_{t\in[0,T]}\left\|X_t-\overline{X}^M_t\right\|^p\right]\right)^{1/p}\leq C_p\Delta t^{1/2},
 \end{eqnarray}
 for all $p\in[1,\infty)$.
   \end{proof}

\section{Linear mean-square stability}   
    \hspace{0.5cm}The goal of  this section is to find a  stepsize for which the tamed euler scheme and the semi-tamed Euler scheme are stable.
The first approach to the stability analysis of a numerical method is to study the stability behavior of the method for a scalar linear equation. So we will focus in the linear case.
Let's consider a linear test equation with real and scalar coefficients.
\begin{eqnarray}
 dX(t)=  aX(t^{-})dt +bX(t^{-})dW(t)+cX(t^{-})dN(t),  \hspace{0.5cm}
 X(0)=X_0.
 \label{ch5linearequation}
\end{eqnarray}
It is proved in \cite{Desmond2} that the exact solution of \eqref{ch5linearequation} is mean-square stable if and only if

 $l :=2a+b^2+\lambda c(2+c)<0$. That is :
\begin{eqnarray}
\label{ch5stabilitycondition}
\lim_{t\longrightarrow \infty}|X(t)|^2=0 \Longleftrightarrow l := 2a+b^2+\lambda c(2+c) <0.
\end{eqnarray}
 \begin{rem}
 In this section, to easy notations, the tamed Euler  appoximation $X^M$ will be replaced by $X$ and the semi-tamed Euler approximation $Y^M$ will be replaced by $Y$.
 \end{rem}
We have the following result for the numerical method \eqref{ch5semitam}.
 
 \begin{thm}
 Under Assumption \ref{ch5stabilitycondition}, the semi-tamed Euler \eqref{ch5semitam} is mean-square stable if and only if 
 \begin{eqnarray*}
 \Delta t<\dfrac{-l}{(a+\lambda c)^2}.
 \end{eqnarray*}
 \end{thm}
 
 \begin{proof}
 Aplying the  semi-tamed Euler scheme to \eqref{ch5linearequation} leads to
 \begin{eqnarray}
 Y_{n+1}=Y_n+aY_n\Delta t+\lambda cY_n\Delta t+bY_n\Delta W_n+cY_n\overline{N}_n.
 \label{ch5compen1}
 \end{eqnarray}
 Squaring both sides of \eqref{ch5compen1} leads to
 \begin{eqnarray}
 Y_{n+1}^2&=&Y_n^2+(a+\lambda c)^2\Delta t^2Y_n^2+b^2Y_n^2\Delta W_n^2+c^2Y_n^2\Delta\overline{N}_n^2+2(a+\lambda c)Y_n^2\Delta t+2bY_n^2\Delta W_n\nonumber\\
 &+&2cY_n^2\Delta\overline{N}_n+2b(a+\lambda c)Y_n^2\Delta t\Delta W_n+2c(a+\lambda c)Y_n^2\Delta t\Delta\overline{N}_n+2bcY_n^2\Delta W_n\Delta\overline{N}_n.
 \label{ch5compen2}
 \end{eqnarray}
 Taking expectation in both sides of \eqref{ch5compen2} and using the relations  $\mathbb{E}(\Delta W_n^2)=\Delta t$, $\mathbb{E}(\Delta\overline{N}_n^2)=\lambda \Delta t$ and  $\mathbb{E}(\Delta W_n)=\mathbb{E}(\Delta\overline{N}_n)=0$ leads to
 \begin{eqnarray*}
 \mathbb{E}|Y_{n+1}|^2=(1+(a+\lambda c)^2\Delta t^2+(b^2+\lambda c^2+2a+2\lambda c)\Delta t)\mathbb{E}|Y_n|^2.
 \end{eqnarray*}
 So the numerical method is stable if and only if
 \begin{eqnarray*}
 1+(a+\lambda c)^2\Delta t^2+(b^2+\lambda c^2+2a+2\lambda c)\Delta t<1.
 \end{eqnarray*}
 That is if and only if
 \begin{eqnarray*}
 \Delta t<\dfrac{-l}{(a+\lambda c)^2}.
 \end{eqnarray*}
 \end{proof}

 \begin{thm}\label{thmncts}
Under  Assumption \ref{ch5stabilitycondition},  the tamed Euler scheme \eqref{ch5tam} is mean-square stable if one of the following conditions  is satisfied :
\begin{itemize}
\item $a(1+\lambda c\Delta t)\leq 0$, $2a-l>0$ and $\Delta t<\dfrac{2a-l}{a^2+\lambda^2c^2}$.
\item $a(1+\lambda c\Delta t)>0$ and $\Delta t<\dfrac{-l}{(a+\lambda c)^2}$.
\end{itemize}

\end{thm}
\begin{proof}
Applying the tamed Euler scheme \eqref{ch5tam} to equation \eqref{ch5linearequation} leads to :
\begin{eqnarray}
X_{n+1}=X_n+\dfrac{aX_n\Delta t}{1+\Delta t|aX_n|}+bX_n\Delta W_n+cX_n\Delta N_n.
\label{ch5eq1}
\end{eqnarray}
By squaring both sides of \eqref{ch5eq1} leads to :
\begin{eqnarray*}
X_{n+1}^2&=&X_n^2+\dfrac{a^2X^2_n\Delta t^2}{(1+\Delta t|aX_n|)^2}+b^2X_n^2\Delta W_n^2+c^2X_n^2\Delta N_n^2+\dfrac{2aX_n^2\Delta t}{1+\Delta t|aX_n|}+2bX_n^2\Delta W_n\\
&+&2cX_n^2\Delta N_n+\dfrac{2abX_n^2}{1+\Delta t|aX_n|}\Delta W_n+\dfrac{2acX_n^2\Delta t}{1+\Delta t|aX_n|}\Delta N_n+2bcX_n^2\Delta W_n\Delta N_n.
\end{eqnarray*}
Using the inequality $ \dfrac{a^2\Delta t^2}{1+\Delta t|aX_n|}\leq a^2\Delta t^2$, the previous equality becomes 
\begin{eqnarray*}
X_{n+1}^2&\leq &X_n^2+a^2X2_n\Delta t^2+b^2X_n^2\Delta W_n^2+c^2X_n^2\Delta N_n^2+\dfrac{2aX_n^2\Delta t}{1+\Delta t|aX_n|}+2bX_n^2\Delta W_n\\
&+&2cX_n^2\Delta N_n+\dfrac{2abX_n^2}{1+\Delta t|aX_n|}\Delta W_n+\dfrac{2acX_n^2\Delta t}{1+\Delta t|aX_n|}\Delta N_n+2bcX_n^2\Delta W_n\Delta N_n.
\end{eqnarray*}
Taking expectation in both sides of the previous equality and using independence and  the fact that $\mathbb{E}(\Delta W_n)=0$, $\mathbb{E}(\Delta W_n^2)=\Delta t$, $\mathbb{E}(\Delta N_n)=\lambda\Delta t$, $\mathbb{E}(\Delta N_n^2)=\lambda \Delta t+\lambda^2\Delta t^2$ leads to :
\begin{eqnarray}
\mathbb{E}|X_{n+1}|^2&\leq& \left[1+a^2\Delta t^2+b^2\Delta t+\lambda^2c^2\Delta t^2+(2+ c)\lambda c\Delta t\right]\mathbb{E}|X_n|^2\nonumber\\
&+&\mathbb{E}\left(\dfrac{2aX^2_n\Delta t(1+\lambda c\Delta t)}{1+\Delta t|aX_n|}\right).
\label{ch5eq2}
\end{eqnarray}
\begin{itemize}
\item  If $a(1+\lambda c\Delta t)\leq 0$, it follows from \eqref{ch5eq2} that
\begin{eqnarray*}
\mathbb{E}|X_{n+1}|^2\leq \{1+(a^2+\lambda^2c^2)\Delta t^2+[b^2+\lambda c(2+c)]\Delta t\}\mathbb{E}|X_n|^2.
\end{eqnarray*}
Therefore, the numerical solution is stable if 
\begin{eqnarray*}
1+(a^2+\lambda^2c^2)\Delta t^2+[b^2+\lambda c(2+c)]\Delta t<1.
\end{eqnarray*}
That is  $\Delta t<\dfrac{2a-l}{a^2+\lambda^2c^2}$.
\item If $a(1+\lambda c\Delta t)> 0$, using the fact that $\dfrac{2aX_n^2\Delta t(1+\lambda c\Delta t)}{1+\Delta t|aX_n|}< 2aY_n^2\Delta t(1+\lambda c\Delta t)$, inequality \eqref{ch5eq2} becomes
\begin{eqnarray}
\mathbb{E}|X_{n+1}|^2\leq\left[1+a^2\Delta t^2+b^2\Delta t+\lambda^2c^2\Delta t^2+2\lambda ac\Delta t^2+(2+ c)\lambda c\Delta t+2a\Delta t\right]\mathbb{E}|X_n|^2.
\label{ch5eq3}
\end{eqnarray}
Therefore, it follows from \eqref{ch5eq3} that the numerical solution is stable if
 
$1+a^2\Delta t^2+b^2\Delta t+\lambda^2c^2\Delta t^2+2\lambda ac\Delta t^2+(2+ c)\lambda c\Delta t+2a\Delta t<1$. That is $\Delta t<\dfrac{-l}{(a+\lambda c)^2}$.

\end{itemize}
\end{proof}

\begin{rem}
In Theorem \ref{thmncts}, we can  easily check that if $l<0$, we have:
\begin{eqnarray}
\left\lbrace \begin{array}{l}
 a(1+\lambda c\Delta t)\leq0,\\
 2a-l>0 \\
 \Delta t<\dfrac{2a-l}{a^2+\lambda^2c^2}\\
 \end{array} \right. 
 \Leftrightarrow 
 \left\lbrace \begin{array}{l}
   a \in (l/2,0], c\geq 0,\\
   \Delta t <\dfrac{2a-l}{a^2+\lambda^2 c^2}\\
 \end{array} \right.
 \bigcup
 \left\lbrace \begin{array}{l}
   a \in (l/2,0), c<0,\\
   \Delta t <\dfrac{2a-l}{a^2+\lambda^2 c^2} \\
   \Delta t\leq \dfrac{-1}{ \lambda c} \\
 \end{array} \right.\\
 \bigcup
 \left\lbrace \begin{array}{l}
   a>0, c<0 \\
   \Delta t <\dfrac{2a-l}{a^2+\lambda^2 c^2}  \\
  \Delta t \geq  \dfrac{-1}{ \lambda c}
 \end{array} \right. 
\end{eqnarray}
\begin{eqnarray}
\left\lbrace \begin{array}{l}
 a(1+\lambda c\Delta t)>0,\\
 \Delta t<\dfrac{-l}{(a+\lambda c)^2}\\
 \end{array} \right. 
 \Leftrightarrow 
 \left\lbrace \begin{array}{l}
   a >0, c>0,\\
   \Delta t < \dfrac{-l}{(a+\lambda c)^2}\\
 \end{array} \right.
 \bigcup
 \left\lbrace \begin{array}{l}
   a >0, c<0,\\
   \Delta t <\dfrac{-l}{(a+\lambda c)^2} \wedge \dfrac{-1}{ \lambda c} \\
 \end{array} \right.\\
 \bigcup
 \left\lbrace \begin{array}{l}
   a<0, c<0 \\
   \Delta t <\dfrac{-l}{(a+\lambda c)^2}  \\
  \Delta t >  \dfrac{-1}{ \lambda c}
 \end{array} \right. 
\end{eqnarray}
 \end{rem}

\section{Nonlinear mean-square stability}
\hspace{0.5cm}In this section, we focus on the mean-square stability of the approximation \eqref{ch5semi}. It is proved in \cite{Desmond2} that under the following conditions, 
\begin{eqnarray*}
\langle x-y, f(x)-f(y)\rangle&\leq& \mu||x-y||^2,\\
||g(x)-g(y)||^2&\leq& \sigma||x-y||^2,\\
||h(x)-h(y)||^2 &\leq &\gamma ||x-y||^2,
\end{eqnarray*}
 where $\mu$, $\sigma$ and $\gamma$ are constants, the eaxct solution of SDE \eqref{ch2jdi1} is mean-square stable if 
 \begin{eqnarray*}
 \alpha :=2\mu+\sigma+\lambda\sqrt{\gamma}(\sqrt{\gamma}+2)<0.
 \end{eqnarray*}
  \begin{rem}
 In this section, to easy notations, the tamed Euler  appoximation $X^M$ will be replaced by $X$ and the semi-tamed Euler approximation $Y^M$ will be replaced by $Y$.
 \end{rem}
 Following the literature of \cite{Xia2}, in order to examine the mean-square stability of the numerical solution given by \eqref{ch5semitam}, we assume that $f(0)=g(0)=h(0)=0$. Also we make the following assumptions.

  \begin{assumption}\label{ch5assumption2}
 There exist a positive constants $\rho$, $\beta$, $\theta$, $K$, $C$ and $a>1$  such that
 \begin{eqnarray*}
 \langle x-y, u(x)-u(y)\rangle\leq -\rho||x-y||^2,\hspace{1cm} ||u(x)-u(y)||\leq K||x-y||,\\
 \langle x-y, v(x)-v(y)\rangle\leq-\beta||x-y||^{a+1}, \hspace{1.5cm} ||v(x)||\leq \overline{\beta}||x||^a,\\
 ||g(x)-g(y)||\leq\theta||x-y||,\hspace{2cm} ||h(x)-h(y)||\leq C||x-y|.
 \end{eqnarray*}
 \end{assumption}
 
 We define $\alpha_1= -2\rho+\theta^2+\lambda C(2+C)$.

 \begin{thm}
 Under Assumptions \ref{ch5assumption2} and the further hypothesis $2\beta-\overline{\beta}>0$, for any stepsize $\Delta t<\dfrac{-\alpha_1}{(K+\lambda C)^2}\wedge\dfrac{2\beta}{[2(K+\lambda C)+\overline{\beta}]\overline{\beta}}\wedge\dfrac{2\beta-\overline{\beta}}{2(K+\lambda C)\overline{\beta}}$, the numerical solution \eqref{ch5semitam} is exponentiallly mean-square stable.
 \end{thm}
 
 \begin{proof}
 The numerical solution \eqref{ch5semitam} is given by 
 \begin{eqnarray*}
 Y_{n+1}=Y_n+\Delta tu_{\lambda}(Y_n)+\dfrac{\Delta tv(Y_n)}{1+\Delta t||v(Y_n)||}+g(Y_n)\Delta W_n+h(Y_n)\Delta\overline{N}_n,
 \end{eqnarray*}
 where $u_{\lambda}=u+\lambda h$.
 
 Taking the inner product in both sides  of the previous equation leads to 
 \begin{eqnarray}
 ||Y_{n+1}||^2&=&||Y_n||^2+\Delta t^2||u_{\lambda}(Y_n)||^2+\dfrac{\Delta t^2||v(Y_n)||^2}{\left(1+\Delta t||v(Y_n)||\right)^2}+||g(Y_n)||^2||\Delta W_n||^2+||h(Y_n)||^2|\Delta\overline{N}_n|^2\nonumber\\
 &+&2\Delta\langle Y_n, u_{\lambda}(Y_n)\rangle+2\Delta t\left\langle Y_n, \dfrac{v(Y_n)}{1+\Delta t||v(Y_n)||}\right\rangle+2\langle Y_n, g(Y_n)\Delta W_n\rangle+2\langle Y_n, h(Y_n)\Delta\overline{N}_n\rangle\nonumber\\
 &+&+2\Delta t^2\left\langle u_{\lambda}(Y_n), \dfrac{v(Y_n)}{1+\Delta t||v(Y_n)||}\right\rangle+2\Delta t\langle u_{\lambda}(Y_n), g(Y_n)\Delta W_n\rangle+2\Delta t\langle u_{\lambda}(Y_n), h(Y_n)\Delta\overline{N}_n\rangle\nonumber\\
 &+&2\Delta t\left\langle\dfrac{v(Y_n)}{1+\Delta t||v(Y_n)||}, g(Y_n)\Delta W_n\right\rangle+2\Delta t\left\langle\dfrac{v(Y_n)}{1+\Delta t||v(Y_n)||}, h(Y_n)\Delta\overline{N}_n\right\rangle\nonumber\\
 &+&2\langle g(Y_n)\Delta W_n, h(Y_n)\Delta\overline{N}_n\rangle.
 \label{ch5meansemi1}
 \end{eqnarray}
 Using Assumptions \ref{ch5assumption2}, it follows that :
 \begin{eqnarray}
 2\Delta t\left\langle Y_n, \dfrac{v(Y_n)}{1+\Delta ||v(Y_n)||}\right\rangle\leq\dfrac{-2\beta\Delta t||Y_n||^{a+1}}{1+\Delta t||v(Y_n)||}
 \label{ch5meansemi2}
 \end{eqnarray}
 \begin{eqnarray}
 2\Delta t^2\left<u_{\lambda}(Y_n), \dfrac{v(Y_n)}{1+\Delta t||v(Y_n)||}\right>&\leq &\dfrac{2\Delta t^2||u_{\lambda}(Y_n)||||v(Y_n)||}{1+\Delta t||v(Y_n||}\nonumber\\
 &\leq&\dfrac{2\Delta t^2(K+\lambda C)\overline{\beta}||Y_n|^{a+1}}{1+\Delta t||v(Y_n)||}.
 \label{ch5meansemi3}
 \end{eqnarray}
 Let's define $\Omega_n=\{\omega\in\Omega : ||Y_n||>1\}$. 
 \begin{itemize}
\item  On $\Omega_n$ we have
 \begin{eqnarray}
 \dfrac{\Delta t^2||v(Y_n)||^2}{\left(1+\Delta t||v(Y_n)||\right)^2}\leq\dfrac{\Delta t||v(Y_n)||}{1+\Delta t||v(Y_n)||}\leq\dfrac{\overline{\beta}\Delta t||Y_n||^{a+1}}{1+\Delta t||v(Y_n)||}.
 \label{ch5meansemi4}
 \end{eqnarray}
 Therefore using \eqref{ch5meansemi2}, \eqref{ch5meansemi3} and \eqref{ch5meansemi4}, equality \eqref{ch5meansemi1} becomes :

\eqref{ch5meansemi1} yields
\begin{eqnarray}
 \|Y_{n+1}\|^2& \leq &\|Y_n\|^2+\Delta t^2 \|u_{\lambda}(Y_n)\|^2+ \|g(Y_n)\|^2\|\Delta W_n\|^2+\|h(Y_n)\|^2|\Delta\overline{N}_n|^2\nonumber\\
 &+&2\Delta t\langle Y_n, u_{\lambda}(Y_n)\rangle+2\langle Y_n, g(Y_n)\Delta W_n
 +2\langle Y_n, h(Y_n)\Delta\overline{N}_n\rangle \nonumber\\
 &+&2\Delta t\langle u_{\lambda}(Y_n), g(Y_n)\Delta W_n\rangle+2\Delta t\langle u_{\lambda}(Y_n), h(Y_n)\Delta\overline{N}_n\rangle\nonumber\\
 &+&2\Delta t\left\langle\dfrac{v(Y_n)}{1+\Delta t\|v(Y_n)\|}, g(Y_n)\Delta W_n\right\rangle+2\Delta t\left\langle\dfrac{v(Y_n)}{1+\Delta t\|v(Y_n)\|}, h(Y_n)\Delta\overline{N}_n\right\rangle\nonumber\\
 &+&2\langle g(Y_n)\Delta W_n, h(Y_n)\Delta\overline{N}_n\rangle+\dfrac{\left[-2\beta\Delta t+2(K+\lambda c)\overline{\beta}\Delta t^2+\overline{\beta}\Delta t\right]\|Y_n\|^{a+1}}{1+\Delta t\|v(Y_n)\|}.
 \label{ch5meansemi4a}
 \end{eqnarray}
 The hypothesis $\Delta t<\dfrac{2\beta-\overline{\beta}}{2(K+\lambda C)\overline{\beta}}$ implies that $-2\beta\Delta t+2(K+\lambda C)\overline{\beta}\Delta t^2+\overline{\beta}\Delta t<0$. Therefore, \eqref{ch5meansemi4a} becomes
 
 \begin{eqnarray}
 ||Y_{n+1}||^2&\leq &||Y_n||^2+\Delta t^2||u_{\lambda}(Y_n)||^2+2\Delta t<Y_n, u_{\lambda}(Y_n)>+||g(Y_n)||^2||\Delta W_n||^2\nonumber\\
 &+&||h(Y_n)||^2||\Delta\overline{N}_n||^2+2\langle Y_n,g(Y_n)\Delta W_n\rangle+2\langle Y_n, h(Y_n)\Delta\overline{N}_n\rangle\nonumber\\
 &+&2\Delta t\langle u_{\lambda}(Y_n), g(Y_n)\Delta W_n\rangle+2\Delta t\langle u_{\lambda}(Y_n), h(Y_n)\Delta\overline{N}_n\rangle\nonumber\\
 &+&2\Delta t\left\langle \dfrac{v(Y_n)}{1+\Delta t||v(Y_n)||}, g(Y_n)\Delta W_n\right\rangle+2\Delta t\left\langle\dfrac{v(Y_n)}{1+\Delta t||v(Y_n)||}, h(Y_n)\Delta\overline{N}_n\right\rangle\nonumber\\
 &+&2\langle g(Y_n)\Delta W_n, h(Y_n)\Delta\overline{N}_n\rangle.
 \label{ch5meansemi4b}
 \end{eqnarray}

 \item On $\Omega_n^c$ we have 
 \begin{eqnarray}
 \dfrac{\Delta t^2||v(Y_n)||^2}{\left(1+\Delta t||v(Y_n)||\right)^2}\leq\dfrac{\Delta t^2||v(Y_n)||^2}{1+\Delta t||v(Y_n)||}\leq\dfrac{\overline{\beta}^2\Delta t^2||Y_n||^{2a}}{1+\Delta t||v(Y_n)||}\leq\dfrac{\overline{\beta}^2\Delta t^2||Y_n||^{a+1}}{1+\Delta t||v(Y_n)||}.
 \label{ch5meansemi5}
 \end{eqnarray}
 Therefore, using \eqref{ch5meansemi2}, \eqref{ch5meansemi3} and \eqref{ch5meansemi5}, equality \eqref{ch5meansemi1} becomes 
 
  \begin{eqnarray}
 ||Y_{n+1}||^2&\leq &||Y_n||^2+\Delta t^2||u_{\lambda}(Y_n)||^2+2\Delta t<Y_n, u_{\lambda}(Y_n)>+||g(Y_n)||^2||\Delta W_n||^2\nonumber\\
 &+&||h(Y_n)||^2||\Delta\overline{N}_n||^2+2\langle Y_n,g(Y_n)\Delta W_n\rangle+2\langle Y_n, h(Y_n)\Delta\overline{N}_n\rangle\nonumber\\
 &+&2\Delta t\langle u_{\lambda}(Y_n), g(Y_n)\Delta W_n\rangle+2\Delta t\langle u_{\lambda}(Y_n), h(Y_n)\Delta\overline{N}_n\rangle\nonumber\\
 &+&2\Delta t\left\langle \dfrac{v(Y_n)}{1+\Delta t||v(Y_n)||}, g(Y_n)\Delta W_n\right\rangle+2\Delta t\left\langle\dfrac{v(Y_n)}{1+\Delta t||v(Y_n)||}, h(Y_n)\Delta\overline{N}_n\right\rangle\nonumber\\
 &+&2\langle g(Y_n)\Delta W_n, h(Y_n)\Delta\overline{N}_n\rangle+\dfrac{\left[-2\beta\Delta t+2(K+\lambda c)\overline{\beta}\Delta t^2+\overline{\beta}^2\Delta t^2\right]||Y_n||^{a+1}}{1+\Delta t t||v(Y_n)||}.
 \label{ch5meansemi5a}
 \end{eqnarray}
 The hypothesis $\Delta t<\dfrac{2\beta}{[2(K+\lambda C)+\overline{\beta}]\overline{\beta}}$ implies  that $-2\beta\Delta t+2(K+\lambda C)\overline{\beta}\Delta t^2+\overline{\beta}^2\Delta t^2<0$. Therefore, \eqref{ch5meansemi5a} becomes 
 
 \begin{eqnarray}
 ||Y_{n+1}||^2&\leq &||Y_n||^2+\Delta t^2||u_{\lambda}(Y_n)||^2+2\Delta t<Y_n, u_{\lambda}(Y_n)>+||g(Y_n)||^2||\Delta W_n||^2\nonumber\\
 &+&||h(Y_n)||^2||\Delta\overline{N}_n||^2+2\langle Y_n,g(Y_n)\Delta W_n\rangle+2\langle Y_n, h(Y_n)\Delta\overline{N}_n\rangle\nonumber\\
 &+&2\Delta t\langle u_{\lambda}(Y_n), g(Y_n)\Delta W_n\rangle+2\Delta t\langle u_{\lambda}(Y_n), h(Y_n)\Delta\overline{N}_n\rangle\nonumber\\
 &+&2\Delta t\left\langle \dfrac{v(Y_n)}{1+\Delta t||v(Y_n)||}, g(Y_n)\Delta W_n\right\rangle+2\Delta t\left\langle\dfrac{v(Y_n)}{1+\Delta t||v(Y_n)||}, h(Y_n)\Delta\overline{N}_n\right\rangle\nonumber\\
 &+&2\langle g(Y_n)\Delta W_n, h(Y_n)\Delta\overline{N}_n\rangle.
 \label{ch5meansemi5b}
 \end{eqnarray}
 \end{itemize}
  Finally, from the discussion above on $\Omega_n$ and $\Omega_n^c$, it follows that on $\Omega$, the following inequality holds for all $\Delta t<\dfrac{2\beta}{[2(K+\lambda C)+\overline{\beta}]\overline{\beta}}\wedge\dfrac{2\beta-\overline{\beta}}{2(K+\lambda C)\overline{\beta}}$
 \begin{eqnarray}
 ||Y_{n+1}||^2&\leq &||Y_n||^2+\Delta t^2||u_{\lambda}(Y_n)||^2+2\Delta t<Y_n, u_{\lambda}(Y_n)>+||g(Y_n)||^2||\Delta W_n||^2\nonumber\\
 &+&||h(Y_n)||^2||\Delta\overline{N}_n||^2+2\langle Y_n,g(Y_n)\Delta W_n\rangle+2\langle Y_n, h(Y_n)\Delta\overline{N}_n\rangle\nonumber\\
 &+&2\Delta t\langle u_{\lambda}(Y_n), g(Y_n)\Delta W_n\rangle+2\Delta t\langle u_{\lambda}(Y_n), h(Y_n)\Delta\overline{N}_n\rangle\nonumber\\
 &+&2\Delta t\left\langle \dfrac{v(Y_n)}{1+\Delta t||v(Y_n)||}, g(Y_n)\Delta W_n\right\rangle+2\Delta t\left\langle\dfrac{v(Y_n)}{1+\Delta t||v(Y_n)||}, h(Y_n)\Delta\overline{N}_n\right\rangle\nonumber\\
 &+&2\langle g(Y_n)\Delta W_n, h(Y_n)\Delta\overline{N}_n\rangle.
 \label{ch5meansemi6}
 \end{eqnarray}
 Taking expectation in both sides of \eqref{ch5meansemi6} and using the martingale properties of $\Delta W_n$ and $\Delta\overline{N}_n$ leads to :
 \begin{eqnarray}
 \mathbb{E}||Y_{n+1}||^2&\leq&\mathbb{E}||Y_n||^2+\Delta t^2\mathbb{E}||u_{\lambda}(Y_n)||^2+2\Delta t\mathbb{E}\langle Y_n, u_{\lambda}(Y_n)\rangle+\Delta t\mathbb{E}||g(Y_n)||^2\nonumber\\
 &+&\lambda\Delta t\mathbb{E}||h(Y_n)||^2.
 \label{ch5meansemi7}
 \end{eqnarray}
 From Assumptions \ref{ch5assumption2}, we have 
 \begin{eqnarray*}
 ||u_{\lambda}(Y_n)||^2\leq (K+\lambda C)^2||Y_n||^2 \hspace{0.5cm} \text{and} \hspace{0.5cm}\langle Y_n, u_{\lambda}(Y_n)\rangle\leq (-\rho+\lambda C)||Y_n||^2.
 \end{eqnarray*}
 So inequality \eqref{ch5meansemi7} gives
 \begin{eqnarray*}
 \mathbb{E}||Y_{n+1}||^2&\leq&\mathbb{E}||Y_n||^2+(K+\lambda C)^2\Delta t^2\mathbb{E}||Y_n||^2+2(-\rho+\lambda C)\Delta t\mathbb{E}||Y_n||^2+\theta^2\Delta t\mathbb{E}||Y_n||^2\\
 &+&\lambda C^2\Delta t\mathbb{E}||Y_n||^2\\
 &=&\left[1-2\rho\Delta t+(K+\lambda C)^2\Delta t^2+2\lambda C\Delta t+\theta^2\Delta t+\lambda C^2\Delta t\right]\mathbb{E}||Y_n||^2.
 \end{eqnarray*}
 Iterating the previous inequality leads to 
 \begin{eqnarray*}
 \mathbb{E}||Y_n||^2\leq\left[1-2\rho\Delta t+(K+\lambda C)^2\Delta t^2+2\lambda C\Delta t+\theta^2\Delta t+\lambda C^2\Delta t\right]^n\mathbb{E}||Y_0||^2.
 \end{eqnarray*}
 In oder to have stability, we impose :
 \begin{eqnarray*}
 1-2\rho\Delta t+(K+\lambda C)^2\Delta t^2+2\lambda C\Delta t+\theta^2\Delta t+\lambda C^2\Delta t<1.
 \end{eqnarray*}
 That is 
 \begin{eqnarray}
 \Delta t<\dfrac{-[-2\rho+\theta^2+\lambda C(2+C)]}{(K+\lambda C)^2} =\dfrac{-\alpha_1}{(K+\lambda C)^2}.
 \end{eqnarray}
 \end{proof}

 In the following , we analyse the mean-square stability of the tamed Euler. We make the following assumptions  which are essentially consequences of Assumptions \ref{ch5assumption2}.
 \begin{assumption}\label{ch5assumption3}
 we assume that there exists positive constants $\beta$, $\overline{\beta}$,  $\theta$, $\mu$, $K$, $C$, $\rho$,  and $a>1$ such that :
 \begin{align}
 \langle x-y, f(x)-f(y)\rangle\leq &-\rho ||x-y||^2-\beta||x-y||^{a+1},\nonumber\\
 ||f(x)|| \leq \overline{\beta}||x||^a+K||x||,\nonumber\\
 ||g(x)-g(y)|| \leq &\theta||x-y||,\nonumber\\
 ||h(x)-h(y)|| \leq &C||x-y||,\nonumber\\
 \langle x-y, h(x)-h(y)\rangle \leq &-\mu ||x-y||^2.
 \label{assumparticular}
 \end{align}
 
 \end{assumption}
 \begin{rem}
  Apart from \eqref{assumparticular}, Assumption \ref{ch5assumption3} is a consequence of Assumption \ref{ch5assumption2}.
 \end{rem}
 
 \begin{thm}
 Under  Assumptions \ref{ch5assumption3} and the further hypothesis
 $\beta-C\overline{\beta}>0$, $\overline{\beta}(1+2C)-2\beta<0$, $K+\theta^2+\lambda C^2-2\mu\lambda+2\lambda CK<0$, the numerical solution \eqref{ch5tam} is mean-square stable for any stepsize
 \begin{eqnarray*}
\Delta t<\dfrac{-[K+\theta^2+\lambda C^2-2\mu\lambda+2\lambda CK]}{2K^2+\lambda^2C^2}\wedge\dfrac{\beta-C\overline{\beta}}{\overline{\beta}^2}.
 \end{eqnarray*}
 \end{thm}
 
 \begin{proof}
 From equation \eqref{ch4tam}, we have 
 \begin{eqnarray}
 ||X_{n+1}||^2&=&||X_n||^2+\dfrac{\Delta t^2||f(X_n)||^2}{\left(1+\Delta t||f(X_n)||\right)^2}+||g(X_n)\Delta W_n||^2+||h(X_n)\Delta N_n||^2\nonumber\\
 &+&\left\langle X_n,\dfrac{\Delta tf(X_n)}{1+\Delta t||f(X_n)||}\right\rangle+2\left\langle X_n+\dfrac{\Delta tf(X_n)}{1+\Delta t||f(X_n)||}, g(X_n)\Delta W_n\right\rangle\nonumber\\
 &+&2\left\langle X_n+\dfrac{\Delta tf(X_n)}{1+\Delta t||f(X_n)||}, h(X_n)\Delta N_n\right\rangle+2\langle g(X_n)\Delta W_n, h(X_n)\Delta N_n\rangle.
 \label{ch5meantamed1}
 \end{eqnarray}
 Using  assumptions \ref{ch5assumption3}, it follows that : 
 \begin{eqnarray*}
 2\left\langle X_n, \dfrac{\Delta tf(X_n)}{1+\Delta t||f(X_n)||}\right\rangle &\leq &\dfrac{-2\Delta t\rho||X_n||^2}{1+\Delta t||f(X_n)||}-\dfrac{2\beta \Delta t||X_n||^{a+1}}{1+\Delta t||f(X_n)||}
 \leq -\dfrac{2\beta \Delta t||X_n||^{a+1}}{1+\Delta t||f(X_n)||}.
 \label{ch5meantamed2}
 \end{eqnarray*}
 
 \begin{eqnarray*}
 ||g(X_n)\Delta W_n||^2 \leq \theta^2 ||X_n||^2||\Delta W_n||^2\hspace{0.5cm}\text{and}\hspace{0.5cm}
 ||h(X_n)\Delta N_n||^2 \leq C^2||X_n||^2|\Delta N_n|^2.\label{ch5meantamed2a}
 \end{eqnarray*}
 
 \begin{eqnarray*}
 2\langle X_n, h(X_n)\Delta N_n\rangle =2\langle X_n, h(X_n)\rangle\Delta N_n \leq -2\mu||X_n||^2|\Delta N_n|.
 \label{ch5meantamed3}
 \end{eqnarray*}
 \begin{eqnarray*}
 2\left\langle\dfrac{\Delta tf(X_n)}{1+h||f(X_n)||}, h(X_n)\Delta N_n\right\rangle &\leq &\dfrac{2\Delta t||f(X_n)||||h(X_n)|||\Delta N_n|}{1+\Delta t||f(X_n)||}\nonumber\\
 &\leq &\dfrac{2\Delta tC\overline{\beta}||X_n||^{a+1}}{1+\Delta t||f(X_n)||}|\Delta N_n|+2CK||X_n||^2|\Delta N_n|.
 \label{ch5meantamed4}
 \end{eqnarray*}
 So from Assumptions \ref{ch5assumption3}, we have
 \begin{eqnarray}
\label{ch5assmeantamed} \left\{
 \begin{array}{lllll}
 \left\langle X_n, \dfrac{\Delta tf(X_n)}{1+\Delta t||f(X_n)||}\right\rangle\leq -\dfrac{2\beta \Delta t||X_n||^{a+1}}{1+\Delta t||f(X_n)||}\\
 
 ||g(X_n)\Delta W_n||^2 \leq \theta^2 ||X_n||^2||\Delta W_n||^2\\
 
 ||h(X_n)\Delta N_n||^2 \leq C^2||X_n||^2|\Delta N_n|^2\\
 
  2\langle Y_n, h(X_n)\Delta N_n\rangle \leq -2\mu||X_n||^2|\Delta N_n|\\
  
   2\left\langle\dfrac{\Delta tf(X_n)}{1+h||f(X_n)||}, h(X_n)\Delta N_n\right\rangle \leq\dfrac{2\Delta tC\overline{\beta}||X_n||^{a+1}}{1+\Delta t||f(X_n)||}|\Delta N_n|+2CK||X_n||^2|\Delta N_n|

 \end{array}
 \right.
 \end{eqnarray}

 Let's define $\Omega_n :=\{w\in\Omega : ||X_n(\omega)||>1\}$.
 \begin{itemize}
 \item
 On $\Omega_n$ we have : 
 \begin{eqnarray}
 \dfrac{\Delta t^2||f(X_n)||^2}{\left(1+\Delta t||f(X_n)||\right)^2}&\leq& \dfrac{\Delta t||f(X_n)||}{1+\Delta t||f(X_n)||}
 \leq \dfrac{\Delta t\overline{\beta}||X_n||^a}{1+\Delta t||f(X_n)||}+K\Delta t||X_n||\nonumber\\
 &\leq& \dfrac{\Delta t\overline{\beta}||X_n||^{a+1}}{1+\Delta t||f(X_n)||}+K\Delta t||X_n||^2.
 \label{ch5meantamed5}
 \end{eqnarray} 
 Therefore using \eqref{ch5assmeantamed} and \eqref{ch5meantamed5}, equality \eqref{ch5meantamed1} becomes 
 \begin{eqnarray}
 ||X_{n+1}||^2&\leq&||X_n||^2+K\Delta t||X_n||^2+\theta^2||X_n||^2||\Delta W_n||^2+C^2||X_n||^2|\Delta N_n|^2\nonumber\\
 &+&2\left\langle X_n+\dfrac{\Delta tf(X_n)}{1+\Delta t||f(X_n)||}, g(X_n)\Delta W_n\right\rangle-2\mu||X_n||^2|\Delta N_n|+2CK|\Delta N_n|\nonumber\\
 &+&2\langle g(X_n)\Delta W_n, h(X_n)\Delta N_n\rangle+\dfrac{\left[-2\beta\Delta t+\overline{\beta}\Delta t+2\overline{\beta}C\Delta t\right]||X_n||^{a+1}}{1+\Delta t||f(X_n)||}.
 \label{ch5meantamed6}
 \end{eqnarray}
 Using the hypothesis  $\overline{\beta}(1+2C)-2\beta<0$, \eqref{ch5meantamed6} becomes 
 \begin{eqnarray}
 ||X_{n+1}||^2&\leq&||X_n||^2+K\Delta t||X_n||^2+\theta^2||X_n||^2||\Delta W_n||^2+C^2||X_n||^2|\Delta N_n|^2\nonumber\\
 &+&2\left\langle X_n+\dfrac{\Delta tf(X_n)}{1+\Delta t||f(X_n)||}, g(X_n)\Delta W_n\right\rangle-2\mu||X_n||^2|\Delta N_n|+2CK|\Delta N_n|\nonumber\\
 &+&2\langle g(X_n)\Delta W_n, h(X_n)\Delta N_n\rangle.
 \label{ch5meantamed7}
 \end{eqnarray}

  \item On $\Omega_n^c$, we have :
 \begin{eqnarray}
 \dfrac{\Delta t^2||f(X_n)||^2}{\left(1+\Delta t||f(X_n)||\right)^2}&\leq& \dfrac{\Delta t^2||f(X_n)||^2}{1+\Delta t||f(X_n)||}
 \leq \dfrac{2\Delta t^2\overline{\beta}^2||X_n||^{2a}}{1+\Delta t||f(X_n)||}+2K^2\Delta t^2||X_n||^2\\
 &\leq& \dfrac{2\Delta t^2\overline{\beta}^2||X_n||^{a+1}}{1+\Delta t||f(X_n)||}+2K^2\Delta t^2||X_n||^2.
 \label{ch5meantamed8}
 \end{eqnarray}
 Therefore, using \eqref{ch5assmeantamed},  and \eqref{ch5meantamed8}, \eqref{ch5meantamed1} becomes
 \begin{eqnarray}
 ||X_{n+1}||^2&\leq&||X_n||^2+2K^2\Delta t^2||X_n||^2+\theta^2||X_n||^2||\Delta W_n||^2+C^2||X_n||^2|\Delta N_n|^2\nonumber\\
 &+&2\left\langle X_n+\dfrac{\Delta tf(X_n)}{1+\Delta t||f(X_n)||}, g(X_n)\Delta W_n\right\rangle-2\mu||Y_n||^2|\Delta N_n|+2CK|\Delta N_n|\nonumber\\
 &+&2\langle g(X_n)\Delta W_n, h(X_n)\Delta N_n\rangle+\dfrac{\left[2C\overline{\beta}\Delta t-2\beta\Delta t+2\overline{\beta}^2\Delta t^2\right]||X_n||^{a+1}}{1+\Delta t||f(X_n)||}.
 \label{ch5meantamed9}
 \end{eqnarray}
 The hypothesis $\Delta t<\dfrac{\beta-C\overline{\beta}}{\overline{\beta}^2}$ implies that $2C\overline{\beta}\Delta t-2\beta\Delta t+2\overline{\beta}^2\Delta t^2<0$. Therefore, \eqref{ch5meantamed9} becomes 
 
 \begin{eqnarray}
 ||X_{n+1}||^2&\leq&||X_n||^2+2K^2\Delta t^2||X_n||^2+\theta^2||X_n||^2||\Delta W_n||^2+C^2||X_n||^2|\Delta N_n|^2\nonumber\\
 &+&2\left\langle X_n+\dfrac{\Delta tf(X_n)}{1+\Delta t||f(X_n)||}, g(X_n)\Delta W_n\right\rangle-2\mu||X_n||^2|\Delta N_n|+2CK|\Delta N_n|\nonumber\\
 &+&2\langle g(X_n)\Delta W_n, h(X_n)\Delta N_n\rangle.
 \label{ch5meantamed10}
 \end{eqnarray}
 
 \end{itemize}
 
From the above discussion on $\Omega_n$ and $\Omega_n^c$, the following inequality holds on $\Omega$ for all $\Delta t<\dfrac{\beta-C\overline{\beta}}{\overline{\beta}^2}$ and $\overline{\beta}(1+2C)-\beta<0$
 \begin{eqnarray}
 ||X_{n+1}||^2&\leq&||X_n||^2+K\Delta t||X_n||^2+2K^2\Delta t^2||X_n||^2+\theta^2||X_n||^2||\Delta W_n||^2+C^2||X_n||^2|\Delta N_n|^2\nonumber\\
 &+&2\left\langle X_n+\dfrac{\Delta tf(X_n)}{1+\Delta t||f(X_n)||}, g(X_n)\Delta W_n\right\rangle-2\mu||X_n||^2|\Delta N_n|+2CK|\Delta N_n|\nonumber\\
 &+&2\langle g(X_n)\Delta W_n, h(X_n)\Delta N_n\rangle.
 \label{ch5meantamed11}
 \end{eqnarray}
 
 Taking Expectation in both sides of \eqref{ch5meantamed11},  using the relation $\mathbb{E}||\Delta W_n||=0$, $\mathbb{E}||\Delta W_n||^2=\Delta t$, $\mathbb{E}|\Delta N_n|=\lambda\Delta t$ and $\mathbb{E}|\Delta N_n|^2=\lambda^2\Delta t^2+\lambda\Delta t$ leads to :
 \begin{eqnarray*}
 \mathbb{E}||X_{n+1}||^2&\leq& \mathbb{E}||X_n||^2+2K^2\Delta t^2\mathbb{E}||X_n||^2+\theta^2\Delta t\mathbb{E}||X_n||^2+\lambda^2C^2\Delta t^2\mathbb{E}||X_n||^2+\lambda C^2\Delta t\mathbb{E}||X_n||^2\\
 &-&2\mu\lambda\Delta t\mathbb{E}||X_n||^2+\lambda CK\Delta t\mathbb{E}||X_n||^2\\
 &=&\left[1+(2K^2+\lambda^2C^2)\Delta t^2+(K+\theta^2+\lambda C^2-2\mu\lambda+2\lambda CK)\Delta t\right]\mathbb{E}||X_n||^2.
 \end{eqnarray*}
 Iterating the last inequality leads to 
 \begin{eqnarray*}
 \mathbb{E}||X_n||^2\leq \left[1+(2K^2+\lambda^2C^2)\Delta t^2+(K+\theta^2+\lambda C^2-2\mu\lambda+2\lambda CK)\Delta t\right]^n\mathbb{E}||X_0||^2.
 \end{eqnarray*}
 In order to have stability, we impose
 \begin{eqnarray*}
 1+(2K^2+\lambda^2C^2)\Delta t^2+(K+\theta^2+\lambda C^2-2\mu\lambda+2\lambda CK)\Delta t<1.
 \end{eqnarray*}
 That is 
 \begin{eqnarray*}
 \Delta t<\dfrac{-[K+\theta^2+\lambda C^2-2\mu\lambda+2\lambda CK]}{2K^2+\lambda^2C^2}.
 \end{eqnarray*}
 \end{proof}

 \section{Numerical Experiments}
 \hspace{0.5cm}In this section, we present some numerical experiments that illustrate our theorical strong convergence and stability results. For the strong convergence illustration of \eqref{ch5tam} and \eqref{ch5semitam}, let's  consider the  stochastic differential equation
 \begin{eqnarray}
 dX_t=(-4X_t-X^3_t)dt+X_tdW_t+X_tdN_t,
 \label{ch5numeric1}
 \end{eqnarray}
 with initial $X_0=1$. $N$ is the scalar poisson process with parameter $\lambda=1$. Here $u(x)=-4x$, $v(x)=-x^3$ $g(x)=h(x)=x$. It is easy to check that $u, v, g$ and $h$ satisfy the Assumptions \ref{ch5assumption1}. 
 
 For the illustration of the  linear mean-square stability , we consider a linear test equation 

$\left\{\begin{array}{ll}
dX(t)=aX(t^-)+bX(t^-)dW(t)+cX(t^-)dN(t)\\
X(0)=1
\end{array}
\right.$

We consider the particular case $a=-1$, $b=2$, $c=-0.9$ and $\lambda=9$. In this case $l=-0.91$, $\dfrac{-l}{(a+\lambda c)^2}<0.084$ and $\dfrac{2a-l}{a^2+\lambda c^2}<0.074$. $a(1+\lambda c\Delta t)<0$ for $\Delta t<0.124$.
We test the stability behaviour of semi-tamed and of tamed Euler for different step-size, $\Delta t=0.02, 0.05$ and $0.08$. We use $7\times 10^3$ sample paths. For all step-size 
$\Delta t<0.083$ the semi-tamed Euler is stable. But for the step-size $\Delta t=0.08>0.074$, the tamed Euler scheme is unstable while the semi-tamed Euler scheme is stable. So the semi-tamed Euler scheme works better than the tamed Euler scheme.
 
\begin{figure}[hbtp]
\includegraphics[scale=0.7]{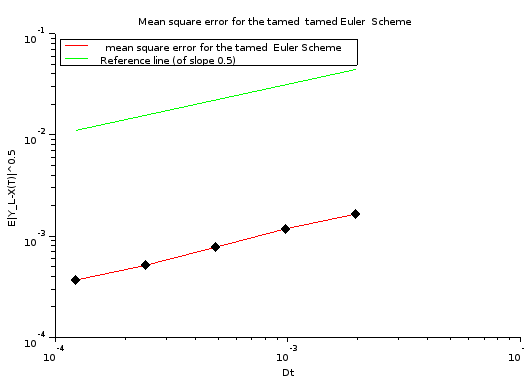}
\caption{Error of the tamed Euler scheme}
\end{figure}

\begin{figure}[hbtp]
\includegraphics[scale=0.7]{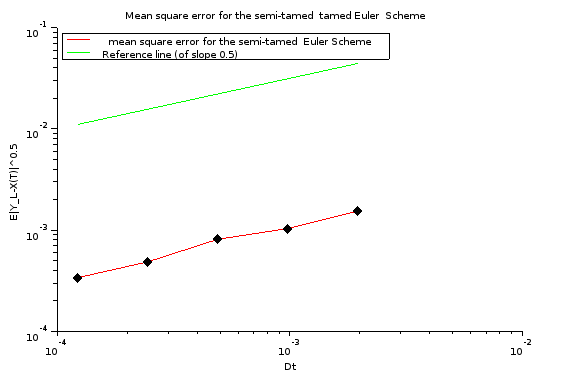}
\caption{Error of the semi-tamed Euler scheme}
\end{figure}

\begin{figure}
\includegraphics[scale=0.65]{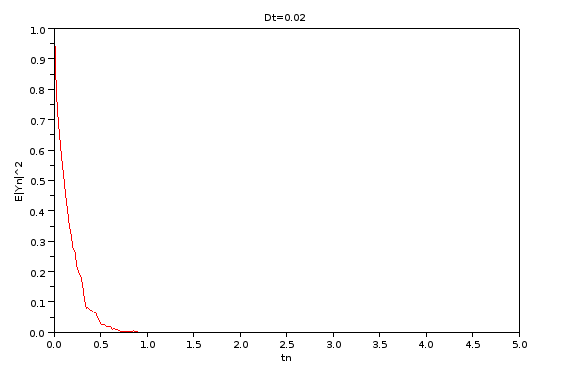}
\caption{Stability tamed Euler}
\end{figure}

\begin{figure}
\includegraphics[scale=0.65]{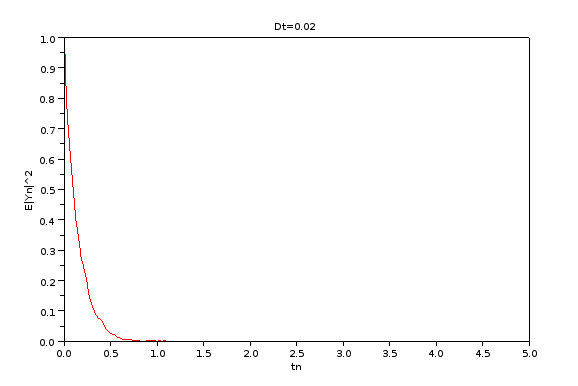} 
\caption{Stability semi-tamed Euler}
\end{figure}

\begin{figure}
\includegraphics[scale=0.65]{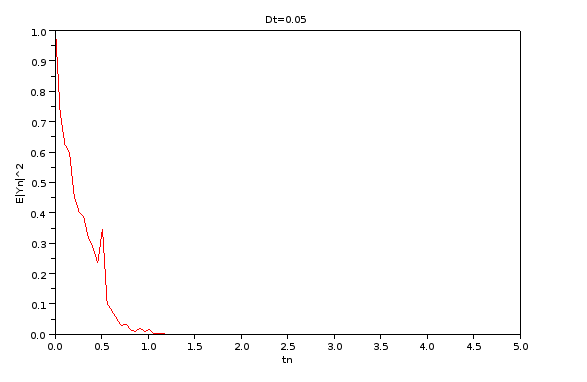}
\caption{Stability tamed Euler}
\end{figure}

\begin{figure}
\includegraphics[scale=0.65]{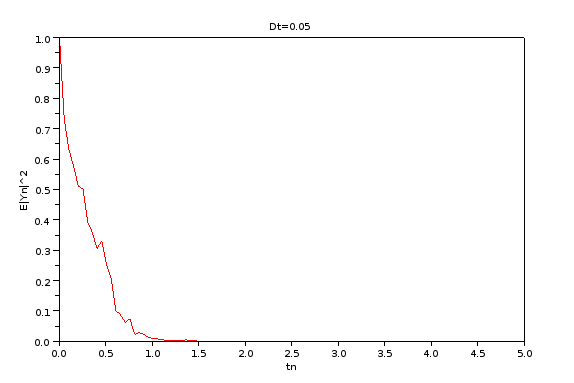}
\caption{Stability semi-tamed Euler }
\end{figure}

\begin{figure}
\includegraphics[scale=0.7]{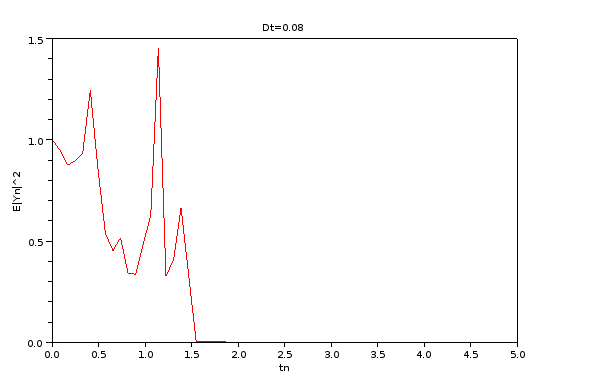}
\caption{ Stability tamed Euler }
\end{figure}

\begin{figure}
\includegraphics[scale=0.7]{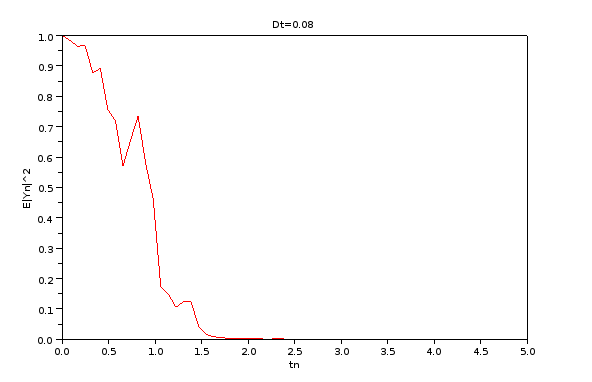}
\caption{Stability semi-tamed Euler }
\end{figure}

\chapter*{Conclusion}\addcontentsline{toc}{chapter}{Conclusion}
\hspace{0.5cm} In this thesis, we provided an overview in probability theory which allowed us to define some basic concepts in stochastic process. Under global Lipschitz condition, we proved the existence and the uniqueness of solution of stochastic differential equation (SDEs) with jumps. In  general, it is difficult to find the exact solution of most SDEs. one tool to approach the exact solution  is the numerical resolution. We provided in this dissertation some numerical techniques to solve SDEs with jumps. More precisely, we investigated the strong convergence of the compensated stochastic theta method (CSTM) under global Lipschitz condition. We investigated the stability of both CSTM and stochastic theta method (STM) and we  proved that for the linear test equation, when  $\dfrac{1}{2}\leq\theta\leq 1$  the CSTM is A-stable while the STM is not. So  CSTM works better than  STM.

\hspace{0.5cm} Under non-global Lipschitz condition Euler explicit method fails to converge strongly while Euler implicit method converges strongly but requires much computational efforts. We extended the tamed Euler scheme by introducing the compensated tamed Euler scheme for SDEs with jumps, which converges strongly with standard order $0.5$ to the exact solution. We also extended the semi-tamed Euler scheme proposed in \cite{Xia2} for SDEs with jumps under non-global Lipschitz condition and proved his strong convergence.  This latter enable us define the tamed Euler scheme and to prove his strong convergence which was not yet done in the literature. In this thesis, we also analysis the stability behaviours of both tamed and semi-tamed Euler schemes in the linear and the nonlinear case. We proved that these two numerical schemes reproduce the exponentially mean-square property of the exact solution. 

\hspace{0.5cm} All the numerical scheme presented in this work are of rate of convergence $0.5$. The tamed Misltein scheme was introduced in \cite{Xia3}, where the authors proved the strong convergence of order $1$ of this scheme for SDEs without jumps. The case with jumps is not yet well developped in the litterature. The weak convergence under non-global Lipschitz condition is not yet investigated. In the future, We would like to focus on the   opened topics mentioned  above.

\chapter*{Appendix}\addcontentsline{toc}{chapter}{Appendix}
The goal of this section is to present some Scilab codes for simulations.
\section*{A.1 Code for simulation of the mean square error}
\begin{verbatim}
lambda = 1;  Xzero = 1;
T = 1; N = 2^(14); dt = T/N; a=1, b=1, c=0.5, theta=1, A=a+lambda*c
M = 5000;
Xerr = zeros(M,5);
for s = 1:M, 
    dW = sqrt(dt)*grand(1,N,'nor',0,1)
    W = cumsum(dW);
    dN=grand(1,N,'poi',dt*lambda)-dt*lambda*ones(1,N);
    W(1)=0
    dP=dN+dt*lambda*ones(1,N);
    P=cumsum(dP,'c');
    P(1)=0;
    X=linspace(0,1,N);
    Xtrue=exp((a-1/2*b^2)*X+b*W+log(1+c)*P);
    for p = 1:5
        R = 2^(p-1); Dt = R*dt; L = N/R;
        Xtemp = Xzero;
            for j = 1:L
            Winc = sum(dW(R*(j-1)+1:R*j));
            Ninc=sum(dN(R*(j-1)+1:R*j));
            Xtemp=(Xtemp+(1-theta)*A*Xtemp*Dt+b*Xtemp*Winc+c*Xtemp*Ninc)/(1-A*theta*Dt);
            end
            Xerr(s,p) = abs(Xtemp - Xtrue(N))^2;
    end
end
Dtvals = dt*(2.^([0:4]));
T=mean(Xerr,'r')^(1/2);
disp(T);
plot2d('ll',Dtvals,T,[5])
plot2d('ll',Dtvals,Dtvals^(1/2),[3])
legends([ 'mean square error for CSTM for theta=1','Reference line'  ],[5,3,2],2)
xtitle("Mean square stability for CSTM");
xlabel("Time")
ylabel("E|Y_L-X(T)|^0.5")
\end{verbatim}

  \section*{A.2 Code for the simulation of the mean square stability}

  \begin{verbatim}  
T=2500, M=5000, Xzero=1,
a=-7,b=1, c=1, lambda=4, Dt=25, 
A=a+lambda*c // drift coefficient for compensated equation
N=T/Dt;
theta=[0.4995, 0.50,0.51];// different value for theta
Xms=zeros(3,N);//initialisation of the mean
for i=1:3
    Xtemp=Xzero*ones(M,1);// initialization of the solution
    for j=1:N
        Winc=grand(M,1,'nor',0,sqrt(Dt)); // generation of random 
        //variables following the normal distribution
        Ninc=grand(M,1,'poi',Dt*lambda)-Dt*lambda*ones(M,1);// generation of
        // compensated poisson process
        B=1-theta(i)*A*Dt
        Xtemp=(Xtemp+(1-theta(i))*A*Xtemp*Dt+b*Xtemp.*Winc+c*Xtemp.*Ninc)/B;
        Xms(i,j)=mean(Xtemp^2);
    end
end
X=linspace(0,T,N);
plot(X,Xms(1,:),'b',X,Xms(2,:),'r',X,Xms(3,:),'g')
legends(['Theta=0.499'  'theta=0.50'  'Theta=0.51'  ], [2,5,3],2) 
xtitle("Mean-square stability for CSTM");
xlabel("tn")
ylabel("E|Yn|^2")
\end{verbatim}

\chapter*{Acknowledgements}\addcontentsline{toc}{chapter}{Acknowledgements}
This dissertation would not have been possible without the guidance and the help of several individuals who in one way or another contributed and extended their valuable assistance in the preparation and completion of this study.

I would like to take this opportunity to express my gratitude to my supervisor Dr. Antoine Tambue to give me the chance to work with him; for his availability, his time to guide my researches, for sincerity, kindness, patience, understanding and encouragement that I will never forget. I am very happy for the way you introduced me in this nice topic which gave us two preprints papers.

I would also like to thank Dr. Mousthapha Sene, my tutor for the time he took to read this work and give a valuable suggestions.

Special thanks to Neil Turok for his wonderful initiative. Thanks to all staff of AIMS-Senegal, tutors and my classmates. I would like to express my gratitude to all lecturers who taught me at AIMS-Senegal, I would like to mention here Pr. Dr. Peter Stollmann. I would like to extend my thanks to the Ph.D. students at the chair of Mathematics of AIMS-senegal for their moral support during my stay in Senegal.

Many thanks to all lecturers in my country (Cameroon) who have being training me since my undergraduate.

Big thanks to my family, to my friends and to everyone who supported me during my studies.

\fussy

\end{document}